\documentclass[10pt]{article} 
\usepackage[utf8]{inputenc}
\usepackage[top=30mm, left=30mm, right=30mm, bottom=30mm]{geometry}
\usepackage{amsmath,amssymb,amsthm}
\usepackage{dsfont}
\usepackage{color}
\usepackage{xcolor}
\usepackage{lineno}
\usepackage[linktoc=page,hidelinks]{hyperref} 
\hypersetup{
	colorlinks,
	linkcolor={red!50!black},
	citecolor={blue!50!black},
	urlcolor={blue!80!black}
}
\usepackage{yfonts}
\usepackage{appendix}
\usepackage{esint} 
\usepackage{babel} 
\usepackage{cancel}
\usepackage{enumerate}
\usepackage{tabu} 
\usepackage{bm} 

\newtheorem{theorem}{Theorem}[section]
\newtheorem{lem}[theorem]{Lemma}

\newtheorem{kor}[theorem]{Corollary}
\newtheorem{prop}[theorem]{Proposition}
\newtheorem{rem}[theorem]{Remark}
\theoremstyle{definition}
\newtheorem{dfn}[theorem]{Definition}

\numberwithin{equation}{section}

\DeclareMathOperator*{\supp}{supp}
\DeclareMathOperator*{\dvg}{div} 
\DeclareMathOperator*{\Tr}{Tr}

\DeclareMathOperator*{\Id}{Id}

\newcommand*{\divv}{\mathrm{div}}

\newcommand*{\Escr}{\mathcal E}
\newcommand*{\Fscr}{\mathcal F}

\newcommand*{\mcR}{\mathcal R}
\newcommand*{\mcQ}{\mathcal Q}
\newcommand*{\N}{\mathbb{N}}

\newcommand*{\R}{\mathbb{R}}
\newcommand*{\C}{\mathbb{C}}

\newcommand*{\Z}{\mathbb{Z}}

\newcommand*{\T}{\mathbb{T}}
\newcommand*{\mft}{\mathfrak{t}}

\newcommand*{\vs}{\varsigma}
\newcommand*{\D}{\Delta}
\newcommand*{\ve}{\varepsilon}

\usepackage{tocloft}
\definecolor{orange}{rgb}{1.0, 0.55, 0.0} 

\newcommand{\footremember}[2]{%
	\footnote{#2}
	\newcounter{#1}
	\setcounter{#1}{\value{footnote}}%
}

	\title{Non-uniqueness of Leray--Hopf solutions for the $3D$ fractional Navier--Stokes equations perturbed by transport noise}
\author{Theresa Lange\footnote{Scuola Normale Superiore Pisa, Piazza dei Cavalieri 7, 56126 Pisa, Italy. E-mail: theresa.lange@sns.it} \and
	Marco Rehmeier\footremember{alley}{Faculty of Mathematics, TU Berlin, 10623 Berlin, Germany. E-mail: rehmeier@tu-berlin.de}
	\and Andre Schenke\footnote{Courant Institute of Mathematical Sciences, New York University, 251 Mercer Street, New York, N.Y. 10012-1185, USA. E-mail: andreschenke90@gmail.com}}
\date{\today}

\begin{document}
	\maketitle

	\begin{abstract}
For the $3D$ fractional Navier--Stokes equations perturbed by transport noise, we prove the existence of infinitely many Hölder continuous analytically weak, probabilistically strong Leray--Hopf solutions starting from the same deterministic initial velocity field. Our solutions are global in time and satisfy the energy inequality pathwise on a non-empty random interval $[0,\tau]$. In contrast to recent related results, we do not consider an additional deterministic suitably chosen force $f$ in the equation. In this unforced regime, we prove the first result of Leray--Hopf nonuniqueness for fractional Navier--Stokes equations with any kind of stochastic perturbation. Our proof relies on convex integration techniques and a flow transformation by which we reformulate the SPDE as a PDE with random coefficients. 
	\end{abstract}


\renewcommand{\baselinestretch}{0.90}\normalsize 
\tableofcontents
\renewcommand{\baselinestretch}{1.0}\normalsize 

\section{Introduction}
In this work we prove local-in-time non-uniqueness for stochastic Leray--Hopf solutions to the $3D$ fractional incompressible Navier--Stokes equations (NSE) with Brownian transport noise
\begin{equation}\label{SPDE}
	\begin{cases}
		du + \big((u\cdot \nabla u) + (-\Delta)^{\alpha}u\big)dt + \nabla p\, dt + \sum_{k\in K} (\sigma_k \cdot \nabla u)\bullet dB^k = 0
		\\
		\divv u = 0
	\end{cases}
\end{equation}
on the periodic torus $\mathbb{T}^3$. 
Here $u$, an $L^2(\T^3,\R^3)$-valued stochastic process, denotes the velocity field, $p$ the scalar pressure, $(-\Delta)^\alpha$ the (nonlocal) fractional Laplace operator with $0 < \alpha <1$, $K$ a finite, non-empty index set, $\sigma_k: \T^3 \to \R^3$ smooth divergence-free vector fields, $B^k$ independent one-dimensional Brownian motions on an arbitrary but fixed filtered probability space $(\Omega, \Fscr, (\Fscr_t)_{t\geq 0}, \mathbb{P})$, and "$\bullet$" the Stratonovich stochastic integral.

\textbf{Main result.}
More precisely, we prove:
For each $0< \alpha \ll 1$, there is a deterministic initial condition $u_0 \in L^2(\T^3,\R^3)$, a strictly positive stopping time $\tau$ and infinitely many analytically weak, probabilistically strong solutions $u_k, k \in \N$, to Equ. \eqref{SPDE} with paths in $C([0,\infty),C^\theta(\T^3,\R^3))$ for small $\theta >0$, and $u_k(0,\cdot) = u_0$ which satisfy the pathwise energy inequality Equ. \eqref{ineq:Leray-u} on $[0,\tau]$. For the precise result, see Theorem \ref{thm:main-result}.

\textbf{State of the art: deterministic regime.}
The problem of existence and uniqueness of solutions for the $3D$ NSE (i.e. \eqref{SPDE} with $\alpha =1$ and $\sigma_k = 0$) poses notoriously difficult challenges. Since the existence of (necessarily unique) global-in-time smooth solutions remains famously unsolved, it is inevitable to turn to the much larger class of weak solutions. However, in \cite{NSE_det_nonU_Annals} (see also \cite{BCV22,BDL23}) it was proven via convex integration techniques (discussed below) that for any smooth $e: [0,\infty) \to [0,\infty)$ there is a weak solution $u = u(e)$ whose kinetic energy profile equals $e$, i.e. $\frac 1 2 ||u(t)||^2_{L^2} = e(t)$ for all $t \geq 0$. This implies non-uniqueness by considering distinct energies with initial value $e(0)=0$. This drastic demonstration of non-uniqueness and the existence of physically anomalous solutions suggests to restrict to subclasses with physically reasonable behavior, called \textit{admissible solutions} and hopefully obtain well-posedness in such subclasses. A natural such class are \emph{Leray--Hopf solutions} $u\in C^0_{\textup{weak}}([0,\infty),L^2(\T^3,\R^3)) \cap L^2([0,\infty),H^1(\T^3,\R^3))$ satisfying the \emph{energy inequality}
\begin{equation}\label{intro:Leray-ineq}
	||u(t)||^2_{L^2} + 2\int_{[s,t]\times \R^d} |\nabla u(r,x)|^2 dr dx \leq ||u(s)||^2_{L^2}
\end{equation}
for a.e. $s \geq 0$ and all $t>s$. Leray--Hopf solutions exist globally in time for all $L^2$-initial data \cite{Leray_NSE_weak-sol}, but their uniqueness remains an open problem of remarkable difficulty.

In the light of this lack of uniqueness results, well-posedness has been studied for variants of the NSE, e.g. the \textit{forced NSE} and the \textit{fractional NSE}. The former differs from the classical NSE by an additional (suitably chosen, hence "free") force $f$, while for the latter the Laplacian is replaced by $(-\Delta)^\alpha, \alpha \in (0,1)$ (a nonlocal operator). Both equations are interesting in itself, for instance from physics considerations or as models with a weakened and non-local diffusive behavior. Of course, both approaches lead to versions of 
Equ. \eqref{intro:Leray-ineq}, by adding an $f$-induced energy term or by replacing the $H^1$-semi norm by a fractional Sobolev norm, respectively. Substantial progress in both directions has been achieved, see \cite{ABC22,ABC23,FHLN22} and \cite{Hypo-paper,DR19,G23}, respectively. For the former case, remarkably it was proven the existence of a force $f$ such that the corresponding class of Leray--Hopf solutions contains at least two elements. For the latter case, non-uniqueness of Leray--Hopf solutions to the fractional NSE with $0< \alpha < \frac 1 3$ was obtained. Note that larger values of $\alpha$ correspond to a stronger diffusion, hence a stronger smoothing effect, so that, at least intuitively, non-uniqueness results are increasingly difficult to obtain as $\alpha$ grows. 

\textbf{Stochastic regime.}
Since adding noise to deterministic systems can not only lead to more accurate models, but can also render ill-posed differential equations well-posed (\textit{regularization by noise}), one might expect to overcome the notorious difficulties for the NSE by suitable stochastic perturbations. However, as shown in \cite{HZZ_stoNSE,HZZ23,HZZ22-ergod,HZZ23-spacetimenoise}, also the stochastic NSE is desperately ill-posed, even in the class of probabilistically strong solutions, for both additive and multiplicative noise, white in time and white or colored in space. Again, this raises the question of a physically reasonable subclass of solutions satisfying a stochastic version of Equ. \eqref{intro:Leray-ineq}. At this point one has to clarify what is meant by \emph{stochastic version of \eqref{intro:Leray-ineq}}. The inequality can be postulated in expectation or pathwise, where clearly the latter implies the former, but not vice versa.  Depending on the choice of noise, an additional positive energy term may appear on the right hand side of Equ. \eqref{intro:Leray-ineq}, which accounts for a noise-induced increase of energy. 

To date, we are not aware of any true well- or ill-posedness result for stochastic Leray--Hopf solutions to the (unforced) NSE with noise, regardless of the choice of noise and of the choice of stochastic version of Equ. \eqref{intro:Leray-ineq}. Recently, similarly to the deterministic case, interesting probabilistically strong local-in-time non-uniqueness results were obtained for stochastic Leray--Hopf solutions with specifically constructed force terms for linear multiplicative \cite{HZZ23-forced} as well as additive \cite{BJLZ23} Brownian noise. In both cases, the stochastic energy inequality holds in expectation. We also mention the very recent work \cite{B24} on global existence and nonuniqueness of ergodic Leray--Hopf solutions to power law flows perturbed by additive noise, which -- even though not containing the fractional or classical NSE -- constitutes another interesting class of equations in hydrodynamics.

\textbf{Our contribution.}
As for the classical NSE ($\alpha = 1$), the uniqueness question for Leray--Hopf solutions to stochastic versions of the (unforced) fractional NSE seems to be entirely open (we point out that for the forced case and $1 < \alpha < \frac  5 4$, a non-uniqueness result is obtained in \cite{BJLZ23}). Our contribution is to provide a negative answer to this question, at least for $\alpha \ll 1$, by constructing a family of probabilistically strong global-in-time solutions to Equ. \eqref{SPDE} with common deterministic initial datum, which satisfy up to a strictly positive stopping time $\tau$ the usual fractional energy inequality
	\begin{equation}\label{ineq:Leray-u}
	\frac 1 2||u(t,\omega)||^2_{L^2} + \int_s^t \int_{\T^3}  \big|(-\Delta)^{\frac \alpha 2}u(r,x,\omega)\big|^2 dxdr \leq \frac 1 2 ||u(s,\omega)||^2_{L^2},\quad \forall 0 \leq s \leq t \leq \tau(\omega)
\end{equation}
pathwise, where $\omega$ denotes elements from the underlying arbitrary probability space $\Omega$. The noise in Equ. \eqref{SPDE} is usually referred to as \emph{transport noise} and has received considerable attention in the recent past, leading to a growing understanding of its use and plausibility, see for instance \cite{FL21,FP22,DP24}, the references therein as well as the introduction of \cite{HLP23} for further literature in this direction. In the present paper, using transport noise allows us
to establish Equ. \eqref{ineq:Leray-u} without an additional positive noise-induced energy-term on its right hand side. Our solutions are $\theta$-Hölder regular, where $\alpha < \theta \ll 1$. At the moment, an improvement to larger values of $\alpha$ and $\theta$ seems to require substantially new techniques. 

We also mention that if we drop the energy inequality, our construction yields global in time probabilistically strong solutions with any prescribed energy profile (which may not satisfy \eqref{ineq:Leray-u}) for exponents $0< \alpha < \tilde{\alpha}_{0}$ (cf. Equ. \eqref{eq:restr_alpha}), but due to the restrictive conditions due to the new terms from the interaction of the nonlocal fractional Laplacian and the flow reformulation of the problem, we can \emph{not} cover the full range $0 < \alpha < \frac{1}{2}$. This is in contrast to the deterministic works \cite{Hypo-paper, DR19}. See Proposition \ref{prop:step1}, which improves the existence theory for \eqref{SPDE} for the extended range $0< \alpha < \tilde{\alpha}_{0}$. 

\textbf{Proof and main novelties.} By a flow transformation based on the flow of the Stratonovich-SDE of the vector fields $\sigma_k$, and akin to \cite{HLP23}, we recast the SPDE Equ. \eqref{SPDE} into a PDE with random coefficients, see Section \ref{sect:flowtransform} and in particular Equ. \eqref{PDE}. Via this transformation, our main result is reformulated as Proposition \ref{prop:main-prop-v}. The proof of the latter relies on (pathwise) convex integration techniques, a powerful method used to construct wild solutions to fluid dynamical (S)PDEs, which has, most notably, led to a proof of Onsager's conjecture for the deterministic Euler equations \cite{DLS_1-10,DLS13,Annals-paper,Isett18,BDLSV18}, as well as to the results from the list of papers mentioned before.

The main idea of the proof of Proposition \ref{prop:main-prop-v} is to combine the work on stochastic Euler equations \cite{HLP23} with the work on (deterministic) Leray--Hopf solutions \cite{Hypo-paper, DR19}. As we are in the transport noise setting, we rely heavily on the groundwork laid by the first-mentioned reference.  While this sounds simple in principle, there were a somewhat surprising amount of technical difficulties that needed to be resolved, of which we now list the most relevant:
\begin{itemize}
 \item As we need a quantitative control of the H\"older norm to prove the energy inequality, we had to track closely the exact dependence of every constant on the energy profile (or rather, the profile's main characteristics $\underline{e}, \bar{e}$ and $|e|_{C^{1}}$). This entailed going through almost all of the proofs of \cite{HLP23}. In particular, we found that all relevant constants only depend on at most the first derivative of $e$, not the second.
  \item To prove the energy inequality, we have to weigh the parameters $a$ and $M_{v}$ against each other (see \eqref{Hlder-est2}), and a product of certain powers needs to be small for small energy profiles. Therefore, we also need to make sure that the constant $M_{v}$ depends increasingly on the energy profile's main characteristics. This, again, means that we have to very carefully track all instances of $M_{v}$ throughout the convex integration scheme and optimize the corresponding estimates. As it turns out, this is indeed possible. $M_v$ is determined by Equ. \eqref{Mv1} 
  below.
 \item The inclusion of the fractional Laplacian term to the Euler equations yields additional stress terms in the convex integration scheme. The most obvious is the ``dissipative error'', which is handled in a similar way as in \cite{Hypo-paper, DR19}. But in addition, there are new contributions to the ``flow error'' and the ``mollification error'' of \cite{HLP23}. Because we are dealing with the ``flow-transformed'' equation for convex integration purposes, these terms are non-trivial to handle. In the end, a Besov space interpolation argument gave us the needed control on these terms, which resulted from a very helpful discussion with Antonio Agresti. For details, see Sections \ref{ssec:Rflow}, \ref{ssec:Rmoll}. 
\end{itemize}

Note that compared to \cite{Hypo-paper, DR19}, we only obtain Leray--Hopf solutions up to a small stopping time. Also in the deterministic case, first, nonunique Leray--Hopf solutions were only constructed for a short time, but one could then extend them since in this case there is a good existence theory for Leray--Hopf solutions from arbitrary initial velocities. To the best of our knowledge, a stochastic counterpart to this global existence is not yet developed for \eqref{SPDE}. Note that our solutions themselves are defined on $[0,\infty)$, but we cannot prove the energy inequality beyond the stopping time.

\textbf{Further literature.}
If one confines oneself to the question of non-uniqueness of weak solutions to stochastic fractional NSEs in the class of weak (not necessarily Leray--Hopf) solutions, results are known for additive and linear multiplicative Brownian noise for $\alpha \in (0,\frac 1 2)$ {\cite{Rehmeier_Hypo,Y21}. A non-uniqueness result for weak (non-Leray--Hopf) solutions to the transport noise NSE ($\alpha=1$) has been obtained in \cite{P24}.
	
Additionally, the two-dimensional fractional NSE has recently been extensively studied. For instance, 
a deterministic Leray--Hopf non-uniqueness result for the $2D$ fractional NSE was obtained in \cite{AC23}, and analogous stochastic results are proven in \cite{Y22}. We also mention the $3D$ Leray--Hopf non-uniqueness result for the Hall--MHD system of \cite{D21}.

For the \textit{hyperdissipative} case, i.e. fractional exponents $\alpha >1$, above the critical value $\alpha =\frac 5 4$, smooth global-in-time solutions exist in the deterministic case, see \cite{FNSE_alpha_5-4_one} and also \cite{Tao09} for the endpoint case. Below $\frac 5 4$, non-uniqueness of weak solutions was proven in \cite{LT20}, and an analogue two-dimensional result is contained in \cite{LQ20}. In an additive and linear multiplicative noise regime, non-uniqueness results in the $3D$ case can be found in \cite{Y22_2}.

Finally, we mention that more generally the literature on convex integration results, in particular ill-posedness results in deterministic and stochastic regimes for a growing list of equations, has grown immensely in the last few years. Instead of presenting a necessarily incomplete reference list here, we only mention that such results have been obtained for, e.g., SQG, Boussinesq, transport(-diffusion), MHD, compressible and hydrostatic Euler equations and further equations.

\textbf{Organisation of the paper.} In Section \ref{sec:Leray-Hopf-main-res} we present the definition of (Leray--Hopf-)solutions to Equ. \eqref{SPDE}, its reformulation as a random PDE, and state our main result, Theorem \ref{thm:main-result} and its PDE version, Proposition \ref{prop:main-prop-v}. We split the proof of the latter into several partial results (most notably the existence result Proposition \ref{prop:step1}) which we formulate and use to conclude the proof in Section \ref{sec:struct_concl}. The proofs of the two key steps, Propositions \ref{prop:step1} and \ref{prop:energy-choice}, are given in Sections \ref{sec:prof-step1} and \ref{sect:step2}, respectively. The former is based on the main iterative proposition \ref{prop:main_iter}, which is proven in Section \ref{sect:mainitprop} via convex integration methods, which are discussed in detail in Sections \ref{sec:CI_scheme} and \ref{sec:proof_main_iter}. Several  of the lengthy calculations and proofs are given in Section \ref{sect:proofs}. The appendix contains an important estimate for the fractional Laplace operator and some results on Besov spaces.

\paragraph{Notation.} We set $\R_+ = [0,\infty)$ and use $x\cdot y$ for the Euclidean inner product. We write $B_r(x)$ for the ball with radius $r$ and center $x$ in a Banach space. Let $\mathcal{S}^{3 \times 3}$ be the space of symmetric $3 \times 3$ matrices.
\\
$C(X,Y)$ is the space of continuous functions between topological spaces $X$ and $Y$. For $k \in \N \cup \{+\infty\}$ and $C_{(\textup{c})}^k(\T^3,\R^m)$ is the Banach space of $k$-times (compactly supported) differentiable maps from $\T^3$ to $\R^m$ with the usual norm $||\cdot||_{C^k}$. The subindex $\sigma$ denotes divergence-free subspaces, for instance $C^\infty_\sigma(\T^3,\R^3)$. For $0 < \theta < 1 $, a Banach space $(X,||\cdot||_X)$ and $U = \T^3$ or a subset of $\R^m$, $C^\theta(U,X)$ denotes the Banach space of Hölder continuous functions with norm 
$$||f||_{C^\theta(U,X)} := \sup_{x \in U} ||f(x)||_X + \sup_{ x \neq y} \frac{\| f(x) - f(y) \|_{E}}{|x-y|^{\theta}},$$
shortly $||\cdot||_\theta$ when no confusion about $U$ and $X$ can occur
The corresponding seminorm, consisting only of the second summand above, is denoted by $[\,\cdot\,]_{C^\theta(U,X)}$ or simply $[\, \cdot \,]_\theta$.
\\
For $ p \in [1,\infty]$, $U \subseteq \R^m$ or $U = \T^3$, the usual spaces of measurable and $p$-integrable functions from $U$ to $X$ are denoted by $L^p(U,X)$ with usual norm $||\cdot||_{L^p}$, and we write $L^p_{\textup{loc}}(U,X)$ for the corresponding local spaces. For $m \in \N$ and $\alpha \in (0,1)$, $H^m(\T^3,\R^l)$ and $H^\alpha(\T^3,\R^l)$ are the usual Hilbert spaces of (fractional) functions from $\T^3$ to $\R^l$ with integrability parameter $p=2$, their usual norms are $||\cdot||_{H^1}$ and $||\cdot||_{H^\alpha}$. For all these spaces, we suppress the state space from the notation if it is $\R$.
\\
Throughout, we use notation such as $||\cdot||_{C^\theta_{\leq s}}$ and $||\cdot||_{C^\theta_{\leq s}C_x}$ for the norm on $C^\theta((-\infty,s],\R^l)$ and $C^\theta((-\infty,s),C(\T^3,\R^l))$, respectively, or $[0,s]$ instead of $(-\infty,s]$, when no confusion about the appearing spaces can occur. In such cases we sometimes write $C^0$ instead of $C$.

We denote the spatial Besov spaces $B_{p,q}^{s} := B_{p,q}^{s}(\T^{d})$ for $s \in \R$, $p,q \in [1,\infty]$ as the subset of distributions $u \in \mathcal{S}'(\T^{d})$ such that
\begin{align*}
    \| u \|_{B_{p,q}^{s}} := \left\| \left( 2^{js} \| \D_{j}u \|_{L^{p}(\T^{d})} \right)_{j \geq -1} \right\|_{\ell^{q}} < \infty.
\end{align*}
Here, $\D_{j}$ denotes the $j$-th Littlewood--Paley block corresponding to a smooth partition of unity, cf. \cite{BCD11, MW17}. For $p, q < \infty$, $B_{p,q}^{s}$ is separable and coincides with the closure of $C^{\infty}(\T^{d})$ with respect to $\| \cdot \|_{B_{p,q}^{s}}$. In terms of duality, we have $B_{p,q}^{s} = \left( B_{p',q'}^{-s} \right)^{*}$, with equivalence of norms, when $\frac{1}{p} + \frac{1}{p'} = \frac{1}{q} + \frac{1}{q'} = 1$, and also $B_{\infty,\infty}^{s} = C_{x}^{s}$ for every non-integer $s  > 0$, again with equivalence of norms, namely 
\begin{equation}\label{eq:Zygmund-Holder}
    C^{-1} \left( \| u \|_{\infty} + [u]_{C_{x}^{s}} \right) \leq \| u \|_{B_{\infty,\infty}^{s}} \leq C \left( \| u \|_{\infty} + [u]_{C_{x}^{s}} \right),
\end{equation}
for some $C > 1$, cf. \cite[3.5.4 Theorem, p. 168 f.]{ST87}. 

For $\alpha \in (0,1)$, the fractional Laplace operator $(-\D)^{\alpha}$ is the operator with symbol $|k|^{2\alpha}$ as a Fourier multiplier, i.e. for any $f \in \mathcal{S}'(\T^{3})$ it has the (formal) Fourier series 
\begin{align*}
    (-\D)^{\alpha} f (x) = \sum_{k \in \Z^{3}} |k|^{2\alpha} \hat{f}_{k} e^{i k \cdot x}.
\end{align*}

\section{Leray--Hopf solutions and main result}\label{sec:Leray-Hopf-main-res}
We begin with the definition and some preliminary observations for (Leray--Hopf) solutions to \eqref{SPDE}
\subsection{Leray--Hopf solutions}
Let $(\Omega, \Fscr, (\Fscr_t)_{t\geq 0},\mathbb{P})$ be a filtered probability space, $B^k, k \in K$, a family of independent standard one-dimensional Brownian motions on $\Omega$ such that $(\Fscr_t)_{t\geq 0}$ is the augmented filtration of the $|K|$-dimensional Brownian motion $B=(B_k)_{k\in K}$. Recall that this filtration is right-continuous.
Our notion of solution is analytically weak, but probabilistically strong in the following sense. 

\begin{dfn}\label{def:sol-to-SPDE}
	\begin{enumerate}
		\item [(i)] 	$(u,p): \R_+\times \T^3 \times \Omega \to \R^3\times \R$ is a \emph{(analytically weak, probabilistically strong) solution} to \eqref{SPDE} with initial condition $u_0 \in L^2(\mathbb{T}^3,\R^3)$, if $(u,p) \in C(\R_+;L^2(\mathbb{T}^3,\R^3\times \R))$ $\mathbb{P}$-a.s., $(u,p)$ is $(\Fscr_t)_{t \geq 0}$-adapted, $u(t,\omega):\T^3 \to \R^3$ is weakly divergence-free for all $t \geq 0$ $\mathbb{P}$-a.s., $u(0,\cdot) = u_0(\cdot)$ $\mathbb{P}$-a.s., 
		and for every $t\geq 0$ and $\varphi \in C_\sigma^\infty(\T^3,\R^3)$
		\begin{align}\label{eq:def-u-sol1}
			\int_{\mathbb{T}^3}u(t)\varphi \,dx = \int_{\T^3}u_0\varphi \,dx &+ \int_0^t \int_{\T^3} u(s)\big(u(s)\cdot \nabla \varphi -(-\Delta)^{\alpha} \varphi \big) dxds \\&\notag+ \sum_{k \in K} \int_0^t\bigg(\int_{\T^3} u(s)\big(\sigma_k \cdot \nabla \varphi\big)dx\bigg)\bullet dB^k_s,\quad \mathbb{P}\text{--a.s.}
		\end{align}
	Often, we regard $u$ instead of $(u,p)$ as the solution.
	\item[(ii)] For an $(\Fscr_t)$-stopping time $\tau: \Omega \to \R_+$, a solution $(u,p)$ to \eqref{SPDE} is a \emph{$\tau$-Leray--Hopf solution}, if $u$ has paths in $L^2_{\textup{loc}}(\R_+,H^\alpha(\T^3,\R^3))$ and satisfies for $\mathbb{P}$-a.e. $\omega \in \Omega$ the stochastic energy inequality \eqref{ineq:Leray-u}
	with exceptional zero set independent of $s$ and $t$.
	\end{enumerate}
\end{dfn}
Note that a $\tau$-Leray--Hopf solution satisfies \eqref{ineq:Leray-u} only up to the random time $\tau$, but is a global in time solution in the sense of part (i) of the previous definition.
\begin{lem}\label{lem:equiv}
For $u$ as in Definition \ref{def:sol-to-SPDE} (i), \eqref{eq:def-u-sol1} is equivalent to the following property. For every semimartingale $h: \R_+\times \Omega \to L^2(\T^3,\R^3)$ of type
\begin{equation}\label{eq:h}
dh = H_0 dt  +\sum_{k \in K}H_k \bullet dB^k,
\end{equation}
with adapted processes
 $H_0, \{H_k\}_{k \in K}:  \R_+ \times \Omega \to C^\infty(\T^3,\R^3)$ , the process $t \mapsto \int_{\T^3} u(t,x)\cdot h(t,x) dx$ is a semimartingale and, $\mathbb{P}$-a.s.,
 \begin{equation}\label{eq:def-u-sol-equiv}
 	d\int_{\T^3} u \cdot h \,dx =\bigg( \int_{\T^3}u \cdot \big(H_0  + (u\cdot \nabla - (-\Delta)^\alpha)h\big)\,dx + \int_{\T^3} p\, \divv h \,dx\bigg)dt + \sum_{k \in K}\bigg[ \int_{\T^3} u \cdot \big(H_k + (\sigma_k \cdot \nabla)h\big)\bigg]\bullet dB^k.
 \end{equation}
\end{lem}
The equivalence can be proven via mollification and Itô's formula as in \cite[App.A]{DHV16}. Equation \eqref{eq:def-u-sol-equiv} is convenient to prove the equivalence of solutions to \eqref{SPDE} and \eqref{PDE} in Lemma \ref{lem:equiv-u-and-v} below.

\begin{rem}
	\begin{enumerate}
		\item [(i)] Note that, in contrast to the general case of (deterministic) Leray--Hopf solutions, we demand (and prove) \eqref{ineq:Leray-u} for \emph{all} $0 \leq s \leq t\, (\leq \tau)$ instead of for only \emph{almost every} $s \geq 0$ and every $s \leq t\, (\leq \tau)$.
		\item[(ii)] We expect also higher order pathwise inequalities to hold, more precisely
	\begin{equation}\label{ineq:higher-order}
		\frac 1 2||u(t,\omega)||^{2q}_{L^2} + \int_s^t ||u(r,\omega)||^{2(q-1)}_{L^2}\int_{\T^3}  \big|(-\Delta)^{\frac \alpha 2}u(r,x,\omega)\big|^2 dxdr \leq \frac 1 2 ||u(s,\omega)||^{2q}_{L^2},\quad \forall q \in \N,
	\end{equation}
	which are crucial for compactness arguments, since typically the set of solution path laws of those solutions satisfying \eqref{ineq:higher-order} is compact in a suitable space of path measures, which in turn is needed for Markovian selections from this space. This way, a proof of non-uniqueness of Markovian selections among the class of Leray--Hopf solutions considered in this work seems possible, although we do not address this question. Note that non-uniqueness of Markovian selections from the larger class of (not necessarily Leray--Hopf) weak solutions was proven for related equations (see, for instance, \cite{HZZ23}).
	\end{enumerate}
	
\end{rem}
\subsection{Main result}
Let $(\Omega, \Fscr, (\Fscr_t)_{t\geq 0}, \mathbb{P})$, $K$ and $B_k$ be as in the previous subsection.
Our main result is the following theorem. 
\begin{theorem}\label{thm:main-result}
	Let $0 < \alpha < \alpha_{0} := \frac{1}{2cb+1}$, where $b$ and $c$ are as in Section \ref{sect:mainitprop}.
	There exists $u_0 \in L^2(\T^3,\R^3)$, an $\mathbb{P}$-a.s. strictly positive $(\Fscr_t)$-stopping time $\tau_0:\Omega \to \R_+$ and infinitely many $\tau_0$-Leray--Hopf solutions to \eqref{SPDE} with initial condition $u_0$ and paths in $C(\R_+, C^\theta(\mathbb{T}^3,\R^3))$ for some $\theta > \alpha$. Moreover, any two of these solutions are distinct on $[0,\tau_0]$, $\mathbb{P}$-a.s.
\end{theorem}
\begin{rem}
	\begin{enumerate}
		\item[(i)] Clearly, it follows that $u_0 \in C^\theta(\T^3,\R^3)$. However, $u_0$ cannot be prescribed, but is an outcome of the proof. In particular, the size of the set of initial conditions for which Theorem \ref{thm:main-result} holds remains to be determined; we only know that this set contains uncountably many elements. Indeed,
		by varying the energy profile  $e \in \Escr$ used in the construction, we see that for every  $\varepsilon$ sufficiently small, we can produce one initial datum $u_{0}^{\varepsilon} \in L^{2}$ with $||u_0^\varepsilon||_{L^2} = \varepsilon$ and infinitely many solutions to this initial datum.
				
		\item[(ii)] The solutions have spatial Hölder regularity with parameter $\theta$ for every $\alpha < \theta < \bar{\theta} := \alpha_{0}$.
		
	\item[(iii)] 
We point out that Theorem \ref{thm:main-result} also improves the existence theory of analytically weak and probabilistically strong (not necessarily Leray--Hopf-)solutions for equation \eqref{SPDE}. Indeed, to our knowledge, this is the first paper concerned with solutions to \eqref{SPDE}.
	
	\item[(iv)] As already said in the introduction, our result does not imply existence or non-uniqueness of \emph{global} Leray--Hopf solutions. Indeed, the typical argument, i.e. to glue our nonunique local Leray--Hopf solutions $u_k$ together (pathwise at the random time $\tau$) with global in time Leray--Hopf solutions starting from the endpoints $u_k(\tau(\omega))$, does not work, since to date it is open whether such global Leray--Hopf solutions exist (not even for \emph{some} initial velocity field, whereas for the gluing, one really needs existence from  \emph{any} initial velocity field).
	\end{enumerate}
\end{rem}

\subsection{Reformulation via flow transformation}\label{sect:flowtransform}
For the proof of Theorem \ref{thm:main-result}, similarly as in \cite{HLP23}, we rewrite the SPDE \eqref{SPDE} as an equivalent PDE with random coefficients. To this end, consider on the manifold $\T^3$ the SDE
\begin{equation*}
\Phi(t,x) = x+ \sum_{k \in K} \int_0^t \sigma_k(\Phi(s,x))\bullet dB^k_s, \quad t\in \R.
\end{equation*}
The vector fields $\sigma_k$ take values in the tangent bundle of $\T^3$, which consists of the tangent spaces $T_x\T^3 = \R^3$, $x \in \T^3$. By assumptions on $\sigma_k$ there exists a unique flow $\Phi = (\Phi(t))_{t\in \R}$ of probabilistically strong solutions such that each $\Phi(t)$ is a Lebesgue  measure-preserving $C^\infty$-diffeomorphism of $\T^3$, see \cite{K90}.
We denote by $\Phi^{-1}$ the inverse flow, by $\Phi(t)^{-1}$ its evaluation at $t$, and set
$$\text{div}_\Phi f := [\divv (f\circ \Phi^{-1})]\circ \Phi,\quad(-\Delta)^\alpha_\Phi f := \big[(-\Delta)^\alpha (f\circ \Phi^{-1})\big]\circ \Phi,\quad \nabla_\Phi f:= [\nabla(f\circ \Phi^{-1})]\circ \Phi,$$
 and note that (unlike $\divv f$, $(-\Delta)^\alpha f$ and $\nabla f$) these operators depend on $t$. If $f$ depends on $t \in \R$, we write $$(-\Delta)^\alpha_\Phi f(t):= (-\Delta)^\alpha_{\Phi(t)}f(t),$$
 and similarly for $\divv_\Phi$ and $\nabla_\Phi$. Consider the following random PDE, posed on the prescribed  arbitrary, fixed filtered probability space $(\Omega,\Fscr, (\Fscr_t)_{t\geq 0},\mathbb{P})$:
\begin{equation}\label{PDE}
	\begin{cases}
		\partial_t v + (-\Delta)_\Phi^\alpha v+ \divv_\Phi(v\otimes v) + \nabla_\Phi q = 0\\
		\divv_\Phi v = 0.
	\end{cases}
\end{equation}

\begin{dfn}\label{def:sol-to-PDE}
	\begin{enumerate}
		\item [(i)] $(v,q): \R_+\times \T^3 \times \Omega \to \R^3\times \R$ is a \emph{(analytically weak, probabilistically strong) solution} to \eqref{PDE} with initial condition $v_0 \in L^2(\T^3,\R^3)$, if $(v,q) \in C(\R_+,L^2(\T^3,\R^3\times \R))$ $\mathbb{P}$-a.s., $(v,q)$ is $(\Fscr_t)$-adapted, $v(0,\cdot) = v_0$ $\mathbb{P}$-a.s.,
		$$\int_{\T^3} v(t,x)\cdot \nabla_{\Phi(t)}\varphi(x)\,dx = 0,\quad \forall \varphi \in C^\infty(\T^3), t \geq 0$$
		with exceptional set independent of $t$ and $\varphi$, and for every semimartingale $h: \R_+\times \Omega \to L^2(\T^3,\R^3)$ of type
		$$dh = H_0 dt  + \sum_{k \in K}H_k \bullet dB^k$$
		with progressively measurable processes $H_0, \{H_k\}_{k \in K}: \Omega \times \R_+ \to C^\infty(\T^3,\R^3)$, the process $t \mapsto \int_{\T^3}v(t,x)\cdot h(x,t)dx$ is a semimartingale satisfying, $\mathbb{P}$-a.s.,
		$$d\int_{\T^3} v \cdot h \,dx = \int_{\T^3} v \cdot\big(H_0 + (v\cdot \nabla_\Phi - (-\Delta)^\alpha_\Phi)h\big)\,dt + \int_{\T^3}q\text{div}_\Phi h\,dt + \sum_{k \in K} \bigg[\int_{\T^3} v \cdot H_k\bigg] \bullet dB^k.$$
		Occasionally, we refer to $v$ instead of $(v,q)$ as the solution.
		\item[(ii)] For an $(\Fscr_t)$-stopping time $\tau:\Omega \to \R_+\cup \{+\infty\}$, a solution $(v,q)$ to \eqref{PDE} is a \emph{$\tau$-Leray--Hopf solution}, if $t\mapsto \int_{\T^3} |(-\Delta)^{\frac \alpha 2}_\Phi v(t,x)|^2 dx$ belongs to $L^2_{\textup{loc}}(\R_+)$ $\mathbb{P}$-a.s. and 
		\begin{equation}\label{ineq:Leray-v}
			\frac 1 2 ||v(t,\omega)||^2_{L^2} + \int_s^t \int_{\T^3} |(-\Delta)^{\frac \alpha 2}_\Phi v(r,x,\omega)|^2 dxdr \leq \frac 1 2 ||v(s,\omega)||^2_{L^2},\quad \forall 0\leq s \leq t \leq \tau(\omega),
		\end{equation}
	for $\mathbb{P}$-a.e. $\omega \in \Omega$, with exceptional set independent of $s$ and $t$. Here we shortened the notation by writing $\Phi$ instead of $\Phi(\omega)$.
	\end{enumerate}
\end{dfn}

\begin{lem}\label{lem:frac-flow-Laplace-id}
	Let $v$ be a solution to \eqref{PDE} with paths in $C(\R_+,C^\theta(\T^3,\R^3))$, $\mathbb{P}$-a.s., such that $0<\alpha < \theta$. Then
	\begin{enumerate}
		\item[(i)] $\mathbb{P}$-a.s., with exceptional set independent of $t$,
		\begin{equation}\label{ineq:Lemma2.7.i}
			\int_{\T^3}|(-\Delta)^{\frac \alpha 2}_\Phi v(t)|^2 dx \leq C_t[v(t)]^2_\theta,\quad t \geq 0,
		\end{equation}
		where $C_t>0$ depends only on $\omega$, $\Phi(t), \alpha$ and $\theta$, and is, for fixed $\omega$, continuous in $t$.
		\item [(ii)]
$u = (u(t))_{t \geq 0}$, $u(t) := v(t)\circ \Phi(t)^{-1},$ has paths in $L^2_{\textup{loc}}(\R_+,H^\alpha(\T^3,\R^3)) \cap C(\R_+,C^\theta(\T^3,\R^3))$, $\mathbb{P}$-a.s., and
	\begin{equation}\label{eq:flow-and-class-Laplac}
		\int_{\T^3} |(-\Delta)_\Phi^{\frac \alpha 2}v(t)|^2dx = \int_{\T^3} |(-\Delta)^{\frac \alpha 2}u(t)|^2 dx ,\quad \forall t\geq 0,
	\end{equation}
$\mathbb{P}$-a.s.
Moreover, $v$ and $u$ have the same kinetic energy profile, i.e.
$$\int_{\mathbb{T}^3} |v(t)|^2 dx = \int_{\mathbb{T}^3} |u(t)|^2 dx, \quad \forall t\geq 0$$
$\mathbb{P}$-a.s. For both assertions, the exceptional set is independent of $t$.
\end{enumerate}
\end{lem}
\begin{proof}
	\begin{enumerate}
	 \item[(i)] By \cite[Cor.B.2]{DeRosa_one-third}, we have, using the measure-preserving property of $\Phi$,
	 \begin{align*}
	 	\int_{\T^3} |(-\Delta)^{\frac \alpha 2}_{\Phi} v(t)|^2 dx &= \int_{\T^3} \big| \big[(-\Delta)^{\frac \alpha 2}(v(t)\circ \Phi(t)^{-1})\big]\circ \Phi(t)\big|^2 dx \\&= \int_{\T^3} \big| (-\Delta)^{\frac \alpha 2}(v(t)\circ \Phi(t)^{-1})\big|^2 dx \leq C(\alpha,\theta)[v(t)\circ \Phi(t)^{-1}]^2_\theta.
	 \end{align*}
	Since 
	 $$[v(t) \circ \Phi(t)^{-1}]_\theta \leq \text{Lip}(\Phi(t)^{-1})^\theta[v(t)]_\theta,$$ 
	 the claim follows with 
	 \begin{equation}\label{eq:C_t}
	 C_t = \textup{Lip}(\Phi^{-1}(t))^{2\theta}C(\alpha,\theta),
\end{equation}
which is continuous in $t$ for fixed $\omega$, where $\textup{Lip}(f)$ denotes the Lipschitz constant of $f: \T^3 \to \T^3$.
	 \item [(ii)] By the regularity of $\Phi$, it is clear that $u \in C(\R_+,C^\theta(\T^3,\R^3))$. Since $\theta > \alpha$ and for every $\varepsilon>0$ the inequality
	 $$\int_{\T^3}|(-\Delta)^{\frac \alpha 2} f|^2 dx \leq C(\varepsilon, \alpha) [f]^2_{\alpha+\varepsilon}$$
	 holds for every $f \in C^\theta(\T^3)$ and a constant $C(\varepsilon,\alpha)>0$ independent of $f$ (see \cite[Cor.B.2]{DeRosa_one-third}), $u \in L^2_{\textup{loc}}(\R_+, H^\alpha(\T^3,\R^3))$ follows.
	 Concerning \eqref{eq:flow-and-class-Laplac}, we have $\mathbb{P}$-a.s.
	 \begin{align*}
	 	\int_{\T^3} |(-\Delta)_\Phi^{\frac \alpha 2}v(t)|^2dx &=  \int_{\T^3} \big|\big[(-\Delta)^{\frac \alpha 2} (v(t)\circ \Phi(t)^{-1})\big]\big|^2 dx = \int_{\T^3} |(-\Delta)^{\frac \alpha 2}u(t)|^2 dx,\quad \forall t \geq 0,
	 \end{align*}
 where the left hand side is finite by (i).
	 The final claim follows directly from the measure-preserving property of each $\Phi(t)$, $t\in \R$.
	\end{enumerate}
\end{proof}
As a consequence, we can prove the following important lemma.
\begin{lem}\label{lem:equiv-u-and-v}
	Let $0< \alpha < \theta$. If $(v,q)$ is a $\tau$-Leray--Hopf solution to \eqref{PDE} such that $v$ has paths in $C(\R_+,C^\theta(\T^3,\R^3))$ and initial condition $v_0$, then $(u,p)$, $u = v\circ \Phi^{-1}$ and $p = q \circ \Phi^{-1}$, is a $\tau$-Leray--Hopf solution to \eqref{SPDE} such that $u$ has  paths in $C(\R_+,C^\theta(\T^3,\R^3))$ and initial condition $u_0 = v_0$. In this case, $v$ and $u$ have the same pathwise kinetic energy profile $\mathbb{P}$-a.s.
\end{lem}
\begin{proof}
	Appealing to Lemma \ref{lem:frac-flow-Laplace-id} (ii), the measure-preserving property of $\Phi$, and since $u(0,x) = v(0,\Phi(0)^{-1}(x)) = v(0,x)$, it remains to prove that $u$ is a solution to \eqref{SPDE} if $v$ is a solution to \eqref{PDE}. To this end, let $h$ be as in \eqref{eq:h} and note $\int_{\T^3} u \cdot h \,dx = \int_{\T^3} v \cdot (h\circ \Phi)\,dx$. By Itô's formula, $h\circ \Phi$ is a semimartingale satisfying
	\begin{equation}
		d(h\circ \Phi) = (H_0 \circ \Phi)dt + \sum_{k\in K} \big[(H_k + (\sigma_k \cdot \nabla)h)\circ \Phi\big]\bullet dB^k(t).
	\end{equation}
Thus, the process $t \mapsto \int_{\T^3} u(t,x)\cdot h(t,x)\,dx$ is a semimartingale satisfying $\mathbb{P}$-a.s.
\begin{align*}
	d\int_{\T^3} u \cdot h \,dx &= \bigg(\int_{\T^3}v\cdot\big(H_0 \circ \Phi + (v\cdot \nabla_\Phi - (-\Delta)^\alpha_\Phi)h\circ \Phi\big)\,dx \bigg) dt
	\\
	&\quad + \bigg(\int_{\T^3}q\mathrm{div}_{\Phi}(h\circ \Phi)\,dx\bigg)dt+ \sum_{k \in K} \bigg[\int_{\T^3} v \cdot \big((H_k + (\sigma_k \cdot \nabla)h)\circ \Phi\big)\bigg]\bullet dB^k \\
	& = \bigg(\int_{\T^3}u \cdot \big(H_0+ (u\cdot \nabla - (-\Delta)^\alpha)h\big)dx + \int_{\T^3} p \mathrm{div} h \,dx\bigg)dt\\
	&\quad + \sum_{k \in K} \bigg[\int_{\T^3} u \big(H_k + (\sigma_k\cdot \nabla)h\big)\bigg]\bullet dB^k.
\end{align*}
Thus, \eqref{eq:def-u-sol-equiv} holds, which, by Lemma \ref{lem:equiv} completes the proof.
\end{proof}
By Lemma \ref{lem:equiv-u-and-v}, the following result is equivalent to Theorem \ref{thm:main-result}.
\begin{prop}\label{prop:main-prop-v}
	For $0 < \alpha < \alpha_{0} := \frac{1}{2cb+1}$, where $b$ and $c$ are as in Section \ref{sect:mainitprop}, there exists $v_0 \in L^2(\mathbb{T}^3,\R^3)$, an $\mathbb{P}$-a.s. strictly positive $(\Fscr_t)$-stopping time $\tau_0: \Omega \to \R_+$ and infinitely many $\tau_0$-Leray--Hopf solutions  to \eqref{PDE} with initial condition $v_0$ and paths in $C(\R_+, C^\theta(\mathbb{T}^3))$ for some $\theta>\alpha$.
	Moreover, any two of these solutions are distinct on $[0,\tau_0]$ $\mathbb{P}$-a.s.
\end{prop}
Thus, the remainder of the paper is devoted to the proof of the previous proposition, which in turn proves our main result, Theorem \ref{thm:main-result}.

\section{Structure and conclusion of the proof}\label{sec:struct_concl}
The proof of Theorem \ref{thm:main-result} consists in the proof of Proposition \ref{prop:main-prop-v}, which is divided into three main steps, as presented now.
\paragraph{Step 1: Construction of H\"older regular solutions.} Set
$$\Escr := \bigg\{e \in C^\infty(\R) \,\big|\, \inf_{t>0} e(t)>0, ||e||_{C^1}<\infty\bigg\}.$$ 
This step consists in proving the following result, which seems to be new itself.
\begin{prop}\label{prop:step1}
	Let $\alpha \in (0,\tilde{\alpha}_{0})$ (where $\tilde{\alpha}_{0} > \alpha_{0}$ is given in Equ. \eqref{eq:restr_alpha} below) and $e \in \Escr$. There is a $\mathbb{P}$-a.s. strictly positive $(\Fscr_t)$-stopping time $\tau: \Omega \to \R_+$ and a solution $v = v(e)$ to \eqref{PDE} with paths in $C(\R_+,C^\theta(\T^3))$ for some $\theta > 0$ and deterministic initial condition $v_0 \in L^2(\T^3;\R^3)$ such that $\mathbb{P}$-a.s.
	\begin{equation}
		||v(t,\omega)||^2_{L^2} = e(t),\quad \forall 0\leq t \leq \tau(\omega).
	\end{equation}
Moreover, if $e_1(0) = e_2(0)$, then $v(e_1)$ and $v(e_2)$ have the same initial condition.
\end{prop}
\begin{kor}\label{cor:alpha}
	For $0 < \alpha < \alpha_{0} = \frac{1}{2cb+1} < \tilde{\alpha}_{0}$, where $b$ and $c$ are as in Section \ref{sect:mainitprop}, the solutions $v$ from the previous proposition are in $C^0_{\leq \tau} C^{\theta}_x$ for $\alpha < \theta < \bar{\theta} := \alpha_{0}$.
\end{kor}

The proofs are given in Section \ref{sec:prof-step1}.
As we have pointed out before, $v_0$ cannot be prescribed \emph{a priori}, and as of now we cannot make any statement about the set of initial conditions arising this way, except its non-emptyness.

\paragraph{Step 2: Choice of energies.}
Here we make a suitable choice for the energy profiles in Proposition \ref{prop:step1} as follows. Its proof is given in Section \ref{sect:step2}.
\begin{prop}\label{prop:energy-choice}
		For every $0<\varepsilon<1$, there is a family of non-increasing energy profiles $\{e^\varepsilon_k\}_{k \in \N} \subseteq \Escr$ such that
		\begin{enumerate}
			\item [(i)] $e^\varepsilon_k(0) = e^\varepsilon_l(0), \quad \forall k,l \in \N$,
			\item [(ii)] for each $k \neq l$, there is a sequence $t_n \to 0$ as $n \to \infty$ with $e^\varepsilon_k(t_n) \neq e^\varepsilon_l(t_n)$ for all $n \in \N$,
			\item [(iii)] there is $0<r=r(\varepsilon)$ such that $(e^\varepsilon_k)'(t) \leq -\frac 1 2$ for all $t\in [0,r)$ and $k \in \N$,
			\item[(iv)] $|e^\varepsilon_k|_{C^1(\R_+)} = 1$, $\inf_{t \geq 0} e^\varepsilon_k(t) = \frac \varepsilon 2$, $\sup_{t\geq 0}e^\varepsilon_k(t) = \varepsilon$ for all $k \in \N$.
		\end{enumerate}
	Let $0< \alpha < \alpha_{0}$. Denoting by $v^\varepsilon_k$ the solution related to $e^\varepsilon_k$  and by $\tau$ the stopping time from Step 1 (independent of $\varepsilon$ and $k$), there is $\theta \in (\alpha, \alpha_{0})$ such that we have $\mathbb{P}$-a.s. 
		\begin{equation}\label{eq:eps-to-0}
			\sup_{k \in \N}	\sup_{t \in [0,\tau]}||v^\varepsilon_k(t)||^2_{C^\theta} \xrightarrow{\varepsilon \to 0} 0,
		\end{equation}
	with convergence is uniform in $\omega \in \Omega'$, where $\Omega'\subseteq \Omega$ is the set of those $\omega$ for which the convergence holds.
\end{prop}

\paragraph{Step 3: Conclusion.}
Via Steps 1 and 2, the proof of Proposition \ref{prop:main-prop-v} is completed as follows.

Let $\varepsilon \in (0,1)$, $0< \alpha < \alpha_{0}$ as in Corollary \ref{cor:alpha}, $(e^\varepsilon_k)_{k \in \N}$ a family of energies as in Step 2, and let $(v^\varepsilon_k)_{k\in \N}$ and $\tau$ be the corresponding solutions and stopping time from Step 1.
By the last assertion of Step 1 and (i) of Proposition \ref{prop:energy-choice}, we have $v^\varepsilon_k (0) = v^\varepsilon_l(0)$ for all $k,l \in \N$, and we denote this common initial condition by $v^\varepsilon_0$. To finish the proof, it is now sufficient to find $\varepsilon>0$ such that each $v^\varepsilon_k$ satisfies \eqref{ineq:Leray-v} pathwise $\mathbb{P}$-a.s. on $[0,\tau_\varepsilon]$, where we set
\begin{equation}\label{stopptime}
\tau_\varepsilon:= \tau \wedge \tau_{\textup{Lip}}\wedge r(\varepsilon) \leq \varepsilon,
\end{equation}
with $r(\varepsilon)>0$ as in (iii) of Proposition \ref{prop:energy-choice} and the $(\Fscr_t)$-stopping time $\tau_{\textup{Lip}}$ is defined by
$$\tau_{\textup{Lip}}(\omega):= \inf\{t>0: \textup{Lip}(\Phi^{-1}(t,\omega))>2\}.$$
Note that $\tau_{\textup{Lip}} >0$ $\mathbb{P}$-a.s., since $\textup{Lip}(\Phi^{-1}(0,\omega)) = 1$ and $t\mapsto \textup{Lip}(\Phi^{-1}(t,\omega))$ is continuous. Thus, comparing with \eqref{eq:C_t} and Lemma \ref{lem:frac-flow-Laplace-id} (i), for $0\leq t \leq \tau_\varepsilon$, the constant on the right hand side of \eqref{ineq:Lemma2.7.i} can be replaced by a constant $C>0$, which depends on $\alpha, \theta$, but not on $\omega$, $\varepsilon$ or $t$.

 Fix $k \in \N$, write $v^\varepsilon = v^\varepsilon_k$ and let $0\leq s \leq t \leq \tau_\varepsilon$. Lemma \ref{lem:frac-flow-Laplace-id} (i) with the aforementioned $\omega$- and $t$-independent constant $C$ on its right hand side entails
\begin{equation*}
	2\int_s^t \int_{\T^3} |(-\Delta)^{\frac \alpha 2}_\Phi v^\varepsilon(r,x)|^2 dx dr \leq (t-s)C\sup_{r \in [0,\tau_\varepsilon]}||v^\varepsilon(r)||^2_{\theta},
\end{equation*}
$\mathbb{P}$-a.s.
By (iii) of Proposition \ref{prop:energy-choice},
\begin{equation*}
	||v^\varepsilon(s)||^2_{L^2}  - ||v^\varepsilon(t)||^2_{L^2} = e^\varepsilon(s)-e^\varepsilon(t) =  -\int_s^t (e^\varepsilon)'(r)dr \geq \frac{t-s}{2}.
\end{equation*}
Thus it remains to obtain $\mathbb{P}$-a.s. the estimate 
\begin{equation*}
	\sup_{r \in [0,\tau_\varepsilon]}||v^\varepsilon(r)||^2_{\theta} \leq (2C)^{-1}.
\end{equation*}
Since $C$ does not depend on $\varepsilon$ or $\omega$, this follows from \eqref{eq:eps-to-0} (recall that the latter holds uniformly in $\omega \in \Omega'$) by choosing $\varepsilon$ sufficiently small. For such $\varepsilon$, we set 
$$
\tau_0 := \tau_\varepsilon,
$$
 where $\tau_\varepsilon$ is as in \eqref{stopptime}. Finally, (ii) of Proposition \ref{prop:energy-choice} implies $v^\varepsilon_k \neq v^\varepsilon_l$ pathwise $\mathbb{P}$-a.s. on $[0,\tau_0]$ for all $k\neq l$. 
\qed 

In the remainder of the paper, we prove Propositions \ref{prop:step1} and \ref{prop:energy-choice}.

\section{Proof of Proposition \texorpdfstring{\ref{prop:step1}}{3.1}}\label{sec:prof-step1}
\subsection{The flow map}
\subsubsection{Mollification of the noise}

Recall that $B = (B_{k})_{k \in I}$ is an $\R^{|K|}$-valued Brownian motion. Let us assume that every realization of $B$ has local $C^{\gamma}$ time regularity, $\gamma \in (1/3, 1/2)$. This means that the associated flow $\phi$ should have the same time regularity in time, which is not sufficient for our purposes. Therefore, we introduce a sequence of mollified flows $(\phi_{n})_{n \in \N}$ for the convex integration scheme as follows. 

Let $\vs_{n} > 0$ to be determined below such that at least $\vs_{n}$ is monotonically decreasing in $n$ and satisfies $(n+1) \vs_{n}^{\gamma - \beta}$, $\beta \in (0,\gamma)$. Then, let $\Theta \colon \R \to \R$ be a smooth mollified with support in $(0,1)$ and set 
\begin{align*}
	\Theta_{n}(t) := \vs_{n}^{-1} \Theta(t \vs_{n}^{-1}), \quad B_{n}(t) := (B * \Theta_{n})(t) = \int_{\R} B(t-s) \Theta_{n}(s) ds,
\end{align*}
where by convention we set $B(t-s) = B(0) = 0$ for $t-s < 0$, and the convolution is component-wise, i.e. $(B_{n})^{k} = B^{k}_{n} = (B^{k} * \Theta_{n})$, $k \in I$. Then obviously $B_{n}$ is smooth for all $t \in \R$ and identically zero for $t < 0$. Then define $\phi_{n}$ as the solution to the equation
\begin{equation}\label{eq:phin_equation}
	\phi_{n}(t,x) := x + \sum_{k \in I} \int_{0}^{t} \sigma_{k}(\phi_{n}(s,x)) dB_{n}^{k}(s), \quad x \in \T^{3}, ~ t \in \R,
\end{equation}
where the integral is understood as a pathwise Riemann-Stieltjes integral. Note that $\phi_{n}(t,x) = \phi_{n}(0,x) = x$ for $t < 0$. We similarly extend the flow $\phi$ to $t < 0$ by setting $\phi(t,x) = \phi(0,x) = x$. As the $\sigma_{k}$ are divergence-free, the map $\phi_{n}(t,\cdot)$ is measure-preserving with probability one. Finally note that by the Wong--Zakai theorem, $\phi_{n}$ is indeed an approximation to the (Stratonovich) flow $\phi$.

\subsubsection{Localisation}
Following \cite{HLP23}, we want to work in the framework of rough paths (cf. \cite{FH20}). To this end, we enhance the Brownian motion $B$ by its step-2 Stratonovich  lift 
\begin{align*}
	\mathbb{B}(s,t) := \int_{s}^{t} (B(r) - B(s)) \otimes \circ dB(r),
\end{align*}
and the mollified process $B_{n}$ by its step-2 Lyons lift (given by the Riemann--Stieltjes integral)
\begin{align*}
	\mathbb{B}_{n}(s,t) := \int_{s}^{t} (B_{n}(r) - B_{n}(s)) \otimes  dB_{n}(r).
\end{align*}
Then both $\bm{B} := (B, \mathbb{B})$, $\bm{B}_{n} := (B_{n}, \mathbb{B}_{n})$ are geometric $\gamma$-H\"older rough paths.

As was shown in \cite[Lemma 2.1, Lemma 2.2]{HLP23} (to which we refer for more details), there exist a suitable increasing and diverging sequence $(K_{L})_{L \in \N}$ with $K_{0} = K_{1} \leq K_{L}$ for all $L \in \N$ and an associated stopping time $\mathfrak{s}_{L} \leq K_{L}$ with $\lim_{L \to \infty} \mathfrak{s}_{L} = \infty$ almost surely. Now define
\begin{equation}
	\mft_{L} := \mathfrak{s}_{L} \wedge \inf \left\{ s \geq 0 ~\colon~ \| B \|_{C_{\leq s}^{\gamma}} > K_{L} \right\} \wedge \inf \left\{ s \geq 0 ~\colon~ \| \mathbb{B} \|_{C_{\leq s}^{2\gamma}} > K_{L} \right\}.
\end{equation}
Then $(\mft_{L})_{L \in \N}$ is a sequence of non-decreasing stopping times such that $\lim_{L \to \infty} \mft_{L} = \infty$ almost surely, and for a constant $C = C(K_{0},\gamma, \beta)$ we have
\begin{alignat}{2}
	\label{eq:diff_phi_CbetaCk} \| \phi_{n+1} - \phi_{n} \|_{C_{\leq \mft_{L}}^{\beta} C_{x}^{\kappa}} &\leq CL (n+1) \vs_{n}^{\gamma-\beta}, \quad && \| \phi_{n} \|_{C_{\leq \mft_{L}}^{\gamma}C_{x}^{\kappa}} ~\leq CL, \\
	\label{eq:diff_phi-1_CbetaCk} \| \phi_{n+1}^{-1} - \phi_{n}^{-1} \|_{C_{\leq \mft_{L}}^{\beta} C_{x}^{\kappa}} &\leq CL (n+1) \vs_{n}^{\gamma-\beta}, \quad && \| \phi_{n}^{-1} \|_{C_{\leq \mft_{L}}^{\gamma}C_{x}^{\kappa}} \leq CL, \\
	\label{eq:phi_CbetaCk} \| \phi_{n} \|_{C_{\leq \mft_{L}}^{r} C_{x}^{\kappa}} &\leq CL \vs_{n}^{\gamma-r}, \quad && \| \phi_{n}^{-1} \|_{C_{\leq \mft_{L}}^{r}C_{x}^{\kappa}} \leq CL \vs_{n}^{\gamma-r}.
\end{alignat}

\subsection{The main iterative proposition}\label{sect:mainitprop}

We make the following choices of parameters: 
\begin{align*}
	\delta_{n} := a^{1-b^{n}}, \quad D_{n} := a^{cb^{n}}, \quad L_{n} := L^{m^{n+1}},
\end{align*}
where
\begin{align*}
	a \geq 2, \quad b = m+\epsilon, \quad c = \frac{b^{4}(1+\epsilon)-1/2}{b - 1 - \epsilon}, \quad \epsilon > 0,
\end{align*}
and $m \geq 4$ is given in Proposition \ref{prop:main_iter} below. The precise choice of these constants will be discussed in Remark \ref{rem:choice_of_b} and proven in Section \ref{ssec:pf_choice_of_parameters}.

Consider the Euler--Reynolds system
\begin{equation}\label{eq:Euler_Reynolds}
	\partial_{t} v + \divv_{\phi}( v \otimes v) + \nabla_{\phi} q + (-\D)^{\alpha}_{\phi} v = \divv_{\phi} \mathring{R}.
\end{equation}
The following result is the key step for the proof of Proposition \ref{prop:step1}. Its proof will be given in Section \ref{sec:proof_main_iter}.
\begin{prop}[Main iterative proposition]\label{prop:main_iter}
	Let $e \in \Escr$ and set
	$$\underline{e} := \inf_{t \in \R} e(t) > 0,\quad \bar{e} := \sup_{t \in \R} e(t) < \infty.$$
	Then there exist constants $\epsilon, m$ as above, a constant $a \geq 2$ depending on $\underline{e},\bar{e},|e|_{C_t^1}, K_{0}$, a constant $\eta \in (0,1)$ depending on $\underline{e},\bar{e}$,  a constant $M_{q}$ depending on $\underline{e},\bar{e}$, a constant $M_{v}$ depending on $\bar{e}$, and a constant $A \in (0,\infty)$ depending on $\underline{e},\bar{e},|e|_{C_t^1}$ with the following property:
	
	Fix $\phi_{n}$ as in \eqref{eq:phin_equation} and let $(v_{n},q_{n},\phi_{n},\mathring{R}_{n})$, $n \in \N$, be a smooth solution of \eqref{eq:Euler_Reynolds} such that $v_n$ has mean zero, satisfying the inductive estimates
	\begin{equation}\label{eq:iter_energy}
		\left| e(t)(1- \delta_{n}) - \int_{\T^{3}} |v_{n}(t,x)|^{2} dx \right| \leq \frac{1}{4} \delta_{n} e(t), \quad t \leq \mathfrak{t},
	\end{equation}
	and for every $L \in \N$
	\begin{align}
		\label{eq:iter_stress}
		\| \mathring{R}_{n} \|_{C_{\leq \mft_{L}}C_{x}} &\leq \eta L_{n} \delta_{n+1}, \\
		\label{eq:iter_press}
		\| q_{n} \|_{C_{\leq \mft_{L}}C_{x}} &\leq M_{q} L_{n} \sum_{k=0}^{n-1}\delta_{k}, \\
		\label{eq:iter_div}
		\| \divv_{\phi_{n}} v_{n} \|_{C_{\leq \mft_{L}}B_{\infty,\infty}^{-1}} &\leq L_{n} \delta_{n+2}^{5/4}, \\
		\label{eq:iter_C1}
		\max \left\{ \| v_{n} \|_{C^{1}_{\leq \mft_{L},x}} ,    \| q_{n} \|_{C^{1}_{\leq \mft_{L},x}}, \| \mathring{R}_{n} \|_{C_{\leq \mft_{L}}C_{x}^{1}}  \right\} &\leq L_{n} D_{n}.
	\end{align}
	Then there exists a second quadruple $(v_{n+1},q_{n+1},\phi_{n+1},\mathring{R}_{n+1})$ solving \eqref{eq:Euler_Reynolds} satisfying \eqref{eq:phin_equation} and the inductive estimates \eqref{eq:iter_energy}--\eqref{eq:iter_C1} with $n$ replaced by $n+1$, and for every $L \in \N$, $L \geq 1$
	\begin{align}
		\label{eq:iter_diffv}
		\| v_{n+1} - v_{n} \|_{C_{\leq \mft_{L}}C_{x}} &\leq M_{v} L_{n}^{4} \delta_{n}^{1/2}, \\
		\label{eq:iter_diffq}
		\| q_{n+1} - q_{n} \|_{C_{\leq \mft_{L}}C_{x}} &\leq M_{q} L_{n} \delta_{n}, \\
		\label{eq:iter_Aestimate} \max \left\{ \| v_{n+1} \|_{C^{1}_{\leq \mft_{L},x}} ,    \| q_{n+1} \|_{C^{1}_{\leq \mft_{L},x}}, \| \mathring{R}_{n+1} \|_{C_{\leq \mft_{L}}C_{x}^{1}}  \right\} &\leq A L_{n+1} \delta_{n}^{1/2} \left( \frac{D_{n}}{\delta_{n+4}} \right)^{1 + \epsilon}.
	\end{align}
	Moreover, the quadruple $(v_{n+1},q_{n+1},\phi_{n+1},\mathring{R}_{n+1})$  evaluated at time $t \in [0,\infty)$ depends only on $e(s)$, $\phi(y,s)$, $v_{k}(y,s)$, $q_{k}(y,s)$, $\phi_{k}(y,s)$, $\mathring{R}_{k}(y,s)$ for arbitrary $s \leq t$, $k \leq n$, and $y \in \T^{3}$, and $v_{n+1}$ has zero mean. 
	
	The (important!) precise form of the constant $a$ can be found in Equ. \eqref{eq:def_a}, $\eta$ in Equ. \eqref{eq:def_eta}, $M_{q}$ in Equ. \eqref{eq:def_Mq}, $M_{v}$ in Section \ref{ssec:choice_Mv}, and $A$ can be found in Equ. \eqref{eq:def_A}.
\end{prop}
	\begin{rem}\label{rem:choice_of_b}
		Let us briefly comment on the form of $b = m + \epsilon$. The idea is that $b > 1$ should be as close to $1$ as possible, i.e. $b = 1 + \epsilon$ for $\epsilon$ small, as in the proof of the Onsager conjecture \cite{Isett18, BDLSV18}. However, due to technical limitations this is not possible at the current stage. In fact, in our case, $\epsilon = 15$ and $m = 23$, so $b = 38$, cf. Section \ref{ssec:choice_of_parameters}. Advances in the methods of stochastic convex integration used here are expected to allow one to take smaller values for $b$.
	\end{rem}

\subsection{Proof of Proposition \texorpdfstring{\ref{prop:step1}}{3.1}} 
Given the main iterative proposition, the proof of Proposition \ref{prop:step1} is obtained as follows.

\begin{proof}[Proof of Proposition \ref{prop:step1}]

Let $e \in \Escr$. We start the convex integration scheme with the initial stage $(v_{0},q_{0},\mathring{R}_{0}) = (0,0,0)$. This triple trivially satisfy Equ. \eqref{eq:Euler_Reynolds}. The Equ. \eqref{eq:iter_stress}, \eqref{eq:iter_press} and \eqref{eq:iter_div}, \eqref{eq:iter_C1} hold trivially, and Equ. \eqref{eq:iter_energy} holds because $\delta_{0} = 1$, so the left-hand side equals zero.

We then apply Proposition \ref{prop:main_iter} iteratively to obtain a sequence of $(v_n)_{n \in \N}$ that satisfy 
\begin{align*}
	\| v_{n} \|_{C_{\leq \mft_{L}} C_{x}} &\leq M_{v} \sum_{k \leq n} L_{k} a^{\frac{1}{2}(1-b^{k})} \leq M_{v} L_{n} \sum_{n \in \N} 2^{\frac{1}{2}(1-4^{k})} \leq 2 M_{v} L_{n}, \\
	\| v_{n} \|_{C_{\leq \mft_{L}} C_{x}^{1}} &\leq L_{n} D_{n},
\end{align*}
where we used the definition of $c$ and the choice $a \geq A^{\frac{1}{\epsilon + 1/2}}$ to get
\begin{align*}
	A \delta_{n}^{1/2} \left( \frac{D_{n}}{\delta_{n+4}} \right)^{1+\epsilon} \leq D_{n+1}.
\end{align*}
Concerning the increments of the velocity field, Proposition \ref{prop:main_iter} ensures
\begin{align*}
	\| v_{n+1} - v_{n} \|_{C_{\leq \mft_{L}} C_{x}} &\leq M_{v} L_{n}^{4} \delta_{n}^{1/2} = M_{v} L^{4m^{n+1}} a^{\frac{1}{2}(1-b^{n})}, \\
	\| v_{n+1} - v_{n} \|_{C_{\leq \mft_{L}} C^{1}_{x}} &\leq 2 L_{n+1} D_{n+1} = 2 L^{m^{n+2}} a^{cb^{n+1}},
\end{align*}
which, by interpolation, implies
\begin{align*}
    \| v_{n+1} - v_{n} \|_{C_{\leq \mft_{L}} C^{\theta}_{x}} &\leq 2^{\theta} M_{v}^{1-\theta}  L^{(1-\theta)4m^{n+1}} L^{\theta m^{n+2}} a^{\frac{1-\theta}{2}} a^{\left( \theta c b - \frac{1-\theta}{2} \right)b^{n}} \\
    &\overset{m \geq 4}{\leq} 2 M_{v}^{1-\theta} L^{m^{n+2}} a^{\frac{1-\theta}{2}} a^{\left( \theta c b - \frac{1-\theta}{2} \right)b^{n}},
\end{align*}
and therefore also for $L = 1$
\begin{equation}\label{Hlder-est}
	||v||_{C^0_{\leq \mft}C^\theta_x} \leq \sum_{n \geq 0}||v_{n+1}-v_n||_{C^0_{\leq \mft}C^\theta_x}\leq  2 M_v^{1-\theta} a^{\frac{1-\theta} 2} \sum_{n\geq 0} a^{(\theta cb -\frac {1-\theta} 2) b^n}.
\end{equation}
Now we choose
\begin{equation}
    0 < \theta < \frac{1}{2cb+1} = \frac{m-1}{2(1+\epsilon)(m+\epsilon)^{5} - 1 - \epsilon} = \alpha_{0},
\end{equation}
which implies $\theta c b - \frac{1-\theta}{2} =: - r < 0$. Then we define
\begin{align*}
    n_{0} := \max \left\{ n \in \N \colon L^{m^{n+2}} > a^{r b^{n-1}}\right\},
\end{align*}
where the maximum exists because $b > m$ and $r > 0$. We see that
\begin{align*}
    n_{0} < \frac{\log \log L + 2 \log m - \log r + \log b - \log \log a}{\log b - \log m} \leq C(1 + \log\log L)
\end{align*}
for some constant $C = C(e)$ not depending on $L$. With this choice of $n_{0}$, we can split the sum and estimate
\begin{align*}
    \sum_{n \in \N} \| v_{n+1} - v_{n} \|_{C_{\leq \mft_{L}} C^{\theta}_{x}} &= \sum_{n \leq n_{0}} \| v_{n+1} - v_{n} \|_{C_{\leq \mft_{L}} C^{\theta}_{x}} + \sum_{n > n_{0}} \| v_{n+1} - v_{n} \|_{C_{\leq \mft_{L}} C^{\theta}_{x}} \\
    &\leq C M_{v}^{1-\theta} a^{1/2} \left( L^{m^{n_{0} + 3}} + 1 \right) \leq C M_{v}^{1-\theta} a^{1/2} L^{m^{C(1 + \log\log L)}},
\end{align*}
where $C$ varies from line to line. This then implies that $v_{n}$ converges in $C_{\leq \mft_{L}} C_{x}^{\theta} \cap C_{\leq \mft_{L}}^{\theta} C_{x}$ to a limit $v$ with a bound uniform in $\omega \in \Omega$:
\begin{align*}
    \| v \|_{C_{\leq \mft_{L},x}^{\theta}} \leq C M_{v}^{1-\theta} a^{1/2} L^{m^{C(1 + \log\log L)}}.
\end{align*}
One can easily convince oneself that $v$ restricted to $(-\infty, \mft_{L-1}]$ is the limit of $v_{n}$ in $C_{\leq \mft_{L-1}} C_{x}^{\theta} \cap C_{\leq \mft_{L-1}}^{\theta} C_{x}$ for every $L \in \N$. In this way, we can uniquely identify a limit in $C_{\mathrm loc} C_{x}^{\theta} \cap C_{\mathrm loc}^{\theta} C_{x}$, which we denote again by $v$, by gluing together limits for different values of $L$.

Note that, as $v_{n}(\cdot \wedge \mft_{L})$ is progressively measurable for every $n \in \N$ and $L \geq 1$, the limit process $v(\cdot \wedge \mft_{L})$ is also progressively measurable, and therefore, $v$ is progressively measurable. As $e(t) > 0$ for all $t$, the above convergence and Equ. \eqref{eq:iter_energy} imply
\begin{align*}
	\int_{\mathbb{T}^{3}} |v(t,x)|^{2} dx = e(t), \quad t \leq \mft.
\end{align*}
Note that $\mft = \mft_{1}$ does not depend on the energy profile $e$.

Now, in the same way as before, we see that $q_{n}$ converges to a H\"older-regular limit. The exact regularity is a bit different, namely as $\| q_{n+1} - q_{n} \|_{C_{\leq \mft_{L}} C} \leq M_{q} L_{n} \delta_{n} = M_{q} L^{m^{n+1}} a^{1-b^{n}}$, we have
\begin{align*}
    \| q_{n+1} - q_{n} \|_{C_{\leq \mft_{L}} C^{\theta'}} \leq 2 M_{q}^{1-\theta'} L^{m^{n+2}} a^{1-\theta} a^{\left( \theta' cb - (1-\theta') \right)b^{n}}.
\end{align*}
This implies that $q$ is H\"older-continuous with H\"older exponent $\theta'$ such that
\begin{equation}
    0 < \theta' < \bar{\theta}' := \frac{1}{cb+1},
\end{equation}
which is approximately twice the H\"older regularity of $v$ for large values of $b$. Up to choosing a smaller $\theta$, we can then suppose $\theta' = 2\theta$ and therefore, the limit of $(q_{n})_{n \in \N}$ satisfies a uniform-in-$\omega$ bound 
\begin{align*}
    \| q \|_{C_{\leq \mft_{L},x}^{2\theta}} \leq C M_{q}^{1-2\theta} a L^{m^{C(1 + \log\log L)}}.
\end{align*}
Finally note that for every $L \in \N$, $\phi_{n} \to \phi$, $\phi^{-1}_{n} \to \phi^{-1}$ in $C_{\leq \mft_{L}} C_{x}^{2}$ almost surely by Equ. \eqref{eq:diff_phi_CbetaCk}, \eqref{eq:diff_phi-1_CbetaCk}, as well as $\mathring{R}_{n} \to 0$ a.s. in $C_{\leq \mft_{L}}C_{x}$ and $\mathrm{div}_{\phi_{n}} v_{n} \to 0$ in $C_{\leq \mft_{L}}B_{\infty,\infty}^{-1}$ a.s. by using the iterative estimates Equ. \eqref{eq:iter_stress}, \eqref{eq:iter_div}, respectively as well as $b > m$. This implies that $(v,q)$ solves Equ. \eqref{PDE} on $[0,\mft_{L}]$. Since $L \in \N$ was arbitrary and $\mft_{L} \to \infty$ as $L \to \infty$, we get global in time solutions. 

Moreover, we see that the stopping time is $\tau = \mft = \mft_1$, which is clearly independent of $e$ and strictly positive, a.s., since $B(0) = 0$. Finally, if we take two energy profiles $e_{1},e_{2}$ with $e_{1}(0) = e_{2}(0)$, we get by Proposition \ref{prop:main_iter} that 
\begin{equation}\label{eq:v1_v2}
v^{(1)}_{n}(0) = v^{(2)}_{n}(0)
\end{equation} 
for all $n \in \N$ and that these initial values only depends on $e(s)$, $\phi(y,s)$, $v_{k}(y,s)$, $q_{k}(y,s)$, $\phi_{k}(y,s)$, $\mathring{R}_{k}(y,s)$ for $k < n$ and $s \leq 0$. As all of these are deterministic, we get that the initial condition for every $n \in \N$ is deterministic. Taking the limit $n \to \infty$ in Equ. \eqref{eq:v1_v2}, we find that $v^{(1)}(0) = v^{(2)}(0)$ and that the initial condition is deterministic.
\end{proof}

\section{Proof of Proposition \texorpdfstring{\ref{prop:energy-choice}}{3.3}}\label{sect:step2}

\begin{proof}[Proof of Proposition \ref{prop:energy-choice}]
	The existence of non-increasing functions $e^\varepsilon_k:\R_+\to \R_+$ with properties (i)-(iv) follows from simple algebraic considerations, and we fix, for every $0< \varepsilon < 1$, such a family.
By Step 1 and assumption, there is a family of solutions $(v^\varepsilon_k)_{k \in \N}$ to \eqref{PDE} with paths in $C(\R_+,C^\theta(\T^3,\R^3))$ $\mathbb{P}$-a.s. for $\theta >\alpha$, and a $\mathbb{P}$-a.s. strictly positive $(\Fscr_t)$-stopping time $\tau: \Omega \to \R_+$, independent of $\varepsilon$ and $k$, such that $\mathbb{P}$-a.s.
	$$||v^\varepsilon_k(t,\omega)||^2_{L^2} =e^\varepsilon_k(t),\quad \forall 0 \leq t \leq \tau(\omega).$$
Since all subsequent estimates involving $e^\varepsilon_k$ and $v^\varepsilon_k$ will only depend on the infimum and supremum of $e^\varepsilon_k$, which is independent of $k$ by (iii) in Proposition \ref{prop:energy-choice}, we now write $e^\varepsilon$ and $v^\varepsilon$ instead of $e^\varepsilon_k$ and $v^\varepsilon_k$.

	It remains to prove \eqref{eq:eps-to-0}.
By \eqref{eq:def_A} and $\bar{e}^\varepsilon = \varepsilon$ and $\underline{e}^\varepsilon = \frac \varepsilon 2$, we have
$$A = 8\tilde{A}_e\bigg(\frac{1 + \bar{e}}{\underline{e}}\bigg)^2.$$
Thus, since $r=9$, $\epsilon = 15$, we obtain (see \eqref{eq:def_a})
\begin{equation}
	a = 2^{\frac{44}{31}}\bigg( \tilde{A}_e\bigg(\frac{1 + \varepsilon}{\varepsilon}\bigg)^3    \bigg)^{\frac{2}{31}}.
\end{equation}
Regarding $M_v$, since $\bar{e}^\varepsilon = \varepsilon$, we see $E = C \varepsilon^{\frac 1 2} < 1$, so that by \eqref{eq:Mv_bounds}, for $\varepsilon$ sufficiently small we have
$$M_v \in \big[\varepsilon^{\frac 1 2}, 2 \varepsilon^{\frac 1 2}\big].$$
In particular $M_v < 1$. Therefore we can estimate (cf. Lemma \ref{lem:choice_A})
	\begin{equation}\label{esttildeA}
1 < 	\tilde{A}_e \leq C \varepsilon ^{-\frac 1 2}.
\end{equation}
Then, revisiting \eqref{Hlder-est}, we see, using the previous expressions and estimates for $a, M_v$ and $\tilde{A}_e$
\begin{align}\label{Hlder-est2}
\notag \sup_{t \in [0,\tau]}	||v^\varepsilon(t)||_{C^\theta_x} &\leq  M_v^{1-\theta} a^{\frac{1-\theta} 2} \sum_{n\geq 0} a^{(\theta cb -\frac {1-\theta} 2) b^n}
\\&
\lesssim \varepsilon^{\frac{1-\theta}{2}}\varepsilon^{-(1-\theta)\frac{3}{31}}\varepsilon^{-(1-\theta)\frac{1}{62}}\sum_{n\geq 0}\tilde{A}_e^{\frac{2}{31}(\theta bc - \frac{1-\theta}{2})b^n}\varepsilon^{\frac{6}{31}(\frac{1-\theta}{2}-\theta b c)b^n}
\\& \notag \lesssim \varepsilon^{(1-\theta)(\frac 1 2 -\frac 3{31}- \frac{1}{62})} \sum_{n\geq 0} \varepsilon^{\frac{6}{31}(\frac{1-\theta}{2}-\theta b c)b^n}.
\end{align}
Since $(1-\theta)(\frac 1 2 - \frac{3}{31}-\frac 1 {62})>0$ and $\frac{6}{31}(\frac{1-\theta}{2}-\theta b c)b^n >0$, we obtain \eqref{eq:eps-to-0}.
\end{proof}

\section{Convex integration scheme}\label{sec:CI_scheme}
It remains to prove the main iterative proposition \ref{prop:main_iter} via convex integration techniques. This is the content of this and the subsequent section.
\subsection{Choice of parameters}\label{ssec:choice_of_parameters}

In this section, we collect the parameters used in the convex integration scheme. First, recall that
\begin{align*}
    \delta_n = a^{1-b^n}, \quad
	D_n = a^{cb^n}, \quad L_{n} := L^{m^{n+1}},
\end{align*}
where
\begin{align*}
    b = m+\epsilon, \quad 
    c =\frac{b^4(1+\epsilon)-\frac{1}{2}}{b-1-\epsilon},
\end{align*}
and the coefficients $a \geq 2, m \geq 4$, $\epsilon > 0$ will be determined in Section \ref{sec:proof_main_iter}. The parameter $\delta_{n}$ decays in $n$ and determines the ``amplitude'' along the iteration of several quantities, cf. Equ. \eqref{eq:iter_energy}--\eqref{eq:iter_div} as well as \eqref{eq:iter_diffv}, \eqref{eq:iter_diffq}. It will ensure that the energy profile is attained, that the stress term and the divergence of $v_{n}$ vanish in the limit (making the limit a solution to the actual, incompressible fractional stochastic Navier--Stokes equations), and that the limiting vector field $v = \sum_{n=0}^{\infty} (v_{n+1} - v_{n})$ and pressure $q = \sum_{n=0}^{\infty} (q_{n+1} - q_{n})$ converge. The parameter $D_{n}$ measures the growth of derivatives (with the heuristic that taking a derivative along the iteration should ``cost'' a factor of roughly $D_{n}$). 
By using the definition, it follows immediately that 
\begin{equation}\label{eq:delta6-5_1-2}
    \delta_{n+2}^{6/5} \delta_{n+1}^{-1/2} \delta_{n}^{-1/2} = a^{1/5} a^{-b^{n}(\frac{6}{5} b^{2} - \frac{1}{2}b - \frac{1}{2})} \overset{b > 1}{<} 1. 
\end{equation}
The parameter $L_{n}$ controls the growth of the iterative estimates on increasingly large time intervals of the form $[0, \mft_{L}]$, $L \in \N$.

We will use two mollification parameters for the convex integration scheme, one for the temporal mollification of the noise, and one for the space-time mollification of the other quantities of the scheme to avoid loss of derivative. To this end, let us fix $\gamma \in (1/3, 1/2)$ close to $1/2$. Then we define
\begin{equation}
    \ell^{\gamma} := \frac{c_{n,\ell}}{C_{\ell}}\frac{\delta_{n+3}^{\frac{4}{3}}}{D_n}, \quad 
	\varsigma_n^{\gamma_*} := \frac{1}{C_{\varsigma}}\frac{\delta_{n+3}^{\frac{4}{3}}}{n+1},
\end{equation}
where $C_{\ell}, C_{\vs} > 1$ are sufficiently large and $\gamma_* \in (0, \gamma)$ will be chosen sufficiently close to $\gamma$, and all independent of $n$, whereas $c_{n,\ell} \in [1,2)$ is chosen such that $\ell^{-1}$ is an integer power of $2$ (to apply Lemma \ref{lem:scaling} in Sections \ref{ssec:Rcomp}, \ref{ssec:Rdiss}).
The additional parameter $\gamma_{*}$ is useful as an application of Equ. \eqref{eq:diff_phi_CbetaCk}--\eqref{eq:phi_CbetaCk} in Section \ref{ssec:Rflow} will require us to pick a further parameter $\gamma' \in (\gamma_{*}, \gamma)$.
From the definitions it is clear that
\begin{equation}\label{eq:ell-sigma-est}
    \ell \leq (n+1) \vs_{n}^{\gamma'}.
\end{equation}
We will need to introduce two further large parameters $\lambda, \mu \in \N$ (depending on $n$) such that $\lambda / \mu \in \N$, which immediately implies that $\lambda \geq \mu$. These two parameters will determine the frequency of space-time oscillations of the building blocks of the convex integration scheme. Following \cite{HLP23}, we choose $r_{*} \geq 7$ and 
\begin{align*}
    \mu := c_{n,\mu}C_{\mu} \ell^{-r_*},
\end{align*}
where again $C_{\mu} > 1$ is possibly large but finite and independent of $n$, and $c_{n,\mu} \in (1,2)$ is chosen such that $\mu \in \N$, and 
\begin{align*}
    \lambda := c_{n,\lambda}\mu^2\varsigma_{n+1}^{\gamma -2}
\end{align*}
where $c_{n,\lambda} \in (1,2)$ is chosen such that $c_{n,\lambda} \vs_{n+1}^{\gamma-2} \in \N$, which implies $\lambda \in \N$ and $\lambda / \mu \in \N$. 

We fix $\eta = \frac{\underline{e}}{8C(1+\sqrt{\bar{e}})}$, where $C$ depends only on the mollifier $\chi$, cf. Equ \ref{eq:def_eta}.
Let us assume without loss of generality the following relations:
\begin{align}
    \label{eq:Dn_ell} D_{n} \ell^{\gamma} &\leq \eta \delta_{n+3}^{9/7} \delta_{n} \leq \eta \delta_{n+1} \leq \eta \delta_{n}, \\
    \label{eq:Dvsest} D_{n} &\leq \vs_{n+1}^{\gamma-1}, \\
    \label{eq:mu-sigma-lambda} \mu^{2} \vs_{n+1}^{\gamma-2} &\leq \lambda \leq D_{n+1}, \\
    \label{eq:vn_97} (n+1)\vs_{n}^{\gamma'} &\leq \delta_{n}^{1/2} \delta_{n+3}^{9/7}, \\
    \lambda^{r} &\geq \mu^{r+5} \vs_{n+1}^{(r+5)(\gamma-1)-2} \left( D_{n} \ell^{-r-4} + \vs_{n+1}^{\gamma-1} \right), \quad \forall r \geq r_{*}, \\
        \label{eq:mu6-5} \frac{1}{\mu} &\leq \delta_{n+2}^{6/5}.
\end{align}
The last condition immediately implies, as by definition $\lambda = c_{n,\lambda} \mu^{2} \vs_{n+1}^{\gamma-2} > \mu^{2} \vs_{n+1}^{\gamma-2}$ and $\vs_{n+1} \leq 1$, that
\begin{equation}\label{eq:mulambda6-5}
    \frac{\mu \vs_{n+1}^{\gamma-1}}{\lambda} \leq \frac{\vs_{n+1}}{\mu} \leq \frac{1}{\mu} \overset{\text{Equ. \eqref{eq:mu6-5}}}{\leq} \delta_{n+2}^{6/5} \overset{\text{Equ. \eqref{eq:delta6-5_1-2}}}{\leq} \delta_{n+1}^{1/2} \delta_{n}^{1/2}.
\end{equation}
The first inequality in \eqref{eq:Dn_ell} will be used just once, namely in the proof of iterative estimate Equ. \eqref{eq:iter_div} using Proposition \ref{prop:divergence} below.

Notice that the condition $\lambda \leq D_{n+1}$ of Equ. \eqref{eq:mu-sigma-lambda} requires that the parameters $A$ and $\epsilon$ in Proposition \ref{prop:main_iter} are sufficiently large, cf. Equ. \eqref{eq:lambda_A} and Equ. \eqref{eq:A_Dn} below. More details on this and on the other parameter choices will be given in Section \ref{sec:proof_main_iter}.

\subsection{Mollification}

Let $\chi \in C_{c}^{\infty}([0,1) \times [-1,1]^{3})$ be a standard mollifier. Set $\chi_{\ell}(t,x) := \ell^{-4} \chi(\ell^{-1}t, \ell^{-1}x)$ and 
\begin{align*}
    v_{\ell} := v_{n} * \chi_{\ell}, \quad q_{\ell} := q_{n} * \chi_{\ell}, \quad \mathring{R}_{\ell} := \mathring{R}_{n} * \chi_{\ell}.
\end{align*}
Note that $v_{n}, q_{n}$ and $\mathring{R}_{n}$ are also defined for negative times and thus satisfy \eqref{eq:Euler_Reynolds} in an analytically strong sense. We shall need in the following that $v_{\ell}, q_{\ell}, \mathring{R}_{\ell}$ satisfy the equation
\begin{equation}
    \partial_{t} v_{\ell} + \chi_{\ell} * \divv_{\phi_{n}}(v_{n} \otimes v_{n}) + \chi_{\ell} * \nabla_{\phi_{n}} q_{n} + \chi_{\ell} * (-\D)^{\alpha}_{\phi_{n}} v_{n} = \chi_{\ell} * \divv_{\phi_{n}} \mathring{R}_{n}.
\end{equation}
Using standard mollification estimates, we have
\begin{align}
  \label{eq:mollif-diff}\| v_{\ell} - v_{n} \|_{C_{\leq \mft_{L}}^{0}C_{x}^{0}} &\leq C \ell \| v_{n} \|_{C_{\leq \mft_{L},x}^{1}} \overset{\text{Equ. \eqref{eq:iter_C1}}}{\leq} C \ell L_{n} D_{n} \overset{\text{Equ. \eqref{eq:Dn_ell}}}{\leq} C \eta L_{n} \delta_{n},
  \\
  \label{eq:mollif-C0} \| v_{\ell} \|_{C_{\leq \mft_{L}}^{0}C_{x}^{0}} &\leq \| v_{n} \|_{C_{\leq \mft_{L}}^{0} C_{x}^{0}} \overset{\text{Equ. \eqref{eq:iter_diffv}}}{\leq} M_{v} L_{n-1}^{4} \sum_{k=0}^{n-1} \delta_{k}^{1/2} \leq 2 M_{v} L_{n-1}^{4} \leq 2 M_{v} L_{n}, \\
  \label{eq:mollif-C1-1} \| v_{\ell} \|_{C_{t,x}^{1}} &\leq \| v_{n} \|_{C_{t,x}^{1}} \overset{\text{Equ. \eqref{eq:iter_C1}}}{\leq} L_{n} D_{n} \\
  \label{eq:mollif-C1-2} \| v_{\ell} \|_{C_{\leq \mft_{L},x}^{1}} &\leq C \ell^{-1} \| v_{n} \|_{C_{\leq \mft_{L},x}^{0}} \overset{\text{Equ. \eqref{eq:mollif-C0}}}{\leq} C \ell^{-1} M_{v} L_{n}  \\
  \| v_{\ell} \|_{C_{\leq \mft_{L}}^{0} C_{x}^{N+1}} &\leq C_{N} \ell^{-N} L_{n} D_{n}, \\
  \nonumber \| v_{\ell} \|_{C_{t}^{0} C_{x}^{\delta}} &\overset{\text{Equ. \eqref{eq:Zygmund-Holder}}}{\leq} C\left( \| v_{\ell} \|_{C_{\leq \mft_{L}}^{0} C_{x}^{0}} + \sup_{t} [v_{\ell}]_{C_{x}^{\delta}} \right) \leq C  \left( \| v_{\ell} \|_{C_{\leq \mft_{L},x}^{0}} + \| v_{\ell} \|_{C_{\leq \mft_{L},x}^{0}}^{1-\delta} \| v_{\ell} \|_{C_{\leq \mft_{L}}^{0} C_{x}^{1}}^{\delta} \right) \\
  \label{eq:mollif-interpol} &\overset{\text{Equ. \eqref{eq:mollif-C1-2}}}{\leq} C \ell^{-\delta} \| v_{n} \|_{C_{\leq \mft_{L},x}^{0}} \leq C M_{v} L_{n} \ell^{-\delta}, \quad \delta \in [0,1].
\end{align}
Similarly to the last estimate, we get from Equ. \eqref{eq:iter_C1} and \eqref{eq:iter_diffv} that
\begin{equation}\label{eq:vn_interpol}
\begin{split}
    \| v_{n} \|_{C_{\leq \mft_{L}}^{0} C_{x}^{\delta}} &\overset{\text{Equ. \eqref{eq:Zygmund-Holder}}}{\leq} C  \left( \| v_{n} \|_{C_{\leq \mft_{L},x}^{0}} + \| v_{n} \|_{C_{\leq \mft_{L},x}^{0}}^{1-\delta} \| v_{n} \|_{C_{\leq \mft_{L}}^{0} C_{x}^{1}}^{\delta} \right) \\
    &\leq C\left( M_{v} + M_{v}^{1-\delta} \right) L_{n} D_{n}^{\delta}, \quad \delta \in (0,1).
\end{split}
\end{equation}

\subsection{Modified Beltrami flows}
We follow \cite[Section 3.3]{HLP23} in defining the fundamental building blocks of our convex integration scheme, based on the Beltrami flows of \cite{DLS13}.
Let $k \in \Z^{3} \backslash \{ 0 \}$. We define
\begin{align*}
    M_{k} := \Id - \frac{k}{|k|} \otimes \frac{k}{|k|}. 
\end{align*}
Let $\lambda_{0} \geq 1$ and $A_{k} \in \R^{3}$ such that $A_{k} \cdot k = 0$, $|A_{k}| = \frac{1}{2}$ and $A_{-k} = A_{k}$ for $k \in \Z^{3}$ with $|k| = \lambda_{0}$. Let us further define the complex vectors
\begin{align*}
    E_{k} := A_{k} + i \frac{k}{|k|} \times A_{k} \in \C^{3}.
\end{align*}
By the symmetry condition on the coefficients, for any matching collection $\{ a_{k} \}_{k \in \Z^{3}, |k| = \lambda_{0}}$ of complex numbers $a_{k} \in \C$ such that $a_{-k} = \bar{a}_{k}$ for every $k$ ($\bar{a}_k$ denotes the complex conjugate of $a_k$), the vector field
\begin{align*}
    E(t,x) := \sum_{|k| = \lambda_{0}} a_{k} E_{k} e^{ik \cdot \phi_{n+1}(t,x)}
\end{align*}
is real-valued and satisfies the relations
\begin{align*}
    \divv_{\phi_{n+1}} E = 0, \quad \divv_{\phi_{n+1}}(E \otimes E) = \nabla_{\phi_{n+1}} \left( \frac{|E|^{2}}{2} \right).
\end{align*}
Moreover, the fact that $\phi_{n+1}$ is measure-preserving implies that
\begin{align*}
    \frac{1}{(2\pi)^{3}} \int_{\T^{3}} E \otimes E ~dx = \frac{1}{2} \sum_{|k| = \lambda_{0}} |a_{k}|^{2} M_{k}.
\end{align*}
Following \cite{HLP23}, we shall call $E$ a \textit{modified Beltrami wave}. Note that since $E$ is time-dependent, it is not a solution of the stationary Euler equations, unlike the classical Beltrami waves of \cite{DLS13}.

\subsection{The transport coefficients}
We introduce a suitable decomposition of the state space of the velocity field: let $c_1, c_2$ be two universal constants such that $\frac{\sqrt{3}}{2} < c_1 < c_2 < 1$, and let $\varphi \in C_c^{\infty}(B_{c_2}(0))$ identical to 1 on $B_{c_1}(0)$. 
Next, let $\mathcal{C}_j, j=1,...,8,$ denote the eight equivalence classes of $\mathbb{Z}^3 / (2\mathbb{Z})^3$, then by defining
\[\varphi_k(x) := \varphi(x-k), \quad k \in \mathbb{Z}^3,\]
observe that if $k \neq l \in \mathcal{C}_j$, then $|k-l| \geq 2 > 2c_2$ and $\varphi_k$ and $\varphi_l$ have disjoint supports.\\
Furthermore define the function
\[\varphi_{\sum}:= \sum_{k \in \mathbb{Z}^3}\varphi_k^2\]
which is smooth, bounded and bounded away from zero. Hence for $v \in \mathbb{R}^3, \tau \in \mathbb{R}$ set
\begin{align*}
	\alpha_k(v) &:= \frac{\varphi_k(v)}{\sqrt{\varphi_{\sum}}},\quad k \in \mathbb{Z}^3,\\
	\psi_k^{(j)}(v,\tau) &:= \sum_{l \in \mathcal{C}_j}\alpha_l(\mu v) e^{-{\rm i}k \cdot \frac{l}{\mu}\tau}, \quad k \in \mathbb{Z}^3, j = 1,..., 8.
\end{align*}
As in \cite[Section 3.4]{HLP23}, we make use of the following properties:
\[\sum_{j=1}^8 \left| \psi_k^{(j)}(v,\tau)\right|^2 =1\]
and for any $r \in \mathbb{N}, k \in \mathbb{Z}^3$ and $j = 1,..., 8$ it holds that
\begin{align*}
	\sup_{v,\tau} \left| D_v^r \psi_k^{(j)}(v,\tau)\right| &\leq C(r, |k|) \mu^r, \quad r \in \mathbb{N},\\
	\sup_{|v| \leq V, \tau} \left| D_v^r \partial_{\tau}^h\psi_k^{(j)}(v,\tau)\right| &\leq C(r, h, |k|) V^h \mu^r,\\
	\sup_{|v| \leq V, \tau} \left| D_v^r \left(\partial_{\tau}\psi_k^{(j)}(v,\tau) +{\rm i} (k \cdot v) \psi_k^{(j)}(v,\tau)\right)\right| &\leq C(r, |k|) \mu^{r-1}
\end{align*}
for any given $V > 0$.

\subsection{The energy pumping term}\label{subsec:energypump}
In view of prescribing the energy profile $e$, we introduce the following term to track the energy of the solutions constructed at each iteration:
\[\tilde{e}(t) := \tilde{e}_{n+1}(t) := \frac{1}{3(2\pi)^3}\left(e(t)\left(1-\delta_{n+1}\right)-\int_{\mathbb{T}^3} |v_{\ell}(t,x)|^2{\rm d}x\right).\]
Then, recalling from \cite[Lemma 3.1]{HLP23}, for 
\begin{equation}\label{eq:def_eta}
\eta := \frac{\underline{e}}{8C(1+\sqrt{\bar{e}})}, 
\end{equation}
where $C$ only depends on the mollifier, we obtain the bounds almost surely for all $t \leq \mathfrak{t}$
\begin{align*}
	\tilde{e}(t) &\leq \frac{1}{3(2\pi)^3} \left(\frac{5}{4}\bar{e} + C \eta(1+\sqrt{\bar{e}})\right)\delta_n,\\
	\tilde{e}(t) &\geq \frac{\underline{e}}{24(2\pi)^3} \delta_n.
\end{align*}
Moreover, let $\gamma_{n} \colon [0,\infty) \to \R_{+}$ be a function such that
\begin{equation}
    \label{eq:def:gamma_n}
    \begin{split}
     	\gamma_n(t) &= \tilde{e}(t), \quad \forall t \leq \mathfrak{t},\\
	\frac{1}{2} \tilde{e}(\mathfrak{t}) &\leq \gamma_n(t) \leq \frac{3}{2} \tilde{e}(\mathfrak{t}), \quad \forall t \geq \mathfrak{t},\\
	\| \gamma_n\|_{C_t^k} &\leq C \| \tilde{e}_{\leq \mathfrak{t}}\|_{C_t^k}, \quad \forall k \leq 2.
    \end{split} 
\end{equation}

\subsection{The perturbations}
Recall the geometric lemma of \cite[Lemma 3.2]{DLS13}: 
\begin{lem}\label{lem:geometric}
    For every $N \in \N$, there exists an $r_{0} > 0$ and $\lambda_{0} > 1$ such that there exist pairwise disjoint, finite subsets $\Lambda_{j} \subset \{ k \in \Z^{3} \colon |k| = \lambda_{0} \}$, for $j \in \{ 1, \ldots, N \}$, as well as smooth positive functions
    \begin{align*}
        \gamma_{k}^{(j)} \in C^{\infty}\left(B_{r_{0}}(\Id) \right), \quad j \in \{ 1, \ldots, N \}, k \in \Lambda_{j},
    \end{align*}
    such that $k \in \Lambda_{j}$ implies $-k \in \Lambda_{j}$ as well as $\gamma_{-k}^{(j)} = \gamma_{k}^{(j)}$. For each $R \in B_{r_{0}}(\Id)$ we have the identity 
    \begin{align*}
        R = \frac{1}{2} \sum_{k \in \Lambda_{j}} \left(\gamma_{k}^{(j)}(R) \right)^{2} M_{k}, \quad \forall R \in B_{r_{0}}(\Id).
    \end{align*}
\end{lem}
We apply the lemma with $N = 8$ to obtain $\lambda_{0} > 1$, $r_{0} > 0$ and pairwise disjoint families $\Lambda_{j}$ as well as positive functions $\gamma_{k}^{(j)} \in C^{\infty}\left(B_{r_{0}}(\Id) \right)$ for $k \in \Lambda_{j}, j \in \{1, \ldots, 8\}$. Define
\begin{equation}\label{eq:def-W}
	W(s, y, \tau, \xi) = \sum_{k \in \Lambda} a_k(s,y,\tau) E_k e^{{\rm i} k \cdot \xi},
\end{equation}
where $\Lambda = \bigcup_{j=1}^{8} \Lambda_{j}$, and we define the amplitude coefficients as (recall that $\gamma_{n}$ was defined in Equ. \eqref{eq:def:gamma_n})
\begin{align}
	\nonumber a_k(s,y,\tau) &:= 1_{\{ k \in \Lambda_{j} \}} \sqrt{\rho_{\ell}(s,y)} \underbrace{\gamma_k^{(j)}\left(\frac{R_{\ell}(s,y)}{\rho_{\ell}(s,y)}\right)}_{= \Gamma(s,y)} \underbrace{\psi_k^{(j)}(\tilde{v}(s,y),\tau)}_{=\Psi(s,y,\tau)},\\
	\nonumber R_{\ell}(s,y) &:= \rho_{\ell}(s,y){\rm Id} - \mathring{R}_{\ell}(s,y),\\
	\label{eq:def_rhol} \rho_{\ell}(s,y) &:= \tilde{\rho}_{\ell}(s,y) + \gamma_n(s),\\
	\nonumber \tilde{\rho}_{\ell}(s,y) &:= \frac{2}{r_0} \sqrt{\eta^2 \delta_{n+1}^2 + |\mathring{R}_{\ell}(s,y)|^2}, \\
	\tilde{v}(s,y) &:= v_{\ell}(y,s) + \dot{\phi}_{n+1}(y,s).
\end{align}
First, we summarize the estimates on the functions $\tilde{\rho}_{\ell}$.
\begin{lem}\label{claim-energy-1}
	We have 
\begin{align*}
	\|\tilde{\rho}_{\ell}\|_{C_{\leq \mft_{L}}C_{x}} &\leq \frac{4}{r_0} L_n \eta \delta_{n+1},\\
	\|\tilde{\rho}_{\ell}\|_{C_{\leq \mft_{L},x}^1} &\leq C L_{n} \ell^{-1} \eta \delta_{n+1}.
\end{align*}
\end{lem}
For the proof see Section \ref{sec-proof-core-est}.

\noindent
Hence we immediately obtain by the above lemma and the choice of $\eta$ in Section \ref{subsec:energypump}
\begin{align*}
	\|\rho_{\ell}\|_{C_{\leq \mft_{L}}C_{x}} &\leq \frac{4}{r_0} L_{n} \eta \delta_{n+1} + \frac{1}{3(2\pi)^3} \left(\frac{5}{4}\bar{e} + C \eta(1+\sqrt{\bar{e}})\right)\delta_n \\
	&\leq C L_{n} \left( \bar{e} +  \eta(1+\sqrt{\bar{e}})\right)\delta_n \leq C L_{n} \bar{e} \delta_n,\\
	\|\rho_{\ell}\|_{C_{t}C_{x}} &\geq C\underline{e} \delta_n.
\end{align*}
In the course of the convex integration analysis we need to control various norms of the amplitude $a$ which we collect in the following proposition. The proof closely follows \cite{HLP23}, however in our context, we need to specify the exact dependence on the energy profile. 

\begin{prop}\label{prop-energy-1}
For $\delta \in [0,1]$ and $t \in [0,\mft_{L}]$
\begin{align}
	\label{eq:a_k_Cdelta} \|a_k\|_{C_y^{\delta}} &\leq C_e^{(1),\delta} {L_n^{\frac{5}{2}}}\mu^{\delta}\varsigma_{n+1}^{\delta(\alpha-1)} \delta_n^{\frac{1}{2}},\\
	\label{eq:dtau_a_k_Cdelta} \|\partial_{\tau} a_k\|_{C_y^{\delta}} &\leq C_e^{(2),\delta}{(1+M_v)} {L_n^\frac{7}{2}} \mu^{\delta}\varsigma_{n+1}^{(\delta+1)(\alpha-1)} \delta_n^{\frac{1}{2}},\\
	\label{eq:dtau_ik_a_k_Cdelta} \|\partial_{\tau} a_k + {\rm i} (k \cdot \tilde{v}) a_k\|_{C_y^{\delta}} &\leq C_e^{(3),\delta}{L_n^{\frac{5}{2}}}\mu^{\delta-1}\varsigma_{n+1}^{\delta(\alpha-1)} \delta_n^{\frac{1}{2}},\\
	\label{eq:ds_a_k_Cdelta} \|\partial_s a_k\|_{C_y} &\leq C_e^{(4),0}{L_n^{\frac{5}{2}}}\mu \varsigma_{n+1}^{\alpha-2} \delta_n^{\frac{1}{2}},
\end{align}
and for $r \in \mathbb{N}$ and $t \in [0,\mft_{L}]$,
\begin{align}
	\label{eq:a_k_Cr} \|a_k\|_{C_y^r} &\leq C_e^{(5),r}{L_n^{r+\frac{3}{2}}}\delta_n^{\frac{1}{2}} \mu^r \varsigma_{n+1}^{r(\alpha-1)},\\
	\label{eq:dtau_a_k_Cr} \|\partial_{\tau} a_k\|_{C_y^r} &\leq C_e^{(6),r}{(1+M_v) L_n^{r + 2}}\mu^r\varsigma_{n+1}^{(r+1)(\alpha-1)} \delta_n^{\frac{1}{2}},\\
	\label{eq:dtau_ik_a_k_Cr} \|\partial_{\tau} a_k + {\rm i} (k \cdot \tilde{v}) a_k\|_{C_y^r} &\leq C_e^{(7),r}{L_n^{r+1}}\mu^{r-1}\varsigma_{n+1}^{r(\alpha-1)} \delta_n^{\frac{1}{2}},\\
	\label{eq:ds_a_k_Cr} \|\partial_s a_k\|_{C_y^r} &\leq C_e^{(8),r} {L_n^{2r+\frac{5}{2}}}\mu^{r+1} \varsigma_{n+1}^{r(\alpha-1)-1} \delta_n^{\frac{1}{2}}\left(D_n\ell^{1-r} +\varsigma_{n+1}^{\alpha-1}\right),
\end{align}
where
\begin{align*}
	C_e^{(1),\delta} &= C\left(\sqrt{\bar{e}}+ \underline{e}^{-\frac{1}{2}} \eta^{1-\delta} + \sqrt{\bar{e}}\eta^{1-\delta} \underline{e}^{-1}\right),\\
	C_e^{(2),\delta} &\sim C_e^{(1),\delta}, \\
	C_e^{(3),\delta} &\sim C_e^{(1),\delta},\\
	C_e^{(4),0} &= C \left(\frac{\eta(1+\sqrt{\bar{e}}) + |e|_{C^1}}{\sqrt{\underline{e}}} + \sqrt{\bar{e}} \left(1 + \frac{\eta}{\underline{e}}\left(1 + \frac{\eta}{\underline{e}}(1+\sqrt{\bar{e}} + \eta^{-1}|e|_{C^1})\right)\right)\right), \\
	C_e^{(5),r} &= C\sqrt{\bar{e}}\left( 1 +C_e^{(r)} \eta  + \eta^2 \sum_{j=1}^{r-1}  C_e^{(j)} C_e^{(r-j)} \right),\\
	C_e^{(6),r} &\sim C_e^{(5),r},\\
	C_e^{(7),r} &\sim C_e^{(5),r},\\
	C_e^{(8),r} &= C_e^{\partial_s (1),r} + C_e^{\partial_s (2),r} + C_e^{\partial_s (3),r},
\end{align*}
with 
\begin{align*}
	C_e^{(r)} &= 1 + \frac{1+\bar{e}}{\bar{e}} \sum_{j=1}^r \left(\frac{\bar{e}}{\underline{e}}\right)^j \sim \frac{1+\bar{e}}{\bar{e}} \left(\frac{\bar{e}}{\underline{e}}\right)^r,\\
	C_e^{\partial_s (1),r}&= C \left(\frac{\eta(1+\sqrt{\bar{e}}) + |e|_{C^1}}{\sqrt{\underline{e}}} + C_e^{\partial_s\sqrt{\rho},r}\right)\left(1+ \eta C_e^{(r)}\right),\\
	C_e^{\partial_s (2),r} &= C\sqrt{\bar{e}}\left(1 + \eta C_e^{(r)}\right)\left(\left(1+\frac{\eta}{\underline{e}}(1+\sqrt{\bar{e}} + \eta^{-1}|e|_{C^1})\right)\frac{\eta}{\underline{e}} + C_e^{\partial_s\Gamma, r} \right),\\
	C_e^{\partial_s (3),r} &= C\sqrt{\bar{e}}\left(1 + \eta C_e^{(r)} + \eta^2 \sum_{j=1}^{r-1} C_e^{(j)}C_e^{(r-j)} \right),\\
	C_e^{\partial_s \sqrt{\rho}, r} &= C \frac{\eta}{\sqrt{\underline{e}}}\left(1+ C_e^{(r)}(\eta(1+\sqrt{\bar{e}})+|e|_{C^1})\right),\\
	C_e^{\partial_s \Gamma,r} &= C\left(\frac{\eta}{\underline{e}}\left(1+\eta C_e^{(r)}\right) +\eta(1+\eta)\left(\frac{\eta}{\underline{e}}\right)^2(1+\sqrt{\bar{e}} +\eta^{-1}|e|_{C^1})\left(C_e^{(r)} + \sum_{i=1}^{r-1} C_e^{(i)}C_e^{(r-i)}\right)\right).
\end{align*}
\end{prop}

For the proof, see Section \ref{sec-proof-core-est}. Furthermore, note that the constants $C_{e}^{(r)}$ and $C_{e}^{(5),r}$ are monotone nondecreasing in $r$. Moreover, by this monotonicity, we see that all the $r$-dependent constants except for $C_e^{(1),r}, C_e^{(2),r}$ and $C_e^{(3),r}$ are monotone nondecreasing in $r > 0$. \\
Denote $C_e^{(j),r+\delta} := (C_e^{(j),r+1})^{\delta}(C_e^{(j),r})^{1-\delta}$. Furthermore it holds for $r \in \N$ and $t \in [0,\mft_{L}]$,
\begin{align*}
	\|(v_{\ell}\cdot\nabla_y)a_k\|_{C_y} &\leq C \|v_{\ell}\|_{C_y}\|a_k\|_{C_y^1} \leq C C_e^{(1),1} \mu \varsigma_{n+1}^{\alpha-1} {M_vL_n^{\frac{7}{2}}} \delta_n^{\frac{1}{2}},\\
	[(v_{\ell}\cdot\nabla_y)a_k]_{C_sC_y^r} &\leq C \|v_{\ell}\|_{C_y}[a_k]_{C_y^{r+1}} + C\sum_{j=1}^r[v_{\ell}]_j[a_k]_{C_sC_y^{r+1-j}}\\
	&\leq C{M_vL_n^{r + \frac{5}{2}}}  C_e^{(5),r+1} \mu^{r+1}\varsigma_{n+1}^{(r+1)(\alpha-1)} \delta_n^{\frac{1}{2}}\\
	&\quad + C \sum_{j=1}^r {M_vL_n^{r-j + \frac{5}{2}}}\ell^{-j} C_e^{(5),r+1-j} \mu^{r+1-j}\varsigma_{n+1}^{(r+1-j)(\alpha-1)} \delta_n^{\frac{1}{2}}\\
	&\leq C \left(\sum_{j=1}^{r+1}C_e^{(5),j}\right) {M_vL_n^{r + \frac{5}{2}}} \mu^{r+1}\varsigma_{n+1}^{(r+1)(\alpha-1)} \delta_n^{\frac{1}{2}}\\
	&\overset{C_{e}^{(5),j} \text{ monotone}}{\leq} CC_e^{(5),r+1} {M_vL_n^{r + \frac{5}{2}}} \mu^{r+1}\varsigma_{n+1}^{(r+1)(\alpha-1)} \delta_n^{\frac{1}{2}}.
\end{align*}

\subsubsection{The principal part}\label{sssec:w_o} 

We define the principal part of the perturbation using the functions defined in the previous section:
\begin{align*}
	w_o(t,x) := W(t,x, \lambda t, \lambda \phi_{n+1}(t,x)).
\end{align*}
By choice of $\gamma_k^{(j)}, \psi_k^{(j)}$, one can immediately see that for $t \in [0,\mft_{L}]$
\begin{equation}\label{eq:wo_C0}
 |w_o(t,x)| \leq C \sum_k |a_k(t,x,\lambda t)| \leq C\sqrt{|\rho_{\ell}(t,x)|}\leq C {L_n^{\frac{1}{2}}} \sqrt{\bar{e}} \delta_n^{\frac{1}{2}}.
\end{equation}
From \cite[Corollary 4.2]{HLP23}, we obtain
\begin{prop}\label{prop-osc-W}
Let $W = W(y,s,\xi\,\tau)$ be defined by \eqref{eq:def-W}. Then we have
\begin{equation}\label{eq:WotimesW} W\otimes W (y,s,\xi,\tau) = R_{\ell}(y,s) + \sum_{1 \leq |k| \leq 2\lambda_{0}} U_k(y, s,\tau) e^{{\rm i}k \cdot \xi},
\end{equation}
where $U_{k} \in C^{\infty}_{\rm loc}(\T^{3}\times \R^{2}, \mathcal{S}^{3 \times 3})$, $k \in \Lambda$, satisfies $U_{k} k = \frac{1}{2} \Tr(U_{k})k$. Moreover, for any fixed $s \leq \mft_{L}, \tau \leq \lambda \mft_{L}$, $L \geq 1$, and for every $\delta \in (0,1], r \in \mathbb{N}, 2 \leq r \leq r_{*}+5$, we have
\begin{align*}
	\|U_k(\cdot, s,\tau)\|_{C_x^{\delta}} &\leq  C {L_n^3}\sqrt{\bar{e}} C_e^{(1),\delta} \mu^{\delta}\varsigma_{n+1}^{\delta(\alpha-1)} \delta_n,\\
	\|U_k(\cdot, s,\tau)\|_{C_x^r} &\leq  C {L_n^{2r+3}}C_e^{(7),r}\mu^r\varsigma_{n+1}^{r(\alpha-1)} \delta_n.
\end{align*}
\end{prop}
For the proof, see Section \ref{sec:pf_wo}. Moreover, we have the following estimates for $w_{o}$.
\begin{lem} \label{lem:w_o}
Fix $r \geq r_{*} + 2$. Then for $\delta \in (0,1)$ sufficiently small it holds 
\begin{align*}
	\|w_o\|_{C_{\leq \mft_{L}}C_{x}^{\delta}}&\leq C {L^{2\delta}}{L_n^{\frac{5}{2}}}C_e^{(1),\delta} \lambda^{\delta}\delta_n^{\frac{1}{2}},\\
	\|w_o\|_{C_{\leq \mft_{L}} C_{x}^{1+\delta}} &\leq C {L^{2(1+\delta)}}{L_n^{\frac{5}{2} + 2\delta}}(C_e^{(1),1} + (C_e^{(5),2})^{\delta}(C_e^{(1),1})^{1-\delta}) \lambda^{1+\delta}\delta_n^{\frac{1}{2}},\\
	\|\partial_t w_o\|_{C_{\leq \mft_{L}} C_{x}^{0}} &\leq C_e^{\partial_t w_o} {L^{3(1+\delta)}L_n^{\frac{9}{2}}}\lambda^{1+\delta} \delta_n^{\frac{1}{2}},
\end{align*}
with $C_e^{\partial_t w_o}$ specified in the proof.
\end{lem}
The proof can also be found in Section \ref{sec:pf_wo}.

\subsubsection{The corrector terms}
Let $\mathcal{P} := I - \mathcal{Q}$ be the classical Leray--Helmholtz projector on zero-mean, divergence-free vector fields. As before, define the corresponding flowed operators
\begin{align*}
    \mathcal{P}_{\phi_{n+1}} v := \left[\mathcal{P} ( v \circ \phi_{n+1}^{-1}) \right] \circ \phi_{n+1}, \quad \mathcal{Q}_{\phi_{n+1}} v := \left[\mathcal{Q} ( v \circ \phi_{n+1}^{-1}) \right] \circ \phi_{n+1}.
\end{align*}
Then using the definitions we can see that 
\begin{align*}
    \divv_{\phi_{n+1}} \mathcal{P}_{\phi_{n+1}} &= \left[ \divv (\mathcal{P}(v \circ \phi_{n+1}^{-1})\right] \circ \phi_{n+1} = 0, \\
    \divv_{\phi_{n+1}} \mathcal{Q}_{\phi_{n+1}} &= \left[ \divv (\mathcal{Q}(v \circ \phi_{n+1}^{-1})\right] \circ \phi_{n+1} = \divv_{\phi_{n+1}} v.
\end{align*}
Further, denoting by $\psi$ the zero-mean solution of the Poisson equation $\D_{\phi_{n+1}} \psi = \divv_{\phi_{n+1}} v$, where $\D_{\phi_{n+1}} v := [\D (v \circ \phi_{n+1}^{-1})] \circ \phi_{n+1}$, we have the alternative representation
\begin{align*}
    \mathcal{Q}_{\phi_{n+1}}v = \nabla_{\phi_{n+1}}\psi + \frac{1}{(2\pi)^{3}} \int_{\T^{3}} v dx,
\end{align*}
if $v \in C^{\infty}(\T^{3};\R^{3})$. As in \cite{HLP23}, the strategy to control $\divv_{\phi_{n+1}} v_{n+1}$, is to recall $v_{n+1} = v_{\ell} + w_{o} + w_{c}$ for an as yet undetermined $w_{c}$, and think of the flowed divergences of the two terms known so far. For the first term it is
\begin{align*}
    \divv_{\phi_{n+1}} v_{\ell} &= \divv_{\phi_{n+1}} v_{\ell} - \left( \divv_{\phi_{n+1}} v_{n} \right) * \chi_{\ell} + \left( \divv_{\phi_{n+1}} v_{n} \right) * \chi_{\ell} - \left( \divv_{\phi_{n}} v_{n} \right) * \chi_{\ell} \\
    &\qquad +  \left( \divv_{\phi_{n}} v_{n} \right) * \chi_{\ell}.
\end{align*}
The terms of the first line look like they might be controllable via mollification estimates and closeness of the flow maps $\phi_{n+1}, \phi_{n}$, so only the last term creates problems. We want to compensate for this term while still having a manageable time derivative. This led in \cite{HLP23} to the choice 
\begin{align*}
    w_c^1 := - \left(\mathcal{Q}^{\phi_n}v_{\ell}\right)\ast \chi_{\ell},
\end{align*}
which turns out to be what is needed\footnote{Since we want to control the divergence here, it is enough to project the divergence-free part of $v_{\ell}$ away and hence use $\mathcal{Q}_{\phi_{n}}$!}. The other term to be controlled is the principal part, and we do this in the ``easy'' way by requiring
\begin{align*}
    \divv_{\phi_{n+1}}(w_{o} + w_{c}^{2}) = 0,
\end{align*}
i.e.
\begin{align*}
    w_c^2 := - \mathcal{Q}^{\phi_{n+1}}w_o.
\end{align*}
We then define the total corrector term as
\begin{align*}
	w_c &:= w_c^1 + w_c^2.
\end{align*}
The main estimates for the corrector terms are summarised in the following lemma.
\begin{lem}\label{lem:w_c}
For $\delta \in (0,1)$ sufficiently small, we have
\begin{align*}
	\|w_c^1\|_{C_{\leq \mft_{L}}C_{x}^{\delta}} &\leq C {L^7 L_n (1+M_v^{2\delta})}\delta_{n+2}^{\frac{6}{5}},\\
	\|w_c^1\|_{C_{\leq \mft_{L}} C_{x}^{1+\delta}} &\leq C {L^7 L_n (1+M_v^{2\delta})}\ell^{-1}\delta_{n+2}^{\frac{6}{5}},\\
	\|w_c^2\|_{C_{\leq \mft_{L}} C_{x}^{\delta}} &\leq C {L^{1+2\delta}}{L_n^{\frac{5}{2} + \delta}}(C_e^{(1),1} + C_e^{(5),1+\delta})\lambda^{\delta-1}\mu\varsigma_{n+1}^{\gamma-1} \delta_n^{\frac{1}{2}},\\
	\|w_c^2\|_{C_{\leq \mft_{L}} C_{x}^{1+\delta}} &\leq C {L^{3+2\delta}}{L_n^{\frac{7}{2} + \delta}}(C_e^{(1),1} + C_e^{(5),2+\delta})\lambda^{\delta}\mu\varsigma_{n+1}^{\gamma-1} \delta_n^{\frac{1}{2}}.
\end{align*}
The definitions of the energy-dependent constants can be found in Proposition \ref{prop-energy-1}.
\end{lem}
The proof will be presented in Section \ref{sec-proof-corr-wc}.

\subsubsection{The total perturbation} 
Finally, we define the new velocity field by
\begin{equation}\label{eq:def-vnplus1}
    v_{n+1} := v_{\ell} + w_{n+1} := v_{\ell} + w_{o} + w_{c},
\end{equation}
where $w_{n+1} := w_{o} + w_{c}^{1} + w_{c}^{2}$ is the total perturbation. The estimates for these are given in the following lemma.
\begin{lem}\label{lem:total_perturb}
    We have for $\delta \in (0,1)$ sufficiently small,
    \begin{align*}
        &\| w_{n+1} \|_{C_{\leq \mft_{L}}^{0} C_{x}^{\delta}} \leq C\left( 1 + M_{v}^{2\delta} + C_e^{(1),\delta} + C_e^{(1),1} + C_e^{(5),1+\delta} \right)  L^{1+2\delta}L_n^{\frac{5}{2} + \delta} \lambda^{\delta}\delta_n^{\frac{1}{2}}, \\
        &\| w_{n+1} \|_{C_{\leq \mft_{L}}^{0} C_{x}^{1+\delta}} \leq C\left( 1 + M_{v}^{2\delta} + C_e^{(1),1}  +  (C_e^{(5),2})^{\delta}(C_e^{(1),1})^{1-\delta} + C_e^{(5),2+\delta} \right) L^{3+2\delta} L_{n}^{7/2+\delta} \delta_{n}^{1/2} \lambda^{1+\delta}, \\
        &\| v_{n+1} \|_{C_{\leq \mft_{L}}^{0} C_{x}^{1+\delta}} \leq C\left(1 + M_{v}^{2\delta} + M_{v} + C_e^{(1),1}  +  (C_e^{(5),2})^{\delta}(C_e^{(1),1})^{1-\delta} + C_e^{(5),2+\delta} \right) L^{3+2\delta} L_{n}^{7/2+\delta} \delta_{n}^{1/2} \lambda^{1+\delta}, \\
        &\| v_{n+1} - v_{n} \|_{C_{\leq \mft_{L}}^{0} C_{x}^{0}}  \leq C\left(  \sqrt{\bar{e}}  + \left(1+ M_{v}^{2\delta} + \eta + C_e^{(1),1} + C_e^{(5),1+\delta} \right) \delta_{n+1}^{1/2}  \right) L^{1+2\delta} L_{n}^{5/2+3\delta} \delta_{n}^{1/2}.
    \end{align*}
\end{lem}
The proof can be found in Section \ref{sec:pf_wtotal}.

\subsection{The Reynolds stress and the new pressure}\label{ssec:stressterms} 

Let us recall the definition of the right-inverse operator for $\divv_{\phi_{n+1}}$ from \cite[Lemma 3.4]{HLP23}.
\begin{lem} \label{lem:inverse}
Let $v \in C^\infty(\T^3,\R^3)$ and $\mcR v$ be the matrix-valued function defined in \cite[Definition 4.2]{DLS13}, so that $\mcR v$ takes values in the space $\mathcal{S}_{0}^{3\times 3}$ of symmetric trace-free matrices and $\dvg \mcR v = v - \frac{1}{(2\pi)^3}\int_{\T^3}v$.
Then the operator $\mathcal{R}_{\phi_{n+1}}$ defined as
\begin{align*}
\mathcal{R}_{\phi_{n+1}} v :=
[\mathcal{R}(v \circ \phi_{n+1}^{-1})] \circ \phi_{n+1}
\end{align*} 
satisfies $\divv_{\phi_{n+1}} (\mcR_{\phi_{n+1}} v) = v - \frac{1}{(2\pi)^3}\int_{\T^3} v$.
\end{lem}

Let us recall the equation at stage $n$:
\begin{align*}
    \partial_{t} v_{n} + \divv_{\phi_{n}}(v_{n} \otimes v_{n}) + \nabla_{\phi_{n}}q_{n} + (-\D)^{\alpha}_{\phi_{n}} v_{n} = \divv_{\phi_{n}} \mathring{R}_{n}.
\end{align*}
Given that we have constructed all the terms on the left-hand side of the above equation at stage $n+1$ in the previous sections, we can simply define
\begin{align*}
    \mathring{R}_{n+1} := \mathcal{R}_{\phi_{n+1}} \left( \partial_{t} v_{n+1} + \divv_{\phi_{n+1}}(v_{n+1} \otimes v_{n+1}) + \nabla_{\phi_{n+1}} q_{n+1} + (-\D)^{\alpha}_{\phi_{n+1}} v_{n+1} \right).
\end{align*}
To investigate the structure of this new stress term, recall that $v_{n+1} = v_{\ell} + w_{o} + w_{c} = v_{\ell} + w_{n+1}$. Moreover, recall the equation for $v_{\ell} = v_{n} * \chi_{\ell}$:
\begin{equation}\label{eq:main_mollified}
    \partial_{t} v_{\ell} + \chi_{\ell} * \divv_{\phi_{n}}(v_{n} \otimes v_{n}) + \chi_{\ell} * \nabla_{\phi_{n}}q_{n} + \chi_{\ell} * (-\D)^{\alpha}_{\phi_{n}} v_{n} = \chi_{\ell} * \divv_{\phi_{n}} \mathring{R}_{n}.
\end{equation}
Now, conceptually speaking, there are three sources of errors in the equation at stage $n+1$:
\begin{enumerate}
 \item \textbf{Perturbation errors}, i.e. errors having to do with $w_{n+1}$, i.e. $R^{\mathrm{tra}}$, $R^{\mathrm{osc}}$, $R^{\mathrm{compr}}$ and $R^{\mathrm{diss}}$.
 \item \textbf{Flow errors} from the mismatch between having all operators flowed along $\phi_{n}$ at stage $n$ and wanting to have all operators flowed along $\phi_{n+1}$ at stage $n+1$ while still wanting to use the equation at stage $n$ to cancel lower-order terms. There is one further error coming from the transport error because there we really want to have a transport-type term $(v_{\ell} \cdot \nabla_{\phi_{n+1}})w_{o}$ and not $\divv_{\phi_{n+1}}(v_{\ell} \otimes w_{o})$. The difference between the two terms is non-zero because of the flow.
 \item \textbf{Mollification errors} from using $v_{\ell}$ in $v_{n+1}$ instead of $v_{n}$. They basically arise from Equ. \eqref{eq:main_mollified}.
\end{enumerate}
Compared with the case of the Euler equations in \cite{HLP23}, all three types of error need to be modified when we introduce a dissipative term $(-\D)^{\alpha}u$ to the equation. We will now proceed to derive the stress decomposition at stage $n+1$. As usual, we do so by plugging in our ansatz $v_{n+1} = v_{\ell} + w_{o} + w_{c}$ and attempt to use our knowledge about the previous stage via $v_{\ell}$. We write terms arising from the fractional Laplacian in blue color.
\begin{align*}
    &\partial_{t} v_{n+1} + \divv_{\phi_{n+1}}(v_{n+1} \otimes v_{n+1}) + \nabla_{\phi_{n+1}} q_{\ell} - \nabla_{\phi_{n+1}} \frac{1}{2}[|w_{o}|^{2} - \tilde{p}_{\ell}] + (-\D)^{\alpha}_{\phi_{n+1}} v_{n+1} \\
    &\overset{\text{Equ. \eqref{eq:main_mollified}}}{=} \underbrace{[\partial_{t} w_{o} + \divv_{\phi_{n+1}}(v_{\ell} \otimes w_{o})  - w_{o} \divv_{\phi_{n+1}} v_{\ell} ]}_{\text{transport error}} {\color{blue}  + \underbrace{(-\D)^{\alpha}_{\phi_{n+1}} w_{n+1}}_{\text{dissipative error}} 
    } \\
    &\quad + \underbrace{\divv_{\phi_{n+1}} \left( w_{o} \otimes w_{o} - \frac{1}{2}\left( |w_{o}|^{2} - \tilde{\rho}_{\ell} \right)\Id - \mathring{R}_{\ell} \right)}_{\text{oscillation error}} \\
    &\quad + \underbrace{ \left( \divv_{\phi_{n+1}} -  \divv_{\phi_{n}} \right) \mathring{R}_{\ell}
    + \left( \divv_{\phi_{n+1}} -  \divv_{\phi_{n}} \right) (v_{\ell} \otimes v_{\ell})
    }_{\text{flow error, I}} \\
    &\quad + \underbrace{ {\color{blue} \left( (-\D)^{\alpha}_{\phi_{n+1}} -  (-\D)^{\alpha}_{\phi_{n}} \right) v_{\ell}
    }
    + \left( \nabla_{\phi_{n+1}} -  \nabla_{\phi_{n}} \right) q_{\ell} + w_{o} \divv_{\phi_{n+1}} v_{\ell}
    }_{\text{flow error, II}} \\
    &\quad + \underbrace{ \left( \divv_{\phi_{n}} \mathring{R}_{\ell} - (\divv_{\phi_{n}} \mathring{R}_{n}) * \chi_{\ell}  \right) + \left( \divv_{\phi_{n}} (v_{\ell} \otimes v_{\ell}) - \divv_{\phi_{n}}(v_{n} \otimes v_{n}) * \chi_{\ell}   \right)}_{\text{mollification error, I}} \\
    &\quad + \underbrace{ {\color{blue} \left( (-\D)^{\alpha}_{\phi_{n}} v_{\ell} - ((-\D)^{\alpha}_{\phi_{n}} v_{n}) * \chi_{\ell}  \right) 
    }
    + \left( \nabla_{\phi_{n}} q_{\ell} - (\nabla_{\phi_{n}} q_{n}) * \chi_{\ell}  \right)}_{\text{mollification error, II}} \\
    &\quad + \underbrace{\partial_{t}w_{c}  +\divv_{\phi_{n+1}} \left( v_{\ell} \otimes w_{c} + w_{o} \otimes v_{\ell} + w_{o} \otimes w_{c} + w_{c} \otimes v_{n+1}  \right)}_{\text{compressibility error}}.
\end{align*}
Note that, strictly speaking, the terms dubbed ``error'' in the above equation are the $\divv_{\phi_{n+1}}$ of the respective stress term, see the next section. 

Compared with \cite{HLP23}, three new error terms appear: the dissipative error and the first terms in the flow error, II and mollification error, II, respectively. Except for the terms involving $\alpha$, all the other terms still satisfy the same bounds as in \cite{HLP23}. We will have to be careful to track the exact energy dependence of all terms, but apart from that we will mostly have to deal with the ``new'' terms.

With this definition we see that 
\begin{align*}
    \partial_{t} v_{n+1} + \divv_{\phi_{n+1}}(v_{n+1} \otimes v_{n+1}) + \nabla_{\phi_{n+1}} q_{n+1} + (-\D)^{\alpha}_{\phi_{n+1}} v_{n+1} = \divv_{\phi_{n+1}} \mathring{R}_{n+1},
\end{align*}
i.e. $(v_{n+1}, q_{n+1}, \phi_{n+1}, \mathring{R}_{n+1})$ is a solution to \eqref{eq:Euler_Reynolds} at stage $n+1$ if we set
\begin{align*}
    q_{n+1} &:= q_{\ell} - \frac{1}{2} \left( |w_{o}|^{2} - \tilde{\rho}_{\ell} \right),
\end{align*}
recalling that $\tilde{\rho}_{\ell} := \frac{2}{r_{0}} \sqrt{\eta^{2} \delta_{n+1}^{2} + |\mathring{R}_{\ell}|^{2}}$.
Note that in the definition Equ. \eqref{eq:def_rhol} of $\rho_{\ell}$, there is an $x$-independent part $\gamma_{n}$. As the pressure appears inside a derivative, we do not need this part here and hence only have $\tilde{\rho}_{\ell}$, not the full $\rho_{\ell}$.

\subsubsection{Transport error} 

\begin{prop} \label{prop:R_trans}
Let us denote $\mathring{R}^{\mathrm{tra}} := \mathcal{R}_{\phi_{n+1}} (\partial_t w_o + (v_\ell \cdot \nabla_{\phi_{n+1}}) w_o) $. Then for every $r \geq r_\star+2$ and $\delta>0$ sufficiently small, it holds that almost surely for every $L \in \N$, $L \geq 1$,
\begin{align*}
\|\mathring{R}^{\mathrm{tra}}\|_{C_{\leq \mathfrak{t}_L} C_x} 
&\leq C\left(\sqrt{\bar{e}} + C_e^{(7),r} + C_e^{(7),r+\delta}+ \left(C_e^{(1),1} + C_e^{(5),r+1} + C_e^{(5),r+\delta+1}\right){M_v} + C_e^{(4),0} \right.\\
	&\qquad\left.+ C_e^{(8),r} + C_e^{(8),r+\delta}\right){L^{r+2\delta}L_n^{2(r+\delta)+\frac{5}{2}}}\lambda^{\delta}\mu^{-1} \delta_n^{\frac{1}{2}},
	\\
\|\mathring{R}^{\mathrm{tra}}\|_{C_{\leq \mathfrak{t}_L} C^1_x} &\leq C\left(C_e^{(3),0} + \left(C_e^{(3),1}\right)^{\delta}\left(C_e^{(3),0}\right)^{1-\delta} + \left(C_e^{(1),1}+\left(C_e^{(5),2}\right)^{\delta}\left(C_e^{(1),1}\right)^{1-\delta}\right){M_v}\right.\\
	&\quad\qquad \left. + C_e^{(4),0}+\left(C_e^{(8),1}\right)^{\delta}\left(C_e^{(4),0}\right)^{1-\delta}\right){L^{3(1+\delta)}L_n^{\frac{7}{2}}}\lambda^{1+\delta} \delta_n^{\frac{1}{2}}.
\end{align*}
\end{prop}
For the proof, see Section \ref{sec-proof-trans-R}.

\subsubsection{Oscillation error} 
\begin{prop} \label{prop:R_osc}
Let us denote $\mathring{R}^{\mathrm{osc}} :=  \mcR_{\phi_{n+1}}\divv_{\phi_{n+1}}(w_o \otimes w_o - \frac{1}{2}\left(|w_o|^2-\tilde{\rho}_\ell \right) Id +\mathring{R}_\ell)$. \\
Then for every $r \geq r_\star+1$ and $\delta>0$ sufficiently small, it holds that almost surely for every $L \in \N$, $L \geq 1$
\begin{align*}
\| \mathring{R}^{\mathrm{osc}} \|_{C_{\leq \mathfrak{t}_L} C_x} 
&\leq 
C \left(\sqrt{\bar{e}} C_e^{(1),1} + C_e^{(7),r+1}+ C_e^{(7),r+\delta+1}\right) {L^{r+2\delta+1}}{L_n^{2(r+\delta) +5}}\lambda^{\delta-1}\mu \varsigma_{n+1}^{\gamma-1} \delta_n,
\\
\| \mathring{R}^{\mathrm{osc}} \|_{C_{\leq \mathfrak{t}_L} C^1_x} 
&\leq
C \left(\sqrt{\bar{e}}C_e^{(1),1} + \left(C_e^{(7),2}\right)^{\delta}\left(\sqrt{\bar{e}}C_e^{(1),1}\right)^{1-\delta}\right){L^{4(1+\delta)}L_n^{3+4\delta}}\lambda^{1+\delta}\delta_n.
\end{align*}
\end{prop}
For the proof see Section \ref{sec-proof-osc-R}.

\subsubsection{Flow error} \label{ssec:Rflow} 

Let us define
\begin{align*}
    \mathring{R}^{\mathrm{flow}}_{1} &:= \mcR_{\phi_{n+1}} \Big[ \left( \divv_{\phi_{n+1}} -  \divv_{\phi_{n}} \right) \mathring{R}_{\ell}
    + \left( \divv_{\phi_{n+1}} -  \divv_{\phi_{n}} \right) (v_{\ell} \otimes v_{\ell}) \\
    &\qquad\qquad\quad  + \left( \nabla_{\phi_{n+1}} -  \nabla_{\phi_{n}} \right) q_{\ell} + w_{o} \divv_{\phi_{n+1}} v_{\ell}  \Big], \\
    \mathring{R}^{\mathrm{flow}}_{2} &:= \mcR_{\phi_{n+1}} \left[ \left( (-\D)^{\alpha}_{\phi_{n+1}} -  (-\D)^{\alpha}_{\phi_{n}} \right) v_{\ell}
    \right].
\end{align*}

\begin{prop}\label{prop:R_flow_1}
For $r \geq r^* + 2$, $\delta \in (0,1)$ sufficiently small and $\gamma' \in (\gamma_*,\gamma)$ it holds that almost surely for every $L \in \N$, $L \geq 1$
\begin{align*}
	\|\mathring{R}^{\mathrm{flow}}_{1}\|_{C_{\leq \mathfrak{t}_L} C_x}  &\leq C\left(\sqrt{\bar{e}} + M_v^{2} + {M_q}+ C_e^{(5),r+1}\right) {L^{r+2+\delta}L_n^{r+\frac{7}{2}}}\ell^{-\delta}(n+1)\varsigma_n^{\gamma'},\\
	\|\mathring{R}^{\mathrm{flow}}_{1}\|_{C_{\leq \mathfrak{t}_L} C_x^{1}}  &\leq C\left(\eta + M_v^{2} + {M_q} + C_e^{(1),\delta}\right){L^{4(1+\delta)}L_n^4} \ell^{-1-\delta}(n+1)\varsigma_n^{\gamma'}.
\end{align*}
\end{prop}
The proof can be found in Section \ref{sec-proof-flow-R-1}.
\begin{prop}\label{prop:R_flow_2}
For $\delta \in (0,1)$ sufficiently small and $\gamma' \in (\gamma_*,\gamma)$ it holds that almost surely for every $L \in \N$, $L \geq 1$
    \begin{align*}
        \| \mathring{R}^{\mathrm{flow}}_{2} \|_{C_{\leq \mathfrak{t}_L} C_x} &\leq CL^{26} M_{v} D_{n}^{\delta+2\alpha} (n+1) \vs_{n}^{\frac{1}{48}\gamma'}, \\
        \| \mathring{R}^{\mathrm{flow}}_{2} \|_{C_{\leq \mathfrak{t}_L} C_x^1} &\leq C M_{v}  L_{n}^{5} \ell^{-1-\delta - 2\alpha} (n+1) \vs_{n}^{\gamma'}.
    \end{align*}
\end{prop}
The proof can be found in Section \ref{sec-proof-flow-R-2}.
Combining these estimates, we finally arrive at
\begin{prop}\label{prop_flow_error}
Let $\mathring{R}^{\mathrm{flow}} := \mathring{R}^{\mathrm{flow}}_{1} + \mathring{R}^{\mathrm{flow}}_{2}$. For $\delta \in (0,1)$ sufficiently small and $\gamma' \in (\gamma_*,\gamma)$, it holds that almost surely for every $L \in \N$, $L \geq 1$
    \begin{align*}
        \| \mathring{R}^{\mathrm{flow}} \|_{C_{\leq \mathfrak{t}_L} C_x} &\leq C\left(\sqrt{\bar{e}} + C_e^{(5),r+1} + M_{v} + M_{v}^{2} + {M_{q}} \right) L^{r+2+\delta}L_n^{r+7/2} \ell^{-\delta-2\alpha} (n+1)\varsigma_n^{\frac{1}{48} \gamma'}, \\ 
        \| \mathring{R}^{\mathrm{flow}} \|_{C_{\leq \mathfrak{t}_L} C^1_x} &\leq C \left(\sqrt{\bar{e}} + C_e^{(1),1} + M_{v} + M_{v}^{2} + M_{q} \right) L^{2} L_{n}^{5} \ell^{-1-\delta - 2\alpha} (n+1) \vs_{n}^{\gamma'}.
    \end{align*}
\end{prop}

\subsubsection{Mollification error} \label{ssec:Rmoll}
Let us define
\begin{align*}
    \mathring{R}^{\mathrm{moll}}_{1} &:= \mcR_{\phi_{n+1}} \divv_{\phi_{n}} \left[ \left(  (v_{\ell} \otimes v_{\ell}) - (v_{n} \otimes v_{n}) * \chi_{\ell}   \right) \right], \\
    \mathring{R}^{\mathrm{moll}}_{2} &:= \mcR_{\phi_{n+1}} \Big[  \left( \divv_{\phi_{n}} \mathring{R}_{\ell} - (\divv_{\phi_{n}} \mathring{R}_{n}) * \chi_{\ell}  \right)
    + \left( \nabla_{\phi_{n}} q_{\ell} - (\nabla_{\phi_{n}} q_{n}) * \chi_{\ell} \right) \\
    &\qquad + \divv_{\phi_{n}} \left( (v_{n} \otimes v_{n}) * \chi_{\ell} \right)  - (\divv_{\phi_{n}}(v_{n} \otimes v_{n})) * \chi_{\ell}  \Big], \\
    \mathring{R}^{\mathrm{moll}}_{3} &:=  \mcR_{\phi_{n+1}} \left[ \left( (-\D)^{\alpha}_{\phi_{n}} v_{\ell} - ((-\D)^{\alpha}_{\phi_{n}} v_{n}) * \chi_{\ell}  \right) \right].
\end{align*}
Recall from \cite{HLP23} the following estimates for the first two terms:
\begin{prop} \label{prop:R_moll_1}
For $\delta \in (0,1)$ sufficiently small, almost surely for every $L \in \N$, $L \geq 1$
\begin{align*}
\|\mathring{R}_{1}^{\mathrm{moll}}\|_{C_{\leq \mathfrak{t}_L} C_x} + \|\mathring{R}^{\mathrm{moll}}_{2} \|_{C_{\leq \mathfrak{t}_L} C_x}
&\leq
C {(1+M_v) L^{7+4\delta}L_n^2}D_n \ell^{\gamma},
\\
\|\mathring{R}_{1}^{\mathrm{moll}}\|_{C_{\leq \mathfrak{t}_L} C^{1}_{x}} + \|\mathring{R}_{2}^{\mathrm{moll}} \|_{C_{\leq \mathfrak{t}_L} C^{1}_{x}}
&\leq
C {(1+M_v) L^{3+2\delta}L_n^2}D_n \ell^{-\delta}.
\end{align*}
\end{prop}
The proof follows easily from a brief inspection of the proof of \cite[Proposition 4.5]{HLP23}.

For the new error term, we get the following.
\begin{prop}\label{prop:R_moll_3}
    Let $\beta = \frac{\gamma}{48}$. For $\delta \in (0,1)$ sufficiently small, almost surely for every $L \in \N$, $L \geq 1$
    \begin{align*}
        \|\mathring{R}_{3}^{\mathrm{moll}}\|_{C_{\leq \mathfrak{t}_L} C_x} &\leq C (1 + M_{v}^{1-\beta} + M_{v}^{1-\delta - 2\alpha} + M_v) L^{16}L_n(D_n \ell^{\gamma} + \ell^{\beta}D_n^{\beta}), \\
        \|\mathring{R}_{3}^{\mathrm{moll}}\|_{C_{\leq \mathfrak{t}_L} C^{1}_{x}} &\leq C (1 + M_{v}^{1-\beta} + M_{v}^{1-\delta - 2\alpha} + M_v) L^{1+4\delta+2\alpha}L_n \ell^{-2\delta-4{\alpha}}.
    \end{align*}
\end{prop}
The proof will be given in Section \ref{sec:proof-Rmoll}.
Finally, combining the previous propositions, we get the following statement for the total mollification error.
\begin{prop}\label{prop:R_moll}
For $\delta \in (0,1)$ sufficiently small, almost surely for every $L \in \N$, $L \geq 1$
\begin{align*}
\|\mathring{R}^{\mathrm{moll}}\|_{C_{\leq \mathfrak{t}_L} C_x}
&\leq
C (1 + M_{v}^{1-\beta} + M_{v}^{1-\delta-2\alpha} + M_{v}) L^{7 + 4\delta} L_n^2  \left( \ell^{\gamma} D_{n} + \ell^{\beta} D_{n}^{\beta} \right),
\\
\|\mathring{R}^{\mathrm{moll}} \|_{C_{\leq \mathfrak{t}_L} C^{1}_{x}} &\leq C(1 + M_{v}^{1-\delta-2\alpha} + M_{v}) L^{3+2\delta} L_n^2 D_n \ell^{-2\delta - 4\alpha}.
\end{align*}
\end{prop}

\subsubsection{Compressibility error}\label{ssec:Rcomp} 
\begin{prop} \label{prop:R_comp}
Let us denote 
\[
\mathring{R}^{\mathrm{comp}} := \mathcal{R}_{\phi_{n+1}} \left[ \partial_{t} w_{c} + \divv_{\phi_{n+1}} \left( v_{\ell} \otimes w_{c} + w_{o} \otimes v_{\ell} + w_{o} \otimes w_{c} + w_{c} \otimes v_{n+1} \right) \right].
\]
Then for every $r \geq r_\star + 1$, $\delta>0$ sufficiently small, almost surely for every $L\in \N$, $L \geq 1$
\begin{align*}
\| \mathring{R}^{\mathrm{comp}} \|_{C_{\leq \mathfrak{t}_{L}} C_{x}}
&\leq
C\left(1 + {M_v^{2\delta}}+ C_e^{{\rm comp},2, \delta} + C_e^{{\rm comp},3, \delta} + C_e^{{\rm comp},4, \delta} \right)L^{12+6\delta}L_n^{2(r+\delta)+\frac{9}{2}}\lambda^{\delta}\delta_n^{\frac{1}{2}}\delta_{n+2}^{\frac{6}{5}},
\\ 
\| \mathring{R}^{\mathrm{comp}} \|_{C_{\leq \mathfrak{t}_{L}} C^{1}_{x}}
&\leq
C\left(1 + {M_v^{2\delta}}+ C_e^{{\rm comp},2,1+\delta} + C_e^{{\rm comp},3,1+\delta} + C_e^{{\rm comp},4,1+\delta}\right){L^{8+5\delta}L_n^{\frac{13}{2} + 4\delta}}\lambda^{1+\delta}\delta_n^{\frac{1}{2}}\delta_{n+2}^{\frac{6}{5}},
\end{align*}
with constants $C_e^{{\rm comp},2,\delta}$, $C_e^{{\rm comp},3,\delta}$, $C_e^{{\rm comp},4,\delta}$, $C_e^{{\rm comp},2,1+\delta}$, $C_e^{{\rm comp},3,1+\delta}$, $C_e^{{\rm comp},4,1+\delta}$ specified in the proof.
\end{prop}
The proof will be presented in Section \ref{sec-proof-comp-R}.

\subsubsection{Dissipative error}\label{ssec:Rdiss} 
Let us denote the dissipative error by
\begin{align*}
    R^{\mathrm{diss}} := \mathcal{R}_{\phi_{n+1}} \left( (-\D)^{\alpha}_{\phi_{n+1}} w_{n+1} \right).
\end{align*}
\begin{prop}\label{prop:R_diss}
    For $\delta>0$ sufficiently small, we have almost surely for every $L\in \N$, $L \geq 1$
    \begin{align*}
        \| R^{\mathrm{diss}} \|_{C_{\leq \mft_{L} }C_{x}^{0}} &\leq C  \left(1+ M_{v}^{2\delta} +  C_{e}^{(1),0} + C_{e}^{(5),r} + C_{e}^{(5),r+\delta}\right)^{1-2\alpha - \delta}  \\
        &\qquad \cdot \left( 1 + M_{v}^{2\delta} + C_{e}^{(1),\delta} + C_{e}^{(1),1} + C_{e}^{(5),1+\delta} \right)^{2\alpha + \delta}\\
        &\qquad \cdot  L^{r+1+3\delta} L_{n}^{r + 3/2 + \delta} \ell^{(1 + d(1/p - 1)-2\delta)(1-2\alpha - \delta)} \lambda^{\delta(2 \alpha + \delta)}  \delta_{n+2}^{6/5 (1 - 2\alpha - \delta)}, \\
        \| R^{\mathrm{diss}} \|_{C_{\leq \mft_{L} }C_{x}^{1}} &\leq C \left(1 + M_{v}^{2\delta} + C_{e}^{(1),\delta} + C_{e}^{(1),1} + C_{e}^{(5),1+\delta} \right)^{1-2\alpha-\delta} \\
        &\qquad \cdot \left(1+ M_{v}^{2\delta} +  C_e^{(1),1} + (C_e^{(5),2})^{\delta}(C_e^{(1),1})^{1-\delta} + C_e^{(5),2+\delta} \right) ^{2\alpha + \delta}\\
        &\qquad \cdot L^{9+4\delta} L_{n}^{7/2+\delta} \delta_{n}^{1/2} \lambda^{2\alpha + 3\delta}. 
    \end{align*}
\end{prop}
The proof will be given in Section \ref{ssec:pf_Rdiss}.

\subsection{The divergence}
\begin{prop} \label{prop:divergence}
For every $\delta$ sufficiently small and $\gamma' \in (0,\gamma)$, for all $L \in \N$, $L \geq 1$ the following holds almost surely
\begin{align*}
\|\divv_{\phi_{n+1}} v_{n+1}\|_{C_{\leq \mathfrak{t}_{L}} B^{-1}_{\infty,\infty}} 
\leq 
C (1 + M_{v}^{2\delta} + M_{v}^{1-\delta} + M_{v}) L^{10} L_{n} (D_{n}^{1+2\delta}\ell^{\gamma} + D_{n}^{\delta} (n+1) \vs_{n}^{\gamma'}).
\end{align*}
\end{prop}
The proof will be given in Section \ref{ssec:pf_divergence}

\subsection{The pressure}
 
Recall the definition of the new pressure $q_{n+1}$ and the energy pumping term $\tilde{\rho}_\ell$
\begin{align*}
q_{n+1} 
= 
q_\ell - \frac12 \left(|w_o|^2-\tilde{\rho}_\ell \right), \quad
\tilde{\rho}_\ell(x,t) 
=
\frac{2}{r_0} \sqrt{\eta^2  \delta_{n+1}^2 + |\mathring{R}_\ell(x,t)|^2}.
\end{align*}
\begin{prop} \label{prop:it_pres}
Let 
\begin{equation}\label{eq:def_Mq}
    M_q := C(1 + \eta + \bar{e}). 
\end{equation}
Then for every $L \in \N,$ $L \geq 1$ it holds almost surely
\begin{align*}
\| q_{n+1}-q_{n} \|_{C_{\leq \mathfrak{t}_{L}}C_{x}}
&\leq
M_{q} L_{n} \delta_{n}.
\end{align*}
Moreover, for every $\delta \in (0,1)$ sufficiently small and for every $L \in \N,$ $L \geq 1$, it holds almost surely
\begin{align*}
\| q_{n+1}-q_{n} \|_{C^1_{\leq \mathfrak{t}_L,x}}
&\leq
C{L^{3(1+\delta)}L_n^5}\left(\sqrt{\bar{e}} \left(C_e^{\partial_t w_o} + C_e^{(1),\delta} + (C_e^{(5),2})^{\delta}(C_e^{(1),1})^{1-\delta}\right) +\eta + 1\right)\lambda^{1+\delta} \delta_n.
\end{align*}
\end{prop}
The proof can be found in Section \ref{ssec:pf_pressure}.

\subsection{The kinetic energy}
 
\begin{prop} \label{prop:it_energy}
Recall the definition of $r_0$ from Section \ref{sssec:w_o}.
Up to choosing $C_\varsigma, C_\mu$ large enough, the following holds true almost surely:
\begin{align*}
		&\left| e(t) (1-\delta_{n+1}) - \int_{\mathbb{T}^3} |v_{n+1}(t,x)|^2 {\rm d}x \right|\\
		&\leq C\left(\left(1 + {M_v^{2\delta}} + C_e^{(1),1} + C_e^{(5),1+\delta} \right)^2 + M_{v} C_{e}^{(1),1} + \sqrt{\bar{e}}(1+C_{e}^{(1),1}) \right)\lambda^{2\delta} \delta_{n+2}^{\frac{6}{5}} + \frac{9\eta}{r_0} \delta_{n+1}.
\end{align*}
\end{prop}
We will give the proof in Section \ref{ssec:pf_energy} below.

\section{Proof of the main iterative proposition} \label{sec:proof_main_iter}
 
In this section, simply speaking, we have to achieve three objectives:
\begin{enumerate}
 \item Choosing $m$ and $r$ such that all of the expressions of powers of $L$ and $L_{n}$ can be bounded by $L_{n+1}$. This will be done in Section \ref{ssec:pf_choice_of_parameters}.
 \item Choosing the \textit{exponential parameters} $m, b, c, \ldots$ such that the various products of parameters in the estimates can be bounded by the desired simple powers such as $\delta_{n+2}$ in each iterative estimate. This will also be done in Section \ref{ssec:pf_choice_of_parameters}.
 \item Choosing the \textit{base parameter} $a$ such that we can absorb all energy-dependent constants for any $n$, and choose the constant $M_{v}$ of Equ. \eqref{eq:iter_diffv}. We will do this in Sections \ref{ssec:choice_a}, \ref{ssec:choice_Mv}, respectively.
\end{enumerate}
In order to derive all conditions on these parameters, we employ the estimates of the previous sections, which will be done in Sections \ref{ssec:pf_C0_est} and \ref{ssec:pf_C1_est}.

\subsection{\texorpdfstring{$C^{0}$}{C0} estimates}\label{ssec:pf_C0_est}
 
The progressive measurability of $v_{n+1}, q_{n+1}$ and $\mathring{R}_{n+1}$ follows directly from their definitions.

\noindent\textbf{Estimate for $q_{n+1} - q_{n}$:}
Iterative inequality Equ. \eqref{eq:iter_diffq} follows directly from Proposition \ref{prop:it_pres}. 

\noindent\textbf{Estimate for $q_{n+1}$:}
With the previous estimate, the choice $q_{0} = 0$ and the iterative inequality of the $n$-th stage, we get:
\begin{align*}
    \| q_{n+1} \|_{C_{\leq \mft_{L}} C_{x}^{0}} \leq \| q_{n+1} - q_{n} \|_{C_{\leq \mft_{L}} C_{x}^{0}} + \sum_{k=0}^{n-1} \| q_{k+1} - q_{k} \|_{C_{\leq \mft_{L}} C_{x}^{0}}  \leq M_{q} L_{n} \sum_{k=0}^{n} \delta_{k}.
\end{align*}

\noindent\textbf{$v_{n+1}$ is mean-free:} Recall the definition of $v_{n+1}$ from Equ. \eqref{eq:def-vnplus1}. The operation of mollification preserves the mean-freeness since, for $f$ mean-free and $g$ smooth,
\begin{align*}
    \int_{\T^{3}} (f * g)(x) dx = \int_{\T^{3}} \int_{\T^{3}} f(x-y) g(y) dy dx \overset{z = x - y}{=} \int_{\T^{3}}\int_{\T^{3}} f(z) g(y) dy dz = \int_{\T^{3}} f(z) dz \cdot \int_{\T^{3}} g(y) dy = 0.
\end{align*}
Similarly, the flowed projector $\mathcal{Q}_{\phi_{n}}$ preserves mean-freeness, as for $f$ mean-free,
\begin{align*}
    \int_{\T^{3}} (\mathcal{Q}_{\phi_{n}} f)(x) dx &\overset{\text{def.}}{=} \int_{\T^{3}} \left( \mathcal{Q}(f \circ \phi_{n}^{-1}) \right) \circ \phi_{n}(x) dx \overset{y = \phi_{n}(x)}{=} \int_{\T^{3}} \mathcal{Q} (f \circ \phi_{n}^{-1}) \cdot 1 dy \\
    &= \int_{\T^{3}} (I - \mathcal{P}) f ( \phi_{n}^{-1}(y)) dy \overset{z = \phi_{n}^{-1}(y)}{=} \int_{\T^{3}} f(z) dz = 0,
\end{align*}
where we have used that $\mathcal{P}$ projects onto vector fields of zero mean \cite[Definition 4.1]{DLS13}.

Therefore, we see that $v_{\ell}$ and $w_{c}^{1}$ are mean-free. Moreover, since $w_{o} + w_{c}^{2} = \mathcal{P}_{\phi_{n}} w_{o}$, the other terms of $v_{n+1}$ are mean-free as well which in turn implies that $\int_{\T^{3}} v_{n+1} dx = 0$.

The remaining inequalities will translate into conditions on the size of $a$ in terms of the energy, as we will see. We will collect the requirements on $a$ in the form of lower bounds $a \geq a_{i}$, $i = 1, \ldots, 5$ and in Section \ref{ssec:choice_a}, we will show that we can choose a suitable $a$ satisfying all the requirements.

\noindent \textbf{Estimate for the energy:}
For Equ. \eqref{eq:iter_energy}, we use Proposition \ref{prop:it_energy} to find
\begin{align*}
		&\left| e(t) (1-\delta_{n+1}) - \int_{\mathbb{T}^3} |v_{n+1}(t,x)|^2 {\rm d}x \right|\\
		&\leq C\left(\left(1 + {M_v^{2\delta}} + C_e^{(1),1} + C_e^{(5),1+\delta} \right)^2 + M_{v} C_{e}^{(1),1} + \sqrt{\bar{e}}(1+C_{e}^{(1),1}) \right)\lambda^{2\delta} \delta_{n+2}^{\frac{6}{5}} + \frac{9\eta}{r_0} \delta_{n+1}.
\end{align*}
We will need $a \geq a_{1}$, where $a_{1}$ will be determined such that
\begin{equation}\label{eq:choice_a_1}
    C\left(\left(1 + {M_v^{2\delta}} + C_e^{(1),1} + C_e^{(5),1+\delta} \right)^2 + M_{v} C_{e}^{(1),1} + \sqrt{\bar{e}}(1+C_{e}^{(1),1}) \right)\lambda^{2\delta} \delta_{n+2}^{\frac{6}{5}} \leq \frac{\eta}{r_0} \delta_{n+1}.
\end{equation}
This will be done in Section \ref{ssec:choice_a}. Once this is done, Equ. \eqref{eq:iter_energy} follows immediately.

\noindent\textbf{Estimate for $\divv_{\phi_{n+1}} v_{n+1}$:}
\begin{align*}
 \|\divv_{\phi_{n+1}} v_{n+1}\|_{C_{\leq \mathfrak{t}_{L}} B^{-1}_{\infty,\infty}} 
&\overset{\text{Prop. \ref{prop:divergence}}}{\leq} 
C (1 + M_{v}^{2\delta} + M_{v}^{1-\delta} + M_{v}) L^{10} L_{n} (D_{n}^{1+2\delta}\ell^{\gamma} + D_{n}^{\delta} (n+1) \vs_{n}^{\gamma'}).
\end{align*}
In order to achieve iterative estimate Equ. \eqref{eq:iter_div}, we will need both 
\begin{equation}\label{eq:choice_m_1}
    L^{10} L_{n} \leq L_{n+1} 
\end{equation}
as well as $a \geq a_{2}$, where $a_{2}$ is determined such that
\begin{equation}\label{eq:choice_a_2}
    C (1 + M_{v}^{2\delta} + M_{v}^{1-\delta} + M_{v}) (D_{n}^{1+2\delta}\ell^{\gamma} + D_{n}^{\delta} (n+1) \vs_{n}^{\gamma'}) \leq \delta_{n+3}^{5/4}.
\end{equation}

\noindent\textbf{Estimate for $\mathring{R}_{n+1}$:} We have collected the $C^{0}$ estimates for all the stress terms in the following table. 

\begin{table}[!ht]\label{tab:stress0}
\begin{center}
  \begin{tabular}[h]{c|c|l}
Term & Proposition & Estimate  \\
\hline
$\| \mathring{R}^{\rm tra} \|_{C_{\leq \mft_{L}}C_{x}}$   &   \ref{prop:R_trans}  			& 		$ C^{\rm tra,0}_{e} {L^{r+2\delta}L_n^{2(r+\delta)+\frac{5}{2}}}\lambda^{\delta}\mu^{-1} \delta_n^{\frac{1}{2}}$	 \\[.5em]\hline 
$\| \mathring{R}^{{\rm osc}}\|_{C_{\leq \mft_{L}}C_{x}}$   & \ref{prop:R_osc} 	&	$C_{e}^{\rm osc,0} {L^{r+2\delta+1}}{L_n^{2(r+\delta) +5}}\lambda^{\delta-1}\mu \varsigma_{n+1}^{\gamma-1} \delta_n$  \\[.5em]\hline 
$\| \mathring{R}^{{\rm flow}}\|_{C_{\leq \mft_{L}}C_{x}}$	 & \ref{prop_flow_error}		 		&   $C_{e}^{{\rm flow,0}} L^{r+2+\delta}L_n^{r+7/2} \ell^{-\delta-2\alpha} (n+1)\varsigma_n^{\frac{1}{48} \gamma'}$  \\[.5em]\hline 
$\| \mathring{R}^{{\rm moll}} \|_{C_{\leq \mft_{L}}C_{x}}$  & \ref{prop:R_moll}  						&	$C_{e}^{{\rm moll,0}}L^{7 + 4\delta} L_n^2  \left( \ell^{\gamma} D_{n} + \ell^{\beta} D_{n}^{\beta} \right)$ \\[.5em]\hline 
$\| \mathring{R}^{{\rm comp}} \|_{C_{\leq \mft_{L}}C_{x}}$ &  \ref{prop:R_comp} & $ C_{e}^{\rm comp,0} L^{12+6\delta}L_n^{2(r+\delta)+\frac{9}{2}}\lambda^{\delta}\delta_n^{\frac{1}{2}}\delta_{n+2}^{\frac{6}{5}}$  \\[.5em]\hline 
$\| \mathring{R}^{{\rm diss}}\|_{C_{\leq \mft_{L}}C_{x}}$ & \ref{prop:R_diss}   & $C_{e}^{\rm diss,0}  L^{r+1+3\delta} L_{n}^{r + 3/2 + \delta}   \delta_{n+2}^{6/5 (1 - 2\alpha - \delta)} $  \\[.5em]
& & $\qquad\cdot \ell^{(1 + d(1/p - 1)-2\delta)(1-2\alpha - \delta)} \lambda^{\delta(2 \alpha + \delta)} $
\end{tabular}
\caption{Summary of estimates for the $C_{\leq \mft_{L}} C^{0}_{x}$-norms of the stress terms.    The precise expression for each named constant (e.g. $C_{e}^{\rm tra,0}$) can be found in their respective propositions.} 
\end{center}
\end{table}
Recall that our goal is to prove that
\begin{equation}\label{eq:Rnplusone_req}
    \| \mathring{R}_{n+1} \|_{C_{\leq \mft_{L}} C_{x}^{0}} \leq \eta L_{n+1} \delta_{n+2}.
\end{equation}
We see from the table that the most restrictive constraints on $m$ come from the requirements
\begin{align}
    \label{eq:choice_m_2} L^{7+4\delta}L_{n}^{2} &\leq L_{n+1}, \\
    \label{eq:choice_m_3} L^{12+6\delta} L_{n}^{2(r+\delta)+9/2} &\leq L_{n+1}.
\end{align}

Let us define 
\begin{align*}
    C_{e}^{\mathring{R},0} := C^{\rm tra,0}_{e} + C^{\rm osc,0}_{e} + C^{\rm flow,0}_{e} + C^{\rm moll,0}_{e} + C^{\rm comp,0}_{e} + C^{\rm diss,0}_{e}.
\end{align*}
\begin{lem}\label{lem:CR0}
    We have
    \begin{align*}
         \| \mathring{R}_{n+1} \|_{C_{\leq \mft_{L}} C_{x}^{0}} \leq C_{e}^{\mathring{R},0} L_{n+1} \left( \lambda^{\delta} \delta_{n}^{1/2} \delta_{n+2}^{6/5} + \ell^{-\delta-2\alpha}  (n+1) \delta_{n+3}^{8/3 \cdot \beta}  \right).
    \end{align*}
    Furthermore, the energy-dependent constant $C_{e}^{\mathring{R},0}$ satisfies
    \begin{align*}
        C_{e}^{\mathring{R},0} &\leq \left(\frac{1+\bar{e}}{\underline{e}} \right)^{2} \left( \frac{\bar{e}}{\underline{e}} \right)^{2r+3} \\
        &\quad \cdot C \bigg(1  + M_{v}^{1-\beta} + M_{v}^{1-\delta-2\alpha} + \sqrt{\bar{e}}\left(1+ \sqrt{\underline{e}} + \frac{|e|_{C^1}}{\underline{e}}\right)\left(1+ \bar{e} + \frac{1+\bar{e}}{1+\sqrt{\bar{e}}}\right)^2\left(\frac{1+\bar{e}}{\underline{e}}\right)^2  \\ 
        &\quad\quad + \left( 1 + M_{v}^{2\delta}+M_{v} + \sqrt{\bar{e}}\frac{1 + \underline{e}}{\underline{e}} +  \sqrt{\bar{e}}\left(1+ \bar{e}+ \frac{\bar{e}}{1+\sqrt{\bar{e}}}\right) \right)\cdot \left( 1 + M_{v}^{2\delta} + M_{v} + \sqrt{\bar{e}} \right)  \bigg). 
    \end{align*}
\end{lem}
The proof can be found in Section \ref{ssec:pf_CR1}. 
To further simplify the expressions, let $\delta > 0$ be small enough such that 
\begin{align}
    \nonumber \delta &< \alpha, \\
    \label{eq:lambda_delta}\lambda^{\delta} \delta_{n}^{1/2} &\leq 1.
\end{align}
We can thus achieve \eqref{eq:Rnplusone_req} if we can ensure that $a \geq \max(a_{3},a_{4})$ is large enough such that
\begin{align}
    \label{eq:choice_a_3} C_{e}^{\mathring{R},0}\delta_{n+2}^{1/5} \leq \frac{\eta}{2}, \\
    \label{eq:choice_a_4} C_{e}^{\mathring{R},0} \ell^{-3\alpha}  (n+1) \delta_{n+3}^{8/3 \cdot \beta} \delta_{n+2}^{-1} \leq \frac{\eta}{2}. 
\end{align}
That this can indeed be done will be shown in Section \ref{ssec:choice_a}.

\noindent \textbf{Estimate for $v_{n+1} - v_{n}$:} We will tend to this estimate in Section \ref{ssec:choice_Mv}. The reason for this is that, as we will see, the constant $M_{v}$ depends on $a$, so we will have to define $a$ first.

\subsection{\texorpdfstring{$C^{1}$}{C1} estimates} \label{ssec:pf_C1_est}
 
Recall that $A$ is chosen such that
\[\max \left\{\|v_{n+1}\|_{C_{\leq \mft_{L},x}^1}, \|q_{n+1}\|_{C_{\leq \mft_{L},x}^1}, \|\mathring{R}_{n+1}\|_{C_{\leq \mft_{L}} C_{x}^1}\right\} \leq A {L_{n+1}}\delta_n^{\frac{1}{2}} \left(\frac{D_n}{\delta_{n+4}}\right)^{1+\epsilon}.\]

\noindent\textbf{Estimate for $\mathring{R}_{n+1}$:} We have collected the $C^{1}$ estimates for all the stress terms in Table \ref{tab:stress1}.

\begin{table}[!ht]
\begin{center}
\begin{tabular}[h]{c|c|l}\label{tab:stress1}
Term & Proposition & Estimate  \\
\hline
$\| \mathring{R}^{\rm tra} \|_{C_{\leq \mft_{L}}C_{x}^{1}}$   &   \ref{prop:R_trans}  			& 		$ C^{\rm tra,1}_{e} {L^{3(1+\delta)}L_n^{7/2}}\lambda^{1+\delta} \delta_n^{1/2}$	 \\[.5em]\hline 
$\| \mathring{R}^{{\rm osc}}\|_{C_{\leq \mft_{L}}C^{1}_{x}}$   & \ref{prop:R_osc} 	&	$C_{e}^{\rm osc,1} {L^{4(1+\delta)}L_n^{3+4\delta}}\lambda^{1+\delta}\delta_n$  \\[.5em]\hline 
$\| \mathring{R}^{{\rm flow}}\|_{C_{\leq \mft_{L}}C^{1}_{x}}$	 & \ref{prop_flow_error}		 		&   $C_{e}^{{\rm flow,1}}  L^{2} L_{n}^{5} \ell^{-1-\delta - 2\alpha} (n+1) \vs_{n}^{\gamma'}$  \\[.5em]\hline 
$\| \mathring{R}^{{\rm moll}} \|_{C_{\leq \mft_{L}}C^{1}_{x}}$  & \ref{prop:R_moll}  						&	$C_{e}^{{\rm moll,1}}L^{3+2\delta} L_n^2 D_n \ell^{-2\delta - 4\alpha}.$ \\[.5em]\hline 
$\| \mathring{R}^{{\rm comp}} \|_{C_{\leq \mft_{L}}C^{1}_{x}}$ &  \ref{prop:R_comp} & $ C_{e}^{\rm comp,1} {L^{8+5\delta}L_n^{\frac{13}{2} + 4\delta}}\lambda^{1+\delta}\delta_n^{\frac{1}{2}}\delta_{n+2}^{\frac{6}{5}}$  \\[.5em]\hline 
$\| \mathring{R}^{{\rm diss}}\|_{C_{\leq \mft_{L}}C^{1}_{x}}$ & \ref{prop:R_diss}   & $C_{e}^{\rm diss,1}  L^{9+4\delta} L_{n}^{7/2+\delta} \delta_{n}^{1/2} \lambda^{2\alpha + 3\delta}$  
\end{tabular}
\caption{Summary of estimates for the $C_{\leq \mft_{L}} C^{1}_{x}$-norms of the stress terms.    The precise expression for each named constant (e.g. $C_{e}^{\rm tra,1}$) can be found in their respective propositions.} 
\end{center}
\end{table}
Let 
\begin{align*}
    C_{e}^{\mathring{R},1} := C^{\rm tra,1}_{e} + C^{\rm osc,1}_{e} + C^{\rm flow,1}_{e} + C^{\rm moll,1}_{e} + C^{\rm comp,1}_{e} + C^{\rm diss,1}_{e}.
\end{align*}
We will defer the long calculations to Section \ref{ssec:pf_CR1} and summarize the result in the following
\begin{lem}\label{lem:CR1}
    The energy-dependent constant $C_{e}^{\mathring{R},1}$ satisfies
    \begin{align*}
        C_e^{\mathring{R},1}&\leq C\left(\frac{1+\bar{e}}{\underline{e}}\right)^2 \frac{\bar{e}}{\underline{e}}\cdot \left(1 + \sqrt{\bar{e}}\left(1+ \sqrt{\underline{e}}+\frac{|e|_{C^1}}{\underline{e}}\right)\left(1+ \bar{e} + \frac{1+\bar{e}}{1+\sqrt{\bar{e}}}\right)^2  + M_{v}^{1-\beta} + M_{v}^{1-\delta-2\alpha} \right.\\
        &\quad +  \sqrt{\bar{e}}\left(1+ \bar{e}+ \frac{1+\bar{e}}{1+\sqrt{\bar{e}}}\frac{\bar{e}}{\underline{e}}\right)\left(1+\sqrt{\bar{e}} + \sqrt{\bar{e}}^{1-\delta} + \sqrt{\bar{e}}\left(1+ \bar{e}+ \frac{1+\bar{e}}{1+\sqrt{\bar{e}}}\right)\right)\\
        &\quad\left. + ({M_v^{2\delta}} +   {M_v})\left(1+ {M_v^{2\delta}}+ M_v + \sqrt{\bar{e}} +\sqrt{\bar{e}}\left(1+ \bar{e} + \frac{1+\bar{e}}{1+\sqrt{\bar{e}}} \right)\right)\right).
    \end{align*}
\end{lem}

We claim that
\begin{equation}\label{eq:lambda_A}
    \lambda^{1+\delta} \leq \tilde{C} \delta_{n}^{1/2} \left(\frac{D_n}{\delta_{n+4}}\right)^{1+\epsilon}
\end{equation}
as well as
\begin{equation}\label{eq:choice_m_4}
    L^{16+2\delta}L_n^{\frac{13}{2}+4\delta} \leq L_{n+1}.
\end{equation}
Both claims will be checked in Section \ref{ssec:pf_choice_of_parameters} below. 
Next,
\begin{align*}
	\|v_{n+1}\|_{C_{\leq \mft_{L}} C_x^1} &\leq \|v_{\ell}\|_{C_{\leq \mft_{L}} C_x^1} + \|w_o\|_{C_{\leq \mft_{L}} C_x^1} + \|w_c\|_{C_{\leq \mft_{L}} C_x^1} \\
	&\leq {L_n}D_n + C {L^{2(1+\delta)}}{L_n^{\frac{5}{2} + 2\delta}}(C_e^{(1),1} + C_e^{(5),1+\delta}) \lambda^{1+\delta}\delta_n^{\frac{1}{2}} \\
	&\quad + C {L^7 L_n (1+M_v^{2\delta}})\ell^{-1}\delta_{n+2}^{\frac{6}{5}} + C {L^{3+2\delta}}{L_n^{\frac{7}{2} + \delta}}(C_e^{(1),1} + C_e^{(5),2+\delta})\lambda^{\delta}\mu\varsigma_{n+1}^{\alpha-1} \delta_n^{\frac{1}{2}}\\
	&\overset{C_{e}^{(5),r} \text{ monotone}}{\leq} C_e^{v,1} L^7L_n^{\frac{7}{2}+2\delta}\lambda^{1+\delta}\delta_n^{\frac{1}{2}},
\end{align*}
where
\[C_e^{v,1} := C\left(1+ {M_v^{2\delta}}+ C_e^{(1),1}  + C_e^{(5),2+\delta}\right), \]
and which leads to the condition 
\begin{equation}\label{eq:choice_m_5}
    L^7L_n^{\frac{7}{2}+2\delta} \leq L_{n+1}.
\end{equation}
Furthermore, for $\|v_{n+1}\|_{C_t^1C_x}$, we use that
\[\partial_t v_{n+1} = {\rm div}^{\phi_{n+1}}\left(\mathring{R}_{n+1} - v_{n+1} \otimes v_{n+1} - q_{n+1} {\rm Id}\right).\]
Hence the $C_t^1C_x$-norm of $v_{n+1}$ is determined by the $C_tC_x^1$-norm of $v_{n+1}$ estimated above, as well as the estimates of the $C_tC_x^1$-norms of $q_{n+1}$ and $\mathring{R}_{n+1}$.

\begin{lem}\label{lem:Cv1}
    The energy-dependent constant $C_{e}^{v,1}$ satisfies
    \begin{align*}
        C_{e}^{v,1} \leq C\frac{1+\bar{e}}{\underline{e}}\frac{\bar{e}}{\underline{e}}\left(1+ {M_v^{2\delta}}+  \sqrt{\bar{e}}\left(1+ \bar{e}+ \frac{1+\bar{e}}{1+\sqrt{\bar{e}}}\right)\right). 
    \end{align*}
\end{lem}
The proof will also be given in Section \ref{ssec:pf_CR1}.

Finally, considering the pressure term, we find
\begin{align*}
	\|q_{n+1}\|_{C_{\leq \mft_{L},x}^1} &\leq \|q_{n+1} - q_n\|_{C_{\leq \mft_{L},x}^1} + \|q_n\|_{C_{\leq \mft_{L},x}^1}\\
	&\leq C{L^{3(1+\delta)}L_n^5}\left(\sqrt{\bar{e}} \left(C_e^{\partial_t w_o} + C_e^{(1),\delta} + C_e^{(5),1+\delta}\right) +\eta + 1\right)\lambda^{1+\delta} \delta_n + {L_n}D_n\\
	&\leq C_e^q {L_{n+1}}\lambda^{1+\delta}\delta_n,
\end{align*}
where we have used another claim to be proven in Section \ref{ssec:pf_choice_of_parameters}:
\begin{equation}\label{eq:choice_m_6}
 {L^{3(1+\delta)}L_n^5} \leq L_{n+1}.
\end{equation}
\begin{lem}\label{lem:Cq1}
    The energy-dependent constant $C_{e}^{q,1}$ satisfies
    \begin{align*}
        C_{e}^{q,1} &\leq C\left(1 + \bar{e}\left(1 + \bar{e} + \frac{1+\bar{e}}{1+\sqrt{\bar{e}}}\right)\left( 1 + {M_v}  + \left(1+ \frac{|e|_{C^1}}{\underline{e}}\right)\left(1 + \frac{1+\bar{e}}{1+\sqrt{\bar{e}}}\right) \right)\right)\frac{1+\bar{e}}{\underline{e}}.
    \end{align*}
\end{lem}

With these estimates in hand, we will determine $A$ via
\[ \max\left(C_e^{\mathring{R},1}, C_e^{v,1}, C_e^{q,1}\right).\]
As the calculation is again rather tedious, we defer it to the end of Section \ref{ssec:pf_CR1}.
\begin{lem}\label{lem:choice_A}
    The energy-dependent constant $A$ satisfies
    \begin{equation}\label{eq:def_A}
        A = \tilde{A}_e \left(\frac{1+\bar{e}}{\underline{e}}\right)^2 \left( \frac{\bar{e}}{\underline{e}} \right)^{3},
    \end{equation}
    where 
    \begin{align*}
        \tilde{A}_{e} &= C \left\{1 + (\bar{e} + \sqrt{\bar{e}})\left(1+\sqrt{\underline{e}}+\frac{|e|_{C^1}}{\underline{e}}\right)\left(1+ \bar{e} + \frac{1+\bar{e}}{1+\sqrt{\bar{e}}}\right)^2 + M_{v}^{1-\delta-2\alpha} + M_{v}^{1-\beta}\right.\\
        &\qquad + \sqrt{\bar{e}}\left(1+ \bar{e}+ \frac{1+\bar{e}}{1+\sqrt{\bar{e}}}\frac{\bar{e}}{\underline{e}}\right)\left(1+ \sqrt{\bar{e}}^{1-\delta} + \sqrt{\bar{e}}\left(1+ M_{v} + \bar{e}+ \frac{1+\bar{e}}{1+\sqrt{\bar{e}}}\right)\right)\\
        &\qquad +(1 + {M_v^{2\delta}} + M_v^{1-\delta} + {M_v} + \sqrt{\bar{e}})\left(1+ {M_v^{2\delta}}+ M_v  +\sqrt{\bar{e}}\left(1+ \bar{e} + \frac{1+\bar{e}}{1+\sqrt{\bar{e}}} \right)\right\}.
    \end{align*}
\end{lem}
Note that (replacing $C$ with $\max(C,1)$ if $C < 1$) we have
\begin{equation}\label{eq:tildeA_1}
    \tilde{A}_e > 1.
\end{equation}
And furthermore, comparing $C_{e}^{\mathring{R},0}$ with $A$, we see that
\begin{equation}\label{eq:CR0_A}
    C_{e}^{\mathring{R},0} \leq \tilde{A}_{e} \left(\frac{1+\bar{e}}{\underline{e}} \right)^{2} \left( \frac{\bar{e}}{\underline{e}} \right)^{2r+3} = A\left( \frac{\bar{e}}{\underline{e}} \right)^{2r}.
\end{equation}

\subsection{Choice of the exponential parameters}\label{ssec:pf_choice_of_parameters}
 
In this section, assuming that the ``base parameter'' $a$ satisfies $a > 1$, we check that all relations between the parameters $\delta_{n}, D_{n}, \vs_{n}, \ell, \mu, \lambda$ we assumed in Section \ref{ssec:choice_of_parameters} hold, i.e. we will derive relations between the ``exponential parameters'' $b, c, m, \epsilon$ that need to hold. Let us recall that $b = m + \epsilon$.

We will first investigate the conditions on $m$ for $L_{n} = L^{m^{n+1}}$. The idea here is that if we have a condition of the form $L^{x}L_{n}^{y} \leq L_{n+1}$, this leads to the two conditions $m \geq x$ (for $n=0$) as well as $m \geq y$. The most important conditions are Equ. \eqref{eq:choice_m_1}, \eqref{eq:choice_m_2}, \eqref{eq:choice_m_3}, \eqref{eq:choice_m_4}, \eqref{eq:choice_m_5} and \eqref{eq:choice_m_6}. From these we get the two conditions
\begin{align*}
    m &\geq 10, \\
    m &\geq 2(r+\delta) + 9/2 \geq 2(r_{*} + 2 + \delta) + 9/2.
\end{align*}
Let us choose $r_{*} = 7$, $r = 9$, which means that we have to choose
\begin{align*}
    m \geq 23.
\end{align*}
Hence, let us choose $m=23$. Now, to go from $n \to n+1$ in the $C^{1}$ estimates Equ. \eqref{eq:iter_Aestimate} and \eqref{eq:iter_C1}, we need to have the condition
\begin{equation}
    \begin{split}\label{eq:A_Dn}
     A \delta_{n}^{1/2} \left( \frac{D_{n}}{\delta_{n+4}} \right)^{1+\epsilon} &= A a^{1/2 - (1+\epsilon)} a^{(-1/2+c(1+\epsilon)+ b^{4}(1+\epsilon))b^{n}} \\
    &= A a^{1/2 - (1+\epsilon)} a^{cb^{n+1}} \leq D_{n+1},
    \end{split}
\end{equation}
as by definition $-1/2+c(1+\epsilon)+ b^{4}(1+\epsilon) = cb$, which gives the following additional condition on $a$:
\begin{equation}\label{eq:choice_a_5}
    a \geq A^{\frac{1}{\epsilon + 1/2}}.
\end{equation}
Now, to determine $\epsilon$, we need to ensure that Equ. \eqref{eq:lambda_A} holds, i.e.
\begin{align*}
    \lambda^{1+\delta} \leq \tilde{C} \delta_{n}^{1/2} \left(\frac{D_n}{\delta_{n+4}}\right)^{1+\epsilon}.
\end{align*}
By the definitions of the parameters involved, this translates to
\begin{align*}
    b^{n}\left( -(1+\epsilon)(b^{4} + c) + \frac{4}{3}(1+\delta)\left[  \frac{2r_{*}}{\gamma} + \frac{2-\gamma}{\gamma_{*}} \right] +  \frac{2r_{*}}{\gamma}c + \frac{1}{2} \right) \leq 0.
\end{align*}
Using the definition $m=23$, $\gamma, \gamma_{*} \approx 1/2$, $\delta \approx 0$, $r_{*} = 7$ as well as the definitions of $c, b$, we see that this inequality holds for $\epsilon \geq 6$. To make sure that $b = m + \epsilon > 36$ (which we will need below), and in keeping with \cite{HLP23}, let us choose $\epsilon = 15$, which implies $b = 38$.

With these choices, one easily verifies that the inequalities of Section \ref{ssec:choice_of_parameters} hold, except for Equ. \eqref{eq:Dn_ell} which requires a condition on $a$, see the next section.

\subsection{Choice of \texorpdfstring{$a$}{a}}\label{ssec:choice_a}
Let us go through all conditions on $a$ from the previous sections.

First, from Equ. \eqref{eq:choice_a_1}, we deduce the condition
\begin{align*}
    &\underbrace{C\left(\left(1 + {M_v^{2\delta}} + C_e^{(1),1} + C_e^{(5),1+\delta} \right)^2 + M_{v} C_{e}^{(1),1} + \sqrt{\bar{e}}(1+C_{e}^{(1),1}) \right)}_{=: \tilde{C}_{e}}\lambda^{2\delta} \delta_{n+2}^{\frac{6}{5}} \\
    &\overset{\text{Equ. \eqref{eq:lambda_delta}}}{\leq} \tilde{C}_{e} \delta_{n}^{-1} \delta_{n+2}^{6/5} \overset{!}{\leq} \frac{\eta}{r_{0}} \delta_{n+1},
\end{align*}
which holds if (note that $\frac{6}{5} b^{2} - b - 1 > 0 \text{ for }b > 2$) 
\begin{align*}
    \frac{r_{0}}{\eta} \tilde{C}_{e} \delta_{n}^{-1} \delta_{n+2}^{6/5} \delta_{n+1}^{-1} &= \frac{r_{0}}{\eta} \tilde{C}_{e} a^{-\frac{4}{5} - b^{n}\left(\frac{6}{5} b^{2} - b - 1 \right)} ~ {\overset{n \geq 0}{\leq}}~  \frac{r_{0}}{\eta} \tilde{C}_{e} a^{ - \frac{6}{5}b^{2} + b + \frac{9}{5}} \overset{!}{\leq} 1.
\end{align*}
Let us simplify the constants a bit more: since $\eta = C \frac{\underline{e}}{1+\sqrt{\bar{e}}}$, we have 
\begin{equation} \label{eq:eta-inv}
\frac{1}{\eta} \leq C \frac{1+\sqrt{\bar{e}}}{\underline{e}} \leq C \frac{(1+\sqrt{\bar{e})^{2}}}{\underline{e}} \leq C \frac{1+\bar{e}}{\underline{e}}.
\end{equation}
Furthermore, we have
\begin{align*}
    \tilde{C}_{e} &\leq C\left(\left(1 + {M_v^{2\delta}} + \sqrt{\bar{e}}\frac{1+\bar{e}}{\underline{e}}\left(1+ \bar{e}+ \frac{\bar{e}}{1+\sqrt{\bar{e}}}\right) \right)^2 +\sqrt{\bar{e}}\frac{1+\underline{e}}{\underline{e}}(\sqrt{\bar{e}} + {M_v}) + \sqrt{\bar{e}}\right)\\
	&\leq C\left(1+ {M_v^{4\delta}} + \bar{e}\frac{(1+\bar{e})^4}{\underline{e}^2} +\sqrt{\bar{e}}\frac{1+\underline{e}}{\underline{e}}(\sqrt{\bar{e}} + {M_v}) + \sqrt{\bar{e}}\right)\\
	&\leq C\left(1+{M_v^{4\delta}} + \underline{e}(1+\bar{e})^2 +\sqrt{\bar{e}}(\sqrt{\bar{e}} + {M_v}) + \sqrt{\bar{e}}\right)\left(\frac{1+\bar{e}}{\underline{e}}\right)^2 \frac{\bar{e}}{\underline{e}},
\end{align*}
and by a simple comparison, we see that
\begin{align*}
    \frac{r_{0}}{\eta} \tilde{C}_{e} \leq \tilde{A}_{e} \left(\frac{1+\bar{e}}{\underline{e}}\right)^3 \frac{\bar{e}}{\underline{e}}.
\end{align*}
As we have $- \frac{6}{5}b^{2} + b + \frac{9}{5} < 0$ for $b > 2$, we see that the appropriate bound can easily be ensured if 
\begin{align*}
    a \geq \left( \tilde{A}_{e} \left(\frac{1+\bar{e}}{\underline{e}}\right)^3 \frac{\bar{e}}{\underline{e}} \right)^{\frac{5}{6b^{2} - 5b - 9}} =: a_{1}.
\end{align*}
Second, from the estimate on the divergence we got our requirement Equ. \eqref{eq:choice_a_2}. Using Equ. \eqref{eq:lambda_delta}, \eqref{eq:Dn_ell} and \eqref{eq:vn_97}, we find
\begin{align*}
    &C (1 + M_{v}^{2\delta} + M_{v}^{1-\delta} + M_{v}) (D_{n}^{1+2\delta}\ell^{\gamma} + D_{n}^{\delta} (n+1) \vs_{n}^{\gamma'}) \\
    &\leq C (1 + M_{v}^{2\delta} + M_{v}^{1-\delta} + M_{v})(D_{n} \delta_{n}^{-1} \ell^{\gamma} + \delta_{n}^{-1/2} (n+1) \vs_{n}^{\gamma'}) \\
    &\leq C (1 + M_{v}^{2\delta} + M_{v}^{1-\delta} + M_{v}) (1+\eta) \delta_{n+3}^{9/7} \\
    &= C (1 + M_{v}^{2\delta} + M_{v}^{1-\delta} + M_{v}) (1+\eta) \delta_{n+3}^{1/28} \delta_{n+3}^{5/4}.
\end{align*}
As $\eta \leq C \sqrt{\bar{e}}$, we see that $C (1 + M_{v}^{2\delta} + M_{v}^{1-\delta} + M_{v}) (1+\eta) \leq A$,  
and hence the divergence estimate holds if we can ensure $A \delta_{n+3}^{1/28}$, i.e.
\begin{align*}
    a \geq A^{\frac{28}{b^{3}-1}} =: a_{2}.
\end{align*}
Third, we got Equ. \eqref{eq:choice_a_3}, which, using Equ. \eqref{eq:CR0_A} easily follows if
\begin{align*}
    a \geq \left( A \left( \frac{1 + \bar{e}}{\underline{e}} \right) \left( \frac{\bar{e}}{\underline{e}} \right)^{2r} \right)^{\frac{5}{b^{2}-1}} =: a_{3}.
\end{align*}
Fourth, the other error terms led to Equ. \eqref{eq:choice_a_4}. We want to show that
\begin{equation}\label{eq:delta_3736}
    \ell^{-3\alpha}(n+1) \delta_{n+3}^{\frac{8}{3}\beta} \delta_{n+2}^{- 37/36} \leq 1.
\end{equation}
Once this has been done, we have 
\begin{align*}
    C_{e}^{\mathring{R},0} \ell^{-3\alpha}(n+1) \delta_{n+3}^{\frac{8}{3}\beta} \leq C_{e}^{\mathring{R},0} \delta_{n+2}^{1/36} \delta_{n+2}, 
\end{align*}
which is bounded above by $\frac{\eta}{2} \delta_{n+2}$ if 
\begin{align*}
    C_{e}^{\mathring{R},0} \delta_{n+2}^{1/36} \leq \frac{\eta}{2},
\end{align*}
which, similarly as in the previous step, follows if
\begin{align*}
    a \geq \left( A \left( \frac{1 + \bar{e}}{\underline{e}} \right) \left( \frac{\bar{e}}{\underline{e}} \right)^{2r} \right)^{\frac{36}{b^{2}-1}} =: a_{4}.
\end{align*}
Therefore, let us show Equ. \eqref{eq:delta_3736}. Using the definitions, the left-hand side is bounded by
\begin{align*}
    (n+1) a^{-37/36 + \frac{4}{3}\left(2\beta - \frac{3\alpha}{\gamma} \right)} a^{b^{n+3}\left( - \frac{8}{3}\beta + \frac{\alpha}{\gamma}\left(4 + 3 \frac{c}{b^{3}} \right) + \frac{37}{36}\frac{1}{b} \right) },
\end{align*}
which is bounded by $1$ as soon as
\begin{equation}\label{eq:cond_alpha_gamma}
    - \frac{8}{3}\beta + \frac{\alpha}{\gamma}\left(4 + 3 \frac{c}{b^{3}} \right) + \frac{37}{36}\frac{1}{b} < 0 \quad \Leftrightarrow \quad \frac{\alpha}{\gamma}\left(4 + 3 \frac{c}{b^{3}} \right) < \frac{8}{3}\beta - \frac{37}{36}\frac{1}{b}.
\end{equation}
For this to be well-defined, we first need
\begin{align*}
    \frac{8}{3}\beta - \frac{37}{36}\frac{1}{b} > 0,
\end{align*}
which, using $\beta = \frac{\gamma}{48}$ translates into a condition on $\gamma$:
\begin{align*}
    2\gamma > \frac{37}{b} = \frac{37}{38},
\end{align*}
which holds for $\gamma \in \left(\frac{37}{76}, 1/2 \right)$. Having established this, we translate Equ. \eqref{eq:cond_alpha_gamma} into a condition for $\alpha$:
\begin{equation}\label{eq:restr_alpha}
    \alpha < \tilde{\alpha}_{0} := \gamma \cdot \frac{\frac{8}{3}\beta - \frac{37}{36}\frac{1}{b}}{4 + 3 \frac{c}{b^{3}}} \sim 1.7 \cdot 10^{-4}.
\end{equation}
As $\alpha < \alpha_{0}$ and 
\begin{align*}
    \alpha_{0} = \frac{1}{2cb+1} \sim 8.7 \cdot 10^{-9} < \tilde{\alpha}_{0}, 
\end{align*}
this condition is clearly satisfied. Therefore, \eqref{eq:delta_3736} holds.

Fifth, in order to make sure that the estimate for the $C^1$-norms work, we needed Equ. \eqref{eq:choice_a_5}, i.e.
\begin{align*}
    a \geq A^{\frac{1}{\epsilon + 1/2}} =: a_{5}.
\end{align*}
Sixth, to guarantee parameter assumption Equ. \eqref{eq:Dn_ell}, we need
\begin{align*}
    D_{n} \ell^{\gamma} \leq \eta \delta_{n+3}^{9/7},
\end{align*}
which by the definitions of the terms involved holds true if
\begin{align*}
    \delta_{n+3}^{4/3 - 9/7} = \delta_{n+3}^{1/21} \leq \frac{\eta}{2},
\end{align*}
which, using Equ. \eqref{eq:eta-inv} as well as $\tilde{A}_{e} \geq 1$, leads to
\begin{align*}
    a \geq A^{\frac{21}{b^{3}-1}} =: a_{6}.
\end{align*}
Now we need to satisfy
\begin{align*}
    a \geq \max\left(a_{1}, \ldots, a_{6} \right).
\end{align*}
By the particular structure of the bounds $a_{i}$, $i = 1, \ldots, 6$, this amounts to finding the largest constant in the definitions of the $a_{i}$ and taking it to the largest exponent. Hence we define
\begin{equation}\label{eq:def_a}
    a := \left( A \left( \frac{1 + \bar{e}}{\underline{e}} \right) \left( \frac{\bar{e}}{\underline{e}} \right)^{2r} \right)^{\frac{1}{\epsilon + 1/2}}.
\end{equation}

\begin{rem}
 Note that the least trivial conditions on the coefficients from the estimates on the stress terms originate from the new terms $\mathring{R}^{{\rm flow}, 2}$ and $\mathring{R}^{{\rm moll}, 3}$.
 It is to be expected that the method used in this paper for bounding the flow and mollification error might break down as one tries to achieve higher regularity of the convex integration solutions, where we hope that in the end $b = 1 + \epsilon$ for a small $\epsilon > 0$ is possible.
\end{rem}

\subsection{Choice of \texorpdfstring{$M_{v}$}{Mv}}\label{ssec:choice_Mv}

Having defined $a$, we can finally choose the remaining constant $M_{v}$ and conclude the proof.
\\ 
\textbf{Estimate for $v_{n+1} - v_{n}$:} Note that 
\begin{align*}
    a = \left( \tilde{A}_{e} \left( \frac{1 + \bar{e}}{\underline{e}} \right)^{3} \left( \frac{\bar{e}}{\underline{e}} \right)^{2r+3} \right)^{\frac{1}{\epsilon + 1/2}} \overset{\text{Equ. \eqref{eq:tildeA_1}}}{\geq}  \left( \frac{1 + \bar{e}}{\underline{e}} \right)^{\frac{3}{16}} \geq \frac{1}{\underline{e}^{\frac{3}{16}}}. 
\end{align*}
From Lemma \ref{lem:total_perturb} it follows that
\begin{align*}
     L_{n}^{-4} \delta_{n}^{-1/2}\| v_{n+1} - v_{n} \|_{C_{\leq \mft_{L}} C_{x}} &\leq C\left(  \sqrt{\bar{e}}  + \left(1+ M_{v}^{2\delta} + \eta + C_e^{(1),1} + C_e^{(5),1+\delta} \right) \delta_{n+1}^{1/2}  \right) \\
     &\leq C\left(\sqrt{\bar{e}} + \left(1 + \sqrt{\bar{e}}\frac{(1+\bar{e})^{2}}{\underline{e}}\right) a^{\frac{1}{2}(1 - b^{n+1})} \right) (1+  M_{v}^{2\delta} ) \\
     &\overset{n \geq 0}{\leq} C\left(\sqrt{\bar{e}} + \left(1 + \sqrt{\bar{e}}\frac{(1+\bar{e})^{2}}{\underline{e}}\right) a^{\frac{1}{2}(1 - b)} \right) (1+  M_{v}^{2\delta} ) \\
     &\leq C\left(\sqrt{\bar{e}} + \left(1 + \sqrt{\bar{e}}\frac{(1+\bar{e})^{2}}{\underline{e}}\right) \underline{e}^{\frac{3}{32}(b - 1)} \right) (1+  M_{v}^{2\delta} ) \\
     &\leq C\left(\sqrt{\bar{e}} +\underline{e}^{\frac{3}{32}(b - 1)} + \frac{\sqrt{\bar{e}}}{\underline{e}} \underline{e}^{\frac{3}{32}(b - 1)}  + \frac{\sqrt{\bar{e}} \bar{e}^{2}}{\underline{e}} \underline{e}^{\frac{3}{32}(b - 1)} \right) (1+  M_{v}^{2\delta} ) \\
     &\leq C \max \left( \sqrt{\bar{e}}, \bar{e}^{\frac{3}{32}(b - 1)}, \bar{e}^{\frac{3}{32}(b - 1) -1/2 }, \bar{e}^{\frac{3}{32}(b - 1) - 3/2 } \right) (1+  M_{v}^{2\delta} ) \\
     &= C \max \left( \sqrt{\bar{e}}, \bar{e}^{\frac{3}{32}(b - 1)},  \bar{e}^{\frac{3}{32}(b - 1) - 3/2 } \right) (1+  M_{v}^{2\delta} ) \\
     &=: E(1+M_{v}^{2\delta}) \overset{!}{\leq} M_{v}.
\end{align*}
Consider $f(x) := \frac{x}{1+x^{2\delta}}$. It is obviously continuous and strictly monotone increasing and satisfies $\lim_{x \to 0} f(x) = 0$ and $\lim_{x \to \infty} f(x) = \infty$. Thus, as $E \in (0,\infty)$, there is a unique $x = M_{v}$ such that
\begin{align}\label{Mv1}
    f(M_{v}) = E, \quad \text{i.e. } E = \frac{M_{v}}{1+M_{v}^{2\delta}}.
\end{align}
The previous equality cannot be solved explicitly for $M_v$, but it will be sufficient to work with the following bounds for $M_v$
\begin{equation}\label{eq:Mv_bounds}
    M_{v} \in \left[E, \max\left( 2E, \left( 2E \right)^{\frac{1}{1-2\delta}} \right) \right].
\end{equation}
With this choice of $M_{v}$, we have defined all parameters of the convex integration scheme. 

\subsection{Estimate of the regularity}
With the choices for the parameters of Section \ref{ssec:pf_choice_of_parameters}, we find that the H\"older exponent $\theta$ of the solution is less than
\begin{align*}
    \bar{\theta} &= \alpha_{0} = \frac{1}{2bc+1} = \frac{m-1}{2(1+\epsilon)(m+\epsilon)^{5} + m - 2 -\epsilon} \\
    &= \frac{23 - 1}{2(1+15)(23+15)^{5} + 23 - 2 -15} \sim 8.68 \cdot 10^{-9},
\end{align*}
which, as in \cite{HLP23}, is extremely low and far from the deterministic Onsager threshold of $1/3$. Because of the restrictions due to the new mollification error terms, we do not expect that this exponent could be improved substantially without finding a better way than Besov interpolation to deal with these error terms.

\section{Proofs}\label{sect:proofs}

\subsection{Core estimates}\label{sec-proof-core-est}
 
In this section we provide proofs for all necessary estimates of the amplitude coefficients $a_{k}$ and its derivatives in various norms.

\begin{proof}[Proof of Lemma \ref{claim-energy-1}] We estimate, using the definition, 
\[\|\tilde{\rho}_{\ell}\|_{C_{\leq \mft_{L}}C_{x}} \leq \frac{2}{r_0} \left(\eta \delta_{n+1} + \|\mathring{R}_{\ell}\|_{C_{\leq \mft_{L}}C_{x}}\right)\leq \frac{4}{r_0}{L_n} \eta \delta_{n+1}.\]
Observe that for $\Phi(z) := \frac{2}{r_0} \sqrt{\eta^2 \delta_{n+1}^2 + z^2}$, it holds on $[0,{L_n}\eta \delta_{n+1}]$ that
\[ |D^j \Phi(z) | \leq C ({L_n}\eta \delta_{n+1})^{1-j}. \]
Therefore, we find
\begin{align*}
	\|\tilde{\rho}_{\ell}\|_{C_{\leq \mft_{L},x}^1} &= \|\tilde{\rho}_{\ell}\|_{C_{\leq \mft_{L}}C_{x}^1} + \|\tilde{\rho}_{\ell}\|_{C_{\leq \mft_{L}}^{1}C_{x}},\\
	\|\tilde{\rho}_{\ell}\|_{C_{\leq \mft_{L}}C_{x}^1} &\leq C \|\mathring{R}_{\ell}\|_{C_{\leq \mft_{L}}C_{x}^1} \leq C \ell^{-1} \|\mathring{R}_n\|_{C_{\leq \mft_{L}}C_{x}} \overset{\text{Equ. \eqref{eq:iter_stress}}}{\leq} C{L_n} \eta \ell^{-1} \delta_{n+1},\\
	\|\tilde{\rho}_{\ell}\|_{C_{\leq \mft_{L}}^{1}C_{x}} &\leq C \|\mathring{R}_{\ell}\|_{C_{\leq \mft_{L}}^{1}C_{x}} \leq C {L_n} \ell^{-1} \eta \delta_{n+1}. \qedhere
\end{align*}
\end{proof}

\begin{proof}[Proof of Proposition \ref{prop-energy-1}] Throughout the proof, we let $t \in [0,\mft_{L}]$. In the following we often suppress the (uniform) time dependence and also write $\| \cdot \|_{C}$ for $\| \cdot \|_{C_x}$ to simplify the notation. We begin by applying the product rule to estimate
\begin{align*}
	\|a_k\|_{C_x^{\delta}} &\leq \|\sqrt{\rho_{\ell}}\|_{C_x^{\delta}} \|\Gamma\|_C \|\Psi\|_C + \|\sqrt{\rho_{\ell}}\|_C \|\Gamma\|_{C_x^{\delta}} \|\Psi\|_C + \|\sqrt{\rho_{\ell}}\|_C \|\Gamma\|_C \|\Psi\|_{C_x^{\delta}},
\end{align*}
\[ \|a_k\|_{C_x^r} \leq C\sum_{r_1+r_2+r_3=r} \|\sqrt{\rho_{\ell}}\|_{C_x^{r_1}} \|\Gamma\|_{C_x^{r_2}} \|\Psi\|_{C_x^{r_3}},\]
\begin{align*}
	\|\partial_{\tau}a_k\|_{C_x^{\delta}} \leq \|\sqrt{\rho_{\ell}}\|_{C_x^{\delta}} \|\Gamma\|_C \|\partial_{\tau}\Psi\|_C\ + \|\sqrt{\rho_{\ell}}\|_C \|\Gamma\|_{C_x^{\delta}} \|\partial_{\tau}\Psi\|_C + \|\sqrt{\rho_{\ell}}\|_C \|\Gamma\|_C \|\partial_{\tau}\Psi\|_{C_x^{\delta}},
\end{align*}
\[ \|\partial_{\tau}a_k\|_{C_x^r} \leq C\sum_{r_1+r_2+r_3=r} \|\sqrt{\rho_{\ell}}\|_{C_x^{r_1}} \|\Gamma\|_{C_x^{r_2}} \|\partial_{\tau}\Psi\|_{C_x^{r_3}},\]
\begin{align*}
	\|\partial_{\tau} a_k + {\rm i} (k \cdot \tilde{v}) a_k\|_{C_x^{\delta}} &\leq \|\sqrt{\rho_{\ell}}\|_{C_x^{\delta}} \|\Gamma\|_C \|\partial_{\tau} \Psi + {\rm i} (k \cdot \tilde{v}) \Psi\|_C + \|\sqrt{\rho_{\ell}}\|_C \|\Gamma\|_{C_x^{\delta}} \|\partial_{\tau} \Psi + {\rm i} (k \cdot \tilde{v}) \Psi\|_C\\
	&\qquad + \|\sqrt{\rho_{\ell}}\|_C \|\Gamma\|_C \|\partial_{\tau} \Psi + {\rm i} (k \cdot \tilde{v}) \Psi\|_{C_x^{\delta}},
\end{align*}
\[ \|\partial_{\tau} a_k + {\rm i} (k \cdot \tilde{v}) a_k\|_{C_x^r} \leq C\sum_{r_1+r_2+r_3=r} \|\sqrt{\rho_{\ell}}\|_{C_x^{r_1}} \|\Gamma\|_{C_x^{r_2}} \|\partial_{\tau} \Psi + {\rm i} (k \cdot \tilde{v} )\Psi\|_{C_x^{r_3}}.\]
\textbf{Proof of} \eqref{eq:a_k_Cdelta}:
Let us consider the individual factors: since $C \underline{e} \delta_n \leq |\rho_{\ell}| \leq C{L_n} \bar{e} \delta_n$, we find 
\begin{align*}
	\|\sqrt{\rho_{\ell}}\|_{C_x^{\delta}} &= \|\sqrt{\rho_{\ell}}\|_C + [\sqrt{\rho_{\ell}}]_{C_x^{\delta}} \leq \|\sqrt{\rho_{\ell}}\|_C + C(\underline{e}\delta_n)^{-\frac{1}{2}}\|\rho_{\ell}\|_{C_x^{\delta}},\\
	&\leq C {L_n^{\frac{1}{2}}}\sqrt{\bar{e}} \delta_n^{\frac{1}{2}} + C (\underline{e}\delta_n)^{-\frac{1}{2}}\|\mathring{R}_{\ell}\|_{C_x^{\delta}} 
	\leq C {L_n^{\frac{1}{2}}}\sqrt{\bar{e}} \delta_n^{\frac{1}{2}} + C {L_n}\underline{e}^{-\frac{1}{2}} \eta^{1-\delta}\delta_n^{-\frac{1}{2}} \delta_{n+1}^{1-\delta} D_n^{\delta}\\
	&\leq C {L_n}(\sqrt{\bar{e}} + \underline{e}^{-\frac{1}{2}} \eta^{1-\delta}) \delta_n^{\frac{1}{2}} D_n^{\delta}.
\end{align*}
Next, it holds by definition that $\|\Gamma\|_{C_{t,x}^{0}}, \|\Psi\|_{C_{t,x}^{0}} \leq C$. Turning to the H\"older norm of $\Gamma$, we find 
\begin{align*}
	\|\Gamma\|_{C_x^{\delta}} &= \|\Gamma\|_C + [\Gamma]_{C_x^{\delta}} \leq C + C\left[\frac{\mathring{R}_{\ell}}{\rho_{\ell}}\right]_{C_x^{\delta}} \leq C + C\|\mathring{R}_{\ell}\|_C[\rho_{\ell}^{-1}]_{C_x^{\delta}} + C [\mathring{R}_{\ell}]_{C_x^{\delta}}\|\rho_{\ell}^{-1}\|_C\\
	&\leq C + C\|\mathring{R}_{\ell}\|_C(\underline{e}\delta_n)^{-2}[\mathring{R}_{\ell}]_{C_x^{\delta}}  + C [\mathring{R}_{\ell}]_{C_x^{\delta}}(\underline{e}\delta_n)^{-1}\\
	&\leq C + C{L_n}\eta \delta_{n+1}(\underline{e}\delta_n)^{-2}{L_n}(\eta \delta_{n+1})^{1-\delta} D_n^{\delta} + C {L_n}(\eta \delta_{n+1})^{1-\delta} D_n^{\delta}(\underline{e}\delta_n)^{-1}\\
	&\leq C + C {L_n}\eta^{1-\delta} \underline{e}^{-1}  (1 + {L_n}\eta \delta_{n+1}(\underline{e}\delta_n)^{-1})\delta_n^{-1}\delta_{n+1}^{1-\delta} D_n^{\delta}\\
	&\leq C {L_n^2} (1+\eta^{1-\delta} \underline{e}^{-1})\delta_n^{-1}\delta_{n+1}^{1-\delta} D_n^{\delta}.
\end{align*}
Finally,
\[\|\Psi\|_{C_x^{\delta}} \leq C \mu^{\delta} \|\tilde{v}\|_{C_x^1}^{\delta} \leq C {L_n^{\delta}}\mu^{\delta} \varsigma_{n+1}^{\delta(\alpha-1)}.\]
Hence we can estimate the H\"older norm of $a_{k}$ 
\[\|a_k\|_{C_x^{\delta}}\leq C_e{L_n^{\frac{5}{2}}}\mu^{\delta}\varsigma_{n+1}^{\delta(\alpha-1)} \delta_n^{\frac{1}{2}}, \]
where
\[ C_e = C\left(\sqrt{\bar{e}}+ \underline{e}^{-\frac{1}{2}} \eta^{1-\delta} + \sqrt{\bar{e}}\eta^{1-\delta} \underline{e}^{-1}\right).\]
\textbf{Proof of} \eqref{eq:a_k_Cr}:
For $\|a_k\|_{C_x^r}$ we proceed as before, estimating each factor. Let us start with
$\|\sqrt{\rho_{\ell}}\|_{C_x^r}$. First, observe that on the interval $[C \underline{e} \delta_n, C{L_n}\bar{e}\delta_n]$, the function $x \mapsto \sqrt{x}$ satisfies for $j \geq 1$
\[ [\sqrt{\cdot}]_j \leq C (\underline{e} \delta_n)^{\frac{1}{2}-j}.\]
Hence
\begin{align*}
	\|\sqrt{\rho_{\ell}}\|_{C_x^r} &\leq C({L_n}\bar{e}\delta_n)^{\frac{1}{2}} + C \sum_{j=1}^r (\underline{e} \delta_n)^{\frac{1}{2}-j} \|\rho_{\ell}\|_{C_x}^{j-1} \|\rho_{\ell}\|_{C_x^r}\\
	&\leq C({L_n}\bar{e}\delta_n)^{\frac{1}{2}} + C \delta_n^{-\frac{1}{2}} \sum_{j=1}^r (\underline{e})^{\frac{1}{2}-j} ({L_n}\bar{e})^{j-1} \|\rho_{\ell}\|_{C_x^r}.
\end{align*}
Note that
\begin{align*}
	\|\rho_{\ell}\|_{C_x^r} &\leq \|\rho_{\ell}\|_{CC} + \sum_{j=1}^r [\tilde{\rho}_{\ell}]_j \leq C{L_n}\bar{e} \delta_n + C \sum_{j=1}^r ({L_n}\eta\delta_{n+1})^{1-j} \|\mathring{R}_{\ell}\|_{C_x}^{j-1}\|\mathring{R}_{\ell}\|_{C_x^r}\\
	&\leq C{L_n}\bar{e} \delta_n + C \sum_{j=1}^r ({L_n}\eta\delta_{n+1})^{1-j} ({L_n}\eta \delta_{n+1})^{j-1}\ell^{1-r} {L_n}D_n\\
	&\leq C{L_n}(1+ \bar{e}) \ell^{1-r}D_n.
\end{align*}
Hence, combining the two previous calculations, we obtain
\begin{align*}
	\|\sqrt{\rho_{\ell}}\|_{C_x^r} &\leq (C {L_n}\bar{e}\delta_n)^{\frac{1}{2}} + C\frac{\sqrt{\underline{e}}(1+ \bar{e})}{\bar{e}} \delta_n^{-\frac{1}{2}} \ell^{1-r}D_n {L_n^r}\sum_{j=1}^r \left(\frac{\bar{e}}{\underline{e}}\right)^j\\
	&\leq C{L_n^r}\sqrt{\bar{e}}\left(1+ \frac{1+ \bar{e}}{\bar{e}}\sum_{j=1}^r \left(\frac{\bar{e}}{\underline{e}}\right)^j\right) \delta_n^{-\frac{1}{2}} \ell^{1-r}D_n\\
	&=:C{L_n^r} \sqrt{\bar{e}} C_e^{(r)} \delta_n^{-\frac{1}{2}} \ell^{1-r}D_n.
\end{align*}
For $\|\Gamma\|_{C_x^r}$, using the boundedness of $\gamma_l^{(j)}$, we immediately see that
\[\|\Gamma\|_{C_x^r} \leq C + C \sum_{r_1 + r_2 =r} \|\mathring{R}_{\ell}\|_{C_x^{r_1}}\|\rho_{\ell}^{-1}\|_{C_x^{r_2}}.\]
Note that on $[C \underline{e} \delta_n, \infty)$, the function $x \mapsto x^{-1}$ satisfies for $j \geq 1$
\[[(\cdot)^{-1} ]_j \leq C (\underline{e}\delta_n)^{-(j+1)}.\]
Hence we obtain
\begin{align*}
	\|\rho_{\ell}^{-1}\|_{C_x^{r_2}} &\leq C (\underline{e}\delta_n)^{-1} + C \sum_{j=1}^{r_2} (\underline{e}\delta_n)^{-j-1}\|\rho_{\ell}\|_{C_x}^{j-1} \|\rho_{\ell}\|_{C_x^{r_2}}\\
	&\leq C\left( \underline{e}^{-1} + \frac{1+ \bar{e}}{\underline{e}\bar{e}}{L_n^{r_2}}\sum_{j=1}^{r_2}\left(\frac{\bar{e}}{\underline{e}} \right)^j\right) \delta_n^{-2} \ell^{1-r_2}D_n \\
	&\leq C {L_n^{r_2}}\underline{e}^{-1} C_e^{(r_2)} \delta_n^{-2} \ell^{1-r_2}D_n =: C_e^{\rho^{-1},r_2}{L_n^{r_2}}\delta_n^{-2}D_n \ell^{1-r_2}.
\end{align*}
We combine the above estimates to find
\begin{align*}
	\|\Gamma\|_{C_x^r} &\leq C {L_n^{r+1}}\left(1+ \underline{e}^{-1} + \eta \sum_{j=1}^r C_e^{\rho^{-1},j}\right) \delta_n^{-1} \ell^{1-r}D_n \leq C {L_n^{r+1}}C_e^{(r)} \delta_n^{-1} \ell^{1-r}D_n.
\end{align*}
For $\|\Psi\|_{C_x^r}$, we use the boundedness of $\| \Psi \|_{C}$ and the interpolation inequality \cite[Equ. $(4.5)$]{ DLS_1-10} to find
\begin{align}
	\|\Psi\|_{C_x^r} &\leq C \left(1 +  \sum_{j=1}^r [\psi_k^{(i)}]_j [\tilde{v}]_r^{\frac{r-j}{r-1}}[\tilde{v}]_1^{(j-1)\frac{r}{r-1}}\right)\notag\\
	&\leq C \left(1 + \sum_{j=1}^r \mu^j ({L_n} D_n \ell^{1-r} + {L} \varsigma_{n+1}^{\alpha -1})^{\frac{r-j}{r-1}}({L_n} \varsigma_{n+1}^{\alpha -1})^{(j-1)\frac{r}{r-1}}\right)\notag\\
	&\leq C {L_n^r} \sum_{j=1}^r \mu^j \left( \max \left(D_n \ell^{1-r}, \varsigma_{n+1}^{\alpha -1}\right)\right)^{\frac{r-j}{r-1}}(\varsigma_{n+1}^{\alpha -1})^{(j-1)\frac{r}{r-1}}.\label{eq:altPsiEstimate}
\end{align}
Note that for $\max \left(D_n \ell^{1-r}, \varsigma_{n+1}^{\alpha -1}\right) = \varsigma_{n+1}^{\alpha -1}$, it holds that
\[\mu^j \left( \max \left(D_n \ell^{1-r}, \varsigma_{n+1}^{\alpha -1}\right)\right)^{\frac{r-j}{r-1}}(\varsigma_{n+1}^{\alpha -1})^{(j-1)\frac{r}{r-1}} = \mu^j \varsigma_{n+1}^{j(\alpha-1)}\]
and in the other case $\max \left(D_n \ell^{1-r}, \varsigma_{n+1}^{\alpha -1}\right) = D_n \ell^{1-r}$, using $D_n \leq \varsigma_{n+1}^{\alpha -1}$ and $\ell^{-1}\leq \mu$, we see that also
\[\mu^j \left( \max \left(D_n \ell^{1-r}, \varsigma_{n+1}^{\alpha -1}\right)\right)^{\frac{r-j}{r-1}}(\varsigma_{n+1}^{\alpha -1})^{(j-1)\frac{r}{r-1}} = \mu^j \left(D_n \ell^{1-r}\right)^{\frac{r-j}{r-1}}(\varsigma_{n+1}^{\alpha -1})^{(j-1)\frac{r}{r-1}} \leq \mu^r \varsigma_{n+1}^{j(\alpha-1)}.\]
Thus, we obtain
\[\|\Psi\|_{C_x^r} \leq C {L_n^r} \mu^r \varsigma_{n+1}^{r(\alpha -1)},\]
with a constant which maximizes both the above cases.

\noindent We collect
\begin{align*}
	&\|a_k\|_{C_x^r} \\
	&\lesssim \|\sqrt{\rho}_{\ell}\|_0 \left(\|\Psi\|_r + \sum_{i=1}^{r-1} \|\Gamma\|_i\|\Psi\|_{r-i} + \|\Gamma\|_r\right) \\
	&+ \sum_{j=1}^{r-1} \|\sqrt{\rho}_{\ell}\|_j \left(\|\Psi\|_{r-j} + \sum_{i=1}^{r-j-1} \|\Gamma\|_i\|\Psi\|_{r-j-i} + \|\Gamma\|_{r-j}\right)\\
	&\quad + \|\sqrt{\rho}_{\ell}\|_r \\
	&\lesssim {L_n^{\frac{1}{2}}}\sqrt{\bar{e}} \delta_n^{\frac{1}{2}} \left({L_n^r}\mu^r \varsigma_{n+1}^{r(\alpha-1)} + {L_n^{r+1}}\sum_{i=1}^{r-1} C_e^{(i)} \delta_n^{-1} \ell^{1-i}D_n \mu^{r-i}\varsigma_{n+1}^{(r-i)(\alpha-1)} +{L_n^{r+1}}C_e^{(r)} \delta_n^{-1} \ell^{1-r}D_n\right)\\
	&\quad + \sum_{j=1}^{r-1} {L_n^j}\sqrt{\bar{e}} C_e^{(j)} \delta_n^{-\frac{1}{2}} \ell^{1-j}D_n \\
	&\qquad \cdot \left({L_n^{r-j}}\mu^{r-j} \varsigma_{n+1}^{(r-j)(\alpha-1)} + {L_n^{r-j+1}}\sum_{i=1}^{r-j-1} C_e^{(i)} \delta_n^{-1} \ell^{1-i}D_n\mu^{r-j-i} \varsigma_{n+1}^{(r-j-i)(\alpha-1)} \right.\\
	&\qquad\quad\left.+ {L_n^{r-j+1}} C_e^{(r-j)} \delta_n^{-1} \ell^{1-(r-j)}D_n\right)\\
	&\quad + {L_n^r}\sqrt{\bar{e}} C_e^{(r)} \delta_n^{-\frac{1}{2}} \ell^{1-r}D_n\\
	&\lesssim {L_n^{r + \frac{3}{2}}}\sqrt{\bar{e}} \delta_n^{\frac{1}{2}} \left(\mu^r \varsigma_{n+1}^{r(\alpha-1)}  + \sum_{i=1}^{r-1}\eta  C_e^{(i)} \mu^{r-i} \varsigma_{n+1}^{(r-i)(\alpha-1)}+C_e^{(r)} \eta \right)\\
	&\quad + \delta_n^{\frac{1}{2}} {L_n^{r+1}}\sum_{j=1}^{r-1} \sqrt{\bar{e}} C_e^{(j)} \eta \left(\mu^{r-j} \varsigma_{n+1}^{(r-j)(\alpha-1)}  + \sum_{i=1}^{r-j-1} C_e^{(i)} \eta \ell^{-i}\mu^{r-j-i} \varsigma_{n+1}^{(r-j-i)(\alpha-1)} +  C_e^{(r-j)} \eta\ell^{-(r-j)}\right)\\
	&\quad + {L_n^r}\sqrt{\bar{e}} C_e^{(r)} \eta \delta_n^{\frac{1}{2}}\ell^{-r},
\end{align*}
where we used $\ell D_n \leq \eta \delta_n$. The above hence finally yields
\[\|a_k\|_{C_x^r} \leq C_e^{[r]} {L_n^{r+\frac{3}{2}}}\delta_n^{\frac{1}{2}} \mu^r \varsigma_{n+1}^{r(\alpha-1)},\]
where collecting the energy-dependent constants in the above yields the definition
\[C_e^{[r]}:=C\sqrt{\bar{e}}\left( 1 +C_e^{(r)} \eta  + \eta^2 \sum_{j=1}^{r-1}  C_e^{(j)} C_e^{(r-j)} \right).\]
\textbf{Proof of} \eqref{eq:dtau_ik_a_k_Cdelta}:
First, let us note that by the above observation
$$
    \left( \partial_{\tau} + (ik \cdot \tilde{v}) \right) a_{k} = \sqrt{\rho_{\ell}}\ \Gamma \left( \partial_{\tau} + (ik \cdot \tilde{v}) \right) \Psi.
$$
Therefore, we estimate
\[ \|\partial_{\tau} \Psi + {\rm i} (k \cdot \tilde{v}) \Psi\|_{C_x^{\delta}} \leq C \mu^{\delta -1} \|\tilde{v}\|_{C_x^1}^{\delta} \leq C{L_n^{\delta}} \mu^{\delta-1}\varsigma_{n+1}^{\delta(\gamma-1)},\]
and use this to conclude that
\[\|\partial_{\tau} a_k + {\rm i} (k \cdot \tilde{v}) a_k\|_{C_x^{\delta}} \leq C_e {L_n^{\frac{5}{2}}}\mu^{\delta-1}\varsigma_{n+1}^{\delta(\gamma-1)} \delta_n^{\frac{1}{2}}.\]
\textbf{Proof of} \eqref{eq:dtau_a_k_Cr}:
For $r \in \mathbb{N}$, we estimate analogously to \cite[p. 34]{HLP23}, using \eqref{eq:altPsiEstimate}, to obtain
\[\|\partial_{\tau} a_k + {\rm i} (k \cdot \tilde{v}) a_k\|_{C_x^r} \leq C_e^{[r]}{L_n^{r+1}}\mu^{r-1}\varsigma_{n+1}^{r(\gamma-1)} \delta_n^{\frac{1}{2}}.\]
\textbf{Proof of} \eqref{eq:ds_a_k_Cdelta} and \eqref{eq:ds_a_k_Cr}: By the product rule we find
\[\partial_s a_k = (\partial_s \sqrt{\rho_{\ell}})\Gamma \Psi + \sqrt{\rho}(\partial_s \Gamma) \Psi + \sqrt{\rho_{\ell}}\Gamma (\partial_s \Psi).\]
Let us treat each term in turn. First of all observe that
\begin{align*}
	\|\partial_s\sqrt{\rho_{\ell}}\|_{C_x} &\leq C (\underline{e}\delta_n)^{-\frac{1}{2}} \|\partial_s \rho_{\ell}\|_{C_x},\\
	\|\partial_s\sqrt{\rho_{\ell}}\|_{C_x^r} &\leq C \sum_{r_1 + r_2 =r} \left\|\rho_{\ell}^{-\frac{1}{2}}\right\|_{C_x^{r_1}} \|\partial_s \rho_{\ell}\|_{C_x^{r_2}}.
\end{align*}
On $[C\underline{e}\delta_n, \infty)$, the function $x \mapsto \frac{1}{\sqrt{x}}$ satisfies for $j \geq 1$
\[ \left[ (\cdot)^{-\frac{1}{2}}\right]_j \leq C (\underline{e}\delta_n)^{-\left(j+\frac{1}{2}\right)}.\]
Hence we can easily see that
\begin{align*}
	\left\|\rho_{\ell}^{-\frac{1}{2}}\right\|_{C_x} &\leq C (\underline{e}\delta_n)^{-\frac{1}{2}},\\
	\left\|\rho_{\ell}^{-\frac{1}{2}}\right\|_{C_x^r} &\leq C (\underline{e}\delta_n)^{-\frac{1}{2}} + C \sum_{j=1}^r (\underline{e}\delta_n)^{-\left(j +\frac{1}{2}\right)} \|\rho_{\ell}\|_{C_x}^{j-1} \|\rho_{\ell}\|_{C_x^r}\\
	&\leq (\underline{e}\delta_n)^{-\frac{1}{2}}\left(C  + C{L_n^r}\delta_n^{-1}\ell^{1-r}D_n \frac{1+\bar{e}}{\bar{e}}\sum_{j=1}^r \left(\frac{\bar{e}}{\underline{e}}\right)^j\right)\\
	&\leq C (\underline{e}\delta_n)^{-\frac{1}{2}} C_e^{(r)} {L_n^r}\delta_n^{-1} \ell^{1-r}D_n.
\end{align*}
Furthermore,
\[ \partial_s \rho_{\ell} = \partial_s \tilde{\rho}_{\ell} + \partial_s \gamma_n.\]
Since $|\partial_s \gamma_n| \leq C \| \tilde{e} \|_{C_{\leq \mathfrak{t}}^1}$, we find 
\[|\partial_s \gamma_n| \leq C (|e|_{C^1} + \sqrt{\bar{e}}D_n).\]
On the other hand, standard mollification estimates imply that
\[ \|\partial_s \mathring{R}_{\ell}\|_r \leq C {L_n} \eta \delta_{n+1} \ell^{-1-r},\]
as well as
\begin{align*}
	\left\|\left(\sqrt{\eta^2\delta_{n+1}^2 + |\cdot|^2}\right)'(\mathring{R}_{\ell})\right\|_r &\leq C({L_n}\eta \delta_{n+1})^{-1} \|\mathring{R}_{\ell}\|_r + C({L_n}\eta \delta_{n+1})^{-r} \|\mathring{R}_{\ell}\|_1^r \leq C \ell^{-r}.
\end{align*}
Hence we can conclude that 
\begin{align*}
	\|\partial_s \tilde{\rho}_{\ell}\|_{C_x} &\leq C \|\partial_s \mathring{R}_{\ell}\|_{C_x} \leq C {L_n}\eta \delta_{n+1} \ell^{-1},\\
	\|\partial_s \tilde{\rho}_{\ell}\|_{C_x^r} &\leq C \sum_{j=0}^r \left\|\left(\sqrt{\eta^2\delta_{n+1}^2 + |\cdot|^2}\right)'(\mathring{R}_{\ell})\right\|_j\|\partial_s \mathring{R}_{\ell}\|_{r-j} \leq C {L_n} \eta \delta_{n+1} \ell^{-1-r},
\end{align*}
which in turn implies (by $D_n \leq \eta \delta_{n+1} \ell^{-1}$)
\begin{align*}
	\|\partial_s \rho_{\ell}\|_{C_x} &\leq C({L_n} \eta \delta_{n+1}\ell^{-1} + |e|_{C^1} + \sqrt{\bar{e}}D_n)\leq C({L_n}\eta(1+\sqrt{\bar{e}}) + |e|_{C^1})\delta_{n+1}\ell^{-1},\\
	\|\partial_s \rho_{\ell}\|_{C_x^r} &\leq C{L_n} \eta \delta_{n+1} \ell^{-1-r}.
\end{align*}
Thus
\begin{align*}
	\|\partial_s \sqrt{\rho_{\ell}}\|_{C_x} &\leq C \frac{{L_n}\eta(1+\sqrt{\bar{e}}) + |e|_{C^1}}{\sqrt{\underline{e}}}\delta_n^{-\frac{1}{2}}\delta_{n+1}\ell^{-1},\\
	\|\partial_s \sqrt{\rho}_{\ell}\|_{C_x^r} &\leq C\left(\left\|\rho_{\ell}^{-\frac{1}{2}}\right\|_{C_x} \|\partial_s \rho_{\ell}\|_{C_x^r} + \sum_{j=1}^{r-1} \left\|\rho_{\ell}^{-\frac{1}{2}}\right\|_{C_x^j} \|\partial_s \rho_{\ell}\|_{C_x^{r-j}} + \left\|\rho_{\ell}^{-\frac{1}{2}}\right\|_{C_x^r} \|\partial_s \rho_{\ell}\|_{C_x}\right)\\
	&\leq C \left((\underline{e}\delta_n)^{-\frac{1}{2}} {L_n} \eta \delta_{n+1}\ell^{-1-r} + \sum_{j=1}^{r-1} (\underline{e}\delta_n)^{-\frac{1}{2}} {L_n^j}C_e^{(j)} \delta_n^{-1} \ell^{1-j}D_n {L_n} \eta \delta_{n+1} \ell^{-1-(r-j)} \right.\\
	&\qquad\quad \left.+ (\underline{e}\delta_n)^{-\frac{1}{2}} {L_n^r}C_e^{(r)} \delta_n^{-1} \ell^{1-r}D_n({L_n}\eta(1+\sqrt{\bar{e}}) + |e|_{C^1})\delta_{n+1}\ell^{-1}\right)\\
	&\leq C {L_n^{r+1}}\frac{\eta}{\sqrt{\underline{e}}}\left(1+ C_e^{(r)}(\eta(1+\sqrt{\bar{e}})+|e|_{C^1})\right)\delta_n^{-\frac{1}{2}} \delta_{n+1}\ell^{-1-r} =: {L_n^{r+1}}C_e^{\partial_s\sqrt{\rho},r} \delta_n^{-\frac{1}{2}} \delta_{n+1}\ell^{-1-r}.
\end{align*}
Next we estimate
\[ \|\partial_s \Gamma \|_{C_x^r} \leq C \sum_{i=0}^r \left\| (D\gamma_k^{(j)}) \left( \frac{R_{\ell}}{\rho_{\ell}}\right) \right\|_{C_x^i} \left( \|(\partial_s\mathring{R}_{\ell})\rho_{\ell}^{-1}\|_{C_x^{r-i}} + \|\mathring{R}_{\ell}\partial_s \rho_{\ell}^{-1}\|_{C_x^{r-i}} \right).\]
By the chain rule, we find that
\begin{align*}
	\left\| (D\gamma_k^{(j)}) \left( \frac{R_{\ell}}{\rho_{\ell}}\right) \right\|_{C_x^r} &\leq C\left(1+ \sum_{i=1}^r \left\|\gamma_k^{(j)}\right\|_{C_x^{i+1}} \left\|\frac{R_{\ell}}{\rho_{\ell}}\right\|_{C_x}^{i-1}\left\|\frac{\mathring{R}_{\ell}}{\rho_{\ell}}\right\|_{C_x^r}\right)\\
	&\leq C \sum_{i=1}^r \left(1+\frac{{L_n}\eta \delta_{n+1}}{\underline{e}\delta_n}\right)^{i-1} {L_n^{r+1}}C_e^{(r)} \delta_n^{-1} \ell^{1-r}D_n\\
	&\leq C {L_n^{2r}}C_e^{(r)} \ell^{1-r} \delta_n^{-1}D_n.
\end{align*}
Next, for any $r \in \mathbb{N}$ we find
\begin{align*}
	&\|(\partial_s\mathring{R}_{\ell})\rho_{\ell}^{-1}\|_{C_x^r} \\
	&\leq C \sum_{j=0}^r \|\partial_s \mathring{R}_{\ell}\|_{C_x^j} \|\rho_{\ell}^{-1}\|_{C_x^{r-j}}\\
	&\leq C {L_n}\eta \delta_{n+1} \ell^{-1-r} (\underline{e}\delta_n)^{-1} + C\sum_{j=1}^{r-1} {L_n}\eta \delta_{n+1} \ell^{-1-j} \underline{e}^{-1} {L_n^{r-j}}C_e^{(r-j)} \delta_n^{-2} \ell^{1-(r-j)}D_n\\
	&\quad + C {L_n}\eta \delta_{n+1} \ell^{-1} \underline{e}^{-1} {L_n^r}C_e^{(r)} \delta_n^{-2} \ell^{1-r}D_n\\
	&= C {L_n}\frac{\eta}{\underline{e}} \delta_{n+1} \delta_n^{-1} \ell^{-1-r} + C \frac{\eta}{\underline{e}} \delta_{n+1} \delta_n^{-2} D_n \ell^{-r} \sum_{j=1}^{r-1} {L_n^{r-j+1}}C_e^{(j)}  + C {L_n^{r+1}}\frac{\eta}{\underline{e}} \delta_{n+1} \delta_n^{-2} D_n \ell^{-r} C_e^{(r)}\\
	&\leq C{L_n^{r+1}}\frac{\eta}{\underline{e}} \left(1+ \eta C_e^{(r)}\right) \delta_{n+1} \delta_n^{-1} \ell^{-1-r},
\end{align*}
as well as
\begin{align*}
	\|(\partial_s\mathring{R}_{\ell}(s,\cdot))\rho_{\ell}^{-1}(s,\cdot)\|_{C_x} \leq C {L_n}\frac{\eta}{\underline{e}}\ell^{-1} \delta_{n+1} \delta_n^{-1},
\end{align*}
and
\[\|\mathring{R}_{\ell}\partial_s \rho_{\ell}^{-1}\|_{C_x^r} \leq C \sum_{j=0}^r \|\mathring{R}_{\ell}\|_j \|\partial_s\rho_{\ell}^{-1}\|_{r-j}.\]
Note that
\[\|\rho_{\ell}^{-2}\|_{C_x^r} \leq C (\underline{e}\delta_n)^{-2} + C\sum_{j=1}^r \sum_{i=1}^j (\underline{e}\delta_n)^{-(i+2)} ({L_n}\bar{e} \delta_n)^{i-1} {L_n}(1+\bar{e})\ell^{1-r}D_n \leq C{L_n^r}\underline{e}^{-2} C_e^{(r)} \delta_n^{-2} \ell^{-r},\]
and so we obtain
\begin{align*}
	\|\partial_s\rho_{\ell}^{-1}\|_{C_x^r} &\leq C \sum_{j=0}^r \|\rho_{\ell}^{-2}\|_{C_x^j} \|\partial_s \rho_{\ell}\|_{C_x^{r-j}}\\
	&\leq C  (\underline{e}\delta_n)^{-2} {L_n} \eta \delta_{n+1} \ell^{-1-r}  + C\sum_{j=1}^{r-1} {L_n^j}\underline{e}^{-2} C_e^{(j)} \delta_n^{-2} \ell^{-j} {L_n} \eta \delta_{n+1} \ell^{-1-(r-j)} \\
	&\quad + C {L_n^r} \underline{e}^{-2} C_e^{(r)} \delta_n^{-2} \ell^{-r}(\eta{L_n}(1+\sqrt{\bar{e}}) + |e|_{C^1})\delta_{n+1}\ell^{-1} \\
	&\leq C{L_n^{r+1}}\frac{\eta}{\underline{e}^2} \left( 1 +  (1+\sqrt{\bar{e}} + \eta^{-1}|e|_{C^1})\sum_{j=1}^r C_e^{(j)} \right) \delta_n^{-2}\delta_{n+1} \ell^{-1-r}\\
	&\leq C{L_n^{r+1}}\frac{\eta}{\underline{e}^2} (1+\sqrt{\bar{e}} + \eta^{-1}|e|_{C^1}) C_e^{(r)} \delta_n^{-2}\delta_{n+1} \ell^{-1-r},\\
	\|\partial_s\rho_{\ell}^{-1}\|_{C_x} &\leq C{L_n}\frac{\eta}{\underline{e}^2} (1+\sqrt{\bar{e}} + \eta^{-1}|e|_{C^1})\delta_n^{-2}\delta_{n+1}\ell^{-1}.
\end{align*}
Thus we can estimate the other factor as
\begin{align*}
	\|\mathring{R}_{\ell}\partial_s \rho_{\ell}^{-1}\|_{C_x^r} &\leq C \|\mathring{R}_{\ell}\|_{C_x} \|\partial_s \rho_{\ell}^{-1}\|_{C_x^r} + C\sum_{j=1}^{r-1} \|\mathring{R}_{\ell}\|_{C_x^j} \|\partial_s \rho_{\ell}^{-1}\|_{C_x^{r-j}} + C\|\mathring{R}_{\ell}\|_{C_x^r} \|\partial_s \rho_{\ell}^{-1}\|_{C_x}\\
	&\leq C {L_n^{r+2}} \left(\frac{\eta}{\underline{e}}\right)^2 (1 + \sqrt{\bar{e}} + \eta^{-1}|e|_{C^1})C_e^{(r)} \delta_n^{-2}\delta_{n+1}^2 \ell^{-1-r},\\
	\|\mathring{R}_{\ell}\partial_s \rho_{\ell}^{-1}\|_{C_x} &\leq C {L_n^2} \left(\frac{\eta}{\underline{e}}\right)^2 (1+\sqrt{\bar{e}} + \eta^{-1}|e|_{C^1})\delta_n^{-2}\delta_{n+1}^2\ell^{-1}.
\end{align*}
Hence we can combine the above estimates to find
\begin{align*}
	&\|\partial_s \Gamma \|_{C_x^r} \\
	&\leq C\left\| (D\gamma_k^{(j)}) \left( \frac{R_{\ell}}{\rho_{\ell}}\right) \right\|_{C_x} \left( \|(\partial_s\mathring{R}_{\ell})\rho_{\ell}^{-1}\|_{C_x^r} + \|\mathring{R}_{\ell}\partial_s \rho_{\ell}^{-1}\|_{C_x^r} \right)\\
	&\quad + C\sum_{i=1}^{r-1} \left\| (D\gamma_k^{(j)}) \left( \frac{R_{\ell}}{\rho_{\ell}}\right) \right\|_{C_x^i} \left( \|(\partial_s\mathring{R}_{\ell})\rho_{\ell}^{-1}\|_{C_x^{r-i}} + \|\mathring{R}_{\ell}\partial_s \rho_{\ell}^{-1}\|_{C_x^{r-i}} \right)\\
	&\quad + C \left\| (D\gamma_k^{(j)}) \left( \frac{R_{\ell}}{\rho_{\ell}}\right) \right\|_{C_x^r} \left( \|(\partial_s\mathring{R}_{\ell})\rho_{\ell}^{-1}\|_{C_x} + \|\mathring{R}_{\ell}\partial_s \rho_{\ell}^{-1}\|_{C_x} \right)\\
	&\leq C \left( {L_n^{r+1}}\frac{\eta}{\underline{e}} \left(1+ \eta C_e^{(r)}\right) \delta_{n+1} \delta_n^{-1} \ell^{-1-r} + {L_n^{r+2}} \left(\frac{\eta}{\underline{e}}\right)^2 (1 + \sqrt{\bar{e}} + \eta^{-1}|e|_{C^1})C_e^{(r)} \delta_n^{-2}\delta_{n+1}^2 \ell^{-1-r} \right)\\
	&\quad + C\sum_{i=1}^{r-1} C{L_n^{2i}}C_e^{(i)}\ell^{1-i} \delta_n^{-1}D_n \ell^{-1-(r-i)} \\
	&\quad \cdot \left( {L_n^{r-i+1}}\frac{\eta}{\underline{e}} \left(1+ \eta C_e^{(r-i)}\right) \delta_{n+1} \delta_n^{-1}  + {L_n^{r-i+2}} \left(\frac{\eta}{\underline{e}}\right)^2 (1 + \sqrt{\bar{e}} + \eta^{-1}|e|_{C^1})C_e^{(r-i)} \delta_n^{-2}\delta_{n+1}^2  \right)\\
	&\quad + C {L_n^{2r}}C_e^{(r)}\ell^{1-r}\delta_n^{-1}D_n \left( C{L_n} \frac{\eta}{\underline{e}}\ell^{-1}\delta_{n+1}\delta_n^{-1} + {L_n^2} \left(\frac{\eta}{\underline{e}}\right)^2 (1+\sqrt{\bar{e}} + \eta^{-1}|e|_{C^1})\delta_n^{-2}\delta_{n+1}^2\ell^{-1} \right)\\
	&\leq {L_n^{2(r+1)}}C_e^{\partial_s \Gamma,r} \delta_{n+1} \delta_n^{-1} \ell^{-1-r},\\
	&\|\partial_s \Gamma(s, \cdot) \|_0 \leq C{L_n^2}\frac{\eta}{\underline{e}}\left(1 + \frac{\eta}{\underline{e}}(1+\sqrt{\bar{e}} + \eta^{-1}|e|_{C^1})\right)\ell^{-1} \delta_{n+1} \delta_n^{-1},
\end{align*}
where
\begin{align*} 
C_e^{\partial_s \Gamma,r} &= C\frac{\eta}{\underline{e}}\left(1+ (1+\eta)C_e^{(r)} + \eta^2 \sum_{i=1}^{r-1}C_e^{(i)}C_e^{(r-i)} \right. \\
 &\qquad\qquad \left. + \frac{\eta}{\underline{e}} (1 + \sqrt{\bar{e}} + \eta^{-1}|e|_{C^1})\left((1+\eta) C_e^{(r)} + \eta \sum_{i=1}^{r-1}C_e^{(i)}C_e^{(r-i)} \right)\right). 
\end{align*}
From \cite[p.64]{HLP23} we obtain
\begin{align*}
	\|\partial_s \Psi\|_{C_x} &\leq C {L_n} \mu \varsigma_{n+1}^{\gamma-2},\\
	\|\partial_s \Psi\|_{C_x^r} &\leq C{L_n^{r+1}} \mu^{r+1} \varsigma_{n+1}^{r(\gamma-1)-1}(D_n \ell^{1-r} + \varsigma_{n+1}^{\gamma-1}).
\end{align*}
Let us put all terms together: recall
\begin{align*}
	\partial_s a_k^{(j)}(s, y, \tau) &= \left(\partial_s \sqrt{\rho_{\ell}(s,y)}\right) \Gamma(s,y) \Psi(s,y,\tau) + \sqrt{\rho_{\ell}(s,y)} (\partial_s \Gamma(s,y)) \Psi(s,y,\tau) \\
	&\quad+ \sqrt{\rho_{\ell}(s,y)} \Gamma(s,y) (\partial_s \Psi(s,y,\tau)).
\end{align*}
For the first term, we find
\begin{align*}
	&\left\|\left(\partial_s \sqrt{\rho_{\ell}}\right) \Gamma \Psi\right\|_{C_x^r} \\ 
	&\leq C \|\partial_s \sqrt{\rho_{\ell}}\|_{C_x} \left(\|\Psi\|_{C_x^r} + \sum_{i=1}^{r-1} \|\Gamma\|_{C_x^i} \|\Psi\|_{C_x^{r-i}} + \|\Gamma\|_{C_x^r} \right)\\
	&\quad + \sum_{j=1}^{r-1} \|\partial_s\sqrt{\rho_{\ell}}\|_{C_x^j} \left(\|\Psi\|_{C_x^{r-j}} + \sum_{i=1}^{r-j-1} \|\Gamma\|_{C_x^i} \|\Psi\|_{C_x^{r-j-i}} + \|\Gamma\|_{C_x^{r-j}} \right)\\
	&\quad + \|\partial_s\sqrt{\rho_{\ell}}\|_{C_x^r}\\
	&\leq C \frac{{L_n}\eta(1+\sqrt{\bar{e}}) + |e|_{C^1}}{\sqrt{\underline{e}}}\delta_n^{-\frac{1}{2}}\delta_{n+1}\ell^{-1}\\
	&\quad \cdot\left({L_n^r}\mu^r \varsigma_{n+1}^{r(\gamma-1)} \right.  + \sum_{i=1}^{r-1} {L_n^{i+1}} C_e^{(i)} \delta_n^{-1} \ell^{1-i}D_n {L_n^{r-i}}\mu^{r-i}\varsigma_{n+1}^{(r-i)(\gamma-1)}\\
	&\qquad\qquad\quad\left.+{L_n^{r+1}} C_e^{(r)} \delta_n^{-1} \ell^{1-r}D_n\right)\\
	&\quad + C\sum_{j=1}^{r-1} {L_n^{j+1}}C_e^{\partial_s\sqrt{\rho},j} \delta_n^{-\frac{1}{2}}  \delta_{n+1}\ell^{-1-j} \\
	&\quad \cdot\left({L_n^{r-j+1}}\mu^{r-j} \varsigma_{n+1}^{(r-j)(\gamma-1)}\right. + \sum_{i=1}^{r-j-1} {L_n^{i+1}}C_e^{(i)} \delta_n^{-1} \ell^{1-i}D_n {L_n^{r-j-i}}\mu^{r-j-i} \varsigma_{n+1}^{(r-j-i)(\gamma-1)} \\
	&\quad\quad\left.+ {L_n^{r-j+1}}C_e^{(r-j)} \delta_n^{-1} \ell^{1-(r-j)}D_n\right)\\
	&\quad + {L_n^{r+1}}CC_e^{\partial_s\sqrt{\rho},r} \delta_n^{-\frac{1}{2}} \delta_{n+1}\ell^{-1-r}\\
	&\leq {L_n^{r+2}}C_e^{\partial_s (1),r} \delta_n^{\frac{1}{2}}\mu^{r+1} \varsigma_{n+1}^{r(\gamma -1)},
\end{align*}
where
\[C_e^{\partial_s (1),r} := C \left(\left(\frac{\eta(1+\sqrt{\bar{e}}) + |e|_{C^1}}{\sqrt{\underline{e}}} +  C_e^{\partial_s\sqrt{\rho},r}\right)\left(1+ \eta C_e^{(r)}\right)\right), \]
using $ \delta_n^{-1} \delta_{n+1} \ell^{-1} \leq \mu$ as well as $\ell^{-j} \leq \left(\mu \varsigma_{n+1}^{\gamma -1}\right)^j$.
\noindent
The second term can be estimated similarly by
\begin{align*}
	&\left\| \sqrt{\rho_{\ell}} \left(\partial_s\Gamma\right) \Psi\right\|_{C_x^r} \\
	&\leq C \|\sqrt{\rho_{\ell}}\|_{C_x} \left(\|\partial_s\Gamma\|_{C_x}\|\Psi\|_{C_x^r} + \sum_{i=1}^{r-1} \|\partial_s\Gamma\|_{C_x^i} \|\Psi\|_{C_x^{r-i}} + \|\partial_s\Gamma\|_{C_x^r} \right)\\
	&\quad + C \sum_{j=1}^{r-1} \|\sqrt{\rho_{\ell}}\|_{C_x^j} \left(\|\partial_s\Gamma\|_{C_x} \|\Psi\|_{C_x^{r-j}}+ \sum_{i=1}^{r-j-1} \|\partial_s\Gamma\|_{C_x^i} \|\Psi\|_{C_x^{r-j-i}} \right. \left.+ \|\partial_s\Gamma\|_{C_x^{r-j}}\right)\\
	&\quad + C\|\sqrt{\rho_{\ell}}\|_{C_x^r} \|\partial_s\Gamma\|_{C_x}\\
	&\leq C {L_n^{\frac{1}{2}}}\sqrt{\bar{e}}\delta_n^{\frac{1}{2}} \left({L_n^2}\left(1+\frac{\eta}{\underline{e}}(1+\sqrt{\bar{e}} + \eta^{-1}|e|_{C^1})\delta_n^{-1}\delta_{n+1}\right)\frac{\eta}{\underline{e}}\ell^{-1} \delta_{n+1}\delta_n^{-1}{L_n^r}\mu^r \varsigma_{n+1}^{r(\gamma-1)} \right.\\
	&\qquad + \sum_{i=1}^{r-1} {L_n^{2(i+1)}}C_e^{\partial_s \Gamma, i}\delta_{n+1}\delta_n^{-1}\ell^{-1-i} {L_n^{r-i}}\mu^{r-i} \varsigma_{n+1}^{(r-i)(\gamma-1)}  \left.+ {L_n^{2(r+1)}}C_e^{\partial_s \Gamma, r}\delta_{n+1}\delta_n^{-1}\ell^{-1-r} \right)\\
	&\quad + C \sum_{j=1}^{r-1} {L_n^j}\sqrt{\bar{e}} C_e^{(j)} \delta_n^{-\frac{1}{2}} \ell^{1-j}D_n \\
	&\qquad\qquad \left({L_n^2}\left(1+\frac{\eta}{\underline{e}}(1+\sqrt{\bar{e}} + \eta^{-1}|e|_{C^1})\delta_n^{-1}\delta_{n+1}\right)\frac{\eta}{\underline{e}}\ell^{-1} \delta_{n+1}\delta_n^{-1} {L_n^{r-j}}\mu^{r-j} \varsigma_{n+1}^{(r-j)(\gamma-1)}\right.\\
	&\qquad\qquad\quad+ \sum_{i=1}^{r-j-1} {L_n^{2(i+1)}}C_e^{\partial_s \Gamma, i}\delta_{n+1}\delta_n^{-1}\ell^{-1-i} {L_n^{r-j-i}}\mu^{r-j-i} \varsigma_{n+1}^{(r-j-i)(\gamma-1)}\\
	&\qquad\qquad\quad\left.+ {L_n^{2(r-j+1)}}C_e^{\partial_s \Gamma, r-j}\delta_{n+1}\delta_n^{-1}\ell^{-1-(r-j)}\right)\\
	&\quad + C{L_n^r}\sqrt{\bar{e}} C_e^{(r)} \delta_n^{-\frac{1}{2}} \ell^{1-r}D_n {L_n^2}\left(1+\frac{\eta}{\underline{e}}(1+\sqrt{\bar{e}} + \eta^{-1}|e|_{C^1})\delta_n^{-1}\delta_{n+1}\right)\frac{\eta}{\underline{e}}\ell^{-1} \delta_{n+1}\delta_n^{-1}\\
	&\leq {L_n^{2r+\frac{5}{2}}}C_e^{\partial_s (2),r} \delta_n^{\frac{1}{2}}\mu^{r+1} \varsigma_{n+1}^{r(\gamma -1)-1}(D_n \ell^{1-r} + \varsigma_{n+1}^{\gamma -1}),
\end{align*}
where
\[C_e^{\partial_s (2),r} := C\sqrt{\bar{e}}\left(1 + \eta C_e^{(r)}\right)\left(\left(1+\frac{\eta}{\underline{e}}(1+\sqrt{\bar{e}} + \eta^{-1}|e|_{C^1})\right)\frac{\eta}{\underline{e}} + C_e^{\partial_s\Gamma, r} \right).\]
Finally, we estimate the third term:
\begin{align*}
	&\left\| \sqrt{\rho_{\ell}} \Gamma\partial_s \Psi\right\|_{C_x^r} \\
	&\leq C \|\sqrt{\rho_{\ell}}\|_{C_x} \left(\|\partial_s\Psi\|_{C_x^r} + \sum_{i=1}^{r-1} \|\Gamma\|_{C_x^i} \|\partial_s\Psi\|_{C_x^{r-i}}+ \|\Gamma\|_{C_x^r} \|\partial_s\Psi\|_{C_x}\right)\\
	&\quad + C\sum_{j=1}^{r-1} \|\sqrt{\rho_{\ell}}\|_{C_x^j} \left( \|\partial_s\Psi\|_{C_x^{r-j}} + \sum_{i=1}^{r-j-1} \|\Gamma\|_{C_x^i} \|\partial_s\Psi\|_{C_x^{r-j-i}} + \|\Gamma\|_{C_x^{r-j}} \|\partial_s\Psi\|_{C_x}\right)\\
	&\quad +C \|\sqrt{\rho_{\ell}}\|_{C_x^r} \|\partial_s\Psi\|_{C_x}\\
	&\leq C {L_n^{\frac{1}{2}}}\sqrt{\bar{e}}\delta_n^{\frac{1}{2}} \left({L_n^{r+1}}\mu^{r+1}\varsigma_{n+1}^{r(\gamma-1)-1}(D_n\ell^{1-r} + \varsigma_{n+1}^{\gamma-1}) \right.\\
	&\qquad\qquad\qquad + \sum_{i=1}^{r-1} {L_n^{i+1}} C_e^{(i)} \delta_n^{-1} \ell^{1-i}D_n {L_n^{r-i+1}}\mu^{r-i+1}\varsigma_{n+1}^{(r-i)(\gamma-1)-1}(D_n\ell^{1-(r-i)} + \varsigma_{n+1}^{\gamma-1})\\
	&\qquad\qquad\qquad\left.+ {L_n^{r+1}} C_e^{(r)} \delta_n^{-1} \ell^{1-r}D_n {L_n} \mu \varsigma_{n+1}^{\gamma-2}\right)\\
	&\quad + C\sum_{j=1}^{r-1} {L_n^j}\sqrt{\bar{e}}C_e^{(j)} \delta_n^{-\frac{1}{2}} \ell^{1-j}D_n \\
	&\qquad\qquad\left( {L_n^{r-j+1}}\mu^{r-j+1}\varsigma_{n+1}^{(r-j)(\gamma-1)-1}(D_n\ell^{1-(r-j)} + \varsigma_{n+1}^{\gamma-1})\right.\\
	&\qquad\qquad\quad+ \sum_{i=1}^{r-j-1} {L_n^{i+1}} C_e^{(i)} \delta_n^{-1} \ell^{1-i}D_n {L_n^{r-j-i+1}}\mu^{r-j-i+1}\varsigma_{n+1}^{(r-j-i)(\gamma-1)-1}(D_n\ell^{1-(r-j-i)} + \varsigma_{n+1}^{\gamma-1})\\
	&\qquad\qquad\quad\left.+ {L_n^{r-j+1}} C_e^{(r-j)} \delta_n^{-1} \ell^{1-(r-j)}D_n {L_n}\mu\varsigma_{n+1}^{\gamma-2}\right)\\
	&\quad + C{L_n^r}\sqrt{\bar{e}}C_e^{(r)} \delta_n^{-\frac{1}{2}} \ell^{1-r}D_n {L_n}\mu\varsigma_{n+1}^{\gamma-2}\\
	&\leq {L_n^{r+\frac{5}{2}}}C_e^{\partial_s (3),r} \delta_n^{\frac{1}{2}}\mu^{r+1} \varsigma_{n+1}^{r(\gamma -1)-1}(D_n \ell^{1-r} + \varsigma_{n+1}^{\gamma -1}),
\end{align*}
where
\[C_e^{\partial_s (3),r} := C\sqrt{\bar{e}}\left(1 + \eta C_e^{(r)} + \eta^2 \sum_{j=1}^{r-1} C_e^{(j)}C_e^{(r-j)} \right).\]
Thus in total
\[\|\partial_s a_k\|_{C_x^r} \leq {L_n^{2r+\frac{5}{2}}} \left(C_e^{\partial_s (1),r} + C_e^{\partial_s (2),r} + C_e^{\partial_s (3),r} \right)\delta_n^{\frac{1}{2}}\mu^{r+1} \varsigma_{n+1}^{r(\gamma -1)-1}(D_n \ell^{1-r} + \varsigma_{n+1}^{\gamma -1}).\]
Furthermore
\begin{align*}
	\|\partial_s a_k\|_{C_x} &\leq C\left( \|\partial_s \sqrt{\rho_{\ell}}\|_{C_x} + \|\sqrt{\rho_{\ell}}\|_{C_x} \left(\|\partial_s \Gamma\|_{C_x}  +  \|\partial_s \Psi\|_{C_x}\right)\right)\\
	&\leq C \left( \frac{{L_n}\eta(1+\sqrt{\bar{e}}) + |e|_{C^1}}{\sqrt{\underline{e}}}\delta_n^{-\frac{1}{2}}\delta_{n+1}\ell^{-1} \right.\\
	&\qquad \left.+ {L_n^{\frac{1}{2}}}\sqrt{\bar{e}}\delta_n^{\frac{1}{2}} \left({L_n^2}\frac{\eta}{\underline{e}}\left(1 + \frac{\eta}{\underline{e}}(1+\sqrt{\bar{e}} + \eta^{-1}|e|_{C^1})\right)\ell^{-1} \delta_{n+1} \delta_n^{-1} + {L_n}\mu \varsigma_{n+1}^{\gamma-2}\right)\right)\\
	&\leq C {L_n^{\frac{5}{2}}}\left(\frac{\eta(1+\sqrt{\bar{e}}) + |e|_{C^1}}{\sqrt{\underline{e}}} + \sqrt{\bar{e}} \left(1 + \frac{\eta}{\underline{e}}\left(1 + \frac{\eta}{\underline{e}}(1+\sqrt{\bar{e}} + \eta^{-1}|e|_{C^1})\right)\right)\right)\mu \varsigma_{n+1}^{\gamma-2} \delta_n^{\frac{1}{2}}.
\end{align*}
This finally concludes the proof of Proposition \ref{prop-energy-1}.
\end{proof}

\subsection{The principal part}\label{sec:pf_wo} 
In this section, we show the main estimates for the principal part of the perturbation, namely those of Proposition \ref{prop-osc-W} and Lemma \ref{lem:w_o}.

\begin{proof}[Proof of Proposition \ref{prop-osc-W}]
Recall that
\[U_k(s,y,\tau,\xi) = \sum_{k \neq k'} a_k(s,y,\tau)a_{k'}(s,y,\tau)E_k\otimes E_{k'} e^{-{\rm i}k'\cdot \xi}.\]
From this, we immediately find, for $t \in [0,\mft_{L}]$,
\begin{align*}
	\|U_k\|_{C_y^{\delta}} &\leq C \sum_{k \neq k'} \|a_k\|_{C_x^{\delta}}\|a_{k'}\|_{C_x} + \|a_k\|_{C_x}\|a_{k'}\|_{C_x^{\delta}} \leq C{L_n^3}\sqrt{\bar{e}} C_e^{(1),\delta} \mu^{\delta}\varsigma_{n+1}^{\delta(\alpha-1)} \delta_n,\\
	\|U_k\|_{C_y^r} &\leq C \sum_{k \neq k'} \sum_{j=0}^r \|a_k\|_{C_x^j} \|a_{k'}\|_{C_x^{r-j}} \leq C {L_n^{2r+3}}\left(C_e^{(5),r}\sqrt{\bar{e}} +\sum_{j=1}^{\lfloor \frac{r}{2} \rfloor} C_e^{(5),j} C_e^{(5),r-j} \right)\delta_n \mu^r \varsigma_{n+1}^{r(\alpha-1)}.
\end{align*}
This concludes the proof.
\end{proof}

\begin{proof}[Proof of Lemma \ref{lem:w_o}]
Recall that since $\lambda = C \mu^2 \varsigma_{n+1}^{\alpha -2}$ and hence $\lambda \geq \mu (D_n \ell^{-1} + \varsigma_{n+1}^{\alpha-1})$, similar to the proof of \cite[Lemma 4.10]{HLP23} and the above estimates we deduce
\begin{align*}
	\|w_o\|_{C_{\leq \mft_{L}}C_{x}^{\delta}} &\leq C {L^{2\delta}}{L_n^{\frac{5}{2}}}C_e^{(1),\delta}\lambda^{\delta} \delta_n^{\frac{1}{2}},\\
	\|w_o\|_{C_{\leq \mft_{L}}C_{x}^{1+\delta}} &\leq C {L^{2(1+\delta)}}{L_n^{\frac{5}{2} + 2\delta}}(C_e^{(1),1} + (C_e^{(5),2})^{\delta}(C_e^{(1),1})^{1-\delta})\lambda^{1+\delta} \delta_n^{\frac{1}{2}}.
\end{align*}
For $\partial_t w_o$ we take the same approach as in \cite[Lemma 4.10]{HLP23}: below we will show
\[ \|\mathring{R}_{n+1}^{\rm trans}\|_{C_{\leq \mft_{L}} C_x^1} \leq C_e^{{\rm trans},1} {L^{3(1+\delta)}L_n^{\frac{7}{2}}}\lambda^{1+\delta} \delta_n^{\frac{1}{2}}.\]
By the stationary phase lemma (Equ. \eqref{eq:SPL_integral} to be precise), we furthermore deduce
\begin{align*}
	\left\|\int_{\mathbb{T}^3} \partial_t w_o\right\|_{C_{\leq \mft_{L}}C_{x}} &\leq C{L}\lambda^{-1}\left([\partial_s a_k]_{C_{t} C_x^1} + \lambda [\partial_{\tau}a_k + {\rm i} k\cdot \tilde{v} a_k]_{C_{t} C_x^1} + \lambda [ k \cdot v_{\ell} a_k]_{C_{t} C_x^1}\right)\\
	&\leq C {LL_n^{\frac{9}{2}}}\left(C_e^{(8),1} + C_e^{(3),1} + {M_v}C_e^{(1),1} + \sqrt{\bar{e}}\right)\mu \varsigma_{n+1}^{\alpha-1}\delta_n^{\frac{1}{2}}.
\end{align*}
Hence we summarize
\begin{align*}
	\|\partial_t w_o\|_{C_{\leq \mft_{L}}C_{x}} &\leq C \left(\|\mathring{R}_{n+1}^{\rm trans}\|_{C_{t} C_x^1} + \|v_{\ell}\|_{C_{t} C_x} \|w_o\|_{C_{t} C_x^1} + \left\|\int_{\mathbb{T}^3} \partial_t w_o\right\|_{C_{t} C_x}\right)\\
	&\leq C_e^{{\rm trans},1} {L^{3(1+\delta)}L_n^{\frac{7}{2}}}\lambda^{1+\delta}\delta_n^{\frac{1}{2}}+ C {L} {L_n^{\frac{7}{2}}}C_e^{(1),1} {M_v} \lambda \delta_n^{\frac{1}{2}}\\
	&\quad + C {L}{L_n^{\frac{9}{2}}}\left(C_e^{(8),1} + C_e^{(3),1} + {M_v}C_e^{(1),1} + \sqrt{\bar{e}}\right)\mu\varsigma_{n+1}^{\alpha-1}\delta_n^{\frac{1}{2}}\\
	&\leq C{L^{3(1+\delta)}L_n^{\frac{9}{2}}}\left(C_e^{{\rm trans},1} + C_e^{(8),1} + C_e^{(3),1} + {M_v}C_e^{(1),1} + \sqrt{\bar{e}}\right) \lambda^{1+\delta} \delta_n^{\frac{1}{2}}\\
	&=: C_e^{\partial_t w_o} {L^{3(1+\delta)}L_n^{\frac{9}{2}}}\lambda^{1+\delta} \delta_n^{\frac{1}{2}},
\end{align*}
which concludes the proof.
\end{proof}

\subsection{The corrector \texorpdfstring{$w_c$}{wc}}\label{sec-proof-corr-wc}

\begin{proof}[Proof of Lemma \ref{lem:w_c}]
Note that 
\begin{align*}
    \| w_{c}^{1} \|_{C_{\leq \mathfrak{t}_{L}} C_{x}^{\delta}} &\leq \| \mathcal{Q}_{\phi_{n}} v_{n} \|_{C_{\leq \mathfrak{t}_{L}} C_{x}^{\delta}}, \\
    \| w_{c}^{1} \|_{C_{\leq \mathfrak{t}_{L}} C_{x}^{1+\delta}} &\leq \ell^{-1} \| \mathcal{Q}_{\phi_{n}} v_{n} \|_{C_{\leq \mathfrak{t}_{L}} C_{x}^{\delta}}.
\end{align*}
Thus we just have to estimate as in \cite[Equ. $(4.13)$]{HLP23}:
\begin{align*}
    \| \mathcal{Q}_{\phi_{n}} v_{n} \|_{C_{\leq \mathfrak{t}_{L}} C_{x}^{\delta}} &\leq C L^{6+2\delta} \| {\rm div}_{\phi_{n}} v_{n} \|_{C_{\leq \mathfrak{t}_{L}} B_{\infty,\infty}^{\delta-1}} \\
    &\leq CL^{6+2\delta}  \| {\rm div}_{\phi_{n}} v_{n} \|_{C_{\leq \mathfrak{t}_{L}} B_{\infty,\infty}^{-1}}^{1-2\delta}  \| {\rm div}_{\phi_{n}} v_{n} \|_{C_{\leq \mathfrak{t}_{L}} B_{\infty,\infty}^{-1/2}}^{2\delta},
\end{align*}
and, using Lemma \ref{lem:B-alpha_chain}, Lemma \ref{lem:Fourier_multipliers} and Lemma \ref{lem:Cr_chain},
\begin{align*}
    \| {\rm div}_{\phi_{n}} v_{n} \|_{C_{\leq \mathfrak{t}_{L}} B_{\infty,\infty}^{-1/2}} &= \| {\rm div}[ v_{n} \circ \phi_{n}^{-1}] \circ \phi_{n} \|_{C_{\leq \mathfrak{t}_{L}} B_{\infty,\infty}^{-1/2}} \leq CL^{7/2} \| {\rm div}[ v_{n} \circ \phi_{n}^{-1}]  \|_{C_{\leq \mathfrak{t}_{L}} B_{\infty,\infty}^{-1/2}} \\
    &\leq CL^{4} \| v_{n} \|_{C_{\leq \mathfrak{t}_{L}} C_{x}^{1/2}} \leq C L^{4} L_{n}(1+M_{v}) D_{n}.
\end{align*}
Therefore we conclude 
\begin{align*}
    \| \mathcal{Q}_{\phi_{n}} v_{n} \|_{C_{\leq \mathfrak{t}_{L}} C_{x}^{\delta}} &\leq C(1+M_{v}^{2\delta})L^{6+10\delta} L_{n} \delta_{n+2}^{5(1-2\delta)/4} D_{n}^{2\delta} \leq C(1+M_{v}^{2\delta})L^{7} L_{n} \delta_{n+2}^{6/5},
\end{align*}
which in turn completes the proof of the estimates for $w_{c}^{1}$. For the other corrector term, we find
\begin{align*}
	\|w_c^2\|_{C_{\leq \mft_{L}} C_{x}^{\delta}} &\leq C{L^{1+\delta}}\lambda^{-1} \sum_k \lambda^{\delta} [a_k]_{C_x^1} + {L^{\delta}} \|a_k\|_{C_x^{1+\delta}}\\
	&\leq C{L^{1+\delta}}\lambda^{-1} \sum_k \lambda^{\delta} C_e^{(1),1}{L_n^{\frac{5}{2}}}\mu\varsigma_{n+1}^{\gamma-1}\delta_n^{\frac{1}{2}} + {L^{\delta}} C_e^{(5),1+\delta}{L_n^{\frac{5}{2}+\delta}}\mu^{1+\delta}\varsigma_{n+1}^{(1+\delta)(\gamma-1)}\delta_n^{\frac{1}{2}}\\
	&\leq C {L^{1+2\delta}}{L_n^{\frac{5}{2} + \delta}}(C_e^{(1),1} + C_e^{(5),1+\delta})\lambda^{\delta-1}\mu\varsigma_{n+1}^{\gamma-1} \delta_n^{\frac{1}{2}},\\
	\|w_c^2\|_{C_{\leq \mft_{L}} C_{x}^{1+\delta}} &\leq C {L^{2+\delta}}\lambda^{-1} \sum_k \lambda^{1+\delta} [a_k]_{C_x^1} + {L^{1+\delta}} \|a_k\|_{C_x^{2+\delta}}\\
	&\leq C {L^{2+\delta}}\lambda^{-1} \sum_k \lambda^{1+\delta} C_e^{(1),1}{L_n^{\frac{5}{2}}}\mu\varsigma_{n+1}^{\gamma-1}\delta_n^{\frac{1}{2}} + {L^{1+\delta}}C_e^{(5),2+\delta}{L_n^{\frac{7}{2}+\delta}}\mu^{2+\delta}\varsigma_{n+1}^{(2+\delta)(\gamma-1)}\delta_n^{\frac{1}{2}}\\
	&\leq C {L^{3+2\delta}}{L_n^{\frac{7}{2} + \delta}}(C_e^{(1),1} + C_e^{(5),2+\delta})\lambda^{\delta}\mu\varsigma_{n+1}^{\gamma-1} \delta_n^{\frac{1}{2}}.
\end{align*}
This conlcudes the proof.
\end{proof}
\subsection{The total perturbation}\label{sec:pf_wtotal}
 
\begin{proof}[Proof of Lemma \ref{lem:total_perturb}]
    As $w_{n+1} = w_{o} + w_{c}^{1} + w_{c}^{2}$, we find, using Lemmas \ref{lem:w_o} and \ref{lem:w_c} 
    \begin{align*}
     &\| w_{n+1} \|_{C_{\leq \mft_{L}} C_{x}^{\delta}} \leq \| w_{o} \|_{C_{\leq \mft_{L}} C_{x}^{\delta}} + \| w_{c}^{1} \|_{C_{\leq \mft_{L}} C_{x}^{\delta}} + \| w_{c}^{2} \|_{C_{\leq \mft_{L}} C_{x}^{\delta}} \\
     &\leq C\left( {L^{2\delta}}{L_n^{\frac{5}{2}}}C_e^{(1),\delta} \lambda^{\delta}\delta_n^{\frac{1}{2}} + {L^7 L_n (1+M_v^{2\delta})}\delta_{n+2}^{\frac{6}{5}}  +  {L^{1+2\delta}}{L_n^{\frac{5}{2} + \delta}}(C_e^{(1),1} + C_e^{(5),1+\delta})\lambda^{\delta-1}\mu\varsigma_{n+1}^{\gamma-1} \delta_n^{\frac{1}{2}} \right) \\
     &\leq C\left( C_e^{(1),\delta} + 1 + M_{v}^{2\delta} + C_e^{(1),1} + C_e^{(5),1+\delta} \right)\left(1 + \frac{\delta_{n+2}^{6/5}}{\lambda^{\delta}\delta_{n}^{1/2}} + \frac{\mu \vs_{n+1}^{\gamma-1}}{\lambda} \right) L^{1+2\delta}L_n^{\frac{5}{2} + \delta} \lambda^{\delta}\delta_n^{\frac{1}{2}} \\
     &\overset{\text{Equ. \eqref{eq:mu-sigma-lambda}, \eqref{eq:delta6-5_1-2}}}{\leq} C\left( 1 + M_{v}^{2\delta} + C_e^{(1),\delta} + C_e^{(1),1} + C_e^{(5),1+\delta} \right)  L^{1+2\delta}L_n^{\frac{5}{2} + \delta} \lambda^{\delta}\delta_n^{\frac{1}{2}}.
    \end{align*}
Similarly,
    \begin{align*}
        &\| w_{n+1} \|_{C_{\leq \mft_{L}} C_{x}^{1+\delta}} \leq  \| w_{o} \|_{C_{\leq \mft_{L}} C_{x}^{1+\delta}} + \| w_{c}^{1} \|_{C_{\leq \mft_{L}} C_{x}^{1+\delta}} + \| w_{c}^{2} \|_{C_{\leq \mft_{L}} C_{x}^{1+\delta}} \\
        &\overset{\text{Lemmas \ref{lem:w_o}, \ref{lem:w_c}}}{\leq} C (C_e^{(1),1} + (C_e^{(5),2})^{\delta}(C_e^{(1),1})^{1-\delta})
        L^{2+2\delta} L_{n}^{5/2+\delta}\lambda^{1+\delta}\delta_n^{\frac{1}{2}} \\
        &\qquad\qquad + C (1 + M_{v}^{2\delta}) L^{7} L_{n} \ell^{-1}\delta_{n+2}^{\frac{6}{5}} + C (C_e^{(1),1} + C_e^{(5),2+\delta}) L^{3+2\delta} L_{n}^{7/2+\delta} \lambda^{1+\delta} \\
        &\leq C 
        \left( 1 + M_{v}^{2\delta} + C_e^{(1),1} +  (C_e^{(5),2})^{\delta}(C_e^{(1),1})^{1-\delta} + C_e^{(5),2+\delta} \right) \left( 1  + \frac{\delta_{n+2}^{6/5}}{\delta_{n}^{1/2}}\frac{\ell^{-1}}{\lambda^{1+\delta}} \right) L^{3+2\delta} L_{n}^{7/2+\delta} \delta_{n}^{1/2} \lambda^{1+\delta}    \\
        &\overset{\text{Equ.  \eqref{eq:delta6-5_1-2}}}{\leq}
        C\left( 1 + M_{v}^{2\delta} + C_e^{(1),1}  +  (C_e^{(5),2})^{\delta}(C_e^{(1),1})^{1-\delta} + C_e^{(5),2+\delta} \right) L^{3+2\delta} L_{n}^{7/2+\delta} \delta_{n}^{1/2} \lambda^{1+\delta}.
    \end{align*}
  The third estimate follows as 
    \begin{align*}
        &\| v_{n+1} \|_{C_{\leq \mft_{L}} C_{x}^{1+\delta}} \leq \| v_{\ell} \|_{C_{\leq \mft_{L}} C_{x}^{\delta}} + \| w_{n+1} \|_{C_{\leq \mft_{L}} C_{x}^{\delta}} \\
        &\overset{\text{Equ. \eqref{eq:mollif-interpol}}}{\leq}  C\left(1 + M_{v}^{2\delta} + M_{v} + C_e^{(1),1}  +  (C_e^{(5),2})^{\delta}(C_e^{(1),1})^{1-\delta} + C_e^{(5),2+\delta} \right) L^{3+2\delta} L_{n}^{7/2+\delta} \delta_{n}^{1/2} \lambda^{1+\delta}.
    \end{align*}
    For the last estimate, by the previous lemmas as well as Equ. \eqref{eq:wo_C0}, \eqref{eq:mollif-diff}, \eqref{eq:Dn_ell} \eqref{eq:delta6-5_1-2}, \eqref{eq:ell-sigma-est}, we get
    \begin{align*}
        \| v_{n+1} - v_{n} \|_{C_{\leq \mft_{L}} C_{x}} &= \| v_{\ell} - v_{n} \|_{C_{\leq \mft_{L}} C_{x}}  + \| w_{n+1} \|_{C_{\leq \mft_{L}} C_{x}} \\
        &\leq C\left(  \sqrt{\bar{e}}  + \left(1+ M_{v}^{2\delta} + \eta + C_e^{(1),1} + C_e^{(5),1+\delta} \right) \delta_{n+1}^{1/2}  \right) L^{1+2\delta} L_{n}^{5/2+3\delta} \delta_{n}^{1/2}. \qedhere
    \end{align*}
\end{proof}

\subsection{Transport error}\label{sec-proof-trans-R}
 
In the next sections, we control all stress terms from Section \ref{ssec:stressterms}. We start with the transport error.
\begin{proof}[Proof of Proposition \ref{prop:R_trans}]
As in \cite[Proof of Proposition 4.3]{HLP23}, we let $b_{k} := (\partial_{\tau} a_{k} + i(k \cdot \tilde{v}) a_{k}) \circ \phi_{n+1}^{-1}$ and $\Omega_{k}^{\lambda}(\cdot) := \Omega_{k}(\lambda \,\cdot)$ (recalling that $\Omega_{k}(\xi) := E_{k} e^{ik \cdot \xi}$), and use the decomposition
\begin{align*} 
\mathring{R}_{n+1}^{\rm trans} &=  \mathring{R}_{n+1}^{{\rm trans}, 1} + \mathring{R}_{n+1}^{{\rm trans}, 2} + \mathring{R}_{n+1}^{{\rm trans}, 3}, \\
\mathring{R}_{n+1}^{{\rm trans}, 1} &:= \lambda \mathcal{R}_{\phi_{n+1}} \left( \sum_{k \in \Lambda} (b_{k} \Omega_{k}^{\lambda}) \circ \phi_{n+1} \right), \\
\mathring{R}_{n+1}^{{\rm trans}, 2} &:=  \mathcal{R}_{\phi_{n+1}} \left( \sum_{k \in \Lambda} (v_{\ell} \circ \phi_{n+1}^{-1} \cdot \nabla_{y}(a_{k} \circ \phi_{n+1}^{-1}) \Omega_{k}^{\lambda}) \circ \phi_{n+1} \right), \\
\mathring{R}_{n+1}^{{\rm trans}, 3} &:=  \mathcal{R}_{\phi_{n+1}} \left( \sum_{k \in \Lambda} (\partial_{s} a_{k} \circ \phi_{n+1}^{-1} \Omega_{k}^{\lambda}) \circ \phi_{n+1} \right).
\end{align*}
We estimate, for $r \geq r_* + 2$ and $t \in [0,\mft_{L}]$, using the stationary phase lemma \ref{lem:stat-phase},
\begin{align*}
	\|\mathring{R}_{n+1}^{{\rm trans}, 1}\|_{C_x^{\delta}} &\leq C{L^{\delta}} \sum_k \lambda^{\delta} \|\partial_{\tau} a_k + {\rm i} k \cdot \tilde{v} a_k\|_{C_x} + {L^r}\lambda^{1+\delta-r} \|\partial_{\tau} a_k + {\rm i} k \cdot \tilde{v} a_k\|_{C_x^r} \\
	&\qquad\qquad + {L^{r+\delta}}\lambda^{1-r} \|\partial_{\tau} a_k + {\rm i} k \cdot \tilde{v} a_k\|_{C_x^{r+\delta}} \\
	&\leq C {L^{\delta}} \sum_k \lambda^{\delta} {L_n^{\frac{5}{2}}}\sqrt{\bar{e}}\mu^{-1}\delta_n^{\frac{1}{2}} + {L^r} \lambda^{1+\delta-r} C_e^{(7),r} {L_n^{r+1}} \mu^{r-1}\varsigma_{n+1}^{r(\gamma-1)} \delta_n^{\frac{1}{2}} \\
	&\qquad\qquad + {L^{r+\delta}}\lambda^{1-r} C_e^{(7),r+\delta} {L_n^{r+\delta+1}} \mu^{r+\delta-1}\varsigma_{n+1}^{(r+\delta)(\gamma-1)} \delta_n^{\frac{1}{2}}\\
	&\leq C {L^{r+2\delta}}{L_n^{r+\delta+1}}\left(\sqrt{\bar{e}} + C_e^{(7),r} + C_e^{(7),r+\delta} \right)\lambda^{\delta}\mu^{-1}\delta_n^{\frac{1}{2}}.
\end{align*}
Similarly, we find
\begin{align*}
	\|\mathring{R}_{n+1}^{{\rm trans}, 2}\|_{C_x^{\delta}} 	&\leq C {L^{\delta}}\sum_k \lambda^{\delta -1} \|(v_{\ell}\cdot\nabla)a_k\|_{C_x} + {L^r}\lambda^{\delta-r} [(v_{\ell}\cdot\nabla)a_k]_{C_x^r} + {L^{r+\delta}}\lambda^{-r}[(v_{\ell}\cdot\nabla)a_k]_{C_x^{r+\delta}}\\
	&\leq C {L^{\delta}}\sum_k \lambda^{\delta -1} C_e^{(1),1} {M_v L_n^{\frac{7}{2}}} \mu\varsigma_{n+1}^{\gamma-1} \delta_n^{\frac{1}{2}} + {L^r} \lambda^{\delta-r} C_e^{(5),r+1} {M_v L_n^{r + \frac{5}{2}}} \mu^{r+1}\varsigma_{n+1}^{(r+1)(\gamma-1)} \delta_n^{\frac{1}{2}}\\
	&\qquad\qquad+ {L^{r+\delta}}\lambda^{-r}C_e^{(5),r+\delta+1} {M_v L_n^{r+\delta + \frac{5}{2}}} \mu^{r+\delta+1}\varsigma_{n+1}^{(r+\delta+1)(\gamma-1)} \delta_n^{\frac{1}{2}}\\
	&\leq C\left(C_e^{(1),1} + C_e^{(5),r+1} + C_e^{(5),r+\delta+1}\right){M_v L_n^{r+\delta+\frac{5}{2}}L^{r+2\delta}} \lambda^{\delta-1} \mu \varsigma_{n+1}^{\gamma-1}\delta_n^{\frac{1}{2}},
\end{align*}
as well as
\begin{align*}
	&\|\mathring{R}_{n+1}^{{\rm trans}, 3}\|_{C_x^{\delta}} \leq C {L^{\delta}}\sum_k \lambda^{\delta-1}\|\partial_s a_k\|_{C_x} + {L^r}\lambda^{\delta-r}[\partial_s a_k]_{C_x^r} + {L^{r+\delta}}\lambda^{-r}[\partial_s a_k]_{C_x^{r+\delta}}\\
	&\leq C {L^{\delta}}\sum_k \lambda^{\delta-1}C_e^{(4),0} {L_n^{\frac{5}{2}}}\mu \varsigma_{n+1}^{\gamma-2}\delta_n^{\frac{1}{2}} + {L^r}\lambda^{\delta-r}C_e^{(8),r} {L_n^{2r+\frac{5}{2}}} \mu^{r+1} \varsigma_{n+1}^{r(\gamma-1)-1} \delta_n^{\frac{1}{2}}\left(D_n\ell^{1-r} +\varsigma_{n+1}^{\gamma-1}\right) \\
	&\qquad\qquad + {L^{r+\delta}}\lambda^{-r}C_e^{(8),r+\delta}{L_n^{2(r+\delta)+\frac{5}{2}}} \mu^{r+\delta+1} \varsigma_{n+1}^{(r+\delta)(\gamma-1)-1} \delta_n^{\frac{1}{2}}\left(D_n\ell^{1-(r+\delta)} +\varsigma_{n+1}^{\gamma-1}\right)\\
	&\leq C\left(C_e^{(4),0} + C_e^{(8),r} + C_e^{(8),r+\delta}\right){L^{r+2\delta}L_n^{2(r+\delta)+\frac{5}{2}}} \lambda^{\delta-1}\mu \varsigma_{n+1}^{\gamma-2} \delta_n^{\frac{1}{2}}.
\end{align*}
In total this gives
\begin{align*}
	\|\mathring{R}_{n+1}^{\rm trans}\|_{C_{\leq \mft_{L}} C_x} &\leq C\left(\sqrt{\bar{e}} + C_e^{(7),r} + C_e^{(7),r+\delta}+ \left(C_e^{(1),1} + C_e^{(5),r+1} + C_e^{(5),r+\delta+1}\right){M_v} + C_e^{(4),0} \right.\\
	&\qquad\left.+ C_e^{(8),r} + C_e^{(8),r+\delta}\right){L^{r+2\delta}L_n^{2(r+\delta)+\frac{5}{2}}}\lambda^{\delta}\mu^{-1} \delta_n^{\frac{1}{2}}.
\end{align*}
Furthermore, using Equ. \eqref{eq:antidiv-orderminusone-r}, we find for $t \in [0,\mft_{L}]$
\begin{align*}
	\|\mathring{R}_{n+1}^{{\rm trans}, 1}\|_{C_x^1} &\leq \|\mathring{R}_{n+1}^{{\rm trans}, 1}\|_{C_x^{1+\delta}} \\
	&\leq C\lambda {L^{3+2\delta}}\sum_k \|(\partial_{\tau}a_k+ {\rm i}k\cdot \tilde{v} a_k)(\Omega_k^{\lambda}\circ \phi_{n+1})\|_{C_x^{\delta}}\\
	&\leq C\lambda {L^{3+2\delta}}\sum_k {L^{\delta}}\|\partial_{\tau}a_k+ {\rm i}k\cdot \tilde{v} a_k\|_{C_x}\|\Omega_k^{\lambda}\|_{C_x^{\delta}} + \|\partial_{\tau}a_k+ {\rm i}k\cdot \tilde{v} a_k\|_{C_x^{\delta}}\\
	&\leq C {L^{3(1+\delta)}L_n^{\frac{5}{2}}}\left(C_e^{(3),0} + C_e^{(3),\delta}\right)\lambda^{1+\delta} \mu^{-1} \delta_n^{\frac{1}{2}},
\end{align*}
and
\begin{align*}
	\|\mathring{R}_{n+1}^{{\rm trans}, 2}\|_{C_x^1} &\leq \|\mathring{R}_{n+1}^{{\rm trans}, 2}\|_{C_x^{1+\delta}}\\
	&\leq C {L^{3+2\delta}}\sum_k \|((v_{\ell}\cdot \nabla_y)a_k)(\Omega_k^{\lambda}\circ \phi_{n+1})\|_{C_x^{\delta}}\\
	&\leq C {L^{3+2\delta}}\sum_k {L^{\delta}}\|(v_{\ell}\cdot \nabla_y)a_k\|_{C_x}\|\Omega_k^{\lambda}\|_{C_x^{\delta}} + \|(v_{\ell}\cdot \nabla_y)a_k\|_{C_x^{\delta}}\\
	&\leq C{L^{3(1+\delta)}L_n^{\frac{7}{2}}M_v}\left(C_e^{(1),1}+\left(C_e^{(5),2}\right)^{\delta}\left(C_e^{(1),1}\right)^{1-\delta}\right)\lambda^{1+\delta}\delta_n^{\frac{1}{2}},
\end{align*}
since by interpolation
\[\|(v_{\ell}\cdot \nabla_y)a_k\|_{C_x^{\delta}} \leq C\left(C_e^{(5),2}\right)^{\delta}\left(C_e^{(1),1}\right)^{1-\delta} {L_n^{\frac{7}{2}}M_v}\left(\mu \varsigma_{n+1}^{\gamma -1}\right)^{1+\delta} \delta_n^{\frac{1}{2}}.\]
Finally, the third term can be estimated in the same way by
\begin{align*}
	\|\mathring{R}_{n+1}^{{\rm trans}, 3}\|_{C_x^1} &\leq \|\mathring{R}_{n+1}^{{\rm trans}, 3}\|_{C_x^{1+\delta}}\\
	&\leq C {L^{3+2\delta}}\sum_k \|(\partial_s a_k)(\Omega_k^{\lambda}\circ \phi_{n+1})\|_{C_x^{\delta}}\\
	&\leq C {L^{3+2\delta}}\sum_k {L^{\delta}}\|\partial_s a_k\|_{C_x}\|\Omega_k^{\lambda}\|_{C_x^{\delta}} + \|\partial_s a_k\|_{C_x^{\delta}}\\
	&\leq C {L^{3(1+\delta)}L_n^{\frac{5}{2}+\delta}}\left(C_e^{(4),0}+\left(C_e^{(8),1}\right)^{\delta}\left(C_e^{(4),0}\right)^{1-\delta}\right)\lambda^{1+\delta}\delta_n^{\frac{1}{2}},
\end{align*}
where again by interpolation we found
\[\|\partial_s a_k\|_{C^{\delta}} \leq C \left(C_e^{(8),1}\right)^{\delta}\left(C_e^{(4),0}\right)^{1-\delta}{L_n^{\frac{5}{2}+2\delta}}\left(D_n+\varsigma_{n+1}^{\gamma-1}\right)^{\delta}\mu \varsigma_{n+1}^{\gamma-2}\delta_n^{\frac{1}{2}}.\]
In total, this gives
\begin{align*}
	\|\mathring{R}_{n+1}^{\rm trans}\|_{C_{\leq \mft_{L}}C_x^1} &\leq C_e^{{\rm trans},1} {L^{3(1+\delta)}L_n^{\frac{7}{2}}}\lambda^{1+\delta} \delta_n^{\frac{1}{2}},\\
	C_e^{{\rm trans},1} &:= C\left(C_e^{(3),0} + \left(C_e^{(3),1}\right)^{\delta}\left(C_e^{(3),0}\right)^{1-\delta} + \left(C_e^{(1),1}+\left(C_e^{(5),2}\right)^{\delta}\left(C_e^{(1),1}\right)^{1-\delta}\right){M_v}\right.\\
	&\quad\qquad \left. + C_e^{(4),0}+\left(C_e^{(8),1}\right)^{\delta}\left(C_e^{(4),0}\right)^{1-\delta}\right). \qedhere
\end{align*}
\end{proof}

\subsection{Oscillation error}\label{sec-proof-osc-R}
 
\begin{proof}[Proof of Proposition \ref{prop:R_osc}]
Following \cite[Section 4.3]{HLP23}, we use the definition of $\mathring{R}_{n+1}^{\rm osc}$ and $w_{o}$, the representation \eqref{eq:WotimesW}  and the stationary phase lemma \ref{lem:stat-phase} to deduce, for $r \geq r_* +1$,
\begin{align*}
	\|\mathring{R}_{n+1}^{\rm osc}\|_{C_{\leq \mft_{L}} C_x^{\delta}} &\leq C {L^{\delta}} \sum_k {L}\lambda^{\delta -1}[U_k]_{C_{\leq \mft_{L}}C_x^1} + {L^{r+1}}\lambda^{\delta-r}[U_k]_{C_{\leq \mft_{L}}C_x^{r+1}} + {L^{r+\delta+1}}\lambda^{-r} [U_k]_{C_{\leq \mft_{L}}C_x^{r+1+\delta}}\\
	&\leq C {L^{r+2\delta+1}}\sum_k \lambda^{\delta-1} {L_n^3}\sqrt{\bar{e}} C_e^{(1),1}\mu\varsigma_{n+1}^{\gamma-1} \delta_n + \lambda^{\delta-r} {L_n^{2r+5}}C_e^{(7),r+1}\mu^{r+1}\varsigma_{n+1}^{(r+1)(\gamma-1)} \delta_n\\
	&\hspace{2cm}+ \lambda^{-r} {L_n^{2(r+\delta)+5}}C_e^{(7),r+\delta+1}\mu^{r+\delta+1}\varsigma_{n+1}^{(r+\delta+1)(\gamma-1)} \delta_n\\
	&\leq C {L^{r+2\delta+1}}{L_n^{2(r+\delta) +5}}\left(\sqrt{\bar{e}} C_e^{(1),1} + C_e^{(7),r+1}+ C_e^{(7),r+\delta+1}\right)\lambda^{\delta-1}\mu \varsigma_{n+1}^{\gamma-1} \delta_n.
\end{align*}
For the first order norm, an application of Equ. \eqref{eq:antidiv-orderminusone-r} yields
\begin{align*}
	&\|\mathring{R}_{n+1}^{\rm osc}\|_{C_{\leq \mft_{L}} C_x^1} \\
	&\leq C{L^{3+2\delta}}\sum_k \left\| {\rm div}_{\phi_{n+1}}\left(U_k - \frac{1}{2}{\rm tr}(U_k) {\rm Id}\right) \Omega_k^{\lambda}\circ \phi_{n+1}\right\|_{C_{\leq \mft_{L}} C_{x}^{\delta}}\\
	&\leq C{L^{3+2\delta}}\sum_k {L^{\delta}\lambda^{\delta}}\left\| {\rm div}_{\phi_{n+1}}\left(U_k - \frac{1}{2}{\rm tr}(U_k) {\rm Id}\right) \right\|_{C_{\leq \mft_{L}}C_{x}^{0}} + \left\| {\rm div}_{\phi_{n+1}}\left(U_k - \frac{1}{2}{\rm tr}(U_k) {\rm Id}\right) \right\|_{C_{\leq \mft_{L}}C_{x}^{\delta}}\\
	&\leq C{L^{3+2\delta}}\sum_k {L^{1+\delta}}\lambda^{\delta}\left\| U_k\right\|_{C_{x}^1} + {L^{1+2\delta}}\left\|U_k \right\|_{C_{x}^{1+\delta}}\\
	&\leq C{L^{3+2\delta}}\sum_k {L^{1+\delta}L_n^3}\lambda^{\delta}\sqrt{\bar{e}}C_e^{(1),1}\mu \varsigma_{n+1}^{\gamma-1}\delta_n + {L^{1+2\delta}L_n^{3+4\delta}}\left(C_e^{(7),2}\right)^{\delta}\left(\sqrt{\bar{e}}C_e^{(1),1}\right)^{1-\delta}\mu^{1+\delta}\varsigma_{n+1}^{(1+\delta)(\gamma-1)}\delta_n\\
	&\leq C {L^{4(1+\delta)}L_n^{3+4\delta}}\left(\sqrt{\bar{e}}C_e^{(1),1} + \left(C_e^{(7),2}\right)^{\delta}\left(\sqrt{\bar{e}}C_e^{(1),1}\right)^{1-\delta}\right)\lambda^{1+\delta}\delta_n. \qedhere
\end{align*}
\end{proof}

\subsection{Flow error}\label{sec-proof-flow-R}

\subsubsection{First term}\label{sec-proof-flow-R-1}
 
Here we rederive the estimates for the first part of the flow error, carefully tracking all energy dependencies of the constants.
\begin{proof}[Proof of Proposition \ref{prop:R_flow_1}]
As in \cite[Section 4.4]{HLP23}, we further decompose the first part of the flow error into
\begin{align*} 
 \mathring{R}_{n+1}^{\rm flow,1} &= \mathring{R}_{n+1}^{{\rm flow},1,1} + \mathring{R}_{n+1}^{{\rm flow},1,2}, \\
 \mathring{R}_{n+1}^{{\rm flow},1,1} &:= \mathcal{R}_{\phi_{n+1}} G_{\ell}, \\
 \mathring{R}_{n+1}^{{\rm flow},1,2} &:= \mathcal{R}_{\phi_{n+1}} \left( w_{o} \mathrm{div}_{\phi_{n+1}} v_{\ell} \right), 
\end{align*}
where $G_{\ell} := \left(\mathrm{div}_{\phi_{n+1}} - \mathrm{div}_{\phi_{n}}  \right) F_{\ell}$, $F_{\ell} := F_{n} * \chi_{\ell}$, and $F_{n} := v_n \otimes v_n -\mathring{R}_n + q_n \mathrm{Id}$.

Now we estimate separately: First we estimate $F_{n}$ via
\[\| F_{n} \|_{C_{\leq \mft_{L}} C_{x}} \leq C M_v^{2} L_n^2 + C{L_n}\eta \delta_{n+1} + C {M_qL_n} \leq C(\eta + M_v^{2} + {M_q}) {L_n^2}.\]
Thus, we deduce from mollification estimates and Equ. \eqref{eq:antidiv-orderminusone} as in \cite[Proof of Proposition 4.8]{HLP23} that
\begin{align*}
	\|\mathring{R}_{n+1}^{{\rm flow}, 1,1}\|_{C_{\leq \mft_{L}} C_x^{\delta}} &\leq C(\eta + M_v^{2} + {M_q}) {L^{8+4\delta}L_n^2} \ell^{-\delta} (n+1) \varsigma_n^{\gamma'},\\
	\|\mathring{R}_{n+1}^{{\rm flow}, 1,1}\|_{C_{\leq \mft_{L}} C_x^{1+\delta}} &\leq C(\eta + M_v^{2} + {M_q}){L^{9+3\delta} L_n^2} \ell^{-1-\delta} (n+1) \varsigma_n^{\gamma'}.
\end{align*}
For the second part, we use again the stationary phase lemma \ref{lem:stat-phase} to find, for $r \geq r_*+2$, and $t \leq \mft_{L}$
\begin{align*}
	\|\mathring{R}_{n+1}^{{\rm flow},1,2}\|_{ C_x^{\delta}} &\leq C {L^{\delta}}\sum_k {L}\lambda^{\delta -1}\|a_k\|_{C_x} \|v_{\ell}\|_{C_x^1} \\
	&\quad + {L^{r+1}}\sum_{r_1+r_2=r} \lambda^{\delta-r} \|a_k\|_{C_x^{r_1}}\|v_{\ell}\|_{C_x^{r_2+1}} + {L^{r+2}}\sum_{r_1+r_2=r+1} \lambda^{-r}\|a_k\|_{C_x^{r_1}}\|v_{\ell}\|_{C_x^{r_2+1}}\\
	&\leq C {L^{\delta}}\sum_k {L}\lambda^{\delta -1}{L_n^{\frac{1}{2}}}\sqrt{\bar{e}}\delta_n^{\frac{1}{2}} {L_n}D_n\\
	&\quad + {L^{r+1}}\lambda^{\delta-r}\Big({L_n^{\frac{1}{2}}}\sqrt{\bar{e}}\delta_n^{\frac{1}{2}}\ell^{-r}{L_n}D_n + C_e^{(5),r}{L_n^{r+\frac{3}{2}}}\mu^r\varsigma_{n+1}^{r(\gamma-1)}\delta_n^{\frac{1}{2}}{L_n}D_n \\
	&\quad\qquad\qquad\qquad + \sum_{j=1}^{r-1}C_e^{(5),j}{L_n^{j+\frac{3}{2}}}\mu^j\varsigma_{n+1}^{j(\gamma-1)}\delta_n^{\frac{1}{2}}{L_n}D_n\ell^{-(r-j)}\Big)\\
	&\quad + {L^{r+2}}\lambda^{-r}\Big({L_n^{\frac{1}{2}}}\sqrt{\bar{e}}\delta_n^{\frac{1}{2}}\ell^{-(r+1)}{L_n}D_n + C_e^{(5),r+1}{L_n^{r+\frac{5}{2}}}\mu^{r+1}\varsigma_{n+1}^{(r+1)(\gamma-1)}\delta_n^{\frac{1}{2}}{L_n}D_n 
	\\
	&\quad\qquad\qquad\qquad + \sum_{j=1}^rC_e^{(5),j}{L_n^{j+\frac{3}{2}}}\mu^j\varsigma_{n+1}^{j(\gamma-1)}\delta_n^{\frac{1}{2}}{L_n}D_n\ell^{-(r+1-j)}\Big)\\
	&\leq C {L^{r+2+\delta}L_n^{r+\frac{7}{2}}}C_e^{(5),r+1} \lambda^{\delta-1}\delta_n^{\frac{1}{2}}D_n.
\end{align*}
Finally, applying Equ \eqref{eq:antidiv-orderminusone-r}, we find for $t \leq \mft_{L}$
\begin{align*}
	&\|\mathring{R}_{n+1}^{{\rm flow}, 1,2}\|_{ C_x^{1+\delta}} \\
	&\leq C {L^{3+2\delta}}\sum_k \|a_k (\Omega_k^{\lambda}\circ \phi_{n+1}){\rm div}^{\phi_{n+1}}v_{\ell}\|_{C_x^{\delta}}\\
	&\leq C {L^{3+2\delta}}\sum_k \|a_k\|_{C_x^{\delta}}\|{\rm div}^{\phi_{n+1}}v_{\ell}\|_{C_x} + {L^{\delta}}\lambda^{\delta}\|a_k\|_{C_x} \|{\rm div}^{\phi_{n+1}}v_{\ell}\|_{C_x} + \|a_k\|_0 \|{\rm div}^{\phi_{n+1}}v_{\ell}\|_{C^{\delta}}\\
	&\leq C {L^{3+2\delta}} \sum_k C_e^{(1),\delta}{L_n^{\frac{5}{2}}}\mu^{\delta}\varsigma_{n+1}^{\delta(\gamma-1)}\delta_n^{\frac{1}{2}}D_n + {L^{1+\delta}L_n^{\frac{3}{2}}}\sqrt{\bar{e}}\delta_n^{\frac{1}{2}}D_n + {L^{1+2\delta}L_n^{\frac{3}{2}}}\sqrt{\bar{e}}\delta_n^{\frac{1}{2}}D_n\ell^{-\delta}\\
	&\leq C {L^{4(1+\delta)}L_n^{\frac{5}{2}}} C_e^{(1),\delta}\lambda^{\delta}\delta_n^{\frac{1}{2}}D_n. \qedhere
\end{align*}
\end{proof}

\subsubsection{Second term}\label{sec-proof-flow-R-2} 
In this section, we consider the new term of the flow error, $\mathring{R}_{2}^{\mathrm{flow}}$. We will handle it using a Besov interpolation argument. Let us first explain why a simpler approach did not work.
\begin{rem}\label{rem:Fourierdef_failure}
    (a) For $\mathring{R}_{2}^{\mathrm{flow}}$, the ``na\"ive'' approach of using the Fourier definition of the fractional Laplacian does not work here, or at least it is very difficult. The problem is that if we use the definition of $\mathring{R}^{\mathrm{flow}}_{2}$, we get
    \begin{align*}
    &\left( (-\D)^{\alpha}_{\phi_{n+1}}v_{\ell} - (-\D)^{\alpha}_{\phi_{n}}v_{\ell} \right)(t,x) \\
    &= \sum_{k} |k|^{2\alpha} \left[ \int_{\T^{3}} v_{\ell}(t,z) e^{ik \cdot \phi_{n+1}(t,z)} dz ~ e^{-ik \cdot \phi_{n+1}(t,x)} - \int_{\T^{3}} v_{\ell}(t,z) e^{ik \cdot \phi_{n}(t,z)} dz ~ e^{-ik \cdot \phi_{n}(t,x)} \right] \\
     &= \sum_{k} |k|^{2\alpha} \bigg[ \int_{\T^{3}} v_{\ell}(t,z) \left[ e^{ik \cdot \phi_{n+1}(t,z)} - e^{ik \cdot \phi_{n}(t,z)} \right] dz ~ e^{-ik \cdot \phi_{n+1}(t,x)} \\
    &\qquad\qquad + \int_{\T^{3}} v_{\ell}(t,z) e^{ik \cdot \phi_{n}(t,z)} dz \left[ e^{-ik \cdot \phi_{n+1}(t,x)} - e^{-ik \cdot \phi_{n}(t,x)} \right] \bigg].
\end{align*}
Can we make this small enough? At first sight, yes, however, there is a trade-off here. We need to achieve at the same time
smallness with respect to $n$, i.e. smallness in terms of the parameters of the convex integration scheme and
smallness with respect to $k$, i.e. summability due to the nonlocal nature of the fractional Laplacian.

It is not clear how to balance these two constraints in a way that makes the convex integration scheme work. This seems to be due to the lack of ``explicit'' damping factors $|k|^{-N}$ in the expressions above. The usual way to create them is to introduce derivatives on the exponential factor and use integration by parts, but this means we will have more derivatives on $v_{\ell}$, and hence factors of $\ell^{-1}$, but we cannot afford much more than $\ell^{-\delta-\alpha}$, and definitely not $\ell^{-N}$ for some integer $N \geq 1$. So it seems like one needs to avoid using this definition of the fractional Laplacian.

(b) Another in principle plausible approach would be to use the ``paralinearisation lemma'' of \cite[Corollary 2.91]{BCD11}. We first tried a similar approach as in \cite[Lemma C.2]{HLP23}, using continuity of the Bessel potential operators $(\mathrm{I} - \D)^{-1}$ and the explicit, local nature of $(\mathrm{I} - \D)$ as well as duality to get a more manageable form of the resulting expressions and then apply the paralinearisation lemma in a Besov space of ``sufficiently high'' regularity.
However, a more careful inspection of the lemma's proof reveals that the constant in the lemma depends on higher derivatives than the second if $s$ is large. We could not find a way to employ this tool in a way that circumvents this issue. Hence we needed a different argument.
\end{rem}

We first apply that $\mcR_{\phi_{n+1}}$ is an order $(-1)$ operator: for $t \in [0,\mft_{L}]$,
\begin{align*}
    \left\| \mcR_{\phi_{n+1}}\left( \left[ \left(-\D \right)^{\alpha}_{\phi_{n+1}} - \left(-\D \right)^{\alpha}_{\phi_{n}} \right] v_{\ell} \right) \right\|_{C^{\delta}_{x}} &= \left\| \mcR_{\phi_{n+1}}\left( \left[ \left(-\D \right)^{\alpha}_{\phi_{n+1}} - \left(-\D \right)^{\alpha}_{\phi_{n}} \right] v_{\ell} \right) \right\|_{B^{\delta}_{\infty,\infty}} \\
    &\overset{\text{Equ. \eqref{eq:antidiv-orderminusone}}}{\leq} CL^{5+4\delta} \left\|  \left[ \left(-\D \right)^{\alpha}_{\phi_{n+1}} - \left(-\D \right)^{\alpha}_{\phi_{n}} \right] v_{\ell} \right\|_{B^{\delta-1}_{\infty,\infty}}.
\end{align*}
We obtain
\begin{align*}
    &\left[ \left(-\D \right)^{\alpha}_{\phi_{n+1}} - \left(-\D \right)^{\alpha}_{\phi_{n}} \right] v_{\ell} \\
    &= (-\D)^{\alpha}(v_{\ell} \circ \phi^{-1}_{n+1}) \circ \phi_{n+1} - (-\D)^{\alpha}(v_{\ell} \circ \phi^{-1}_{n}) \circ \phi_{n} \\
    &= (-\D)^{\alpha}(v_{\ell} \circ \phi^{-1}_{n+1} - v_{\ell} \circ \phi^{-1}_{n}) \circ \phi_{n+1} + \left[(-\D)^{\alpha}(v_{\ell} \circ \phi^{-1}_{n}) \circ \phi_{n+1} - (-\D)^{\alpha}(v_{\ell} \circ \phi^{-1}_{n}) \circ \phi_{n} \right] \\
    &=: I + II.
\end{align*}
Note that both terms have a similar form. In particular, the following holds:
\begin{lem}
    We have for $t \in [0,\mft_{L}]$
    \begin{alignat*}{3}
        \| I \|_{B_{\infty,\infty}^{\delta-1}} &\leq 2 C L^{7} \| f_{I} \circ \psi_{I} - f_{I} \|_{B_{\infty,\infty}^{\delta+2\alpha-1}}, \qquad \text{for} \quad \psi_{I} :=& \phi_{n+1}^{-1} \circ \phi_{n}, \quad f_{I} &:= v_{\ell}, \\
        \| II \|_{B_{\infty,\infty}^{\delta-1}} &\leq 2 C L^{4} \| f_{II} \circ \psi_{II} - f_{II} \|_{B_{\infty,\infty}^{\delta+2\alpha-1}}, \quad \text{for} \quad \psi_{II} :=& ~\phi_{n+1} \circ \phi_{n}^{-1}, \quad f_{II} &:= (-\D)^{\alpha}(v_{\ell} \circ \phi_{n}^{-1}).
    \end{alignat*}
\end{lem}

\begin{proof}
Since we will be using different dualisations
\begin{align*}
    (B_{\infty,\infty}^{s}, B_{1,1}^{-s})
\end{align*}
for several values of $s$, we will denote which one we use by writing ``$B^{s}_{\infty,\infty}$-dualisation'' whenever we use it. We will first treat the term $I$. We will explain this term in full technical detail and shorten the analogous computations for subsequent terms. For $t \in [0,\mft_{L}]$, we have
\begin{align*}
    \| I \|_{B^{\delta-1}_{\infty,\infty}} &= \left\| (-\D)^{\alpha}(v_{\ell} \circ \phi^{-1}_{n+1} - v_{\ell} \circ \phi^{-1}_{n}) \circ \phi_{n+1} \right\|_{B^{\delta-1}_{\infty,\infty}} \\
    &\overset{B^{\delta-1}_{\infty,\infty}-\text{dualisation}}{=} \sup_{\substack{g \in B^{1-\delta}_{1,1} \\ \| g \|_{B^{1-\delta}_{1,1} = 1}}} \int_{\T^{3}} \left[ (-\D)^{\alpha}(v_{\ell} \circ \phi^{-1}_{n+1} - v_{\ell} \circ \phi^{-1}_{n}) \circ \phi_{n+1} \right] \cdot g ~dx \\
    &\leq \sup_{\substack{g \in C^{\infty} \\ \| g \|_{B^{1-\delta}_{1,1} \leq 2}}} \int_{\T^{3}} \left[ (-\D)^{\alpha}(v_{\ell} \circ \phi^{-1}_{n+1} - v_{\ell} \circ \phi^{-1}_{n}) \circ \phi_{n+1} \right] \cdot g ~dx \\
    &= \sup_{\substack{g \in C^{\infty} \\ \| g \|_{B^{1-\delta}_{1,1} \leq 2}}} \int_{\T^{3}} \left[ (-\D)^{\alpha}(v_{\ell} \circ \phi^{-1}_{n+1} - v_{\ell} \circ \phi^{-1}_{n}) \right] \cdot (g \circ \phi_{n+1}^{-1}) ~dx \\
    &= \sup_{\substack{g \in C^{\infty} \\ \| g \|_{B^{1-\delta}_{1,1} \leq 2}}} \int_{\T^{3}} \left[ v_{\ell} \circ \phi^{-1}_{n+1} - v_{\ell} \circ \phi^{-1}_{n} \right] \cdot (-\D)^{\alpha} (g \circ \phi_{n+1}^{-1}) ~dx \\
    &= \sup_{\substack{g \in C^{\infty} \\ \| g \|_{B^{1-\delta}_{1,1} \leq 2}}} \int_{\T^{3}} \left[ v_{\ell} \circ (\phi^{-1}_{n+1} \circ \phi_{n}) - v_{\ell} \right] \cdot \left[ (-\D)^{\alpha} (g \circ \phi_{n+1}^{-1}) \right] \circ \phi^{-1}_{n} ~dx,
\end{align*}
and using a $B^{\delta+2\alpha-1}_{\infty,\infty}$-dualisation, we find
\begin{align*}
     &\| I \|_{B^{\delta-1}_{\infty,\infty}} \leq  C \sup_{\substack{g \in C^{\infty} \\ \| g \|_{B^{1-\delta}_{1,1} \leq 2}}} \left\| v_{\ell} \circ (\phi^{-1}_{n+1} \circ \phi_{n}) - v_{\ell} \right\|_{B^{\delta+2\alpha-1}_{\infty,\infty}} \cdot \left\| (-\D)^{\alpha} (g \circ \phi_{n+1}^{-1}) \circ \phi^{-1}_{n} \right\|_{B^{1-\delta-2\alpha}_{1,1}}  \\
     &\overset{\text{Lemma \ref{lem:B11_chain}}}{\leq} C L^{4-(1-\delta-2\alpha)} \left\| v_{\ell} \circ (\phi^{-1}_{n+1} \circ \phi_{n}) - v_{\ell} \right\|_{B^{\delta+2\alpha-1}_{\infty,\infty}} \sup_{\substack{g \in C^{\infty} \\ \| g \|_{B^{1-\delta}_{1,1} \leq 2}}} \left\| (-\D)^{\alpha} (g \circ \phi_{n+1}^{-1}) \right\|_{B^{1-\delta-2\alpha}_{1,1}}  \\
     &\overset{\text{Lemma \ref{lem:Fourier_multipliers}}}{\leq}   CL^{3+\delta + 2\alpha} \left\|  v_{\ell} \circ (\phi^{-1}_{n+1} \circ \phi_{n}) - v_{\ell} \right\|_{B^{\delta+2\alpha-1}_{\infty,\infty}} \cdot C \sup_{\substack{g \in C^{\infty} \\ \| g \|_{B^{1-\delta}_{1,1} \leq 2}}} \left\| g \circ \phi_{n+1}^{-1} \right\|_{B^{1-\delta}_{1,1}}  \\
     &\overset{\text{Lemma \ref{lem:B11_chain}}}{\leq} CL^{3+\delta + 2\alpha} \left\| v_{\ell} \circ (\phi^{-1}_{n+1} \circ \phi_{n}) - v_{\ell} \right\|_{B^{\delta+2\alpha-1}_{\infty,\infty}} \cdot CL^{4-(1-\delta)} \sup_{\substack{g \in C^{\infty} \\ \| g \|_{B^{1-\delta}_{1,1} \leq 2}}} \left\| g  \right\|_{B^{1-\delta}_{1,1}}  \\
     &\leq 2CL^{6+2\delta+2\alpha}  \left\| v_{\ell} \circ (\phi^{-1}_{n+1} \circ \phi_{n}) - v_{\ell} \right\|_{B^{\delta+2\alpha-1}_{\infty,\infty}}.
\end{align*}
Similarly, the second term satisfies (letting $f_{II} := (-\D)^{\alpha}(v_{\ell} \circ \phi_{n}^{-1})$) for $t \in [0,\mft_{L}]$
\begin{align*}
    \| II \|_{B^{\delta-1}_{\infty,\infty}} 
    &\leq \sup_{\substack{g \in C^{\infty} \\ \| g \|_{B^{1-\delta}_{1,1} \leq 2}}} \int_{\T^{3}} \left[ f_{II} \circ (\phi_{n+1} \circ \phi_{n}^{-1}) - f_{II} \right] \cdot (g \circ \phi_{n}^{-1}) ~dx \\
     &\overset{B^{\delta-1}_{\infty,\infty}-\text{dualisation}}{\leq}  C \sup_{\substack{g \in C^{\infty} \\ \| g \|_{B^{1-\delta}_{1,1} \leq 2}}} \left\|  f_{II} \circ (\phi_{n+1} \circ \phi_{n}^{-1}) - f_{II} \right\|_{B^{\delta-1}_{\infty,\infty}} \cdot \left\| g \circ \phi_{n}^{-1} \right\|_{B^{1-\delta}_{1,1}}  \\
     &\overset{\text{Lemma \ref{lem:B11_chain}}}{\leq} 2CL^{3+\delta}  \left\| f_{II} \circ (\phi_{n+1} \circ \phi_{n}^{-1}) - f_{II} \right\|_{B^{\delta-1}_{\infty,\infty}}. \qedhere
\end{align*}
\end{proof}
Now we set $s = \delta - 1 < 0$ and estimate, using duality:
\begin{align*}
    \| f \circ \psi - f \|_{B_{\infty,\infty}^{s}} 
    &\leq \sup_{\substack{g \in C^{\infty} \\ \| g \|_{B^{-s}_{1,1} \leq 2}}} \int_{\T^{3}} f \cdot \left[ g \circ \psi^{-1} - g \right] ~dx \\ 
    &\overset{B^{s+\varepsilon}_{\infty,\infty}-\text{dualisation}}{\leq}  C \sup_{\substack{g \in C^{\infty} \\ \| g \|_{B^{-s}_{1,1} \leq 2}}}  \| f \|_{B_{\infty,\infty}^{s+\varepsilon}} \| g \circ \psi^{-1} - g \|_{B_{1,1}^{-s - \varepsilon}},
\end{align*}
for some $\varepsilon > 0$ to be determined below.
\begin{lem}\label{lem:flow_interpolation}
    Let $L \in \N$, $L \geq 1$, $d \in \N$, and let $s < 0$ and $\varepsilon > 0$ be such that $-s - \varepsilon/2 \in (0,1)$. 
    
    Then for every $g \in B_{1,1}^{-s}(\T^{d})$ and $\psi = \phi_{1} \circ \phi_{2}$, if either $\phi_{i} \in \{ \phi^{-1}_{n+1} \circ \phi_{n}, \phi_{n+1} \circ \phi^{-1}_{n} \}$ or $\phi_{i} \in \{ \phi_{n}(t,\cdot), \phi_{n}(r,\cdot) \}$ for some fixed $t \leq \mft_{L},r$, we have 
    \begin{equation}\label{eq:psi}
        \| g \circ \psi^{-1} - g \|_{B_{1,1}^{-s - \varepsilon}} \leq  C_{d}(L,s,\varepsilon)\| g \|_{B_{1,1}^{-s}} \| \psi - \Id \|_{L^{\infty}}^{\frac{\varepsilon^{2}}{48}}.
    \end{equation}
    Furthermore, we have $C_{3}(L,s,1) \leq CL^{8}$.
\end{lem}

\begin{proof}
We want to prove this equality using interpolation because the $L^{1}$-type Besov norms are more difficult to handle than the $L^{\infty}$-type Besov norms are. Using $L^{p}$-interpolation, we can ``split'' our problem into a bounded term of $L^{1}$-type and a small term of $L^{\infty}$-type, each raised to a power $\vartheta$, $1-\vartheta$, respectively.
We need the resulting powers of the ``small'' term to be sufficiently small for the scheme to work.

First, we can increase the integrability/summability indices $(1,1) \to (p,p)$ for $p > 1$ to be determined below at the cost of slightly worse regularity.
To be precise, we find
\begin{align*}
    \| g \circ \psi^{-1} - g \|_{B_{1,1}^{-s - \varepsilon}} &\overset{\text{Lemma \ref{lem_pincr}}}{\leq} C(p) \| g \circ \psi^{-1} - g \|_{B_{p,1}^{-s - \varepsilon}} \overset{\text{Lemma \ref{lem_qincr}}}{\leq} C(p) C(\varepsilon, p) \| g \circ \psi^{-1} - g \|_{B_{p,p}^{-s - \varepsilon + \varepsilon/2}} \\
    &= C(\varepsilon, p) \| g \circ \psi^{-1} - g \|_{B_{p,p}^{-s - \varepsilon/2}}.
\end{align*}
Now, the trick consists in understanding $g \circ \psi^{-1} - g$ as the application of a linear operator
$$
    Th := h \circ \psi^{-1} - h
$$
and to apply interpolation theory to this operator.

To this end, let us show that for the endpoints $p \in \{1 , \infty \}$ we have
\begin{align*}
    &T \colon B_{1,1}^{-s-\varepsilon/4} \to B_{1,1}^{-s - \varepsilon/2} \quad \text{with norm bounded by } CL^{8+2s+\varepsilon}, \\
    &T \colon B_{\infty,\infty}^{-s-\varepsilon/4} \to B_{\infty,\infty}^{-s - \varepsilon/2} \quad \text{with norm bounded by } CL^{-2s - \varepsilon/2} \| \psi - \Id \|_{L^{\infty}}^{(1-\theta)(-s - \varepsilon/4)}.
\end{align*}
\textbf{Case 1: $p=1$.} Here we have
\begin{align*}
    \| g \circ \psi^{-1} - g \|_{B_{1,1}^{-s-\varepsilon/2}} &\leq \| g \circ \psi^{-1} \|_{B_{1,1}^{-s-\varepsilon/2}} + \| g \|_{B_{1,1}^{-s-\varepsilon/2}} \\
    &\overset{\text{Lemma \ref{lem:B11_chain}}}{\leq} 2L^{4 - (-s -\varepsilon/2)}L^{4 - (-s -\varepsilon/2)} \| g \|_{B_{1,1}^{-s-\varepsilon/2}} \\
    &\overset{\text{Lemma \ref{lem:sincr}}}{\leq} C(\varepsilon/4) L^{8+2s+\varepsilon} \| g \|_{B_{1,1}^{-s-\varepsilon/4}}.
\end{align*}
\textbf{Case 2: $p=\infty$.} Recall that $0 < -s -\varepsilon/2 < 1$ and let $\theta \in (0,1)$ such that
\begin{align*}
    \theta^{-1}(-s - \varepsilon/2) = -s - \varepsilon/4, \quad \text{i.e.} \quad \theta = \frac{-s-\varepsilon/2}{-s-\varepsilon/4}.
\end{align*}
We now have\footnote{Here, $\mathcal{C}^{\alpha}$ denote the classical H\"older spaces; note that the embedding of the Zygmund spaces $C^{\alpha}$ into the H\"older spaces works because in this case, $0 < \alpha < 1$.} 
\begin{align*}
    \| g \circ \psi^{-1} - g \|_{B_{\infty,\infty}^{-s-\varepsilon/2}} \overset{\text{Equ. \eqref{eq:Zygmund-Holder}}}{\leq} C \| g \circ \psi^{-1} - g \|_{\mathcal{C}^{-s-\varepsilon/2}} = C(\| g \circ \psi^{-1} - g \|_{L^{\infty}} + [g \circ \psi^{-1} - g]_{\mathcal{C}^{-s-\varepsilon/2}}).
\end{align*}
The second term can be estimated as
\begin{align*}
    [g \circ \psi^{-1} - g]_{\mathcal{C}^{-s-\varepsilon/2}} &\overset{\text{def.}}{=} \sup_{x \neq y} \frac{|g \circ \psi^{-1}(x) - g(x) - g \circ \psi^{-1}(y) - g(y)|}{|x-y|^{-s - \varepsilon/2}} =: \sup_{x\neq y} \frac{I_{g,\psi}(x,y)}{|x-y|^{-s - \varepsilon/2}} \\
    &\leq \left( \sup_{x \neq y} \frac{I_{g,\psi}(x,y)}{|x-y|^{\theta^{-1}(-s - \varepsilon/2})} \right)^{\theta} \left( \sup_{x \neq y} I_{g,\psi}(x,y) \right)^{1-\theta} \\
    &\leq [g \circ \psi^{-1} - g]_{\mathcal{C}^{-s-\varepsilon/4}}^{\theta} \left( \sup_{x \neq y} I_{g,\psi}(x,y) \right)^{1-\theta} \\
    &\leq C(\theta) \| g \circ \psi^{-1} - g \|_{B_{\infty,\infty}^{-s-\varepsilon/4}}^{\theta} \left( \sup_{x \neq y} I_{g,\psi}(x,y) \right)^{1-\theta} \\
    &\overset{\text{Lemma \ref{lem:Cr_chain}}}{\leq} C(\theta) \left(L^{-s - \varepsilon/4} \right)^{2} \| g \|_{B_{\infty,\infty}^{-s-\varepsilon/4}}^{\theta}\left( \sup_{x \neq y} I_{g,\psi}(x,y) \right)^{1-\theta}.
\end{align*}
The last factor can now be estimated readily as follows:
\begin{align*}
    I_{g,\psi}(x,y) &\leq | g \circ \psi^{-1}(x) - g(x)| + | g \circ \psi^{-1}(y) - g(y)| \\
    &\leq [g]_{\mathcal{C}^{-s-\varepsilon/4}} \left( | \psi^{-1}(x) - x|^{-s-\varepsilon/4} + | \psi^{-1}(y) - y|^{-s-\varepsilon/4}  \right) \\
    &\leq C\| g \|_{B_{\infty,\infty}^{-s-\varepsilon/4}} \| \psi - \Id \|_{L^{\infty}}^{-s-\varepsilon/4}.
\end{align*}
Therefore, we find
\begin{align*}
    \| g \circ \psi^{-1} - g \|_{B_{\infty,\infty}^{-s-\varepsilon/2}} \leq C L^{-2s-\varepsilon/2} \| g \|_{B_{\infty,\infty}^{-s-\varepsilon/4}} \| \psi - \Id \|_{L^{\infty}}^{(1-\theta)(-s-\varepsilon/4)}.
\end{align*}
Then, using complex interpolation theory (cf. \cite[Theorem 2.6, p. 51]{Lunardi18}), we find
\begin{align*}
    \| T \|_{\mathcal{L}([B_{1,1}^{-s-\varepsilon/4}, B_{\infty,\infty}^{-s-\varepsilon/4}]_{\vartheta}, [B_{1,1}^{-s - \varepsilon/2}, B_{\infty,\infty}^{-s - \varepsilon/2}]_{\vartheta})} &\leq \| T \|_{\mathcal{L}(B_{1,1}^{-s-\varepsilon/4} ,B_{1,1}^{-s - \varepsilon/2})}^{1-\vartheta} \| T \|_{\mathcal{L}(B_{\infty,\infty}^{-s-\varepsilon/4}, B_{\infty,\infty}^{-s - \varepsilon/2})}^{\vartheta} \\
    &\leq C L^{(1-\vartheta)(8+2s+\varepsilon)} L^{\vartheta (-2s-\varepsilon/2)} \| \psi - \Id \|_{L^{\infty}}^{\vartheta(1-\theta)(-s - \varepsilon/4)}.
\end{align*}
Finally, applying \cite[Theorem 14.4.30, p. 345 f.]{HvNVW23} allows us to identify the interpolation spaces 
\begin{align*}
    [B_{1,1}^{-s-\varepsilon/4}, B_{\infty,\infty}^{-s-\varepsilon/4}]_{\vartheta} &= B^{-s-\varepsilon/4}_{p,p}, \\
    [B_{1,1}^{-s - \varepsilon/2}, B_{\infty,\infty}^{-s - \varepsilon/2}]_{\vartheta} &= B_{p,p}^{-s - \varepsilon/2}
\end{align*}
for 
\begin{align*}
    \frac{1}{p} = \frac{1-\vartheta}{1} + \frac{\vartheta}{\infty}, \quad \text{i.e. } \vartheta = 1-\frac{1}{p} \in (0,1) \quad \text{since } p \in (1,\infty). 
\end{align*}
Thus we find that
\begin{align*}
    \| T \|_{\mathcal{L}(B^{-s-\varepsilon/4}_{p,p}, B^{-s-\varepsilon/2}_{p,p})} \leq CL^{(1-\vartheta)(8+2s+\varepsilon)} L^{\vartheta (-2s-\varepsilon/2)} \| \psi - \Id \|_{L^{\infty}}^{\vartheta(1-\theta)(-s - \varepsilon/4)}.
\end{align*}
In summary, this implies that 
\begin{align*}
    \| g \circ \psi^{-1} - g \|_{B_{p,p}^{-s - \varepsilon/2}} &= \| T g \|_{B_{p,p}^{-s - \varepsilon/2}} \leq \| T \|_{\mathcal{L}(B_{p,p}^{-s-\varepsilon/4}, B_{p,p}^{-s-\varepsilon/2})} \| g \|_{B_{p,p}^{-s-\varepsilon/4}} \\
    &\leq C L^{(1-\vartheta)(8+2s+\varepsilon)} L^{\vartheta (-2s-\varepsilon/2)} \| \psi - \Id \|_{L^{\infty}}^{\vartheta(1-\theta)(-s - \varepsilon/4)} \| g \|_{B_{p,p}^{-s-\varepsilon/4}} \\
    &\leq C L^{(1-\vartheta)(8+2s+\varepsilon)} L^{\vartheta (-2s-\varepsilon/2)} \| \psi - \Id \|_{L^{\infty}}^{\vartheta(1-\theta)(-s - \varepsilon/4)} \| g \|_{B_{1,1}^{-s}},
\end{align*}
where we have used the Sobolev-type embedding theorem Lemma \ref{lem:Sobolev_embedding} in the case 
\begin{align*}
    p_{1} &= q_{1} = 1, \\
    p_{2} &= q_{2} = p,
\end{align*}
and we finally choose $p$ such that
\begin{align*}
    d\left( 1 - \frac{1}{p} \right) = \varepsilon / 4, \quad \text{i.e.} \quad p = \frac{1}{1 - \frac{\varepsilon}{4d}} \in (1,\infty), \quad \text{ which works for } 0 < \varepsilon < 4d. 
\end{align*}
The exponent of the norm of $\psi - \Id$ can be calculated since
\begin{align*}
    \vartheta &= 1 - \frac{1}{p} = \frac{\varepsilon}{4d} \in (0,1), \\
    1 - \theta &= \frac{\varepsilon / 4}{- s - \varepsilon/4} \in (0,1).
\end{align*}
Hence,
\begin{align*}
    \vartheta(1-\theta)(-s - \varepsilon/4) = \frac{\varepsilon}{4d}  \cdot \frac{\varepsilon / 4}{- s - \varepsilon/4} \cdot (-s - \varepsilon/4) = \frac{\varepsilon^{2}}{16 d}.
\end{align*}
Finally, using the above choices, we find that the constant has the form
\begin{align*}
    C_{d}(L,s,\varepsilon) := C L^{(1-\vartheta)(8+2s+\varepsilon)} L^{\vartheta (-2s-\varepsilon/2)} = C L^{8(1-\frac{\varepsilon}{4d}) + 2s(1 - \frac{\varepsilon}{2d}) + \varepsilon(1 - \frac{3\varepsilon}{8d})}.
\end{align*}
As $\varepsilon > 0$ and $s < 0$, we immediately infer that $C_{3}(L,s,1) \leq C L^{8}$. This concludes the proof. 
\end{proof}

With the preparations of the previous section, we are now ready to prove Proposition \ref{prop:R_flow_2}.
\begin{proof}[Proof of Proposition \ref{prop:R_flow_2}]
The restrictions on $\varepsilon$ from the previous section are
\begin{enumerate}
 \item $\varepsilon > 0$;
 \item $\varepsilon < 4d = 12$;
 \item \begin{align*}
    0 < -s -\varepsilon/2 < 1 \quad &\Leftrightarrow \quad 0 < -2s - \varepsilon < 2 \quad \Leftrightarrow \quad 2s < -\varepsilon < 2(1+s) \\
    &\Leftrightarrow -2s > \varepsilon > -2(1+s) \quad \overset{s = \delta - 1}{\Leftrightarrow} \quad 2(1 - \delta) > \varepsilon > -2\delta.
\end{align*}
\end{enumerate}
The intersection of the three conditions gives
\begin{align*}
    \varepsilon \in (0, 2(1-\delta)).
\end{align*}
We obviously have that for $\psi \in \{ \psi_{I}, \psi_{II} \}$
\begin{equation}\label{eq:psi-Id}
    \| \psi - \Id \|_{L^{\infty}} \leq \max \left( \| \phi_{n+1} - \phi_{n} \|_{L^{\infty}}, \| \phi_{n+1}^{-1} - \phi_{n}^{-1} \|_{L^{\infty}} \right) \leq C (n+1) \varsigma_{n}^{\gamma'},
\end{equation}
and hence we find
\begin{align*}
    \| I + II \|_{B_{\infty,\infty}^{\delta-1}} &\overset{\text{Lemma \ref{lem:flow_interpolation}}}{\leq} CL^{7} L^{8} \sup_{\substack{g \in C^{\infty} \\ \| g \|_{B^{1-\delta}_{1,1} \leq 2}}} \left( \| f_{I} \|_{B_{\infty,\infty}^{s + \varepsilon}} + \| f_{II} \|_{B_{\infty,\infty}^{s + \varepsilon}} \right) \| g \|_{B_{1,1}^{1-\delta}} \| \psi - \Id \|_{L^{\infty}}^{\frac{\varepsilon^{2}}{48}} \\
    &\leq 2CL^{15} \left( \| v_{\ell} \|_{B_{\infty,\infty}^{\delta-1 + \varepsilon}} + \| (-\D)^{\alpha}(v_{\ell} \circ \phi_{n}^{-1})  \|_{B_{\infty,\infty}^{\delta-1 + \varepsilon}} \right) \| \psi - \Id \|_{L^{\infty}}^{\frac{\varepsilon^{2}}{48}} \\
    &\overset{\text{Lemma \ref{lem:Fourier_multipliers}, Lemma \ref{lem:Cr_chain}}}{\leq} CL^{4 - (\delta + 2\alpha- 1 + \varepsilon)} L^{15} \| v_{\ell} \|_{B_{\infty,\infty}^{\delta + 2\alpha -1 + \varepsilon}} \| \psi - \Id \|_{L^{\infty}}^{\frac{\varepsilon^{2}}{48}}.
\end{align*}
Now, there is a \textbf{trade-off}: 
\begin{enumerate}
 \item For the terms with $v_{\ell}$, we want $\varepsilon > 0$ to be \textbf{as small as possible}, because each derivative of $v_{\ell}$ costs a factor $\ell^{-1}$.
 \item On the other hand, for the $\psi - \Id$ term, we want $\varepsilon$ to be \textbf{as large as possible}, since this determines how small the resulting term will be.
\end{enumerate}
Since the term $-1$ appears in the regularity exponent of $v_{\ell}$, $1\leq \varepsilon \ll 2$ is admissible. So, let us fix 
$$
\varepsilon = 1. 
$$
This implies
\begin{align*}
    \| I + II \|_{B_{\infty,\infty}^{\delta-1}} &\leq CL^{20} \| v_{\ell} \|_{B_{\infty,\infty}^{\delta + 2\alpha}} \| \psi - \Id \|_{L^{\infty}}^{\frac{1}{48}} \overset{\text{Equ. \eqref{eq:mollif-interpol},\eqref{eq:psi-Id}}}{\leq} CL^{20} M_{v} \ell^{-\delta-2\alpha} (n+1) \vs_{n}^{\frac{1}{48} \gamma'}.
\end{align*}
Therefore,
\begin{align*}
    \| \mathring{R}^{\mathrm{flow}}_{2} \|_{C_{t}^{0} C^{0}_{x}} \leq CL^{26} {M_{v}} \ell^{-\delta-2\alpha} (n+1) \vs_{n}^{\frac{1}{48}\gamma'},
\end{align*}
as we wanted to show.
Regarding the $C^{1}$-norm, we have
\begin{align*}
    \| \mathring{R}^{\mathrm{flow}}_{2} \|_{C_{\leq \mathfrak{t}_L} C^1_x}  
    &\leq C \left\| \mathcal{R}_{\phi_{n+1}} \left[ \left( (-\D)_{\phi_{n+1}}^{\alpha} - (-\D)_{\phi_{n}}^{\alpha} \right) v_{\ell} \right] \right\|_{C_{\leq \mathfrak{t}_L} C^{1+\delta}_x} \\
    &\overset{\text{Equ. \eqref{eq:antidiv-orderminusone-r}}}{\leq} C L^{1+2\delta} \left( \| (-\D)_{\phi_{n+1}}^{\alpha} v_{\ell} \|_{C_{\leq \mathfrak{t}_L} C^{\delta}_{x}} + \| (-\D)_{\phi_{n}}^{\alpha} v_{\ell} \|_{C_{\leq \mathfrak{t}_L} C^{\delta}_{x}} \right) \\
    &= C L^{1+2\delta} \left( \| (-\D)^{\alpha} (v_{\ell} \circ \phi_{n+1}^{-1}) \circ \phi_{n+1} \|_{C_{\leq \mathfrak{t}_L} C^{\delta}_{x}} + \| (-\D)^{\alpha} (v_{\ell} \circ \phi_{n}^{-1}) \circ \phi_{n} \|_{C_{\leq \mathfrak{t}_L} C^{\delta}_{x}} \right) \\
    &\overset{\text{Lemma \ref{lem:Cr_chain}, Lemma \ref{lem:Fourier_multipliers} }}{\leq} CL^{1+4\delta + 2\alpha} \| v_{\ell} \|_{C_{\leq \mathfrak{t}_L} C^{\delta + 2\alpha}_{x}} \\
    &\overset{\text{Equ. \eqref{eq:mollif-interpol}}}{\leq} CL^{2} \ell^{-1-\delta - 2\alpha} {M_{v}} L_{n} \ell \\
    &\overset{\text{Equ. \eqref{eq:ell-sigma-est}}}{\leq} C {M_{v}}  L^{2} L_{n} \ell^{-1-\delta - 2\alpha} (n+1) \vs_{n}^{\gamma'}.\qedhere
\end{align*}
\end{proof}

\subsection{Mollification error}\label{sec:proof-Rmoll} 

\begin{proof}[Proof of Proposition \ref{prop:R_moll_3}]
Throughout this proof, let  $t \in [0,\mft_{L}]$. The idea here is similar to the one for the flow error, except that instead of a difference $\phi_{n+1} - \phi_{n}$, in the most difficult terms, we will have differences of the form $\phi_{n}(t) - \phi_{n}(s)$. 

We start in a similar way as for $\mathring{R}_{2}^{\mathrm{flow}}$:
\begin{align*}
    \left\| \mathring{R}^{\mathrm{moll}}_{3}(t) \right\|_{C^{\delta}_{x}} &:= \left\|  \mcR_{\phi_{n+1}} \left[ \left( (-\D)^{\alpha}_{\phi_{n}} v_{\ell} - ((-\D)^{\alpha}_{\phi_{n}} v_{n}) * \chi_{\ell}  \right) \right] \right\|_{C^{\delta}_{x}} \\
    &\overset{\text{Equ. \eqref{eq:antidiv-orderminusone}}}{\leq}C L^{5+4\delta} \left\|  (-\D)^{\alpha}_{\phi_{n}} v_{\ell} - ((-\D)^{\alpha}_{\phi_{n}} v_{n}) * \chi_{\ell}   \right\|_{B^{\delta-1}_{\infty,\infty}} \\
    &\leq C L^{5+4\delta} \left( \left\|  (-\D)^{\alpha}_{\phi_{n}} (v_{\ell} - v_{n}) \right\|_{B^{\delta-1}_{\infty,\infty}} +  \left\| (-\D)^{\alpha}_{\phi_{n}} v_{n} -  ((-\D)^{\alpha}_{\phi_{n}} v_{n}) * \chi_{\ell}   \right\|_{B^{\delta-1}_{\infty,\infty}} \right) \\
    &=: CL^{5+4\delta}( I + II).
\end{align*}
Let us consider $I$ first:
\begin{align*}
    I 
    &\overset{\text{Lemma \ref{lem:B-alpha_chain}}}{\leq} CL^{4-(\delta-1)} \left\|  (-\D)^{\alpha}( (v_{\ell} - v_{n}) \circ \phi_{n}^{-1} ) \right\|_{B^{\delta-1}_{\infty,\infty}} \\
    &\overset{\text{Lemma \ref{lem:Fourier_multipliers}}}{\leq}  CL^{5-\delta} \| (v_{\ell} - v_{n}) \circ \phi_{n}^{-1} \|_{B^{\delta + 2\alpha -1}_{\infty,\infty}} \\
    &\overset{\text{Lemma \ref{lem:B-alpha_chain}}}{\leq}  CL^{10-2\delta-2\alpha} \| v_{\ell} - v_{n} \|_{B^{\delta + 2\alpha -1}_{\infty,\infty}}.
\end{align*}
If we only had a spatial mollification, we would be done now using standard mollification estimates. However, we have a mollification in both space and time. So we need to work a bit more and employ Lemma \ref{lem:mollification} for $s = \delta+2\alpha-1$, $f = v_{n}$, $\kappa = 1$ and $\beta = \gamma \in (0,1/2)$:
\begin{align*}
    I(t) &\leq CL^{10-2\delta-2\alpha} \| v_{\ell} - v_{n} \|_{B_{\infty,\infty}^{\delta+2\alpha-1}}(t) \\
    &\overset{\text{Lemma \ref{lem:mollification}}}{\leq} CL^{10} \left(\ell \| v_{n} \|_{C_{t}^{0} B_{\infty,\infty}^{\delta + 2\alpha}} +  \ell^{\gamma} \| v_{n} \|_{C_{t}^{\gamma}B_{\infty,\infty}^{\delta + 2\alpha -1}}\right) \\
    &{\leq} CL^{10} \ell^{\gamma} \left( \| v_{n} \|_{C_{t}^{0} C_{x}^{1}} + \| v_{n} \|_{C_{t}^{\gamma} C_{x}^{0}} \right) \\
    &\overset{\text{Equ. \eqref{eq:iter_C1}}}{\leq}  CL^{10} L_{n} D_{n} \ell^{\gamma}.
\end{align*}
Now we estimate the term $II$.
We again apply Lemma \ref{lem:mollification}, this time with $f = (-\D)^{\alpha}_{\phi_{n}} v_{n}$, $s = \delta-1$, $\kappa = 1$ and $\beta = \frac{\gamma}{48} \in (0,1/2)$. The significance of this latter choice will become clear in Equ. \eqref{eq:Db_est}.
\begin{align*}
     II(t) &= \left\| (-\D)^{\alpha}_{\phi_{n}} v_{n} -  ((-\D)^{\alpha}_{\phi_{n}} v_{n}) * \chi_{\ell}   \right\|_{B^{\delta-1}_{\infty,\infty}}(t) \\
     &\overset{\text{Lemma \ref{lem:mollification}}}{\leq} C \ell \left\| (-\D)^{\alpha}_{\phi_{n}} v_{n} \right\|_{C_{t}^{0} B_{\infty,\infty}^{\delta}} + \ell^{\beta}  \| (-\D)^{\alpha}_{\phi_{n}} v_{n} \|_{C_{t}^{\beta}B_{\infty,\infty}^{\delta -1}} \\
     &\overset{\text{def.}}{=} C \ell \left\| (-\D)^{\alpha} (v_{n} \circ \phi_{n}^{-1}) \circ \phi_{n} \right\|_{C_{t}^{0} B_{\infty,\infty}^{\delta}} + 
     \ell^{\beta}  \| (-\D)^{\alpha}( v_{n} \circ \phi_{n}^{-1}) \circ \phi_{n} \|_{C_{t}^{\beta}B_{\infty,\infty}^{\delta -1}}.
\end{align*}
The first term is simple: Using $B_{\infty,\infty}^{\delta} = C^{\delta}_{x}$, we find
\begin{align*}
    \left\| (-\D)^{\alpha} (v_{n} \circ \phi_{n}^{-1}) \circ \phi_{n} \right\|_{C_{t}^{0} B_{\infty,\infty}^{\delta}} &\overset{\text{Lemmas \ref{lem:Cr_chain}, \ref{lem:Fourier_multipliers}}}{\leq}  CL^{2\delta+2\alpha} \| v_{n} \|_{C_{t}^{0}B_{\infty,\infty}^{\delta+2\alpha}} \leq CL^{2\delta+2\alpha} \| v_{n} \|_{C_{t}^{0}C_{x}^{1}} \\
    &\leq CL L_{n} D_{n}.
\end{align*}
The main difficulty lies in estimating the second term. We want to ``peel away'' all the operations to get to a norm of $v_{n}$ only. An obvious strategy would be to attempt to use \cite[Lemma C.3]{HLP23}, however, as we would have to use it once, then increase the regularity, and then use it again, this latter use is not covered by that lemma as it concerns the $C_{t}^{\beta}C_{x}^{0}$-norm, not the $C_{t}^{\beta}C_{x}^{2\alpha}$-norm. 
Here, we will transfer some of the flow operations onto a test function. This has the added benefit of simplifying the time-dependence of the terms we need to estimate, as the test function is independent of time, whereas $v_{n}$ and $\phi_{n}$ are time dependent. More precisely, setting $F(t,x) := (-\D)^{\alpha}( v_{n} \circ \phi_{n}^{-1})(t,x)$ 
\begin{align*}
    &[ F \circ \phi_{n} ]_{C_{t}^{\beta}B_{\infty,\infty}^{\delta -1}} = \sup_{\substack{s,r \in [0,t], \\ s \neq r}} \frac{\left\| F(s,\phi_{n}(s)) - F(r,\phi_{n}(r)) \right\|_{B_{\infty,\infty}^{\delta-1}}}{|s-r|^{\beta}} \\
    &= \sup_{\substack{s,r \in [0,t], \\ s \neq r}} \frac{1}{|s-r|^{\beta}} \sup_{\substack{g \in B_{1,1}^{1-\delta} \\ \| g \|_{B_{1,1}^{1-\delta}  } \leq 1}} \int_{\T^{3}} \left( F(s,\phi_{n}(s,x)) - F(r,\phi_{n}(r,x)) \right) g(x) ~dx \\
    &\leq \sup_{\substack{s,r \in [0,t], \\ s \neq r}} \frac{1}{|s-r|^{\beta}} \sup_{\substack{g \in C^{\infty} \\ \| g \|_{B_{1,1}^{1-\delta}  } \leq 2}} \int_{\T^{3}} F(s,y) (g \circ \phi_{n}^{-1})(s,y) ~dy - \int_{\T^{3}} F(r,y)  (g \circ \phi_{n}^{-1})(r,y) ~dy \\
    &\leq \sup_{\substack{s,r \in [0,t], \\ s \neq r}} \frac{1}{|s-r|^{\beta}} \sup_{\substack{g \in C^{\infty} \\ \| g \|_{B_{1,1}^{1-\delta}  } \leq 2}} \int_{\T^{3}} (F(s,y) - F(r,y)) (g \circ \phi_{n}^{-1})(s,y) ~dy \\
    &\qquad + \sup_{\substack{s,r \in [0,t], \\ s \neq r}} \frac{1}{|s-r|^{\beta}} \sup_{\substack{g \in C^{\infty} \\ \| g \|_{B_{1,1}^{1-\delta}  } \leq 2}} \int_{\T^{3}} F(r,y) \left((g \circ \phi_{n}^{-1})(s) - (g \circ \phi_{n}^{-1})(r) \right) ~dy \\
    &=: D + E.
\end{align*}
This is similar to the terms we had to deal with in the flow error, except for differences of $\phi_{n+1},\phi_{n}$, here we have differences of $\phi_{n}(t)$ and $\phi_{n}(s)$. We apply the same method we used for the flow error.

Let us consider the easier term, $D$, first:
\begin{align*}
    &\sup_{\substack{s,r \in [0,t], \\ s \neq r}} \frac{1}{|s-r|^{\beta}} \sup_{\substack{g \in C^{\infty} \\ \| g \|_{B_{1,1}^{1-\delta}  } \leq 2}} \int_{\T^{3}} (F(s,y) - F(r,y)) (g \circ \phi_{n}^{-1})(s,y) ~dy \\
    &\overset{\text{Lemma \ref{lem:B11_chain}}}{\leq} C \sup_{\substack{s,r \in [0,t], \\ s \neq r}}  \sup_{\substack{g \in C^{\infty} \\ \| g \|_{B_{1,1}^{1-\delta}  } \leq 2}} \frac{\| F(s) - F(r) \|_{B^{\delta-1}_{\infty,\infty}}}{|s-r|^{\beta}} CL^{4-(1-\delta)} \| g  \|_{B_{1,1}^{1-\delta}} \\
    &\leq CL^{3+\delta} \sup_{\substack{s,r \in [0,t], \\ s \neq r}}  \frac{\| (-\D)^{\alpha} \left[ (v_{n} \circ \phi_{n}^{-1})(s) - (v_{n} \circ \phi_{n}^{-1})(r) \right] \|_{B^{\delta-1}_{\infty,\infty}}}{|s-r|^{\beta}} \\
    &\overset{\text{Lemma \ref{lem:Fourier_multipliers}}}{\leq} CL^{3+\delta} \sup_{\substack{s,r \in [0,t], \\ s \neq r}}  \frac{\|  (v_{n} \circ \phi_{n}^{-1})(s) - (v_{n} \circ \phi_{n}^{-1})(r) \|_{B^{\delta+2\alpha-1}_{\infty,\infty}}}{|s-r|^{\beta}}.
\end{align*}
We need to be careful here, since outside of $D$, we only have a ``weak'' factor of $\ell^{\beta} \approx \ell^{1/100}$ as opposed to the $\ell^{\gamma} \approx \ell^{1/2}$ in \cite{HLP23}, so we cannot afford ``full'' terms of $D_{n}$ as we could there. We thus analyze the term inside the supremum:
\begin{align*}
    &\frac{1}{|s-r|^{\beta}} \|  (v_{n} \circ \phi_{n}^{-1})(t) - (v_{n} \circ \phi_{n}^{-1})(s) \|_{B^{\delta+2\alpha-1}_{\infty,\infty}} \\
    &\leq \frac{1}{|s-r|^{\beta}} \sup_{\substack{g \in C^{\infty} \\ \| g \|_{B_{1,1}^{1-\delta-2\alpha}  } \leq 2}} \int_{\T^{3}} v_{n}(s,y) \cdot \left( g \circ \phi_{n}(s) \right) ~dy - \int_{\T^{3}}  v_{n}(r,y) \cdot \left( g \circ \phi_{n}(r) \right) ~dy \\
    &\overset{\text{duality}}{\leq} C \sup_{\substack{g \in C^{\infty} \\ \| g \|_{B_{1,1}^{1-\delta-2\alpha}  } \leq 2}} \frac{\| v_{n}(s) - v_{n}(r) \|_{B_{\infty,\infty}^{\delta+2\alpha-1}} }{|s-r|^{\beta}} \| g \circ \phi_{n}(s) \|_{B_{1,1}^{1-\delta-2\alpha}} \\
    &\qquad + C \sup_{\substack{g \in C^{\infty} \\ \| g \|_{B_{1,1}^{1-\delta-2\alpha}  } \leq 2}} \| v_{n}(r) \|_{B_{\infty,\infty}^{\delta+2\alpha}} \frac{\| g \circ \phi_{n}(s) - g \circ \phi_{n}(r) \|_{B_{1,1}^{-\delta-2\alpha}}}{|s-r|^{\beta}} \\
    &=: D_{a} + D_{b}.
\end{align*}
For the term $D_{a}$, we find
\begin{align*}
    D_{a} &\overset{\text{Lemma \ref{lem:B11_chain}}}{\leq} \frac{\| v_{n}(s) - v_{n}(r) \|_{C^{0}_{x}} }{|s-r|^{\beta}} CL^{4-(1-\delta-2\alpha)} \sup_{\substack{g \in C^{\infty} \\ \| g \|_{B_{1,1}^{1-\delta-2\alpha}  } \leq 2}} \| g \|_{B_{1,1}^{1-\delta-2\alpha}} \\
    &\leq 4 CL^{3+\delta+2\alpha} \left( \frac{\| v_{n}(s) - v_{n}(r) \|_{C^{0}_{x}} }{|s-r|} \right)^{\beta} \| v_{n} \|_{C_{t}^{0} C_{x}^{0}}^{1-\beta} \\
    &\leq CL^{4} \| v_{n} \|_{C_{t}^{1} C_{x}^{0}}^{\beta} \| v_{n} \|_{C_{t}^{0} C_{x}^{0}}^{1-\beta} \\
    &\overset{\text{Equ. \eqref{eq:iter_C1}, \eqref{eq:iter_diffv}}}{\leq} CL^{4} \left( L_{n} D_{n} \right)^{\beta} \left( 2 {M_{v}} L_{n} \right)^{1-\beta} \\
    &\leq CL^{4} L_{n} {M_{v}^{1-\beta} } D_{n}^{\beta}.
\end{align*}
The term $D_{b}$ is treated as follows: setting $\psi := \phi_{n}(s) \circ \phi^{-1}_{n}(r)$, we find 
 \begin{align*}
  &D_{b} 
  \\
  &\overset{\text{Lem. \ref{lem:Cr_chain}}}{\leq} CL^{\delta+2\alpha} \| v_{n} \|_{C_{t}^{0} C_{x}^{\delta+2\alpha}}
  \frac{1}{|s-r|^{\beta}}
  \sup_{\substack{g \in C^{\infty} \\ \| g \|_{B_{1,1}^{1-\delta-2\alpha}  } \leq 2}} 
  \sup_{\substack{h \in C^{\infty} \\ \| h \|_{B_{\infty,\infty}^{\delta+2\alpha}  } \leq 2}} \| g \circ \psi - g \|_{B_{1,1}^{-\delta-2\alpha}} \| h \|_{B_{\infty,\infty}^{\delta+2\alpha}  } \\
  &\overset{\text{Equ. \eqref{eq:vn_interpol}, Lem. \ref{lem:flow_interpolation}}}{\leq}  CL^{9+\delta+2\alpha} {(M_{v} + M_{v}^{1-\delta-2\alpha})} L_{n} D_{n}^{\delta+2\alpha} 
  \frac{1}{|s-r|^{\beta}}
  \sup_{\substack{g \in C^{\infty} \\ \| g \|_{B_{1,1}^{1-\delta-2\alpha}  } \leq 2}} \| g \|_{B_{1,1}^{1-\delta-2\alpha}} \| \psi - \Id \|_{L^{\infty}}^{1/48} \\
  &\leq  CL^{9+\delta+2\alpha} {(M_{v} + M_{v}^{1-\delta-2\alpha})} L_{n} D_{n}^{\delta+2\alpha}  \frac{\| \phi_{n}(r) - \phi_{n}(s) \|_{L^{\infty}}^{1/48}}{|s-r|^{\beta}}.
\end{align*}
In the penultimate step, we applied Lemma \ref{lem:flow_interpolation} with $s = \delta + 2\alpha - 1$. Now, with $\beta = \frac{\gamma}{48}$, we find
\begin{equation}\label{eq:Db_est}
    \begin{split}
       D_{b} &\leq CL^{9+\delta+2\alpha} {(M_{v} + M_{v}^{1-\delta-2\alpha})} L_{n} D_{n}^{\delta+2\alpha} \left( \frac{\| \phi_{n}(r) - \phi_{n}(s) \|_{L^{\infty}}}{|s-r|^{\alpha}} \right)^{1/48} \\
        &\leq CL^{9+\delta+2\alpha} {(M_{v} + M_{v}^{1-\delta-2\alpha})} L_{n} D_{n}^{\delta+2\alpha} \left( CL \right)^{1/48} \\
        &\leq CL^{10} L_{n} {(M_{v} + M_{v}^{1-\delta-2\alpha})} D_{n}^{\delta+2\alpha}.
    \end{split}
\end{equation}
Finally, we need to estimate the term 
\begin{align*}
  E &:= \sup_{\substack{s,r \in [0,t], \\ s \neq r}} \frac{1}{|s-r|^{\beta}} \sup_{\substack{g \in C^{\infty} \\ \| g \|_{B_{1,1}^{1-\delta}  } \leq 2}} \int_{\T^{3}} F(r,y) \left((g \circ \phi_{n}^{-1})(s) - (g \circ \phi_{n}^{-1})(r) \right) ~dy.
\end{align*}
First, we consider
\begin{align*}
     & \int_{\T^{3}} F(r,y) \left((g \circ \phi_{n}^{-1})(s) - (g \circ \phi_{n}^{-1})(r) \right) ~dy \\
     &\overset{\text{Lemma \ref{lem:Fourier_multipliers}}}{\leq} C \| v_{n} \circ \phi_{n}^{-1} \|_{B_{\infty,\infty}^{\delta+2\alpha}} \| (g \circ \phi_{n}^{-1})(s) - (g \circ \phi_{n}^{-1})(r) \|_{B_{1,1}^{-\delta}} \\ 
     &\underset{\text{Equ. \eqref{eq:vn_interpol}}}
     {\overset{\text{Lemma \ref{lem:Cr_chain},}}{\leq}} CL^{\delta+2\alpha}L_{n}D_{n}^{\delta+2\alpha} {(M_{v} + M_{v}^{1-\delta-2\alpha})} \| (g \circ \phi_{n}^{-1})(s) - (g \circ \phi_{n}^{-1})(r) \|_{B_{1,1}^{-\delta}}.
\end{align*}
We consider the last factor separately:
\begin{align*}
    &\| (g \circ \phi_{n}^{-1})(s) - (g \circ \phi_{n}^{-1})(r) \|_{B_{1,1}^{-\delta}} \leq \sup_{\substack{h \in C^{\infty} \\ \| h \|_{B_{\infty,\infty}^{\delta}} \leq 2}} \int_{\T^{3}} \left[ g \circ \phi_{n}^{-1}(s) - g \circ \phi_{n}^{-1}(r) \right] \cdot h ~dx \\
    &\leq C \sup_{\substack{h \in C^{\infty} \\ \| h \|_{B_{\infty,\infty}^{\delta}} \leq 2}} \| g \circ (\phi_{n}^{-1}(s) \circ \phi_{n}(r) ) - g  \|_{B_{1,1}^{-\delta}} \| h \circ \phi_{n}(r) \|_{B_{\infty,\infty}^{\delta}} \\
    &\overset{\text{Lemma \ref{lem:Cr_chain}}}{\leq} CL^{\delta}  \| g \circ \psi - g  \|_{B_{1,1}^{-\delta}},
\end{align*}
with $\psi := \phi_{n}^{-1}(s) \circ \phi_{n}(r)$. We apply Lemma \ref{lem:flow_interpolation} with $s = \delta - 1 < 0$ and  $\varepsilon = 1$ to find (recall that $\beta = \gamma/48$)
\begin{align*}
    \| g \circ \psi - g  \|_{B_{1,1}^{-\delta}} &\leq L^{8} \| g \|_{B_{1,1}^{1-\delta}} \| \psi - \Id \|_{L^{\infty}}^{1/48} 
    \leq CL^{8} \| g \|_{B_{1,1}^{1-\delta}} \left( \frac{\| \phi_{n}^{-1}(s) - \phi_{n}^{-1}(r) \|_{C^{0}_{x}}}{|s-r|^{\gamma}} \right)^{1/48} |s-r|^{\frac{\gamma}{48}} \\
    &\leq C L^{9} \| g \|_{B_{1,1}^{1-\delta}} |s-r|^{\beta}.
\end{align*}
Plugging our above results into the definition of $E$ yields:
\begin{align*}
    E \leq C L^{9+2\delta+2\alpha} L_{n} D_{n}^{\delta+2\alpha} {(M_{v} + M_{v}^{1-\delta-2\alpha})} \| g \|_{B_{1,1}^{1-\delta}}.
\end{align*}
This finally gives (as $\delta + 2\alpha < \beta$)
\begin{align*}
    II(t) &\leq C \ell L L_{n} D_{n} + \ell^{\beta} \Big( CL^{4} L_{n} {M_{v}^{1-\beta}} D_{n}^{\beta} + CL^{10} L_{n} { (M_{v} + M_{v}^{1-\delta-2\alpha})} D_{n}^{\delta+2\alpha} \\
    &\qquad\qquad\qquad +  CL^{9+2\delta+2\alpha} L_{n} D_{n}^{\delta+2\alpha} { (M_{v} + M_{v}^{1-\delta-2\alpha})} \sup_{\substack{g \in C^{\infty} \\ \| g \|_{B_{1,1}^{1-\delta}  } \leq 2}} \| g \|_{B_{1,1}^{1-\delta}} \Big) \\
    &\leq CL^{10} L_{n} (1 + M_{v}^{1-\beta}  + M_{v}^{1-\delta-2\alpha} + M_{v}) \left( \ell D_{n} + \ell^{\beta} D_{n}^{\beta} + \ell^{\beta} D_{n}^{\delta+2\alpha} \right) 
    \\
    &\leq C L^{10} L_{n} (1 + M_{v}^{1-\beta}  + M_{v}^{1-\delta-2\alpha} + M_{v}) (\ell D_{n} + \ell^{\beta} D_{n}^{\beta}).
\end{align*}
Now we combine all estimates to find
\begin{align*}
    \left\| \mathring{R}^{\mathrm{moll}}_{3}(t) \right\|_{C^{\delta}_{x}} &\leq CL^{5+4\delta} (I + II) \leq C L^{16} L_{n} (1 + M_{v}^{1-\beta}  + M_{v}^{1-\delta-2\alpha} + M_{v}) ( \ell^{\gamma} D_{n} + \ell^{\beta} D_{n}^{\beta} ).
\end{align*}
To estimate the $C^{1}$-norm of $\mathring{R}_{3}^{\mathrm{moll}}$, let us calculate
\begin{align*}
    &\| \mathring{R}_{3}^{\mathrm{moll}} \|_{C_{\leq \mathfrak{t}_{L}}C_{x}^{1}} 
    \\
    &\overset{\text{Equ. \eqref{eq:antidiv-orderminusone-r}}}{\leq} C L^{1+2\delta} \left\|  \left( (-\D)^{\alpha}_{\phi_{n}} v_{\ell} - ((-\D)^{\alpha}_{\phi_{n}} v_{n}) * \chi_{\ell}  \right) \right\|_{C_{\leq \mathfrak{t}_{L}}C_{x}^{\delta}} \\
    &\leq CL^{1+2\delta} \| (-\D)^{\alpha}_{\phi_{n}} ( v_{\ell} - v_{n} ) \|_{C_{\leq \mathfrak{t}_{L}}C_{x}^{\delta}} + CL^{1+2\delta} \left\| (-\D)^{\alpha}_{\phi_{n}} v_{n} - \left[ (-\D)^{\alpha}_{\phi_{n}} v_{n} \right] * \chi_{\ell} \right\|_{C_{\leq \mathfrak{t}_{L}}C_{x}^{\delta}} \\
    &=: I + II.
\end{align*}
We have for $t \leq \mathfrak{t}_{L}$
\begin{align*}
    I 
    &\overset{\text{Lemmas \ref{lem:Cr_chain}, \ref{lem:Fourier_multipliers}}}{\leq}  CL^{1+4\delta + 2\alpha} \left( \| v_{\ell} \|_{C_{x}^{\delta+2\alpha}} + \| v_{n} \|_{C_{x}^{\delta+2\alpha}} \right) \\
    &\overset{\text{Equ. \eqref{eq:mollif-interpol} \eqref{eq:vn_interpol}}}{\leq}  CL^{1+4\delta + 2\alpha} \left( {M_{v}} L_{n} \ell^{-\delta-2\alpha} + (M_{v}^{1-\delta-2\alpha} + M_{v}) L_{n} D_{n}^{\delta+2\alpha} \right) \\
    &\leq C (M_{v}^{1-\delta-2\alpha} + M_{v}) L^{1+4\delta + 2\alpha} L_{n}  \ell^{-\delta - 2\alpha}  .
\end{align*}
Now we consider the other term: Using interpolation, we find
\begin{align*}
    II 
    &\leq  CL^{1+2\delta} \left( \|(-\D)^{\alpha}_{\phi_{n}} v_{n} \|_{C_{x}^{\delta}} + \|  (-\D)^{\alpha}_{\phi_{n}} v_{n} \|_{C_{x}^{0}}^{1-\delta} \| \left\{ (-\D)^{\alpha}_{\phi_{n}} v_{n} \right\} * D\chi_{\ell} \|_{C_{x}^{0}}^{\delta} \right) \\
    &\leq CL^{1+2\delta} \ell^{-\delta} \|(-\D)^{\alpha}_{\phi_{n}} v_{n} \|_{C_{x}^{\delta}} \\
    &\overset{\text{Lemma \ref{lem:Cr_chain}}}{\leq} CL^{1+2\delta} \ell^{-\delta} L^{2\delta + 2\alpha} \| v_{n} \|_{C_{x}^{\delta+2\alpha}}  \\
    &\leq C (M_{v}^{1-\delta-2\alpha} + M_{v}) L^{1+4\delta+2\alpha}L_{n}  \ell^{-\delta}  D_{n}^{\delta+2\alpha} \\
    &\leq C (M_{v}^{1-\delta-2\alpha} + M_{v}) L^{1+4\delta+2\alpha} L_{n}  \ell^{-2\delta-4\alpha}.
\end{align*} 
Therefore, 
\begin{align*}
    \| \mathring{R}_{3}^{\mathrm{moll}} \|_{C_{\leq \mathfrak{t}_{L}}C_{x}^{1}} \leq C (M_{v}^{1-\delta-2\alpha} + M_{v}) L^{1+4\delta+2\alpha} L_{n} \ell^{-2\delta-4\alpha}. \qedhere
\end{align*}
\end{proof}

\subsection{Compressibility error}\label{sec-proof-comp-R}
 
\begin{proof}[Proof of Proposition \ref{prop:R_comp}]
Throughout this proof, let $t \in [0,\mft_{L}]$.
We use the decomposition of \cite[Proof of Proposition 4.11]{HLP23}:
\begin{align*} 
\mathring{R}_{n+1}^{\rm comp} &= \mathring{R}_{n+1}^{{\rm comp},1} + \mathring{R}_{n+1}^{{\rm comp}, 2} + \mathring{R}_{n+1}^{{\rm comp},3} + \mathring{R}_{n+1}^{{\rm comp},4} \\
\mathring{R}_{n+1}^{{\rm comp},1} &:= \mathcal{R}_{\phi_{n+1}} \partial_{t} w_{c}^{1}, \\
\mathring{R}_{n+1}^{{\rm comp},2} &:= \mathcal{R}_{\phi_{n+1}} \partial_{t} w_{c}^{2}, \\
\mathring{R}_{n+1}^{{\rm comp},3} &:= \mathcal{R}_{\phi_{n+1}} \mathrm{div}_{\phi_{n+1}}\left( v_{n+1} \otimes w_{c} + w_{c} \otimes v_{n+1} - w_{c} \otimes w_{c} \right), \\
\mathring{R}_{n+1}^{{\rm comp},4} &:= \mathcal{R}_{\phi_{n+1}} \mathrm{div}_{\phi_{n+1}}\left( v_{\ell} \otimes w_{c} \right).
\end{align*}
Following \cite{HLP23}, it is easy to see that the constants in
\begin{align*}
	\|\mathring{R}_{n+1}^{{\rm comp},1}\|_{C_{\leq \mft_{L}} C_x^{\delta}} &\leq C (1 + M_v^{2\delta}) {L^{12+5\delta} L_n } \delta_{n+2}^{\frac{6}{5}}\ell^{-3\delta},\\
	\|\mathring{R}_{n+1}^{{\rm comp},1}\|_{C_{\leq \mft_{L}} C_x^{1+\delta}} &\leq C (1 + M_v^{2\delta}) {L^7 L_n}\delta_{n+2}^{\frac{6}{5}} \ell^{-1-3\delta},
\end{align*}
are independent of the energy. Next, for $r \geq r_* +1$ and $t \leq \mft_{L}$, by the stationary phase lemma \ref{lem:stat-phase}
\begin{align*}
	\|\mathring{R}_{n+1}^{{\rm comp},2}\|_{C_x^{\delta}} &\leq C {L^{r+3+2\delta}}\sum_k \lambda^{\delta-2} [\partial_s a_k]_{C_x^1} + \lambda^{\delta-r-1} [\partial_s a_k]_{C_x^{r+1}} + \lambda^{-r-1}[\partial_s a_k]_{C_x^{r+1+\delta}}\\
	&\qquad\qquad\qquad\qquad + \lambda^{\delta-1}[\partial_{\tau}a_k]_{C_x^1} + \lambda^{\delta-r} [\partial_{\tau}a_k]_{C_x^{r+1}} + \lambda^{-r} [\partial_{\tau}a_k]_{C_x^{r+1+\delta}}\\
	&\qquad\qquad\qquad\qquad + \varsigma_{n+1}^{\gamma-1}\left(\lambda^{\delta-1}[a_k]_{C_x^2} +\lambda^{\delta-r}[a_k]_{C_x^{r+2}}+\lambda^{-r}[a_k]_{C_x^{r+2+\delta}} \right)\\
	&\quad + C{L^{11+4\delta}}\varsigma_{n+1}^{\gamma-1}\|w_c^2\|_{C_x^{\delta}}\\
	&\leq C {L^{r+3+2\delta}}\sum_k \lambda^{\delta-2} C_e^{(8),1}{L_n^{\frac{9}{2}}}\mu^2\varsigma_{n+1}^{\gamma-2}\delta_n^{\frac{1}{2}}(D_n+\varsigma_{n+1}^{\gamma-1}) \\
	&\hspace{2cm}+ \lambda^{\delta-r-1}C_e^{(8),r+1}{L_n^{2r+\frac{5}{2}}}\mu^{r+2}\varsigma_{n+1}^{(r+1)(\gamma-1)-1}\delta_n^{\frac{1}{2}}(D_n\ell^{-r}+\varsigma_{n+1}^{\gamma-1}) \\
	&\hspace{2cm}+ \lambda^{-r-1}C_e^{(8),r+1+\delta}{L_n^{2(r+\delta)+\frac{5}{2}}}\mu^{r+2+\delta}\varsigma_{n+1}^{(r+\delta+1)(\gamma-1)-1}\delta_n^{\frac{1}{2}}(D_n\ell^{-(r+\delta)}+\varsigma_{n+1}^{\gamma-1})\\
	&\hspace{2cm}+ \lambda^{\delta-1}C_e^{(2),1}(1+{M_v}){L_n^3}\mu\varsigma_{n+1}^{2(\gamma-1)}\delta_n^{\frac{1}{2}}\\
	&\hspace{2cm} + \lambda^{\delta -r}C_e^{(6),r+1}(1+{M_v}){L_n^{r+3}}\mu^{r+1}\varsigma_{n+1}^{(r+2)(\gamma-1)}\delta_n^{\frac{1}{2}}\\
	&\hspace{2cm} + \lambda^{-r}C_e^{(6),r+1+\delta}(1+{M_v}){L_n^{r+\delta+3}}\mu^{r+1+\delta}\varsigma_{n+1}^{(r+2+\delta)(\gamma-1)}\delta_n^{\frac{1}{2}}\\
	&\hspace{2cm} +\varsigma_{n+1}^{\gamma-1}\left(\lambda^{\delta-1}C_e^{(5),2}{L_n^{\frac{7}{2}}}\mu^2\varsigma_{n+1}^{2(\gamma-1)}\delta_n^{\frac{1}{2}}\right.\\
	&\hspace{3.5cm} + \lambda^{\delta-r}C_e^{(5),r+2}{L_n^{r+\frac{7}{2}}}\mu^{r+2}\varsigma_{n+1}^{(r+2)(\gamma-1)}\delta_n^{\frac{1}{2}}\\
	&\hspace{3.5cm}\left.+\lambda^{-r}C_e^{(5),r+2+\delta}{L_n^{r+\delta+\frac{7}{2}}}\mu^r\varsigma_{n+1}^{(r+\delta+2)(\gamma-1)}\delta_n^{\frac{1}{2}}\right)\\
	&\quad +C{L^{12+6\delta}L_n^{\frac{5}{2}+\delta}}\left(C_e^{(1),1} + C_e^{(5),1+\delta}\right)  \lambda^{\delta-1}\mu\varsigma_{n+1}^{2\gamma-2}\delta_n^{\frac{1}{2}}\\
	&\leq C_e^{{\rm comp},2, \delta} {L^{12+6\delta}L_n^{2(r+\delta)+\frac{5}{2}}}\lambda^{\delta-1}\mu\varsigma_{n+1}^{2\gamma-2}\delta_n^{\frac{1}{2}},
\end{align*}
where
\begin{align*}
	C_e^{{\rm comp},2, \delta} &:= C \left(C_e^{(8),1} + C_e^{(8),r+1} + C_e^{(8),r+\delta+1} + \left(C_e^{(2),1} + C_e^{(6),r+1}  +C_e^{(6),r+\delta+1}\right){(1+M_v)}\right.\\
	&\qquad\quad \left.+ C_e^{(5),2} + C_e^{(5),r+2} + C_e^{(5),r+\delta+ 2} + C_e^{(1),1} + C_e^{(5),1+\delta}\right).
\end{align*}
Further (recalling the notation $u_{c}(t,x) := i \sum_{k \in \Lambda} \nabla_{\phi_{n+1}} a_{k}(t,x,\lambda t) \times \frac{k}{|k|^{2}} \times \Omega_{k}(\lambda \phi_{n+1})$ and the identity $w_{c}^{2} = \frac{1}{\lambda} \mathcal{Q}\left( u_{c} \circ \phi_{n+1}^{-1} \right) \circ \phi_{n+1}$ from \cite[p.45]{HLP23}), we have
\begin{align*}
	\|\mathring{R}_{n+1}^{{\rm comp},2}\|_{C_x^{1+\delta}}&\leq \lambda^{-1} \|(\mathcal{RQ}\partial_t(u_c\circ \phi_{n+1}^{-1}))\circ \phi_{n+1}\|_{C_x^{1+\delta}} + \|\mathcal{R}^{\phi_{n+1}}(\dot{\phi}_{n+1}\cdot \nabla^{\phi_{n+1}}w_c^2)\|_{C_x^{1+\delta}}\\
	&\leq C{L^{3+2\delta}}\left(\lambda^{-1} {L^{\delta}}\|\partial_t (u_c \circ \phi_{n+1}^{-1})\|_{C_x^{\delta}} + \varsigma_{n+1}^{\gamma -1}({L}\|w_c^2\|_{C_x^1} + {L^{2+\delta}}\|w_c^2\|_{C_x^{1+\delta}})\right)\\
	&\leq C{L^{3+2\delta}}\left(\lambda^{-1} {L^{\delta}}\left(\sum_k \|(\nabla \partial_t)(a_k \circ \phi_{n+1}^{-1})\|_{C_x^{\delta}} + \|(\nabla \partial_t)(a_k \circ \phi_{n+1}^{-1})\|_{C_x} \lambda^{\delta}\right)\right.\\
	&\qquad\qquad\qquad \left.+{L^{5+3\delta}}{L_n^{\frac{7}{2}+\delta}}(C_e^{(1),1} + C_e^{(5),2+\delta}) \lambda^{\delta}\mu\varsigma_{n+1}^{2\gamma-2}\delta_n^{\frac{1}{2}}\right)\\
	&\leq C{L^{3(1+\delta)}}\left(\lambda^{-1} {L^{\delta}}\left(\sum_k {L^{\delta}}\left(\|\partial_s a_k\|_{C_x^{1+\delta}} + \varsigma_{n+1}^{\gamma-1}\left(\|a_k\|_{C_x^{1+\delta}} + \|a_k\|_{C_x^{2+\delta}}\right)\right) \right.\right.\\
	&\qquad\qquad\qquad\qquad\left.+ \|\partial_s a_k\|_{C_x^1} + \varsigma_{n+1}^{\gamma-1}\left(\|a_k\|_{C_x^1} + \|a_k\|_{C_x^2}\right) + \|(\nabla \partial_t)(a_k \circ \phi_{n+1}^{-1})\|_{C_x} \lambda^{\delta}\right)\\
	&\qquad\qquad\qquad \left.+{L^{5+2\delta}}{L_n^{\frac{7}{2}+\delta}}(C_e^{(1),1} + C_e^{(5),2+\delta}) \lambda^{\delta}\mu\varsigma_{n+1}^{2\gamma-2}\delta_n^{\frac{1}{2}}\delta_n^{\frac{1}{2}}\right)\\
	&\leq C{L^{3(1+\delta)}}\\
	&\quad\left(\lambda^{-1} {L^{\delta}}\left(\sum_k {L^{\delta}}\left(C_e^{(8),1+\delta}{L_n^{\frac{13}{2}+2\delta}}\mu^{2+\delta}\varsigma_{n+1}^{(1+\delta)(\gamma-1)-1}\delta_n^{\frac{1}{2}}(D_n\ell^{-\delta} + \varsigma_{n+1}^{\gamma-1}) \right.\right.\right.\\
	&\hspace{3.5cm}\left.+ \varsigma_{n+1}^{\gamma-1}C_e^{(5),2+\delta}{L_n^{\frac{5}{2}+\delta}}\delta_n^{\frac{1}{2}}\mu^{2+\delta}\varsigma_{n+1}^{(2+\delta)(\gamma-1)}\right) \\
	&\qquad\qquad\qquad\left.+ (1+\lambda^{\delta})C_e^{(8),1}{L_n^{\frac{13}{2}}}\mu^2\varsigma_{n+1}^{\gamma-2}\delta_n^{\frac{1}{2}}(D_n+ \varsigma_{n+1}^{\gamma-1}) + \varsigma_{n+1}^{\gamma-1}C_e^{(5),2}{L_n^{\frac{5}{2}}}\delta_n^{\frac{1}{2}}\mu^2\varsigma_{n+1}^{2(\gamma-1)} \right)\\
	&\qquad \left.+ {L^{5+2\delta}}{L_n^{\frac{7}{2}+\delta}}(C_e^{(1),1} + C_e^{(5),2+\delta}) \lambda^{\delta}\mu\varsigma_{n+1}^{2\gamma-2}\delta_n^{\frac{1}{2}}\delta_n^{\frac{1}{2}}\right)\\
	&\leq C_e^{{\rm comp},2,1+\delta}{L^{8+5\delta}L_n^{\frac{13}{2}+2\delta}}\lambda^{\delta}\mu\varsigma_{n+1}^{2\gamma-2}\delta_n^{\frac{1}{2}}\delta_n^{\frac{1}{2}},
\end{align*}
where the constant is given by
\[C_e^{{\rm comp},2,1+\delta}:= C\left(C_e^{(1),1} + C_e^{(5),2+\delta} + C_e^{(8),1+\delta}\right).\]
Next we apply Equ. \eqref{eq:RdivCr} which yields
\begin{align*}
	\|\mathring{R}_{n+1}^{{\rm comp},3}\|_{C_x^{\delta}} &\leq C {L^{2\delta}}\left(\left(\|v_{\ell}\|_{C_x^{\delta}}  + \|w_o\|_{C_x^{\delta}} + \|w_c\|_{C_x^{\delta}}\right)\|w_c\|_{C_x}+ \left(\|v_{\ell}\|_{C_x}  + \|w_o\|_{C_x} + \|w_c\|_{C_x}\right)\|w_c\|_{C_x^{\delta}}\right)\\
	&\leq C {L^{2\delta}}\Big({\ell^{-\delta}M_vL_n}  + {L^{2\delta}L_n^{\frac{5}{2}}}C_e^{(1),\delta}\lambda^{\delta}\delta_n^{\frac{1}{2}} + (1 + M_v^{2\delta}) {L^7L_n }\delta_{n+1}^{\frac{6}{5}} \\
	&\qquad\qquad+{L^{1+2\delta}L_n^{\frac{5}{2}+\delta}}\left(C_e^{(1),1} + C_e^{(5),1+\delta}\right)\lambda^{\delta-1}\mu\varsigma_{n+1}^{\gamma-1}\delta_n^{\frac{1}{2}}\Big)\\
	&\quad \times\left({(1 + M_v^{2\delta}) L^7L_n}\delta_{n+1}^{\frac{6}{5}}+{L^{1+2\delta}L_n^{\frac{5}{2}+\delta}}\left(C_e^{(1),1} + C_e^{(5),1+\delta}\right)\lambda^{\delta-1}\mu\varsigma_{n+1}^{\gamma-1}\delta_n^{\frac{1}{2}}\right)\\
	&=: C_e^{{\rm comp},3,\delta}{L^{2+6\delta} L_n^{5+ 2\delta}}\lambda^{\delta}\delta_n^{\frac{1}{2}}\delta_{n+2}^{\frac{6}{5}},
\end{align*}
where
\[C_e^{{\rm comp},3,\delta} := C \left(1 + {M_v} + {M_v^{2\delta}}+ C_e^{(1),1} +C_e^{(5),1+\delta}\right)\left(1 + {M_v^{2\delta}}+C_e^{(1),1} + C_e^{(5),1+\delta}\right).\]
Moreover, using Equ. \eqref{eq:antidiv-orderminusone-r}
\begin{align*}
	&\|\mathring{R}_{n+1}^{{\rm comp},3}\|_{C_x^{1+ \delta}} \\ 
	&\leq C {L^{2(1+\delta)}}\left(\left(\|v_{\ell}\|_{C_x^{\delta}} + \|w_o\|_{C_x^{\delta}}+ \|w_c\|_{C_x^{\delta}}\right)\|w_c\|_{C_x^{1+\delta}} + \left(\|v_{\ell}\|_{C_x^{1+\delta}} + \|w_o\|_{C_x^{1+\delta}}\right)\|w_c\|_{C_x^{\delta}}\right)\\
	&\leq C {L^{2(1+\delta)}}\Big(\ell^{-\delta} {M_vL_n} + C_e^{(1),\delta} {L^{2\delta}L_n^{\frac{5}{2}}}\lambda^{\delta} \delta_n^{\frac{1}{2}} + (1 + M_v^{2\delta}){L^7L_n }\delta_{n+2}^{\frac{6}{5}} \\
	&\qquad\qquad\qquad +(C_e^{(1),1} + C_e^{(5),1+\delta}){L^{1+2\delta}L_n^{\frac{5}{2}+2\delta}}\lambda^{\delta-1}\mu\varsigma_{n+1}^{\gamma-1} \delta_n^{\frac{1}{2}}\Big)\\
	&\quad \times\left((1 + M_v^{2\delta}){L^7L_n}\ell^{-1}\delta_{n+2}^{\frac{6}{5}} + \left(C_e^{(1),1} + C_e^{(5),2+\delta}\right){L^{3+2\delta}L_n^{\frac{7}{2}+2\delta}}\lambda^{\delta}\mu\varsigma_{n+1}^{\gamma-1}\delta_n^{\frac{1}{2}}\right) \\
	&\qquad + C {L^{2(1+\delta)}}\left(\ell^{-\delta}{L_n}D_n + \left(C_e^{(1),\delta} + C_e^{(5),1+\delta}\right){L^{2+2\delta}L_n^{\frac{5}{2}+2\delta}}\lambda^{1+\delta}\delta_n^{\frac{1}{2}}\right)\\
	&\quad \times \left((1 + M_v^{2\delta}){L^7L_n }\delta_{n+2}^{\frac{6}{5}} + (C_e^{(1),1} + C_e^{(5),1+\delta}){L^{1+2\delta}L_n^{\frac{5}{2}+2\delta}}\lambda^{\delta} \mu\varsigma_{n+1}^{\gamma-1} \delta_n^{\frac{1}{2}}\right)\\
	&\leq C_e^{{\rm comp},3,1+\delta}{L^{16+2\delta}L_n^{6+4\delta}}\lambda^{1+\delta}\delta_n^{\frac{1}{2}}\delta_{n+2}^{\frac{6}{5}},
\end{align*}
where we have set
\begin{align*}
	C_e^{{\rm comp},3,1+\delta} &:= C\left(\left(1 + + {M_v^{2\delta}} + {M_v}  + C_e^{(1),1} + C_e^{(5),1+\delta}\right)\left(1 + {M_v^{2\delta}}+ C_e^{(1),1} + C_e^{(5),2+\delta}\right)\right.\\
	&\qquad\quad \left.+\left(1+ C_e^{(1),\delta} + C_e^{(5),1+\delta}\right)\left(1 + {M_v^{2\delta}}+C_e^{(1),1} + C_e^{(5),1+\delta}\right)\right).
\end{align*}
Finally we decompose
\begin{align*}
	\mathring{R}_{n+1}^{{\rm comp},4} &= \mathcal{R}^{\phi_{n+1}}\left(v_{\ell}{\rm div}^{\phi_{n+1}}(w_o)\right) + \mathcal{R}^{\phi_{n+1}}\left((w_o \cdot \nabla^{\phi_{n+1}})v_{\ell}\right)\\
	&= \sum_k \mathcal{R}^{\phi_{n+1}}\bigg(\Big[(v_{\ell}\circ \phi_{n+1}^{-1})\left({\rm div}(a_k\circ\phi_{n+1}^{-1}) + \lambda (a_k \circ \phi_{n+1}^{-1})\right)E_k e^{{\rm i}\lambda k \cdot x} \\
	&\qquad\qquad\qquad + (a_k \circ\phi_{n+1}^{-1})E_k\cdot\nabla(v_{\ell} \circ \phi_{n+1}^{-1})e^{{\rm i}\lambda k \cdot x}\Big]\circ \phi_{n+1}\bigg).
\end{align*}
We then obtain by the stationary phase lemma, Lemma \ref{lem:stat-phase}, for $r \geq r_* + 2$ and $t \leq \mft_{L}$,
\begin{align*}
	&\|\mathring{R}_{n+1}^{{\rm comp},4}\|_{C_x^{\delta}} \\
	&\leq C {L^{\delta}}\sum_k \lambda^{\delta-1}\left(\|v_{\ell}\|_{C_x}\left({L}\|a_k\|_{C_x^1} + \lambda \|a_k \|_{C_x}\right)+ {L}\|a_k\|_{C_x}\|v_{\ell}\|_{C_x^1} \right)\\
	&\qquad\qquad+ \lambda^{\delta-r}\left(\|(v_{\ell}\circ \phi_{n+1}^{-1}){\rm div}(a_k\circ\phi_{n+1}^{-1})\|_{C_x^r} + \lambda \|(v_{\ell}\circ \phi_{n+1}^{-1})(a_k \circ \phi_{n+1}^{-1})\|_{C_x^r} \right.\\
	&\hspace{2cm}\left.+ \|(a_k \circ\phi_{n+1}^{-1})E_k\cdot\nabla(v_{\ell} \circ \phi_{n+1}^{-1})\|_{C_x^r} \right)\\
	&\qquad\qquad+ \lambda^{-r} \left(\| (v_{\ell}\circ \phi_{n+1}^{-1}){\rm div}(a_k\circ\phi_{n+1}^{-1})\|_{C_x^{r+\delta}} + \lambda \|(v_{\ell}\circ \phi_{n+1}^{-1})(a_k \circ \phi_{n+1}^{-1})\|_{C_x^{r+\delta}}  \right.\\
	&\hspace{2cm}\left.+ \|(a_k \circ\phi_{n+1}^{-1})E_k\cdot\nabla(v_{\ell} \circ \phi_{n+1}^{-1})\|_{C_x^{r+\delta}}\right)\\
	&\leq C {L^{\delta}}\sum_k \lambda^{\delta-1}\left(\|v_{\ell}\|_{C_x}\left({L}\|a_k\|_{C_x^1} + \lambda \|a_k \|_{C_x}\right)+ {L}\|a_k\|_{C_x}\|v_{\ell}\|_{C_x^1} \right)\\
	&\qquad\quad+ \lambda^{\delta-r}\left(\|v_{\ell}\|_{C_x}{L^{r+1}}\|a_k\|_{C_x^{r+1}} + {L^r}\|v_{\ell}\|_{C_x^r}\|a_k\|_{C_x^1} + \sum_{j=1}^{r-1} {L^j}\|v_{\ell}\|_{C_x^j} {L^{r+1-j}}\|a_k\|_{C_x^{r+1-j}} \right.\\
	&\hspace{1.8cm}\left.+ \lambda \left(\|v_{\ell}\|_{C_x}{L^r}\|a_k\|_{C_x^r} + {L^r}\|v_{\ell}\|_{C_x^r}\|a_k\|_{C_x} + \sum_{j=1}^{r-1} {L^j}\|v_{\ell}\|_{C_x^j} {L^{r-j}}\|a_k\|_{C_x^{r-j}}\right) \right.\\
	&\hspace{1.8cm}\left.+{L}\|v_{\ell}\|_{C_x^1}{L^r}\|a_k\|_{C_x^r} + {L^{r+1}}\|v_{\ell}\|_{C_x^{r+1}}\|a_k\|_{C_x} + \sum_{j=1}^{r-1} {L^{j+1}}\|v_{\ell}\|_{C_x^{j+1}} {L^{r-j}}\|a_k\|_{C_x^{r-j}} \right)\\
	&\qquad\quad+ \lambda^{-r} \left(\|v_{\ell}\|_{C_x}{L^{r+2}}\|a_k\|_{C_x^{r+2}} + {L^{r+1}}\|v_{\ell}\|_{C_x^{r+1}}\|a_k\|_{C_x^1} + \sum_{j=1}^r {L^j}\|v_{\ell}\|_{C_x^j} {L^{r+2-j}}\|a_k\|_{C_x^{r+2-j}} \right.\\
	&\hspace{1.8cm}\left.+ \lambda \left(\|v_{\ell}\|_{C_x}{L^{r+1}}\|a_k\|_{C_x^{r+1}} + {L^{r+1}}\|v_{\ell}\|_{C_x^{r+1}}\|a_k\|_{C_x} + \sum_{j=1}^r {L^j}\|v_{\ell}\|_{C_x^j} {L^{r+1-j}}\|a_k\|_{C_x^{r+1-j}}\right) \right.\\
	&\hspace{1.8cm}\left.+{L}\|v_{\ell}\|_{C_x^1}{L^{r+1}}\|a_k\|_{C_x^{r+1}} + {L^{r+2}}\|v_{\ell}\|_{C_x^{r+2}}\|a_k\|_{C_x} + \sum_{j=1}^r {L^{j+1}}\|v_{\ell}\|_{C_x^{j+1}} {L^{r+1-j}}\|a_k\|_{C_x^{r+1-j}} \right),
\end{align*}
which we estimate further by using Proposition \ref{prop-energy-1}
\begin{align*}
    &\|\mathring{R}_{n+1}^{{\rm comp},4}\|_{C_x^{\delta}} \\
    &\leq C {L^{\delta}}\sum_k {LL_n^{\frac{7}{2}}}\lambda^{\delta-1}\left({M_v}\left(C_e^{(1),1}\mu\varsigma_{n+1}^{\gamma-1}\delta_n^{\frac{1}{2}} + \lambda \sqrt{\bar{e}}\delta_n^{\frac{1}{2}}\right)+ \sqrt{\bar{e}}\delta_n^{\frac{1}{2}}D_n \right)\\
	&\qquad\qquad+ \lambda^{\delta-r}{L^{r+1}L_n^{r+\frac{7}{2}}}\left({M_v}C_e^{(5),r+1}\mu^{r+1}\varsigma_{n+1}^{(r+1)(\gamma-1)}\delta_n^{\frac{1}{2}} + {D_n}\ell^{1-r}C_e^{(1),1}\mu\varsigma_{n+1}^{\gamma-1}\delta_n^{\frac{1}{2}} \right.\\
	&\hspace{3cm} + \sum_{j=1}^{r-1} D_n\ell^{1-j} C_e^{(5),r+1-j}\mu^{r+1-j}\varsigma_{n+1}^{(r+1-j)(\gamma-1)}\delta_n^{\frac{1}{2}} \\
	&\hspace{3cm}+ \lambda \left({M_v}C_e^{(5),r}\mu^r\varsigma_{n+1}^{r(\gamma-1)}\delta_n^{\frac{1}{2}} + D_n\ell^{1-r}\sqrt{\bar{e}}\delta_n^{\frac{1}{2}} \right.\\
	&\hspace{4cm}\left.+ \sum_{j=1}^{r-1} D_n\ell^{1-j} C_e^{(5),r-j}\mu^{r-j}\varsigma_{n+1}^{(r-j)(\gamma-1)}\delta_n^{\frac{1}{2}}\right)\\
	&\hspace{3cm}+D_nC_e^{(5),r}\mu^r\varsigma_{n+1}^{r(\gamma-1)}\delta_n^{\frac{1}{2}} + D_n\ell^{-r}\sqrt{\bar{e}}\delta_n^{\frac{1}{2}} \\
	&\hspace{3cm}\left.+ \sum_{j=1}^{r-1}D_n\ell^{-j}C_e^{(5),r-j}\mu^{r-j}\varsigma_{n+1}^{(r-j)(\gamma-1)}\delta_n^{\frac{1}{2}} \right)\\
	&\qquad\qquad+ \lambda^{-r} {L^{r+2}L_n^{r+\frac{9}{2}}}\left({M_v}C_e^{(5),r+2}\mu^{r+2}\varsigma_{n+1}^{(r+2)(\gamma-1)}\delta_n^{\frac{1}{2}} +D_n\ell^{-r}C_e^{(1),1}\mu\varsigma_{n+1}^{\gamma-1}\delta_n^{\frac{1}{2}} \right.\\
	&\hspace{3cm}+ \sum_{j=1}^r D_n\ell^{1-j} C_e^{(5),r+2-j}\mu^{r+2-j}\varsigma_{n+1}^{(r+2-j)(\gamma-1)}\delta_n^{\frac{1}{2}} \\
	&\hspace{3cm}+ \lambda \left({M_v}C_e^{(5),r+1}\mu^{r+1}\varsigma_{n+1}^{(r+1)(\gamma-1)}\delta_n^{\frac{1}{2}} + D_n\ell^{-r}\sqrt{\bar{e}}\delta_n^{\frac{1}{2}} \right.\\
	&\hspace{4cm}\left.+ \sum_{j=1}^r D_n\ell^{1-j} C_e^{(5),r+1-j}\mu^{r+1-j}\varsigma_{n+1}^{(r+1-j)(\gamma-1)}\delta_n^{\frac{1}{2}}\right)\\
	&\hspace{3cm}+D_nC_e^{(5),r+1}\mu^{r+1}\varsigma_{n+1}^{(r+1)(\gamma-1)}\delta_n^{\frac{1}{2}} + D_n\ell^{-(r+1)}\sqrt{\bar{e}}\delta_n^{\frac{1}{2}} \\
	&\hspace{3cm}\left.+ \sum_{j=1}^rD_n\ell^{-j} C_e^{(5),r+1-j}\mu^{r+1-j}\varsigma_{n+1}^{(r+1-j)(\gamma-1)}\delta_n^{\frac{1}{2}} \right)\\
	&\leq C {L^{\delta}}\delta_n^{\frac{1}{2}}\sum_k {LL_n^{\frac{7}{2}}}\lambda^{\delta-1}\mu\varsigma_{n+1}^{\gamma-1}\left({M_v}\left(C_e^{(1),1} + \sqrt{\bar{e}}\right)+ \sqrt{\bar{e}} \right)\\
	&\qquad+ \lambda^{\delta-r}\mu^{r+1}\varsigma_{n+1}^{(r+1)(\gamma-1)}{L^{r+1}L_n^{r+\frac{7}{2}}}\\
	&\qquad\quad \times \left((1+{M_v})C_e^{(5),r+1} +C_e^{(1),1}+ \lambda \left((1+{M_v})C_e^{(5),r}+ \sqrt{\bar{e}} \right)+C_e^{(5),r} + \sqrt{\bar{e}} \right)\\
	&\qquad+ \lambda^{-r} \mu^{r+2}\varsigma_{n+1}^{(r+2)(\gamma-1)}{L^{r+2}L_n^{r+\frac{9}{2}}}\\
	&\qquad\quad\times\left((1+{M_v})C_e^{(5),r+2}+C_e^{(1),1}+ \lambda \left((1+{M_v})C_e^{(5),r+1} + \sqrt{\bar{e}} \right)+C_e^{(5),r+1} + \sqrt{\bar{e}}\right)\\
	&\leq C_e^{{\rm comp},4,\delta} {L^{r+\delta+2}L_n^{r+\frac{9}{2}}}\lambda^{\delta-1}\mu\varsigma_{n+1}^{\gamma-1}\delta_n^{\frac{1}{2}}
\end{align*}
where we used $D_n\ell^{-j} \leq \mu^j \varsigma_{n+1}^{j(\gamma-1)}$, and where
\[C_e^{{\rm comp},4,\delta} := C(1+{M_v})\left(\sqrt{\bar{e}} + C_e^{(1),1} + C_e^{(5),r+2}\right).\]
Further
\begin{align*}
	&\|\mathring{R}_{n+1}^{{\rm comp},4}\|_{C_x^{1+ \delta}} \\
	&\leq C {L^{3+2\delta}}\left(\|(v_{\ell}\circ \phi_{n+1}^{-1}){\rm div}(a_k\circ\phi_{n+1}^{-1})\|_{C_x^{\delta}} + \lambda \|(v_{\ell}\circ \phi_{n+1}^{-1})(a_k \circ \phi_{n+1}^{-1})\|_{C_x^{\delta}} \right.\\
	&\hspace{2cm}\left.+ \|(a_k \circ\phi_{n+1}^{-1})E_k\cdot\nabla(v_{\ell} \circ \phi_{n+1}^{-1})\|_{C_x^{\delta}} \right)\\
	&\leq C {L^{3+2\delta}}\left({L^{1+\delta}}\left(\|v_{\ell}\|_{C_x}\|a_k\|_{C_x^{1+\delta}} + \|v_{\ell}\|_{C_x^{\delta}}\|a_k\|_{C_x^1}\right) + \lambda {L^{\delta}}\left(\|v_{\ell}\|_{C_x}\|a_k\|_{C_x^{\delta}} + \|v_{\ell}\|_{C_x^{\delta}}\|a_k\|_{C_x}\right) \right.\\
	&\hspace{2cm}\left.+ {L^{1+\delta}}\left(\|a_k\|_{C_x}\|v_{\ell}\|_{C_x^{1+\delta}} + \|a_k\|_{C_x^{\delta}}\|v_{\ell}\|_{C_x^1}\right)\right)\\
	&\leq C {L^{3+2\delta}}\left({L^{1+\delta}}{L_n^{\frac{7}{2}+\delta}}\left({M_v}C_e^{(5),1+\delta}\mu^{1+\delta}\varsigma_{n+1}^{(1+\delta)(\gamma-1)}\delta_n^{\frac{1}{2}} + {M_v}\ell^{-\delta}C_e^{(1),1}\mu\varsigma_{n+1}^{\gamma-1}\delta_n^{\frac{1}{2}}\right)\right.\\
	&\hspace{2cm}\left. + \lambda {L^{\delta}L_n^{\frac{7}{2}}}\left({M_v}C_e^{(1),\delta}\mu^{\delta}\varsigma_{n+1}^{\delta(\gamma-1)}\delta_n^{\frac{1}{2}} + {M_v}\ell^{-\delta}\sqrt{\bar{e}}\delta_n^{\frac{1}{2}}\right) \right.\\
	&\hspace{2cm}\left.+ {L^{1+\delta}L_n^{\frac{7}{2}}}\left(\sqrt{\bar{e}}\delta_n^{\frac{1}{2}}D_n\ell^{-\delta} + C_e^{(1),\delta}\mu^{\delta}\varsigma_{n+1}^{\delta(\gamma-1)}\delta_n^{\frac{1}{2}}D_n\right)\right)\\
	&\leq C {L^{4+3\delta}L_n^{\frac{7}{2}+\delta}}\left({M_v}C_e^{(5),1+\delta}\mu^{1+\delta}\varsigma_{n+1}^{(1+\delta)(\gamma-1)}\delta_n^{\frac{1}{2}} + {M_v}\ell^{-\delta}C_e^{(1),1}\mu\varsigma_{n+1}^{\gamma-1}\delta_n^{\frac{1}{2}}\right.\\
	&\hspace{3cm} + \lambda \left({M_v}C_e^{(1),\delta}\mu^{\delta}\varsigma_{n+1}^{\delta(\gamma-1)}\delta_n^{\frac{1}{2}} + {M_v}\ell^{-\delta}\sqrt{\bar{e}}\delta_n^{\frac{1}{2}}\right)\\
	&\hspace{3cm}\left.+ \sqrt{\bar{e}}\delta_n^{\frac{1}{2}}D_n\ell^{-\delta} + C_e^{(1),\delta}\mu^{\delta}\varsigma_{n+1}^{\delta(\gamma-1)}\delta_n^{\frac{1}{2}}D_n\right)\\
	&\leq C_e^{{\rm comp},4,1+\delta}{L^{4+3\delta}L_n^{\frac{7}{2}+\delta}} (\lambda^{\delta}\mu \varsigma_{n+1}^{\gamma-1}+\lambda \mu^{\delta}\varsigma_{n+1}^{\delta(\gamma-1)})\delta_n^{\frac{1}{2}}
\end{align*}
where
\[C_e^{{\rm comp},4,1+\delta} := C\left({M_v}\left(C_e^{(5),1+\delta} + C_e^{(1),1}+\sqrt{\bar{e}}\right) + \sqrt{\bar{e}} + C_e^{(1),\delta} \right).\]
Thus for the total compressibility error, we conclude
\begin{align*}
	\|\mathring{R}_{n+1}^{\rm comp}\|_{C_{\leq \mft_{L}} C_x} &\leq C (1 + M_v^{2\delta}) {L^7 {L_n} } \delta_{n+2}^{\frac{6}{5}}\ell^{-3\delta} \\
	&\quad+ C_e^{{\rm comp},2, \delta} {L^{12+6\delta}L_n^{2(r+\delta)+\frac{9}{2}}}\lambda^{\delta-1}\mu\varsigma_{n+1}^{2\gamma-2}\delta_n^{\frac{1}{2}}\\
	&\quad + C_e^{{\rm comp},3,\delta}{L^{2+6\delta}L_n^{5+2\delta}}\lambda^{\delta}\delta_n^{\frac{1}{2}}\delta_{n+2}^{\frac{6}{5}}\\
	&\quad + C_e^{{\rm comp},4,\delta} {L^{r+\delta+2}L_n^{r+\frac{9}{2}}}\lambda^{\delta-1}\mu\varsigma_{n+1}^{\gamma-1}\delta_n^{\frac{1}{2}}\\
	&\leq C\left(1 + {M_v^{2\delta}}+ C_e^{{\rm comp},2, \delta} + C_e^{{\rm comp},3, \delta} + C_e^{{\rm comp},4, \delta} \right){L^{12+6\delta}L_n^{2(r+\delta)+\frac{9}{2}}}\lambda^{\delta}\delta_n^{\frac{1}{2}}\delta_{n+2}^{\frac{6}{5}},
\end{align*}
as well as
\begin{align*}
	\|\mathring{R}_{n+1}^{\rm comp}\|_{C_{\leq \mft_{L}} C_x^1} &\leq C (1 + M_v^{2\delta}) {L^7L_n }\delta_{n+2}^{\frac{6}{5}} \ell^{-1-3\delta} \\
	&\quad + C \left(C_e^{(1),1} + C_e^{(5),2+\delta} + C_e^{(8),1+\delta}\right){L^{8+5\delta}L_n^{\frac{13}{2}+2\delta}}\lambda^{\delta}\mu\varsigma_{n+1}^{2\gamma-2}\delta_n^{\frac{1}{2}}\\
	&\quad+ C_e^{{\rm comp},3,1+\delta}{L^{16+2\delta}L_n^{6+4\delta}}\lambda^{1+\delta}\delta_n^{\frac{1}{2}}\delta_{n+2}^{\frac{6}{5}}\\
	&\quad+ C_e^{{\rm comp},4,1+\delta} {L^{4+3\delta}L_n^{\frac{7}{2}+\delta}} (\lambda^{\delta}\mu \varsigma_{n+1}^{\gamma-1}+\lambda \mu^{\delta}\varsigma_{n+1}^{\delta(\gamma-1)})\delta_n^{\frac{1}{2}}\\
	&\leq C\left(1 + {M_v^{2\delta}}+ C_e^{{\rm comp},2,1+\delta} + C_e^{{\rm comp},3,1+\delta} + C_e^{{\rm comp},4,1+\delta}\right){L^{8+5\delta}L_n^{\frac{13}{2} + 4\delta}}\lambda^{1+\delta}\delta_n^{\frac{1}{2}}\delta_{n+2}^{\frac{6}{5}}. \qedhere
\end{align*}
\end{proof}

\subsection{Dissipative error}\label{ssec:pf_Rdiss} 
Throughout this proof, let $t \in [0,\mft_{L}]$.
We will employ similar arguments as \cite{Hypo-paper, DR19}. The trick here consists in applying commutativity of the operators $\mathcal{R}$ and $(-\D)^{\alpha}$, both of which are Fourier multipliers, bijectivity of $\phi_{n+1}$ as well as interpolation to get
\begin{align*}
    \| \mathcal{R}_{\phi_{n+1}}\left( (-\D)^{\alpha}_{\phi_{n+1}}w_{n+1} \right) \|_{C^{0}} &= \| \mathcal{R} \left\{ (-\D)^{\alpha}[w_{n+1} \circ \phi_{n+1}^{-1}] \right\} \circ \phi_{n+1} \|_{C^{0}} \\
    &\leq C \| \mathcal{R} \left\{ (-\D)^{\alpha}[w_{n+1} \circ \phi_{n+1}^{-1}] \right\} \circ \phi_{n+1} \|_{C^{\delta/2}}  \\
    &\overset{\text{Lemma \ref{lem:Cr_chain}}}{\leq} CL^{\delta/2} \| (-\D)^{\alpha}  \left\{ \mathcal{R} [w_{n+1} \circ \phi_{n+1}^{-1}] \right\} \|_{C^{\delta/2}} \\
    &\overset{\text{Theorem \ref{thm:fract_Lap_Holder}}}{\leq} CL^{\delta/2} \left[\mcR(w_{n+1} \circ \phi_{n+1}^{-1}) \right]_{2\alpha+\delta} \\
    &\leq CL^{\delta/2} \| \mcR(w_{n+1} \circ \phi_{n+1}^{-1}) \|_{C^{0}}^{1-2\alpha-\delta} \| D \mcR(w_{n+1} \circ \phi_{n+1}^{-1}) \|_{C^{0}}^{2\alpha+\delta} \\
    &= CL^{\delta/2} \| \mcR_{\phi_{n+1}} w_{n+1} \|_{C^{0}}^{1-2\alpha-\delta} \| D \left[ \mcR_{\phi_{n+1}} w_{n+1} \circ \phi_{n+1}^{-1} \right] \|_{C^{0}}^{2\alpha+\delta}.
\end{align*}
To calculate the first term, we need to calculate the three perturbative terms.
\begin{lem} \label{lem:Rw:C0} For $t \in [0,\mft_{L}]$, we have
    \begin{align*}
        \| \mcR_{\phi_{n+1}} w_{n+1} \|_{C^{0}_{x}} \leq  C L^{r+2\delta} L_{n}^{r + 3/2 + \delta} \left(1 + C_{e}^{(1),0} + C_{e}^{(5),r} + C_{e}^{(5),r+\delta}\right) \ell^{1 + d(1/p - 1)-2\delta} \delta_{n+2}^{6/5}.
    \end{align*}
\end{lem}
\begin{proof}
Recall that
\begin{align*}
    w_{o} = \sum_{k \in \Lambda} a_{k} E_{k} e^{i \lambda k \cdot \phi_{n+1}(x)} = \left( \sum_{k} (a_{k} \circ \phi_{n+1}^{-1}) E_{k} e^{i \lambda k \cdot } \right) \circ \phi_{n+1} =: \left( \sum_{k} \bar{a}_{k} E_{k} e^{i \lambda k \cdot } \right) \circ \phi_{n+1}
\end{align*}
Therefore, applying the stationary phase lemma, Lemma \ref{lem:stat-phase}, we find 
\begin{align*}
    &\| \mcR_{\phi_{n+1}} w_{o} \|_{C_{x}^{0}} \leq C \| \mcR_{\phi_{n+1}} w_{o} \|_{C_{x}^{\delta}} \leq C L^{\delta} \sum_{k} \left( \lambda^{\delta-1} \| \bar a_k \|_{C^{0}_{x}} + \lambda^{\delta - r} \| \bar a_k \|_{C^{r}_{x}} + \lambda^{-r} \| \bar a_k \|_{C^{r+\delta}_{x}} \right) \\
    &\leq C L^{r+2\delta} \lambda^{\delta-1} \sum_{k} \left(  \| a_k \|_{C^{0}_{x}} + \lambda^{- (r-1)} \| a_k \|_{C^{r}_{x}} + \lambda^{-(r+\delta-1)} \| a_k \|_{C^{r+\delta}_{x}} \right) \\
    &\leq L^{r+2\delta} L_{n}^{r+3/2+\delta} \delta_{n}^{1/2} \lambda^{\delta-1} \Bigg( C_{e}^{(1),0} + C_{e}^{(5),r} \mu \vs_{n+1}^{\gamma-1} \left( \frac{\mu \vs_{n+1}^{\gamma-1}}{\lambda} \right)^{r-1}  \\
    &\qquad \qquad \qquad + C_{e}^{(5),r+\delta}  \mu \vs_{n+1}^{\gamma-1} \left( \frac{\mu \vs_{n+1}^{\gamma-1}}{\lambda} \right)^{r+\delta-1} \Bigg).
\end{align*}
Note that
\begin{align*}
    \lambda^{-(r-1)} \left( \mu \vs_{n+1}^{\gamma-1} \right)^{r} &\leq 1, \\
    \lambda^{-(r-2)} \left( \mu \vs_{n+1}^{\gamma-1} \right)^{r+1} &\leq 1.
\end{align*}
Therefore, we find that we have simplified the original estimates to
\begin{align*}
    \| \mcR_{\phi_{n+1}} w_{o} \|_{C_{x}^{0}} \leq C L^{r+2\delta} L_{n}^{r+3/2+\delta}  \delta_{n}^{1/2} \lambda^{\delta-1}
    \left( C_{e}^{(1),0} + C_{e}^{(5),r} + C_{e}^{(5),r+\delta} \right).
\end{align*}
In a similar way we find that
\begin{align*}
    \| \mcR_{\phi_{n+1}} w_{c}^{2} \|_{C_{x}^{0}} = \| \mcR_{\phi_{n+1}} \mathcal{Q}^{\phi_{n+1}} w_{o} \|_{C_{x}^{0}} \leq C L^{r+2\delta} L_{n}^{r+3/2+\delta} \delta_{n}^{1/2} \lambda^{\delta-1}
    \left( C_{e}^{(1),0} + C_{e}^{(5),r} + C_{e}^{(5),r+\delta} \right).
\end{align*}
The remaining corrector term is handled differently. Recall that
\begin{align*}
    w_{c}^{1} := - \mathcal{Q}_{\phi_{n}} v_{n} * \chi_{\ell} = - \mathcal{Q}_{\phi_{n}} v_{n} * \chi_{\ell}^{0},
\end{align*}
where $\mathcal{Q}$ is $I - \mathcal{P}$. Recall further the following estimate from \cite[Equ. $(4.13)$]{HLP23}:
\begin{equation}\label{eq:Qphin}
    \| \mcQ_{\phi_{n}} v_{n} \|_{C_{\leq \mft_{L}} C_{x}^{\delta}} \leq C (1 + M_{v}^{2\delta}) L^{7} L_{n} \delta_{n+2}^{6/5}.
\end{equation}
Following \cite{HLP23}, we find
\begin{align*}
    \| \mcR_{\phi_{n+1}} w_{c}^{1} \|_{C^{0}_{x}} &\leq C \| \mcR_{\phi_{n+1}} w_{c}^{1} \|_{C^{\delta}_{x}} \overset{\text{Equ. \eqref{eq:antidiv-orderminusone}}}{\leq} CL^{5+4\delta} \| \mathcal{Q}_{\phi_{n}} v_{n} * \chi_{\ell}^{0} \|_{B_{\infty,\infty}^{\delta-1}} \\
    &\overset{\text{def.}}{=}  CL^{5+4\delta} \left\| \int_{0}^{\ell} \int_{\T^{3}} \left(\mathcal{Q}_{\phi_{n}} v_{n}\right)(x-y, t-s) \chi_{\ell}^{0}(y,s) dy ds \right\|_{B_{\infty,\infty}^{\delta-1}}
    \\
    &\overset{\text{Lemma \ref{lem:Besov_convolution}}}{\leq} C L^{5+4\delta} \ell \| \mathcal{Q}_{\phi_{n}} v_{n} \|_{C_{\leq \mft_{L}} L^{\infty}_{x}} \| \chi^{0}_{\ell} \|_{C_{\leq \mft_{L}} B_{p,\infty}^{2\delta - 1}} \\
    &\overset{\text{Equ. \eqref{eq:Qphin}}}{\leq} C(1 + M_{v}^{2\delta}) \ell L^{12+4\delta} L_{n} \delta_{n+2}^{6/5} \| \chi^{0}_{\ell} \|_{C_{\leq \mft_{L}}B_{p,\infty}^{2\delta - 1}} .
\end{align*}
Here, $p = p(\delta) \in (1,\infty)$ is a number very close to $1$ to be determined  in Section \ref{sec:proof_main_iter}, and $\chi^{0}_{\ell}$ denotes the (spatially) mean-free part of $\chi_{\ell}$, and we have the following estimate for its Besov norm:
\begin{align*}
    \| \chi^{0}_{\ell} \|_{B_{p,\infty}^{2\delta - 1}} &= \| \ell^{-4} \chi^{0}(\ell^{-1}x, \ell^{-1}t) \|_{B_{p,\infty}^{2\delta - 1}} = \ell^{-1} \left\| \ell^{-3} \chi(\ell^{-1}x, \ell^{-1}t) - \fint \chi(y, \ell^{-1}t) dy \right\|_{B_{p,\infty}^{2\delta - 1}} \\
    &\overset{\text{Lemma \ref{lem:scaling}}}{\leq} \ell^{-1} \ell^{d/p - d - (2\delta-1)} \| \chi \|_{B_{p,\infty}^{2\delta-1}} \leq C \ell^{d(1/p - 1)-2\delta}.
\end{align*}
Therefore,
\begin{align*}
    \| \mcR_{\phi_{n+1}} w_{c}^{1} \|_{C^{0}_{x}} \leq C (1 + M_{v}^{2\delta}) L^{12+\delta} L_{n} \delta_{n+2}^{6/5} \ell^{1 + d(1/p - 1)-2\delta}.
\end{align*}
Putting it together we find
\begin{align*}
    &\| \mcR_{\phi_{n+1}} w_{n+1} \|_{C^{0}_{x}} \\
    &\leq C L^{r+2\delta} L_{n}^{r + 3/2 + \delta} \delta_{n}^{1/2} \lambda^{\delta-1} \left( C_{e}^{(1),0} + C_{e}^{(5),r} + C_{e}^{(5),r+\delta} \right) + C(1 + M_{v}^{2\delta}) L^{12+4\delta} L_{n} \delta_{n+2}^{6/5} \ell^{1 + d(1/p - 1)-2\delta} \\
    &\leq C L^{r+2\delta} L_{n}^{r + 3/2 + \delta} \left(1 + M_{v}^{2\delta} + C_{e}^{(1),0} + C_{e}^{(5),r} + C_{e}^{(5),r+\delta}\right) \ell^{1 + d(1/p - 1)-2\delta} \delta_{n+2}^{6/5}. \qedhere
\end{align*}
\end{proof}
Now we need to consider the term
$$
\| D \left[ \mcR_{\phi_{n+1}} w_{n+1} \circ \phi_{n+1}^{-1} \right] \|_{C^{0}_{x}} = \| D[ \mcR (w_{n+1} \circ \phi_{n+1}^{-1})] \|_{C_{x}^{0}}.
$$
\begin{lem}\label{lem:Rw:C1}
    For any $j \in \{1,2,3\}$, $t \in [0,\mft_{L}]$, 
    \begin{equation}
        \| \partial_{j} \mathcal{R}(w_{n+1} \circ \phi_{n+1}^{-1} ) \|_{C_{x}^{0}} \leq C(1+ M_{v}^{2\delta} + C_{e}^{(1),\delta} + C_{e}^{(1),1} + C_{e}^{(5),1+\delta}) L^{3 + 4\delta} L_{n}^{5/2+\delta} \delta_{n}^{1/2} \lambda^{\delta} .
    \end{equation}
\end{lem}
\begin{proof}
    We have
    \begin{align*}
        \| \partial_{j} \mathcal{R}(w_{n+1} \circ \phi_{n+1}^{-1} ) \|_{C_{x}^{0}} 
        &= \| \partial_{j} \left( \mathcal{R}_{\phi_{n+1}}w_{n+1}  \circ \phi_{n+1}^{-1} \right) \|_{C_{x}^{0}} \\
        &= \left\| \sum_{l} \partial_{l} \left( \mathcal{R}_{\phi_{n+1}}w_{n+1}  \right) \cdot \partial_{l} \left( \phi_{n+1,j}^{-1} \right) \right\|_{C^{0}_{x}} \\
        &\leq CL \|  \mathcal{R}_{\phi_{n+1}}w_{n+1} \|_{C^{1+\delta}_{x}} \\
        &\overset{\text{Equ. \eqref{eq:antidiv-orderminusone-r}}}{\leq} CL^{2+2\delta} \| w_{n+1} \|_{C_{x}^{\delta}} \\
        &\overset{\text{Lemma \ref{lem:total_perturb}}}{\leq} CL^{2+2\delta}L^{1+2\delta} L_{n}^{5/2+\delta} \lambda^{\delta} \delta_{n}^{1/2}  \left(1 + M_{v}^{2\delta} +  C_e^{(1),\delta} + C_e^{(1),1} + C_e^{(5),1+\delta} \right). \qedhere
    \end{align*}
\end{proof}

\begin{proof}[Proof of Proposition \ref{prop:R_diss}]
We apply the two previous lemmas to find 
\begin{align*}
    &\| \mathcal{R}_{\phi_{n+1}} \left( (-\D)^{\alpha}_{\phi_{n+1}}w_{n+1} \right) \|_{C^{0}} \leq CL^{\delta/2}  \| \mcR_{\phi_{n+1}} w_{n+1} \|_{C^{0}}^{1-2\alpha-\delta} \| D \left[ \mcR_{\phi_{n+1}} w_{n+1} \circ \phi_{n+1}^{-1} \right] \|_{C^{0}}^{2\alpha+\delta} \\
    &\overset{\text{Lemmas \ref{lem:Rw:C0}, \ref{lem:Rw:C1}}}{\leq} CL^{\delta/2} \\
    &\qquad\qquad \cdot 
    \left( L^{r+2\delta} L_{n}^{r+3/2+\delta} \left( 1 + M_{v}^{2\delta} + C_{e}^{(1),0} + C_{e}^{(5),r} + C_{e}^{(5),r+\delta}\right) \ell^{1 + d(1/p - 1)-2\delta} \delta_{n+2}^{6/5} \right)^{1-2\alpha - \delta} 
    \\
    &\qquad\qquad \cdot 
    \left( L^{3 + 4\delta} L_{n}^{5/2+\delta} \delta_{n}^{1/2} \lambda^{\delta}(1 + M_{v}^{2\delta}+ C_{e}^{(1),\delta} + C_{e}^{(1),1} + C_{e}^{(5),1+\delta}) \right)^{2\alpha + \delta} \\
    &\leq C L^{r+1+3\delta} L_{n}^{r + 3/2 + \delta} \ell^{(1 + d(1/p - 1)-2\delta)(1-2\alpha - \delta)} \lambda^{\delta(2 \alpha + \delta)}  \delta_{n+2}^{6/5 (1 - 2\alpha - \delta)}\\
    &\qquad\qquad \cdot \left(1 + M_{v}^{2\delta} +  C_{e}^{(1),0} + C_{e}^{(5),r} + C_{e}^{(5),r+\delta} \right)^{1-2\alpha - \delta}  \left( 1 + M_{v}^{2\delta} + C_{e}^{(1),\delta} + C_{e}^{(1),1} + C_{e}^{(5),1+\delta} \right)^{2\alpha + \delta}.
\end{align*}
Now let us turn to the $C^{1}$-norm. 
We estimate, for $i \in \{1,2,3\}$,
\begin{align*}
    &\| \partial_{i} \mathcal{R}_{\phi_{n+1}} \left( (-\D)^{\alpha}_{\phi_{n+1}} w_{n+1} \right) \|_{C^{0}_{x}} = \left\| \partial_{i} \left\{ \mathcal{R} [(-\D)^{\alpha} (w_{n+1} \circ \phi_{n+1}^{-1})] \circ \phi_{n+1} \right\} \right\|_{C^{0}_{x}} \\
    &= \left\| \sum_{j} \partial_{j} \mathcal{R} [(-\D)^{\alpha} (w_{n+1} \circ \phi_{n+1}^{-1})] \circ \phi_{n+1} \cdot \partial_{j} \phi^{i}_{n+1} \right\|_{C^{0}_{x}} \\
    &{\leq} C L  \sum_{j}  \left\| \partial_{j} \mathcal{R} (-\D)^{\alpha} (w_{n+1} \circ \phi_{n+1}^{-1}) \right\|_{C_{x}^{0}} \\
    &=  C L  \sum_{j}  \left\| (-\D)^{\alpha} \partial_{j} \mathcal{R}  (w_{n+1} \circ \phi_{n+1}^{-1}) \right\|_{C_{x}^{0}} \\
    &\overset{\text{Theorem \ref{thm:fract_Lap_Holder}}}{\leq} C L \sum_{j} [ \partial_{j} \mathcal{R}  (w_{n+1} \circ \phi_{n+1}^{-1})]_{2\alpha + \delta} \\
    &\overset{\text{interpolation}}{\leq} CL \sum_{j} \| \partial_{j} \mathcal{R}  (w_{n+1} \circ \phi_{n+1}^{-1}) \|_{C^{0}_{x}}^{1 - (2\alpha + \delta)} \| D \partial_{j} \mathcal{R}  (w_{n+1} \circ \phi_{n+1}^{-1}) \|_{C^{0}}^{2\alpha + \delta}.
\end{align*}
The first term (to a different power) has already been estimated above in Lemma \ref{lem:Rw:C1}.
For the second term, we proceed as follows: For $i \in \{1,2,3\}$ we find, using the chain and product rule,
\begin{align*}
    &\left\| \partial_{i} \partial_{j} \mathcal{R}  (w_{n+1} \circ \phi_{n+1}^{-1}) \right\|_{C^{0}_{x}} = \left\| \partial_{i} \partial_{j} \left[ (\mathcal{R}_{\phi_{n+1}}  w_{n+1}) \circ \phi_{n+1}^{-1} \right] \right\|_{C^{0}_{x}} \\
    &= \left\| \partial_{i} \sum_{l} \partial_{l} (\mathcal{R}_{\phi_{n+1}}  w_{n+1}) \circ \phi_{n+1}^{-1} \cdot  \partial_{l} \phi_{n+1,j}^{-1} \right\|_{C^{0}_{x}} \\
    &= \left\| \sum_{k,l} \partial_{k} \partial_{l} (\mathcal{R}_{\phi_{n+1}}  w_{n+1}) \circ \phi_{n+1}^{-1} \cdot \partial_{k} \phi_{n+1,i}^{-1} \partial_{l} \phi_{n+1,j}^{-1} 
    + \sum_{l} \partial_{l} (\mathcal{R}_{\phi_{n+1}}  w_{n+1}) \circ \phi_{n+1}^{-1} \cdot \partial_{i} \partial_{l} \phi_{n+1,j}^{-1}\right\|_{C^{0}_{x}} \\
    &\leq \left( C L^{2}  \sum_{k,l} \| \partial_{k}\partial_{l} \mathcal{R}_{\phi_{n+1}}  w_{n+1} \|_{C_{x}^{0}} + \sum_{l} C L  \| \partial_{l} \mathcal{R}_{\phi_{n+1}}  w_{n+1} \|_{C^{0}_{x}} \right) \\
    &\leq C L^{2}\| \mathcal{R}_{\phi_{n+1}}  w_{n+1} \|_{C^{2+\delta}_{x}}.
\end{align*}
Now, an application of Equ. \eqref{eq:antidiv-orderminusone-r} yields
\begin{align*}
    &\left\| \partial_{i} \partial_{j} \mathcal{R}  (w_{n+1} \circ \phi_{n+1}^{-1}) \right\|_{C^{0}_{x}}\\ 
    &\leq C L^{5+2\delta}   \|  w_{n+1} \|_{C^{1+\delta}_{x}} 
    \\
    &\overset{\text{Lemma \ref{lem:total_perturb}}}{\leq} C L^{5+2\delta}  L^{3+2\delta} L_{n}^{7/2+\delta} \delta_{n}^{1/2} \lambda^{1+2\delta}  \left(1 + M_{v}^{2\delta} + C_e^{(1),1} + (C_e^{(5),2})^{\delta}(C_e^{(1),1})^{1-\delta}  + C_e^{(5),2+\delta} \right)  \\
    &\leq C L^{8 + 4\delta} L_{n}^{7/2+\delta} \delta_{n}^{1/2} \lambda^{1+2\delta} \left( 1 + M_{v}^{2\delta} + C_e^{(1),1} + (C_e^{(5),2})^{\delta}(C_e^{(1),1})^{1-\delta} + C_e^{(5),2+\delta} \right).
\end{align*}
The above estimates are then combined to give 
\begin{align*}
    &\| \partial_{i} \mathcal{R}_{\phi_{n+1}} (-\D)^{\alpha}_{\phi_{n+1}} w_{n+1} \|_{C^{0}_{x}}  \leq CL \sum_{j} \| \partial_{j} \mathcal{R}  (w_{n+1} \circ \phi_{n+1}^{-1}) \|_{C^{0}_{x}}^{1 - (2\alpha + \delta)} \| D \partial_{j} \mathcal{R}  (w_{n+1} \circ \phi_{n+1}^{-1}) \|_{C^{0}}^{2\alpha + \delta} \\
    &\leq CL \left(L^{3 + 4\delta} L_{n}^{5/2+\delta} \delta_{n}^{1/2} \lambda^{\delta} \right)^{1-2\alpha-\delta} \left( 1 + M_{v}^{2\delta} + + C_{e}^{(1),\delta} + C_{e}^{(1),1} + C_{e}^{(5),1+\delta} \right)^{1-2\alpha-\delta} \\
    &\qquad \cdot \left( L^{8 + 4\delta} L_{n}^{7/2+\delta}\delta_{n}^{1/2} \lambda^{1+\delta} \right)^{2\alpha+\delta} \left(1 + M_{v}^{2\delta} + +  C_e^{(1),1} + (C_e^{(5),2})^{\delta}(C_e^{(1),1})^{1-\delta} + C_e^{(5),2+\delta} \right) ^{2\alpha + \delta} \\
    &\leq CL^{9+4\delta} L_{n}^{7/2+\delta} \delta_{n}^{1/2} \lambda^{2\alpha + 3\delta}  \left( 1 + M_{v}^{2\delta} + + C_{e}^{(1),\delta} + C_{e}^{(1),1} + C_{e}^{(5),1+\delta} \right)^{1-2\alpha-\delta} \\
    &\qquad \cdot \left(1 + M_{v}^{2\delta} + +  C_e^{(1),1} + (C_e^{(5),2})^{\delta}(C_e^{(1),1})^{1-\delta} + C_e^{(5),2+\delta} \right) ^{2\alpha + \delta}.  \qedhere
\end{align*}
\end{proof}

\subsection{Estimating the divergence}\label{ssec:pf_divergence}
 
\begin{proof}[Proof of Proposition \ref{prop:divergence}]
 We proceed as in \cite[Section 4.6]{HLP23} and decompose
\begin{align*}
    \mathrm{div}_{\phi_{n+1}} v_{n+1} &= \mathrm{div}_{\phi_{n+1}} v_{\ell} - \left( \mathrm{div}_{\phi_{n+1}} v_{n} \right) * \chi_{\ell} \\
    &\quad + \left( \mathrm{div}_{\phi_{n+1}} v_{n} \right) * \chi_{\ell} - \left( \mathrm{div}_{\phi_{n}} v_{n} \right) * \chi_{\ell} \\
    &\quad + \left( \mathrm{div}_{\phi_{n+1}} \mathcal{Q}_{\phi_{n}} v_{n} \right) * \chi_{\ell} - \mathrm{div}_{\phi_{n}} \left( \left( \mathcal{Q}_{\phi_{n}} v_{n} \right) * \chi_{\ell} \right) \\
    &\quad + \mathrm{div}_{\phi_{n}} \left( \left( \mathcal{Q}_{\phi_{n}} v_{n} \right) * \chi_{\ell} \right) - \mathrm{div}_{\phi_{n+1}} \left( \left( \mathcal{Q}_{\phi_{n}} v_{n} \right) * \chi_{\ell} \right).
\end{align*}
The first and third line are easily estimated using mollification estimates: applying \cite[Lemma 4.4]{HLP23} and Equ. \eqref{eq:QSchauder}, we find\footnote{Note that for $n=0$, $v_{n} = 0$, and for $n > 0$, we know that $v_{n} = v_{n-1} * \chi_{\ell} + w_{o,n-1} + w_{c,n-1}$, and we can apply mollification estimates for the first term and our estimates for the $C_{x}^{1+\delta}$ norms for the second and third term to get the estimate $\| v_{n} \|_{C_{\leq \mft_{L}} C_{x}^{1+\delta}} \leq CL_{n} D_{n}^{1+2\delta}$.}
\begin{align*}
    \| \mathrm{div}_{\phi_{n+1}} v_{\ell} - \left( \mathrm{div}_{\phi_{n+1}} v_{n} \right) * \chi_{\ell} \|_{C_{\leq \mft_{L}} C_{x}} &\leq C L^{2} \| v_{n} \|_{C_{\leq \mft_{L}} C_{x}^{1}} \ell^{\gamma} \leq C L^{2} L_{n} D_{n} \ell^{\gamma}, \\
    \| \left( \mathrm{div}_{\phi_{n+1}} \mathcal{Q}_{\phi_{n}} v_{n} \right) * \chi_{\ell} - \mathrm{div}_{\phi_{n}} \left( \left( \mathcal{Q}_{\phi_{n}} v_{n} \right) * \chi_{\ell} \right) \|_{C_{\leq \mft_{L}} C_{x}} &\leq  C L^{2} \| \mathcal{Q}_{\phi_{n}} v_{n} \|_{C_{\leq \mft_{L}} C_{x}^{1}} \ell^{\gamma} \\
    &\leq C L^{4+2\delta} \| v_{n} \|_{C_{\leq \mft_{L}} C_{x}^{1+\delta}}  \ell^{\gamma} \\
    &\leq C L^{4+2\delta} L_{n} D_{n}^{1+2\delta} \ell^{\gamma}.
\end{align*}
The second line is estimated using \cite[Lemma 4.7]{HLP23} and Lemma \ref{lem:Besov_convolution} and Equ. \eqref{eq:vn_interpol}
\begin{align*}
    \| \left( \mathrm{div}_{\phi_{n+1}} v_{n} \right) * \chi_{\ell} - \left( \mathrm{div}_{\phi_{n}} v_{n} \right) * \chi_{\ell} \|_{C_{\leq \mft_{L}} B_{\infty,\infty}^{-1}} &\leq C L^{3} \| v_{n} \|_{C_{\leq \mft_{L}} C_{x}^{\delta}} (n+1) \vs_{n}^{\gamma'} \\
    &\leq C (M_{v} + M_{v}^{1-\delta}) L^{3}L_{n} D_{n}^{\delta}.
\end{align*}
Finally, the fourth line follows in a similar way to give
\begin{align*}
    &\| \mathrm{div}_{\phi_{n}} \left( \left( \mathcal{Q}_{\phi_{n}} v_{n} \right) * \chi_{\ell} \right) - \mathrm{div}_{\phi_{n+1}} \left( \left( \mathcal{Q}_{\phi_{n}} v_{n} \right) * \chi_{\ell} \right) \|_{C_{\leq \mft_{L}} B_{\infty,\infty}^{-1}} \\
    &\leq CL^{3} \| (\mathcal{Q}_{\phi_{n}} v_{n}) * \chi_{\ell} \|_{C_{\leq \mft_{L}} C_{x}^{\delta}} (n+1) \vs_{n}^{\gamma'} \\
    &\leq C (1 + M_{v}^{2\delta}) L^{10} L_{n} \delta_{n+2}^{6/5} (n+1) \vs_{n}^{\gamma'}.
\end{align*}
Combining the four estimates yields the desired result.
\end{proof}

\subsection{Estimating the pressure}\label{ssec:pf_pressure}
\begin{proof}[Proof of Proposition \ref{prop:it_pres}]

We have
\begin{align*}
	\|q_{n+1}-q_n\|_{C_{\leq \mft_{L}} C_{x}} &\leq \| q_{n+1} - q_{\ell} \|_{C_{\leq \mft_{L}} C_{x}} + \|q_{\ell} -q_n\|_{C_{\leq \mft_{L}} C_{x}}\\
	&\leq C\left({L_n} \bar{e} \delta_n + {L_n} \delta_{n+1} \right) + \ell {L_n}D_n\\
	&\leq C\left(1+ \eta + \bar{e}\right) {L_n}\delta_n =: {M_q L_n} \delta_n.
\end{align*}
Furthermore, using
\begin{align*}
	\| \partial_t |w_o|^2 \|_{C_{\leq \mft_{L}} C_{x}} &\leq C\| w_o\|_{C_{\leq \mft_{L}} C_{x}} \|\partial_t w_o\|_{C_{\leq \mft_{L}} C_{x}} \leq C \sqrt{\bar{e}}C_e^{\partial_t w_o} {L^{3(1+\delta)}L_n^5}  \lambda^{1+\delta}\delta_n,
\end{align*}
we obtain for $t \leq \mft_{L}$
\begin{align*}
	\|q_{n+1}-q_n\|_{C_{t,x}^1} &\leq \| q_{n+1} - q_{\ell} \|_{C_{t,x}^1} + \|q_{\ell} -q_n\|_{C_{t,x}^1}\\
	&\leq C\left( \| |w_o|^2 \|_{C_t^1C_x} + \| |w_o|^2 \|_{C_tC_x^1} + \| \tilde{\rho} \|_{C_{t,x}^1}\right) + \|q_{\ell} -q_n\|_{C_{t,x}^1}\\
	&\leq C\left( C_e^{\partial_t w_o} {L^{3(1+\delta)}L_n^5} \sqrt{\bar{e}} \lambda^{1+\delta}\delta_n \right.\\
	&\qquad \left.+ {L^{2(1+\delta)}}{L_n^{3 + 2\delta}}\sqrt{\bar{e}}(C_e^{(1),\delta} + (C_e^{(5),2})^{\delta}(C_e^{(1),1})^{1-\delta}) \lambda^{1+\delta}\delta_n + {L_n}\eta \delta_{n+1}\ell^{-1} \right) + {L_n}D_n\\
	&\leq C{L^{3(1+\delta)}L_n^5}\left(\sqrt{\bar{e}} \left(C_e^{\partial_t w_o} + C_e^{(1),\delta} + (C_e^{(5),2})^{\delta}(C_e^{(1),1})^{1-\delta}\right) +\eta + 1\right)\lambda^{1+\delta} \delta_n. \qedhere
\end{align*}
\end{proof}

\subsection{Estimating the kinetic energy}\label{ssec:pf_energy} 
\begin{proof}[Proof of Proposition \ref{prop:it_energy}]

Recall that in this section $L = 1$. 
    We first observe that
	\begin{align*}
		&\left| e(t) (1-\delta_{n+1}) - \int_{\mathbb{T}^3} |v_{n+1}(t,x)|^2 {\rm d}x \right|\\
		&\leq \left| \int_{\mathbb{T}^3} \left(|v_{n+1}|^2 - |v_{\ell}|^2 - |w_o|^2\right)(t,x) {\rm d}x\right| + 3\left|\int_{\mathbb{T}^3} \tilde{\rho}_{\ell}(t,x){\rm d}x\right| \\
		&\qquad + \sum_{1\leq |k|\leq 2\lambda_0} \left|\int_{\mathbb{T}^3}{\rm tr}(U_k(t,x)) e^{{\rm i}\lambda k \cdot \phi_{n+1}(t,x)}{\rm d}x\right|.
	\end{align*}
    Let us start with the first term. Using
	\[|v_{n+1}|^2 - |v_{\ell}|^2 - |w_o|^2 = |w_c|^2 + 2 v_{\ell} \cdot w_o + 2 v_{\ell}\cdot w_c + 2 w_o \cdot w_c,\]
	we find
	\begin{align*}
		\left| \int_{\mathbb{T}^3} \left(|v_{n+1}|^2 - |v_{\ell}|^2 - |w_o|^2\right)(t,x) {\rm d}x\right|
		&\leq \|w_c\|_0(\|w_c\|_0 + 2\|w_o\|_0) + 2 \left| \int_{\mathbb{T}^3} v_{\ell} \cdot w_o{\rm d}x\right| \\
		&\qquad + 2 \left|\int_{\mathbb{T}^3} v_{\ell}\cdot w_c {\rm d}x\right|.
	\end{align*}
	We use the parameter relation Equ. \eqref{eq:mulambda6-5} to simplify 
	\begin{align*}
		\|w_c\|_{C_{\leq \mft} C_{x}} &\leq C (1 + M_v^{2\delta})\delta_{n+2}^{\frac{6}{5}} + C(C_e^{(1),1} + C_e^{(5),1+\delta})\lambda^{\delta-1}\mu\varsigma_{n+1}^{\gamma-1} \delta_n^{\frac{1}{2}} \\
		&\overset{\text{Equ.  \eqref{eq:mulambda6-5}}}{\leq} C \left(1 + M_v^{2\delta}+ C_e^{(1),1} + C_e^{(5),1+\delta}\right)\lambda^{\delta}\delta_{n+2}^{\frac{6}{5}}.
    \end{align*}
    Now, using Equ. \eqref{eq:SPL_integral} we find
    \begin{align*}
		\left| \int_{\mathbb{T}^3} v_{\ell} \cdot w_o{\rm d}x\right| &\leq C \lambda^{-1} \sum [v_{\ell} \cdot a_k]_1 \leq C \lambda^{-1} \sum \|v_{\ell}\|_0 [a_k]_1 + [v_{\ell}]_1\|a_k\|_0\\
		&\leq C\lambda^{-1}\left( M_{v} C_{e}^{(1),1} \mu \varsigma_{n+1}^{\gamma-1} \delta_{n}^{1/2} +  D_{n} \sqrt{\bar{e}}  \right) \delta_{n}^{1/2} \\
		&\leq C\left( M_{v} C_{e}^{(1),1} + \sqrt{\bar{e}} \right) \left( \frac{\mu \varsigma_{n+1}^{\gamma-1} }{\lambda} + \frac{D_{n}}{\lambda} \right)\delta_n^{\frac{1}{2}} \\
		&\overset{\text{Equ. \eqref{eq:Dvsest}, \eqref{eq:mulambda6-5}}}{\leq} C\left( M_{v} C_{e}^{(1),1} + \sqrt{\bar{e}} \right)\delta_{n+2}^{6/5}\delta_n^{\frac{1}{2}}.
    \end{align*}
    We use the Cauchy--Schwarz inequality to find
    \begin{align*}
		\left|\int_{\mathbb{T}^3} v_{\ell}\cdot w_c {\rm d}x\right|&\leq C \sqrt{\bar{e}} \|w_c\|_{C_{\leq \mft} C_{x}} \leq C\sqrt{\bar{e}} \left(1 + M_v^{2\delta}+ C_e^{(1),1} + C_e^{(5),1+\delta}\right)\lambda^{\delta}\delta_{n+2}^{\frac{6}{5}}.
	\end{align*}
	Hence
	\begin{align*}
		&\left| \int_{\mathbb{T}^3} \left(|v_{n+1}|^2 - |v_{\ell}|^2 - |w_o|^2\right)(t,x) {\rm d}x\right|\\
		&\leq C\left(1 + M_v^{2\delta}+ C_e^{(1),1} + C_e^{(5),1+\delta} \right)^2\lambda^{2\delta}\delta_{n+2}^{\frac{12}{5}}\\
		&\quad + C \left(1 + M_v^{2\delta}+ C_e^{(1),1} + C_e^{(5),1+\delta} \right)C_e^{(1),\delta}\lambda^{2\delta} \delta_n\delta_{n+2}^{\frac{6}{5}}\\
		&\quad + C (M_{v} C_{e}^{(1),1} + \sqrt{\bar{e}} ) \delta_{n+2}^{\frac{6}{5}} \delta_n^{\frac{1}{2}}\\
		&\quad + C\sqrt{\bar{e}} \left(1 + M_v^{2\delta}+ C_e^{(1),1} + C_e^{(5),1+\delta}\right)\lambda^{\delta}\delta_{n+2}^{\frac{6}{5}}\\
		&\leq C\left(\left(1 + {M_v^{2\delta}} + C_e^{(1),1} + C_e^{(5),1+\delta} \right)^2 + M_{v} C_{e}^{(1),1} + \sqrt{\bar{e}} \right)\lambda^{2\delta} \delta_{n+2}^{\frac{6}{5}}.
	\end{align*}
	Next, the second term can be estimated by
	\[3\left|\int_{\mathbb{T}^3} \tilde{\rho}_{\ell}(t,x){\rm d}x\right| \leq \frac{9\eta}{r_0} \delta_{n+1}.\]
	The third term can be dealt with by another application of the stationary phase lemma, Equ. \eqref{eq:SPL_integral}
	\[\sum_{1\leq |k|\leq 2\lambda_0} \left|\int_{\mathbb{T}^3}{\rm tr}(U_k(t,x)) e^{{\rm i}\lambda k \cdot \phi_{n+1}(t,x)}{\rm d}x\right| \leq C \sqrt{\bar{e}} C_e^{(1),1} \frac{ \mu \varsigma_{n+1}^{\gamma -1} }{\lambda}\delta_n \overset{\text{Equ. \eqref{eq:mulambda6-5}}}{\leq} C \sqrt{\bar{e}} C_e^{(1),1} \delta_{n+2}^{6/5} \delta_n. \] 
	Thus, in total we find
	\begin{align*}
		&\left| e(t) (1-\delta_{n+1}) - \int_{\mathbb{T}^3} |v_{n+1}(t,x)|^2 {\rm d}x \right|\\
		&\leq C\left(\left(1 + {M_v^{2\delta}} + C_e^{(1),1} + C_e^{(5),1+\delta} \right)^2 + M_{v} C_{e}^{(1),1} + \sqrt{\bar{e}}(1+C_{e}^{(1),1}) \right)\lambda^{2\delta} \delta_{n+2}^{\frac{6}{5}} + \frac{9\eta}{r_0} \delta_{n+1}. \qedhere
	\end{align*}
\end{proof}

\subsection{Estimates for energy-dependent constants}\label{ssec:pf_CR1} 

\begin{proof}[Proof of Lemma \ref{lem:CR0}]
Let us study the first two and the last two lines of the table first. We see that by Equ. \eqref{eq:mu6-5}, \eqref{eq:mulambda6-5}
\begin{align*}
    &C^{\rm tra,0}_{e} {L^{r+2\delta}L_n^{2(r+\delta)+\frac{5}{2}}}\lambda^{\delta}\mu^{-1} \delta_n^{\frac{1}{2}} + C_{e}^{\rm osc,0} {L^{r+2\delta+1}}{L_n^{2(r+\delta) +5}}\lambda^{\delta-1}\mu \varsigma_{n+1}^{\gamma-1} \delta_n \\
    &\leq \left( C^{\rm tra,0}_{e} + C^{\rm osc,0}_{e} \right) L^{r+2\delta+1} L_{n}^{2(r+\delta)+5} \lambda^{\delta} \delta_{n}^{1/2} \delta_{n+2}^{6/5}.
\end{align*}
This yields immediately that 
\begin{align*}
    &\| \mathring{R}^{\rm tra} \|_{C_{\leq \mft_{L}} C_{x}^{0}} + \| \mathring{R}^{\rm osc} \|_{C_{\leq \mft_{L}} C_{x}^{0}} + \| \mathring{R}^{\rm comp} \|_{C_{\leq \mft_{L}} C_{x}^{0}} + \| \mathring{R}^{\rm diss} \|_{C_{\leq \mft_{L}} C_{x}^{0}} \\
    &\leq \left( C^{\rm tra,0}_{e} + C^{\rm osc,0}_{e} + C^{\rm comp,0}_{e} + C^{\rm diss,0}_{e} \right) L^{r+2\delta+1} L_{n}^{2(r+\delta)+5} \lambda^{\delta} \delta_{n}^{1/2} \delta_{n+2}^{6/5} \\
    &\leq C_{e}^{\mathring{R},0} L^{r+2\delta+1} L_{n}^{2(r+\delta)+5} \lambda^{\delta} \delta_{n}^{1/2} \delta_{n+2}^{6/5}.
\end{align*}
Now we analyse $\mathring{R}^{\rm moll}$. We immediately see that, by definition
\begin{align*}
    \ell^{\gamma} D_{n} = \frac{c_{n,\ell}}{C_{\ell}} \delta_{n+3}^{4/3} \leq 2\delta_{n+3}^{4/3}.
\end{align*}
The other term is much less small:
\begin{align*}
    \left( \ell D_{n} \right)^{\beta} = \left( \ell^{1-\gamma} \ell^{\gamma} D_{n} \right)^{\beta} \overset{\gamma \in (0,1/2)}{\leq} \left( \ell^{\gamma} 2\delta_{n+3}^{4/3} \right)^{\beta} \leq 2 \delta_{n+3}^{2\beta \cdot 4/3}.
\end{align*}
Finally, for $\mathring{R}^{\rm flow}$, we see that, using $\gamma' \in (\gamma_{*},\gamma)$ hence $\frac{\gamma'}{\gamma_{*}} > 1$, as well as $\beta = \frac{\gamma}{48}$,
\begin{align*}
    \varsigma_n^{\frac{1}{48} \gamma'} = \left( \varsigma_n^{\gamma_{*}} \right)^{\frac{1}{48} \frac{\gamma'}{\gamma_{*}}} \overset{\text{def.}}{=} \left( \frac{1}{C_{\vs}} \frac{\delta_{n+3}^{4/3}}{n+1} \right)^{\frac{1}{48} \frac{\gamma'}{\gamma_{*}}} \leq \delta_{n+3}^{4/3 \cdot \frac{1}{48} \frac{\gamma'}{\gamma_{*}}} \leq \delta_{n+3}^{4/3 \cdot \frac{1}{48}} \overset{\gamma \in (0,1/2)}{\leq} \delta_{n+3}^{4/3 \cdot 2\gamma \frac{1}{48}} = \delta_{n+3}^{2\beta \cdot 4/3}.
\end{align*}
This implies
\begin{align*}
    \| \mathring{R}^{\rm flow} \|_{C_{\leq \mft_{L}} C_{x}^{0}} + \| \mathring{R}^{\rm moll} \|_{C_{\leq \mft_{L}} C_{x}^{0}} &\leq \left( C^{\rm flow,0}_{e} + C^{\rm moll,0}_{e} \right) L^{r+2+ \delta} L_{n}^{r+7/2} \ell^{-\delta-2\alpha}  (n+1) \delta_{n+3}^{8/3 \cdot \beta} \\
    &\leq C_{e}^{\mathring{R},0} L^{r+2+ \delta} L_{n}^{r+7/2} \ell^{-\delta-2\alpha}  (n+1) \delta_{n+3}^{8/3 \cdot \beta}.
\end{align*}
Together with assumptions Equ. \eqref{eq:choice_m_2}, Equ. \eqref{eq:choice_m_3}, this implies the claim.

Now let us turn to estimating the constant $C_{e}^{\mathring{R},0}$ itself.

We recall
\begin{align*}
	C_e^{(1),1} &:= C\left(\sqrt{\bar{e}} + \underline{e}^{-\frac{1}{2}} + \sqrt{\bar{e}}\underline{e}^{-1}\right),\\
	C_e^{(4),0} &:= C \left(\frac{\eta(1+\sqrt{\bar{e}}) + |e|_{C^1}}{\sqrt{\underline{e}}} + \sqrt{\bar{e}} \left(1 + \frac{\eta}{\underline{e}}\left(1 + \frac{\eta}{\underline{e}}(1+\sqrt{\bar{e}} + \eta^{-1}|e|_{C^1})\right)\right)\right),\\
	C_e^{(5),r} &:= C\sqrt{\bar{e}}\left(1+ \eta C_e^{(r)} +\eta^2 \sum_{j=1}^{r-1}C_e^{(j)}C_e^{(r-j)}\right), \\
	C_e^{(8), r} &:= C_e^{\partial_s (1),r} + C_e^{\partial_s (2),r} + C_e^{\partial_s (3),r},\\
	C_e^{\partial_s (1),r}&= C \left(\frac{\eta(1+\sqrt{\bar{e}}) + |e|_{C^1}}{\sqrt{\underline{e}}} + C_e^{\partial_s\sqrt{\rho},r}\right)\left(1+ \eta C_e^{(r)}\right),\\
	C_e^{\partial_s (2),r} &= C\sqrt{\bar{e}}\left(1 + \eta C_e^{(r)}\right)\left(\left(1+\frac{\eta}{\underline{e}}(1+\sqrt{\bar{e}} + \eta^{-1}|e|_{C^1})\right)\frac{\eta}{\underline{e}} + C_e^{\partial_s\Gamma, r} \right),\\
	C_e^{\partial_s (3),r} &= C\sqrt{\bar{e}}\left(1 + \eta C_e^{(r)} + \eta^2 \sum_{j=1}^{r-1} C_e^{(j)}C_e^{(r-j)} \right),\\
	C_e^{\partial_s \sqrt{\rho}, r} &= C \frac{\eta}{\sqrt{\underline{e}}}\left(1+ C_e^{(r)}(\eta(1+\sqrt{\bar{e}})+|e|_{C^1})\right),\\
\end{align*}
as well as
\begin{align*}
    &C_e^{\partial_s \Gamma,r} = \\
    &C\frac{\eta}{\underline{e}}\left(1+ (1+\eta)C_e^{(r)} + \eta^2 \sum_{i=1}^{r-1}C_e^{(i)}C_e^{(r-i)} + \frac{\eta}{\underline{e}} (1 + \sqrt{\bar{e}} + \eta^{-1}|e|_{C^1})\left((1+\eta) C_e^{(r)} + \eta \sum_{i=1}^{r-1}C_e^{(i)}C_e^{(r-i)} \right)\right)
\end{align*}
and
\[C_e^{(r)}  \leq C \frac{1+\bar{e}}{\bar{e}} \left(\frac{\bar{e}}{\underline{e}}\right)^r,\]
and we recall that ${M_q} = C\left(1+ \eta + \bar{e}\right)$. Next we will simplify the above constants. To this end, we obtain further upper bounds to be inserted back into the analysis. First of all
\[C_e^{(1),1} = C \left(\sqrt{\bar{e}} + \underline{e}^{-\frac{1}{2}} + \sqrt{\bar{e}}\underline{e}^{-1}\right) = C\sqrt{\bar{e}}\left(1+\frac{1}{\sqrt{\bar{e}\underline{e}}} + \frac{1}{\underline{e}}\right) \leq C \sqrt{\bar{e}}\left(1+\frac{1}{\underline{e}}\right) = C \sqrt{\bar{e}}\frac{1+\underline{e}}{\underline{e}}.\]
Further using $\eta \leq C \frac{\underline{e}}{1+\sqrt{\bar{e}}}$ we obtain
\begin{align*}
	C_e^{(4),0} &\leq C \left(\sqrt{\underline{e}}+ \frac{|e|_{C^1}}{\sqrt{\underline{e}}} + \sqrt{\bar{e}} \left(1 + \frac{|e|_{C^1}}{(1+\sqrt{\bar{e}})\underline{e}}\right)\right) \leq C \left( \sqrt{\bar{e}}+ \frac{|e|_{C^1}}{\sqrt{\underline{e}}}\left( 1 + \frac{1}{1+\sqrt{\bar{e}}}\sqrt{\frac{\bar{e}}{\underline{e}}}\right)\right)\\
	&\leq C\sqrt{\bar{e}}\left(1+ \frac{|e|_{C^1}}{\underline{e}}\right).
\end{align*}
Next, due to the above representation of $C_e^{(r)}$ we conclude
\begin{align*}
	C_e^{(5),r} &\leq C\sqrt{\bar{e}}\left(1+ \eta \frac{1+\bar{e}}{\bar{e}} \left(\frac{\bar{e}}{\underline{e}}\right)^r +\eta^2 \left(\frac{1+\bar{e}}{\bar{e}}\right)^2 \left(\frac{\bar{e}}{\underline{e}}\right)^r\right)\\
	&\leq C\sqrt{\bar{e}}\left(1+ \frac{1+\bar{e}}{1+\sqrt{\bar{e}}} \left(\frac{\bar{e}}{\underline{e}}\right)^{r-1} +\left(\frac{1+\bar{e}}{1+\sqrt{\bar{e}}}\right)^2 \left(\frac{\bar{e}}{\underline{e}}\right)^{r-2}\right)\\
	&\leq C\sqrt{\bar{e}}\left(1+ \frac{1+\bar{e}}{1+\sqrt{\bar{e}}} \left(\frac{\bar{e}}{\underline{e}}\right)^{r-1} + (1+\bar{e}) \left(\frac{\bar{e}}{\underline{e}}\right)^{r-2}\right).
\end{align*}
Furthermore, we estimate
\begin{align*}
	C_e^{(5),1+\delta} &:= \left(C_e^{(5),2}\right)^{\delta}\left(C_e^{(1),1}\right)^{1-\delta} \leq C\sqrt{\bar{e}} \left(1+ \frac{1+\bar{e}}{1+\sqrt{\bar{e}}} \frac{\bar{e}}{\underline{e}} + (1+\bar{e})\frac{\bar{e}}{\underline{e}}\right)^{\delta}\left(\frac{1+\underline{e}}{\underline{e}}\right)^{1-\delta}\\
	&\leq C\sqrt{\bar{e}}\frac{1+\bar{e}}{\underline{e}}\left(1+ \bar{e}+ \frac{\bar{e}}{1+\sqrt{\bar{e}}}\right),
\end{align*}
and thus
\[\left(C_e^{(5),2}\right)^{\delta}\left(\sqrt{\bar{e}}C_e^{(1),1}\right)^{1-\delta} \leq C\sqrt{\bar{e}}^{2-\delta}\frac{1+\bar{e}}{\underline{e}}\left(1+ \bar{e}+ \frac{\bar{e}}{1+\sqrt{\bar{e}}}\right). \]
Next,
\begin{align*}
	C_e^{\partial_s \sqrt{\rho}, r} &\leq C \frac{\eta}{\sqrt{\underline{e}}}\left(1+ \frac{1+\bar{e}}{\bar{e}} \left(\frac{\bar{e}}{\underline{e}}\right)^r(\underline{e}+|e|_{C^1})\right) \leq C \sqrt{\bar{e}}\frac{1+\bar{e}}{1+\sqrt{\bar{e}}} \left(\frac{\bar{e}}{\underline{e}}\right)^r\left(1+ \frac{|e|_{C^1}}{\underline{e}}\right),\\
	C_e^{\partial_s \sqrt{\rho}, 1} &\leq C\sqrt{\bar{e}}\frac{1+\bar{e}}{1+\sqrt{\bar{e}}} \frac{\bar{e}}{\underline{e}}\left(1+ \frac{|e|_{C^1}}{\underline{e}}\right).
\end{align*}
Now, in a similar way, for $r \geq 2$, using $\left(1+\sqrt{\bar{e}}\right)^{2} \geq 1 + \bar{e}$, we may estimate
\begin{align*}
	C_e^{\partial_s \Gamma,r} &\leq C\frac{\eta}{\underline{e}}\left(1+ (1+\eta)\frac{1+\bar{e}}{\bar{e}} \left(\frac{\bar{e}}{\underline{e}}\right)^r + \eta^{2}\left(\frac{1+\bar{e}}{\bar{e}}\right)^2 \left(\frac{\bar{e}}{\underline{e}}\right)^r \right.\\
	&\qquad\qquad \left. + \frac{\eta}{\underline{e}} (1 + \sqrt{\bar{e}} + \eta^{-1}|e|_{C^1})\left((1+\eta)\frac{1+\bar{e}}{\bar{e}} \left(\frac{\bar{e}}{\underline{e}}\right)^r + \eta\left(\frac{1+\bar{e}}{\bar{e}}\right)^2 \left(\frac{\bar{e}}{\underline{e}}\right)^r \right)\right) \\
	&\leq C\left(1 + \sqrt{\underline{e}}+ \frac{|e|_{C^1}}{\underline{e}}\right)\left(\frac{1+\bar{e}}{\underline{e}}\right)^2 \left(\frac{\bar{e}}{\underline{e}}\right)^{r-2},
\end{align*}
whereas for $r=1$
\begin{align*}
	C_e^{\partial_s \Gamma,1} &\leq C\frac{\eta}{\underline{e}}\left(1 + \left(1 +\eta + \frac{|e|_{C^1}}{\underline{e}}\right)C_e^{(1)}\right) \leq C\frac{\eta}{\underline{e}}\left(1 + \left(1 +\eta + \frac{|e|_{C^1}}{\underline{e}}\right)\frac{1+\bar{e}}{\underline{e}}\right)\\
	&\leq C\left(1 + \frac{|e|_{C^1}}{\underline{e}}\right)\frac{1+\bar{e}}{\underline{e}}.
\end{align*}
Thus, we conclude for $r \geq 2$
\begin{align*}
	C_e^{\partial_s (1),r}&\leq C \left(\sqrt{\underline{e}} + \frac{|e|_{C^1}}{\sqrt{\underline{e}}} + \sqrt{\bar{e}}\frac{1+\bar{e}}{1+\sqrt{\bar{e}}} \left(\frac{\bar{e}}{\underline{e}}\right)^r\left(1+ \frac{|e|_{C^1}}{\underline{e}}\right)\right)\left(1+ \frac{1+\bar{e}}{1+\sqrt{\bar{e}}}\left(\frac{\bar{e}}{\underline{e}}\right)^{r-1}\right)\\
	&\leq C \sqrt{\bar{e}}\left(1+ \frac{|e|_{C^1}}{\underline{e}}\right)\left(1+ \frac{1+\bar{e}}{1+\sqrt{\bar{e}}}\left(\frac{\bar{e}}{\underline{e}}\right)^{r-1}\right)^2 \frac{\bar{e}}{\underline{e}},
\end{align*}
as well as
\begin{align*}
	C_e^{\partial_s (2),r} &\leq C\sqrt{\bar{e}}\left(1 + \frac{1+\bar{e}}{1+\sqrt{\bar{e}}}\left(\frac{\bar{e}}{\underline{e}}\right)^{r-1}\right)\left(1+ \sqrt{\underline{e}} + \frac{|e|_{C^1}}{\underline{e}}\right)\left(\frac{1+\bar{e}}{\underline{e}}\right)^2 \left(\frac{\bar{e}}{\underline{e}}\right)^{r-2}.
\end{align*}
On the other hand, for $r=1$
\begin{align*}
	C_e^{\partial_s (1),1}&\leq C \left(\sqrt{\underline{e}} + \frac{|e|_{C^1}}{\sqrt{\underline{e}}} + \sqrt{\bar{e}}\frac{1+\bar{e}}{1+\sqrt{\bar{e}}} \frac{\bar{e}}{\underline{e}}\left(1+ \frac{|e|_{C^1}}{\underline{e}}\right)\right)\left(1+ \frac{1+\bar{e}}{1+\sqrt{\bar{e}}}\right)\\
	&\leq C \sqrt{\bar{e}}\left(1+ \frac{|e|_{C^1}}{\underline{e}}\right)\left(1+ \frac{1+\bar{e}}{1+\sqrt{\bar{e}}}\right)^2\frac{\bar{e}}{\underline{e}},\\
	C_e^{\partial_s (2),1} &\leq C\sqrt{\bar{e}}\left(1 + \frac{|e|_{C^1}}{\underline{e}}\right)\left(1 + \frac{1+\bar{e}}{1+\sqrt{\bar{e}}}\right) \frac{1+\bar{e}}{\underline{e}}.
\end{align*}
Finally, for $r\geq 2$,
\begin{align*}
	C_e^{\partial_s (3),r} &\leq C\sqrt{\bar{e}}\left(1 + \frac{1+\bar{e}}{1+\sqrt{\bar{e}}}\left(\frac{\bar{e}}{\underline{e}}\right)^{r-1} +  \left(\frac{1+\bar{e}}{1+\sqrt{\bar{e}}}\right)^2 \left(\frac{\bar{e}}{\underline{e}}\right)^{r-2} \right)\\
	&\leq C\sqrt{\bar{e}}\left(1 + \bar{e} + \frac{1+\bar{e}}{1+\sqrt{\bar{e}}}\frac{\bar{e}}{\underline{e}} \right)\left(\frac{\bar{e}}{\underline{e}}\right)^{r-2},
\end{align*}
and for $r =1$
\begin{align*}
	C_e^{\partial_s (3),1} &\leq C\sqrt{\bar{e}}\left(1 + \frac{1+\bar{e}}{1+\sqrt{\bar{e}}}\right).
\end{align*}
The previous calculations are then combined to yield (for $r \geq 2$)
\begin{align*}
	&C_e^{(8), r} = C_e^{\partial_s (1),r} + C_e^{\partial_s (2),r} + C_e^{\partial_s (3),r}\\
	&\leq C \sqrt{\bar{e}}\left(1+\sqrt{\underline{e}}+ \frac{|e|_{C^1}}{\underline{e}}\right)\left(\left(1+ \frac{1+\bar{e}}{1+\sqrt{\bar{e}}}\left(\frac{\bar{e}}{\underline{e}}\right)^{r-1}\right)^2 \frac{\bar{e}}{\underline{e}} \right.\\
	&\hspace{2.5cm}\left.+ \left(1+ \frac{1+\bar{e}}{1+\sqrt{\bar{e}}}\left(\frac{\bar{e}}{\underline{e}}\right)^{r-1}\right)  \left(\frac{1+\bar{e}}{\underline{e}}\right)^2 \left(\frac{\bar{e}}{\underline{e}}\right)^{r-2}+ \left(1+ \bar{e} + \frac{1+\bar{e}}{1+\sqrt{\bar{e}}}\frac{\bar{e}}{\underline{e}}\right)\left(\frac{\bar{e}}{\underline{e}}\right)^{r-2} \right)\\
	&\leq C \sqrt{\bar{e}}\left(1+ \sqrt{\underline{e}} + \frac{|e|_{C^1}}{\underline{e}}\right)\left(1+ \bar{e} + \frac{1+\bar{e}}{1+\sqrt{\bar{e}}}\right)^2\left(\frac{1+\bar{e}}{\underline{e}}\right)^2 \left(\frac{\bar{e}}{\underline{e}}\right)^{2r-1},
\end{align*}
and
\begin{align*}
	C_e^{(8), 1} &:= C_e^{\partial_s (1),1} + C_e^{\partial_s (2),1} + C_e^{\partial_s (3),1} \leq C \sqrt{\bar{e}}\left(1+ \frac{|e|_{C^1}}{\underline{e}}\right)\left(1 + \frac{1+\bar{e}}{1+\sqrt{\bar{e}}}\right)^2 \frac{1+\bar{e}}{\underline{e}}.
\end{align*}
Thus
\[C_e^{(8),\delta} := \left(C_e^{(8),1}\right)^{\delta}\left(C_e^{(4),0}\right)^{1-\delta}\leq C \sqrt{\bar{e}}\left(1+ \frac{|e|_{C^1}}{\underline{e}}\right)\left(1 + \frac{1+\bar{e}}{1+\sqrt{\bar{e}}}\right)^2 \frac{1+\bar{e}}{\underline{e}}.\]
We are now in a position to estimate $C_{e}^{\mathring{R},0}$. Note that by definition, using the monotonicity of the constants $C_{e}^{(j),r}$ for $j \geq 5$, we find
\begin{align*}
    &C_{e}^{\mathring{R},0} \\
    &\leq C \left(1+ \sqrt{\bar{e}} + C_{e}^{(7),r}+ C_{e}^{(7),r+\delta} + M_{v}\left( C_{e}^{(1),1} + C_{e}^{(5),r+1} + C_{e}^{(5),r+\delta+1} \right)  + C_{e}^{(4),0} + C_{e}^{(8),r}+C_{e}^{(8),r+\delta} \right. \\
    &\qquad + \sqrt{\bar{e}}C_{e}^{(1),1} + C_{e}^{(7),r+1}+ C_{e}^{(7),r+1+\delta} + M_{v} + M_{v}^{2} + M_{q} + M_{v}^{1-\beta} + M_{v}^{1-\delta-2\alpha}  \\
    &\qquad\left. + M_{v}^{2\delta} + C_{e}^{(1),0} + C_{e}^{(1),\delta} + C_{e}^{(1),1} + C_{e}^{(5), r+\delta} + C_{e}^{\mathrm{comp},2,\delta} + C_{e}^{\mathrm{comp},3,\delta} + C_{e}^{\mathrm{comp},4,\delta} \right) \\
    &\overset{\text{monotonicity}}{\leq} C \left( 1 + M_{v}^{1-\beta} + M_{v}^{1-\delta-2\alpha} + M_{q} + \sqrt{\bar{e}}\left( 1 + C_{e}^{(1),1} \right) \right. \\
    &\qquad\qquad + C_{e}^{(1),0} + C_{e}^{(1),\delta} + C_{e}^{(1),1} + C_{e}^{(4),0}+ C_{e}^{(5),r+3} + C_{e}^{(8),r+2}  \\
    &\qquad\qquad + (1 + M_{v}^{2\delta}+M_{v} + C_{e}^{(1),1} + C_{e}^{(5),1+\delta}) \cdot (1 + M_{v}^{2\delta}+M_{v} + C_{e}^{(1),1} + C_{e}^{(5),r+2}  )
\end{align*}
Note that, by considering the two cases $\eta \leq 1$, $\eta > 1$, one finds that 
$$ 
C_{e}^{(1),0} + C_{e}^{(1),\delta} + C_{e}^{(1),1} \leq 2 \left( C_{e}^{(1),0} + C_{e}^{(1),1} \right).
$$ 
Using this and the above estimates, we find
\begin{align*}
    &C_{e}^{\mathring{R},0} \\
    &\leq C \left( 1  + M_{v}^{1-\beta} + M_{v}^{1-\delta-2\alpha} + M_{q} + \sqrt{\bar{e}}\left( 1 + \sqrt{\bar{e}}\frac{1 + \underline{e}}{\underline{e}} \right) + \sqrt{\bar{e}}\frac{1 + \underline{e}}{\underline{e}} + \sqrt{\bar{e}}\left(1 + \frac{|e|_{C^{1}}}{\underline{e}} \right) \right. \\
    &\qquad + \sqrt{\bar{e}}\left(1+ \frac{1+\bar{e}}{1+\sqrt{\bar{e}}} \left(\frac{\bar{e}}{\underline{e}}\right)^{(r+3)-1} + (1+\bar{e}) \left(\frac{\bar{e}}{\underline{e}}\right)^{(r+3)-2}\right) \\
    &\qquad + \sqrt{\bar{e}}\left(1+ \sqrt{\underline{e}} + \frac{|e|_{C^1}}{\underline{e}}\right)\left(1+ \bar{e} + \frac{1+\bar{e}}{1+\sqrt{\bar{e}}}\right)^2\left(\frac{1+\bar{e}}{\underline{e}}\right)^2 \left(\frac{\bar{e}}{\underline{e}}\right)^{2(r+2)-1} \\
    &\qquad + \left( 1 + M_{v}^{2\delta}+M_{v} + \sqrt{\bar{e}}\frac{1 + \underline{e}}{\underline{e}} +  \sqrt{\bar{e}}\frac{1+\bar{e}}{\underline{e}}  \left(1+ \bar{e}+ \frac{\bar{e}}{1+\sqrt{\bar{e}}}\right) \right) \\
    &\quad\quad \left. \cdot \left( 1 + M_{v}^{2\delta}+M_{v} + \sqrt{\bar{e}}\frac{1+\bar{e}}{\underline{e}} + \sqrt{\bar{e}}\left(1+ \frac{1+\bar{e}}{1+\sqrt{\bar{e}}} \left(\frac{\bar{e}}{\underline{e}}\right)^{(r+2)-1} + (1+\bar{e}) \left(\frac{\bar{e}}{\underline{e}}\right)^{(r+2)-2}\right)  \right) \right) \\
    &\leq \left(\frac{1+\bar{e}}{\underline{e}} \right)^{2} \left( \frac{\bar{e}}{\underline{e}} \right)^{2r+3} \\
    &\qquad \cdot C \bigg(1  + M_{v}^{1-\beta} + M_{v}^{1-\delta-2\alpha} + \sqrt{\bar{e}}\left(1+ \sqrt{\underline{e}} + \frac{|e|_{C^1}}{\underline{e}}\right)\left(1+ \bar{e} + \frac{1+\bar{e}}{1+\sqrt{\bar{e}}}\right)^2\left(\frac{1+\bar{e}}{\underline{e}}\right)^2  \\ 
    &\qquad\quad + \left( 1 + M_{v}^{2\delta}+M_{v} + \sqrt{\bar{e}}\frac{1 + \underline{e}}{\underline{e}} +  \sqrt{\bar{e}}\left(1+ \bar{e}+ \frac{\bar{e}}{1+\sqrt{\bar{e}}}\right) \right)\\
    &\qquad\qquad \cdot \left( 1 + M_{v}^{2\delta} + M_{v} + \sqrt{\bar{e}} \right)  \bigg). \qedhere 
\end{align*}
\end{proof}

\begin{proof}[Proof of Lemma \ref{lem:CR1}]
 Plugging in the expressions for all constants and simplifying a bit, we find
\begin{align*}
	C_e^{\mathring{R},1} &\leq C\left(1 + C_e^{(3),0} + C_e^{(3),\delta} + \left(1+M_{v}+C_e^{(1),1}+C_e^{(5),1+\delta}\right){M_v} + C_e^{(4),0}+C_e^{(8),\delta} \right.\\
	&\qquad + \sqrt{\bar{e}}C_e^{(1),1} + \left(C_e^{(7),2}\right)^{\delta}\left(\sqrt{\bar{e}}C_e^{(1),1}\right)^{1-\delta}\\
	&\qquad + \eta + M_{v}^{1-\beta} + M_{v}^{1-\delta-2\alpha} + {M_q} + C_e^{(1),\delta}\\
	&\qquad + {M_v^{2\delta}}+ C_e^{{\rm comp},2,1+\delta} + C_e^{{\rm comp},3,1+\delta} + C_e^{{\rm comp},4,1+\delta}\\
	&\qquad  + \left( 1 + {M_{v}^{2\delta}} + C_{e}^{(1),\delta} + C_{e}^{(1),1} + C_{e}^{(5),1+\delta} \right)^{1-2{\alpha}-\delta} \\
	&\qquad\quad \left. \cdot \left(1 + {M_{v}^{2\delta}} + C_e^{(1),1} + C_e^{(5),1+\delta} + C_e^{(5),2+\delta} \right)^{2{\alpha} + \delta}\right).
\end{align*}
Now we recall that
\begin{align*}
	C_e^{{\rm comp},2, 1+\delta} &= C\left(C_e^{(1),1} + C_e^{(5),2+\delta} + C_e^{(8),1+\delta}\right),\\
	C_e^{{\rm comp},3,1+\delta} &= C\left(\left(1+{M_v} + {M_v^{2\delta}} + C_e^{(1),1} + C_e^{(5),1+\delta}\right)\left(1+{M_v^{2\delta}}+ C_e^{(1),1} + C_e^{(5),2+\delta}\right)\right.\\
	&\qquad\quad \left.+\left(1+ C_e^{(1),\delta} + C_e^{(5),1+\delta}\right)\left(1+{M_v^{2\delta}}+C_e^{(1),1} + C_e^{(5),1+\delta}\right)\right),\\
	C_e^{{\rm comp},4,1+\delta} &= C\left({M_v}\left(C_e^{(5),1+\delta}  +C_e^{(1),1}+\sqrt{\bar{e}}\right) + \sqrt{\bar{e}} + C_e^{(1),\delta} \right),
\end{align*}
and we denote
\[C_e^{(5),1+\delta} := \left(C_e^{(5),2}\right)^{\delta}\left(C_e^{(1),1}\right)^{1-\delta}.\]

We will simplify these constants further. Using the monotonicity of $C_{e}^{(5),r}$, we find
\begin{align*}
	C_e^{\mathring{R},1} &:= C\left(1+C_e^{(3),0} + C_e^{(3),\delta} + \left(1+ M_v + C_e^{(1),1}+C_e^{(5),1+\delta}\right){M_v} + C_e^{(4),0}+C_e^{(8),\delta} \right.\\
	&\quad\qquad + \sqrt{\bar{e}}C_e^{(1),1} + \left(C_e^{(7),2}\right)^{\delta}\left(\sqrt{\bar{e}}C_e^{(1),1}\right)^{1-\delta} + \eta + M_v^{1-\beta} + M_v^{1-\delta-2\alpha} + {M_q} + C_e^{(1),\delta}\\
	&\quad\qquad + {M_v^{2\delta}}+ C_e^{{\rm comp},2,1+\delta} + C_e^{{\rm comp},3,1+\delta} + C_e^{{\rm comp},4,1+\delta}\\
	&\quad\qquad  + \left( 1 + M_{v}^{2\delta} + C_{e}^{(1),\delta} + C_{e}^{(1),1} + C_{e}^{(5),1+\delta} \right)^{1-2{\alpha}-\delta} \\
	&\qquad\qquad \left. \cdot \left(1 + M_{v}^{2\delta} + C_e^{(1),1} + C_e^{(5),1+\delta} + C_e^{(5),2+\delta} \right)^{2{\alpha} + \delta}\right)\\
	&\leq C\left(1+C_e^{(3),0} + C_e^{(3),\delta} + \left(1+ M_v + C_e^{(1),1}+C_e^{(5),1+\delta}\right){M_v} + C_e^{(4),0}+C_e^{(8),\delta} \right.\\
	&\quad\qquad + \sqrt{\bar{e}}C_e^{(1),1} + \left(C_e^{(7),2}\right)^{\delta}\left(\sqrt{\bar{e}}C_e^{(1),1}\right)^{1-\delta}  + \eta + M_v^{1-\beta} + M_v^{1-\delta-2\alpha} + {M_q} + C_e^{(1),\delta} + {M_v^{2\delta}}\\
	&\quad\qquad + C_e^{(1),1} + C_e^{(5),2+\delta} + C_e^{(8),1+\delta}\\
	&\quad\qquad + \left(1+{M_v} + {M_v^{2\delta}} + C_e^{(1),1} + C_e^{(5),1+\delta}\right)\left(1+{M_v^{2\delta}}+ C_e^{(1),1} + C_e^{(5),2+\delta}\right)\\
	&\qquad\quad +\left(1+ C_e^{(1),\delta} + C_e^{(5),1+\delta}\right)\left(1+{M_v^{2\delta}}+C_e^{(1),1} + C_e^{(5),1+\delta}\right)\\
	&\quad\qquad \left. + {M_v}\left(C_e^{(5),1+\delta} + C_e^{(1),1}+\sqrt{\bar{e}}\right) + \sqrt{\bar{e}} \right).
\end{align*}
This implies
\begin{align*}
    C_e^{\mathring{R},1}&\leq C\left(1+C_e^{(1),0} + C_e^{(1),\delta}  + C_e^{(4),0}+C_e^{(8),\delta} \right.  + \sqrt{\bar{e}}C_e^{(1),1} + \left(C_e^{(7),2}\right)^{\delta}\left(\sqrt{\bar{e}}C_e^{(1),1}\right)^{1-\delta}\\
	&\quad + \eta + M_v^{1-\beta} + M_v^{1-\delta-2\alpha} + {M_q} + C_e^{(1),\delta} + C_e^{(1),1} + C_e^{(5),2+\delta} + C_e^{(8),1+\delta}\\
	&\quad + \left(1+C_e^{(1),1} + C_e^{(5),1+\delta}\right)\left(1+C_e^{(1),1} + C_e^{(5),2+\delta}\right)\\
	&\quad+\left(1+ C_e^{(1),\delta} + C_e^{(5),1+\delta}\right)\left(1+C_e^{(1),1} + C_e^{(5),1+\delta}\right)\\
	&\quad  + \sqrt{\bar{e}}   + {M_v^{2\delta}}\left(1+ {M_v^{2\delta}}+ C_e^{(1),1} + C_e^{(5),2+\delta} \right)\\
	&\quad + {M_v}\left(1+{M_v^{2\delta}}+ M_v + C_e^{(1),\delta} +C_e^{(1),1} +C_e^{(5),1+\delta} + C_e^{(5),2+\delta} \right)\\
	&\leq C\left(1+\sqrt{\bar{e}}+ (1+\sqrt{\bar{e}})C_e^{(1),1}  + C_e^{(4),0} + \left(C_e^{(5),2}\right)^{\delta}\left(\sqrt{\bar{e}}C_e^{(1),1}\right)^{1-\delta}+ \eta  + {M_q}\right. \\
	&\quad  + C_e^{(8),1+\delta} + M_{v}^{1-\beta} + M_{v}^{1-\delta-2\alpha}\\
	&\quad + \left(1+C_e^{(1),1} + C_e^{(5),1+\delta}\right)\left(1+C_e^{(1),1} + C_e^{(5),2+\delta}\right)\\
	&\quad +\left(1+ C_e^{(1),\delta} + C_e^{(5),1+\delta}\right)\left(1+C_e^{(1),1} + C_e^{(5),1+\delta}\right)\\
	&\quad  + {M_v^{2\delta}}\left(1+ {M_v^{2\delta}}+ C_e^{(5),2+\delta}+ C_e^{(1),1}\right)\\
	&\quad \left.+ {M_v}\left(1+C_e^{(1),1}+{M_v^{2\delta}} + M_v + C_e^{(5),2+\delta}\right) \right).
\end{align*}
We will need most of the estimates for the constants used in the proof of Lemma \ref{lem:CR0} above.

We further estimate 
\begin{align*}
	C_e^{(5),2+\delta} &\leq C_e^{(5),3} \leq C \sqrt{\bar{e}}\frac{\bar{e}}{\underline{e}}\left(1+ \bar{e} + \frac{\bar{e}}{1+\sqrt{\bar{e}}} \frac{1+\bar{e}}{\underline{e}}\right),
\end{align*}
which leads us to 
\[\max\left(C_e^{(5),1+\delta}, C_e^{(5),2+\delta}\right) \leq C\sqrt{\bar{e}}\frac{1+\bar{e}}{\underline{e}}\left(1+ \bar{e}+ \frac{1+\bar{e}}{1+\sqrt{\bar{e}}}\frac{\bar{e}}{\underline{e}}\right).\]
Collecting the above estimates, we find
\begin{align*}
C_e^{\mathring{R},1}&\leq C\left(1  + (1+\sqrt{\bar{e}})\sqrt{\bar{e}}\frac{1+\underline{e}}{\underline{e}}  + \sqrt{\bar{e}}\left(1+\frac{|e|_{C^1}}{\underline{e}}\right) + \sqrt{\bar{e}}^{2-\delta}\frac{1+\bar{e}}{\underline{e}}\left(1+ \bar{e}+ \frac{\bar{e}}{1+\sqrt{\bar{e}}}\right) +  \eta + \bar{e}\right. \\
	&\quad + \sqrt{\bar{e}}\frac{\bar{e}}{\underline{e}}\left(1+ \bar{e} + \frac{\bar{e}}{1+\sqrt{\bar{e}}} \frac{1+\bar{e}}{\underline{e}}\right) +  \sqrt{\bar{e}}\left(1+ \sqrt{\underline{e}} + \frac{|e|_{C^1}}{\underline{e}}\right)\left(1+ \bar{e} + \frac{1+\bar{e}}{1+\sqrt{\bar{e}}}\right)^2\left(\frac{1+\bar{e}}{\underline{e}}\right)^2 \left(\frac{\bar{e}}{\underline{e}}\right)^{3}\\
	&\quad + \left(\sqrt{\bar{e}}\frac{1+\underline{e}}{\underline{e}} + \sqrt{\bar{e}}\frac{1+\bar{e}}{\underline{e}}\left(1+ \bar{e}+ \frac{1+\bar{e}}{1+\sqrt{\bar{e}}}\frac{\bar{e}}{\underline{e}}\right)\right)^2 + M_{v}^{1-\beta} + M_{v}^{1-\delta-2\alpha}\\
	&\quad+\left(1+ \sqrt{\bar{e}}\frac{1+\underline{e}}{\underline{e}} + \sqrt{\bar{e}}\frac{1+\bar{e}}{\underline{e}}\left(1+ \bar{e}+ \frac{\bar{e}}{1+\sqrt{\bar{e}}}\right)\right)\left(\sqrt{\bar{e}}\frac{1+\underline{e}}{\underline{e}} + \sqrt{\bar{e}}\frac{1+\bar{e}}{\underline{e}}\left(1+ \bar{e}+ \frac{\bar{e}}{1+\sqrt{\bar{e}}}\right)\right)\\
	&\quad\left. + ({M_v^{2\delta}} +   {M_v})\left(1+ \sqrt{\bar{e}}+ {M_v^{2\delta}}+ M_{v} + \sqrt{\bar{e}}\frac{1+\underline{e}}{\underline{e}}+\sqrt{\bar{e}}\frac{\bar{e}}{\underline{e}}\left(1+ \bar{e} + \frac{\bar{e}}{1+\sqrt{\bar{e}}} \frac{1+\bar{e}}{\underline{e}}\right)\right)\right) \\
    &\leq C\left(1 + \bar{e}+ \sqrt{\bar{e}}+ (1+\sqrt{\bar{e}})\sqrt{\bar{e}}\frac{1+\underline{e}}{\underline{e}}  + \sqrt{\bar{e}}\left(1+\frac{|e|_{C^1}}{\underline{e}}\right) + \sqrt{\bar{e}}^{2-\delta}\frac{1+\bar{e}}{\underline{e}}\left(1+ \bar{e}+ \frac{\bar{e}}{1+\sqrt{\bar{e}}}\right)\right. \\
	&\quad + \sqrt{\bar{e}}\frac{\bar{e}}{\underline{e}}\left(1+ \bar{e} + \frac{\bar{e}}{1+\sqrt{\bar{e}}} \frac{1+\bar{e}}{\underline{e}}\right) +  \sqrt{\bar{e}}\left(1+ \sqrt{\underline{e}} + \frac{|e|_{C^1}}{\underline{e}}\right)\left(1+ \bar{e} + \frac{1+\bar{e}}{1+\sqrt{\bar{e}}}\right)^2\left(\frac{1+\bar{e}}{\underline{e}}\right)^2 \left(\frac{\bar{e}}{\underline{e}}\right)^{3} \\
	&\quad + \bar{e}\left(\frac{1+\bar{e}}{\underline{e}}\right)^2\left(1+ \bar{e}+ \frac{1+\bar{e}}{1+\sqrt{\bar{e}}}\frac{\bar{e}}{\underline{e}}\right)^2+ M_{v}^{1-\beta} + M_{v}^{1-\delta-2\alpha}\\
	&\quad+\left(1 + \sqrt{\bar{e}}\frac{1+\bar{e}}{\underline{e}}\left(1+ \bar{e}+ \frac{\bar{e}}{1+\sqrt{\bar{e}}}\right)\right)\sqrt{\bar{e}}\frac{1+\bar{e}}{\underline{e}}\left(1+ \bar{e}+ \frac{\bar{e}}{1+\sqrt{\bar{e}}}\right)\\
	&\quad\left. + ({M_v^{2\delta}} +   {M_v})\left(1+ \sqrt{\bar{e}}+ {M_v^{2\delta}}+ M_{v} + \sqrt{\bar{e}}\frac{1+\underline{e}}{\underline{e}}+\sqrt{\bar{e}}\frac{\bar{e}}{\underline{e}}\left(1+ \bar{e} + \frac{\bar{e}}{1+\sqrt{\bar{e}}} \frac{1+\bar{e}}{\underline{e}}\right)\right)\right),
\end{align*}
which we simplify further to get
\begin{align*}
    C_e^{\mathring{R},1}&\leq C\left(1 +  \sqrt{\bar{e}}\left(1+ \sqrt{\underline{e}} + \frac{|e|_{C^1}}{\underline{e}}\right) \left(1+\left(1+ \bar{e} + \frac{1+\bar{e}}{1+\sqrt{\bar{e}}}\right)^2\left(\frac{1+\bar{e}}{\underline{e}}\right)^2 \left(\frac{\bar{e}}{\underline{e}}\right)^{3}\right)  + M_{v}^{1-\delta-2\alpha}\right.\\
	&\quad + M_{v}^{1-\beta}+ \sqrt{\bar{e}}\frac{1+\bar{e}}{\underline{e}}\left(1+ \bar{e}+ \frac{1+\bar{e}}{1+\sqrt{\bar{e}}}\frac{\bar{e}}{\underline{e}}\right)\left(1+\sqrt{\bar{e}} + \sqrt{\bar{e}}^{1-\delta} + \sqrt{\bar{e}}\frac{1+\bar{e}}{\underline{e}}\left(1+ \bar{e}+ \frac{1+\bar{e}}{1+\sqrt{\bar{e}}}\frac{\bar{e}}{\underline{e}}\right)\right)\\
	&\quad\left. + ({M_v^{2\delta}} +   {M_v})\left(1+ \sqrt{\bar{e}}+ {M_v^{2\delta}}+ M_{v} + \sqrt{\bar{e}}\frac{1+\underline{e}}{\underline{e}}+\sqrt{\bar{e}}\frac{\bar{e}}{\underline{e}}\left(1+ \bar{e} + \frac{\bar{e}}{1+\sqrt{\bar{e}}} \frac{1+\bar{e}}{\underline{e}}\right)\right)\right) \\
    &\leq C\left(1  + \sqrt{\bar{e}}\left(1+ \sqrt{\underline{e}}+\frac{|e|_{C^1}}{\underline{e}}\right)\left(1+ \bar{e} + \frac{1+\bar{e}}{1+\sqrt{\bar{e}}}\right)^2\left(\frac{1+\bar{e}}{\underline{e}}\right)^2 \left(\frac{\bar{e}}{\underline{e}}\right)^{3}+ M_{v}^{1-\delta-2\alpha} + M_{v}^{1-\beta} \right.\\
	&\quad\qquad + \sqrt{\bar{e}}\frac{1+\bar{e}}{\underline{e}}\left(1+ \bar{e}+ \frac{1+\bar{e}}{1+\sqrt{\bar{e}}}\frac{\bar{e}}{\underline{e}}\right)\left(1+\sqrt{\bar{e}} + \sqrt{\bar{e}}^{1-\delta} + \sqrt{\bar{e}}\frac{1+\bar{e}}{\underline{e}}\left(1+ \bar{e}+ \frac{1+\bar{e}}{1+\sqrt{\bar{e}}}\frac{\bar{e}}{\underline{e}}\right)\right)\\
	&\quad\left. + ({M_v^{2\delta}} +   {M_v})\left(1+ \sqrt{\bar{e}}+ {M_v^{2\delta}}+ M_{v} + \sqrt{\bar{e}}\frac{1+\underline{e}}{\underline{e}}+\sqrt{\bar{e}}\frac{\bar{e}}{\underline{e}}\left(1+ \bar{e} + \frac{\bar{e}}{1+\sqrt{\bar{e}}} \frac{1+\bar{e}}{\underline{e}}\right)\right)\right) \\
    &\leq C\left(1  + \sqrt{\bar{e}}\left(1+ \sqrt{\underline{e}}+\frac{|e|_{C^1}}{\underline{e}}\right)\left(1+ \bar{e} + \frac{1+\bar{e}}{1+\sqrt{\bar{e}}}\right)^2\left(\frac{1+\bar{e}}{\underline{e}}\right)^2 \left( \frac{\bar{e}}{\underline{e}}\right)^{3} + M_{v}^{1-\delta-2\alpha} + M_{v}^{1-\beta} \right. \\
	&\quad + \sqrt{\bar{e}}\left(\frac{1+\bar{e}}{\underline{e}}\right)^2 \frac{\bar{e}}{\underline{e}}\left(1+ \bar{e}+ \frac{1+\bar{e}}{1+\sqrt{\bar{e}}}\frac{\bar{e}}{\underline{e}}\right)\left(1+\sqrt{\bar{e}} + \sqrt{\bar{e}}^{1-\delta} + \sqrt{\bar{e}}\left(1+ \bar{e}+ \frac{1+\bar{e}}{1+\sqrt{\bar{e}}}\right)\right)\\
	&\quad\left. + ({M_v^{2\delta}} +   {M_v})\left(1+ \sqrt{\bar{e}}+ {M_v^{2\delta}}+ M_{v} + \sqrt{\bar{e}}\frac{1+\underline{e}}{\underline{e}}+\sqrt{\bar{e}}\frac{\bar{e}}{\underline{e}}\left(1+ \bar{e} + \frac{\bar{e}}{1+\sqrt{\bar{e}}} \frac{1+\bar{e}}{\underline{e}}\right)\right)\right),
\end{align*}
and finally
\begin{align*}
    C_e^{\mathring{R},1}&\leq C\left(\frac{1+\bar{e}}{\underline{e}}\right)^2 \frac{\bar{e}}{\underline{e}}\cdot \left(1 + \sqrt{\bar{e}}\left(1+ \sqrt{\underline{e}}+\frac{|e|_{C^1}}{\underline{e}}\right)\left(1+ \bar{e} + \frac{1+\bar{e}}{1+\sqrt{\bar{e}}}\right)^2  + M_{v}^{1-\beta} + M_{v}^{1-\delta-2\alpha} \right.\\
	&\quad +  \sqrt{\bar{e}}\left(1+ \bar{e}+ \frac{1+\bar{e}}{1+\sqrt{\bar{e}}}\frac{\bar{e}}{\underline{e}}\right)\left(1+\sqrt{\bar{e}} + \sqrt{\bar{e}}^{1-\delta} + \sqrt{\bar{e}}\left(1+ \bar{e}+ \frac{1+\bar{e}}{1+\sqrt{\bar{e}}}\right)\right)\\
	&\quad\left. + ({M_v^{2\delta}} +   {M_v})\left(1+ {M_v^{2\delta}}+ M_v + \sqrt{\bar{e}} +\sqrt{\bar{e}}\left(1+ \bar{e} + \frac{1+\bar{e}}{1+\sqrt{\bar{e}}} \right)\right)\right). \qedhere
\end{align*}
\end{proof}

\begin{proof}[Proof of Lemma \ref{lem:Cv1}]
We now turn to the constant for the velocity:
\begin{align*}
	C_e^{v,1} &:= C\left(1+ {M_v^{2\delta}}+ C_e^{(1),1} + C_e^{(5),1+\delta} + C_e^{(5),2+\delta}\right)\\
	&\leq C\left(1+ {M_v^{2\delta}}+ \sqrt{\bar{e}}\frac{1+\underline{e}}{\underline{e}} + \sqrt{\bar{e}}\frac{1+\bar{e}}{\underline{e}}\left(1+ \bar{e}+ \frac{1+\bar{e}}{1+\sqrt{\bar{e}}}\frac{\bar{e}}{\underline{e}}\right)\right)\\
	&\leq C\frac{1+\bar{e}}{\underline{e}}\frac{\bar{e}}{\underline{e}}\left(1+ {M_v^{2\delta}}+  \sqrt{\bar{e}}\left(1+ \bar{e}+ \frac{1+\bar{e}}{1+\sqrt{\bar{e}}}\right)\right). \qedhere
\end{align*}
\end{proof}

\begin{proof}[Proof of Lemma \ref{lem:Cq1}]
Recall that
\begin{align*}
	C_e^{q,1} &:= C\left(\sqrt{\bar{e}}\left(C_e^{\partial_t w_o} + C_e^{(1),\delta} + C_e^{(5),1+\delta}\right) +\eta + 1\right),\\
	C_e^{\partial_t w_o} &:= C\left(C_e^{{\rm trans},1} + C_e^{(8),1} + C_e^{(3),1} + {M_v}C_e^{(1),1} + \sqrt{\bar{e}}\right),\\
	C_e^{{\rm trans},1} &:= C\left(C_e^{(3),0} + C_e^{(3),\delta} + \left(C_e^{(1),1}+C_e^{(5),1+\delta}\right){M_v} + C_e^{(4),0} +C_e^{(8),\delta}\right).
\end{align*}
Using the estimates we have obtained thus far, we find
\begin{align*}
	C_e^{{\rm trans},1} &\leq C\left(\sqrt{\bar{e}} + \sqrt{\bar{e}}\frac{1+\underline{e}}{\underline{e}} + \left(\sqrt{\bar{e}}\frac{1+\underline{e}}{\underline{e}} + \sqrt{\bar{e}}\frac{1+\bar{e}}{\underline{e}}\left(1+ \bar{e}+ \frac{\bar{e}}{1+\sqrt{\bar{e}}}\right)\right){M_v} \right.\\
	&\qquad\left.+ \sqrt{\bar{e}}\left(1+ \frac{|e|_{C^1}}{\underline{e}}\right) + \sqrt{\bar{e}}\left(1+ \frac{|e|_{C^1}}{\underline{e}}\right)\left(1 + \frac{1+\bar{e}}{1+\sqrt{\bar{e}}}\right)^2 \frac{1+\bar{e}}{\underline{e}}\right)\\
	&\leq C\sqrt{\bar{e}}\frac{1+\bar{e}}{\underline{e}}\left(1 + \bar{e} + \frac{1+\bar{e}}{1+\sqrt{\bar{e}}}\right)\left( 1 + {M_v}  + \left(1+ \frac{|e|_{C^1}}{\underline{e}}\right)\left(1 + \frac{1+\bar{e}}{1+\sqrt{\bar{e}}}\right) \right),
\end{align*}
which gives
\begin{align*}
	C_e^{\partial_t w_o} &\leq C\left(\sqrt{\bar{e}}\frac{1+\bar{e}}{\underline{e}}\left(1 + \bar{e} + \frac{1+\bar{e}}{1+\sqrt{\bar{e}}}\right)\left( 1 + {M_v}  + \left(1+ \frac{|e|_{C^1}}{\underline{e}}\right)\left(1 + \frac{1+\bar{e}}{1+\sqrt{\bar{e}}}\right) \right) \right.\\
	&\quad\qquad\left.+ \sqrt{\bar{e}}\left(1+ \frac{|e|_{C^1}}{\underline{e}}\right)\left(1 + \frac{1+\bar{e}}{1+\sqrt{\bar{e}}}\right)^2 \frac{1+\bar{e}}{\underline{e}} + (1+{M_v})\sqrt{\bar{e}}\frac{1+\underline{e}}{\underline{e}} + \sqrt{\bar{e}}\right)\\
	&\leq C\sqrt{\bar{e}}\frac{1+\bar{e}}{\underline{e}}\left(1 + \bar{e} + \frac{1+\bar{e}}{1+\sqrt{\bar{e}}}\right)\left( 1 + {M_v}  + \left(1+ \frac{|e|_{C^1}}{\underline{e}}\right)\left(1 + \frac{1+\bar{e}}{1+\sqrt{\bar{e}}}\right) \right),
\end{align*}
and therefore
\begin{align*}
	C_e^{q,1} &\leq C\left(1 + \bar{e}\left(1 + \bar{e} + \frac{1+\bar{e}}{1+\sqrt{\bar{e}}}\right)\left( 1 + {M_v}  + \left(1+ \frac{|e|_{C^1}}{\underline{e}}\right)\left(1 + \frac{1+\bar{e}}{1+\sqrt{\bar{e}}}\right) \right)\right)\frac{1+\bar{e}}{\underline{e}}. \qedhere
\end{align*}
\end{proof}

\begin{proof}[Proof of Lemma \ref{lem:choice_A}]
We estimate
\begin{align*}
	&\max\left(C_e^{\mathring{R},1}, C_e^{v,1}, C_e^{q,1}\right)\\
	&\leq C\left(\frac{1+\bar{e}}{\underline{e}}\right)^2 \left( \frac{\bar{e}}{\underline{e}} \right)^{3} \max\left\{1 + \sqrt{\bar{e}}\left(1+\sqrt{\underline{e}} + \frac{|e|_{C^1}}{\underline{e}}\right)\left(1+ \bar{e} + \frac{1+\bar{e}}{1+\sqrt{\bar{e}}}\right)^2\right.+ M_{v}^{1-\delta-2\alpha} \\
	&\hspace{1.5cm} + M_{v}^{1-\beta} + \sqrt{\bar{e}}\left(1+ \bar{e}+ \frac{1+\bar{e}}{1+\sqrt{\bar{e}}}\frac{\bar{e}}{\underline{e}}\right)\left(1 + \sqrt{\bar{e}}^{1-\delta} + \sqrt{\bar{e}}\left(1+ \bar{e}+ \frac{1+\bar{e}}{1+\sqrt{\bar{e}}}\right)\right)  \\
	&\hspace{1.5cm} +({M_v^{2\delta}} +   {M_v})\left(1+ {M_v^{2\delta}}+ M_v +\sqrt{\bar{e}}\left(1+ \bar{e} + \frac{1+\bar{e}}{1+\sqrt{\bar{e}}} \right)\right),\\
	&\quad\qquad 1+ {M_v^{2\delta}}+ \sqrt{\bar{e}}\left(1+ \bar{e}+ \frac{1+\bar{e}}{1+\sqrt{\bar{e}}}\right),\\
	&\quad\qquad\left. 1 + \bar{e}\left(1 + \bar{e} + \frac{1+\bar{e}}{1+\sqrt{\bar{e}}}\right)\left({M_v}  + \left(1+ \frac{|e|_{C^1}}{\underline{e}}\right)\left(1 + \frac{1+\bar{e}}{1+\sqrt{\bar{e}}}\right) \right)\right\},
\end{align*}
which we further simplify via
\begin{align*}
    &\max\left(C_e^{\mathring{R},1}, C_e^{v,1}, C_e^{q,1}\right)\\
	&\leq C\left(\frac{1+\bar{e}}{\underline{e}}\right)^2\left( \frac{\bar{e}}{\underline{e}} \right)^{3} \max\left\{1 + \sqrt{\bar{e}}\left(1+\sqrt{\underline{e}}+\frac{|e|_{C^1}}{\underline{e}}\right)\left(1+ \bar{e} + \frac{1+\bar{e}}{1+\sqrt{\bar{e}}}\right)^2 + M_{v}^{1-\delta-2\alpha} \right. \\
	&\hspace{1.5cm} + M_{v}^{1-\beta}+ \sqrt{\bar{e}}\left(1+ \bar{e}+ \frac{1+\bar{e}}{1+\sqrt{\bar{e}}}\frac{\bar{e}}{\underline{e}}\right)\left(1+  \sqrt{\bar{e}}^{1-\delta} + \sqrt{\bar{e}}\left(1+ \bar{e}+ \frac{1+\bar{e}}{1+\sqrt{\bar{e}}}\right)\right)\\
	&\hspace{1.5cm} +({M_v^{2\delta}} +   {M_v})\left(1+ {M_v^{2\delta}}+ M_v +\sqrt{\bar{e}}\left(1+ \bar{e} + \frac{1+\bar{e}}{1+\sqrt{\bar{e}}} \right)\right),\\
	&\quad\qquad\left. 1 + \bar{e}\left(1 + \bar{e} + \frac{1+\bar{e}}{1+\sqrt{\bar{e}}}\right)\left(  {M_v}  + \left(1+ \frac{|e|_{C^1}}{\underline{e}}\right)\left(1 + \frac{1+\bar{e}}{1+\sqrt{\bar{e}}}\right) \right)\right\}\\
	&\leq \left(\frac{1+\bar{e}}{\underline{e}}\right)^2 \left( \frac{\bar{e}}{\underline{e}} \right)^{3}\\
	&\quad \cdot C \left\{1 + (\bar{e} + \sqrt{\bar{e}})\left(1+\sqrt{\underline{e}}+\frac{|e|_{C^1}}{\underline{e}}\right)\left(1+ \bar{e} + \frac{1+\bar{e}}{1+\sqrt{\bar{e}}}\right)^2 + M_{v}^{1-\delta-2\alpha} + M_{v}^{1-\beta}\right.\\
	&\qquad + \sqrt{\bar{e}}\left(1+ \bar{e}+ \frac{1+\bar{e}}{1+\sqrt{\bar{e}}}\frac{\bar{e}}{\underline{e}}\right)\left(1+ \sqrt{\bar{e}}^{1-\delta} + \sqrt{\bar{e}}\left(1+ M_{v} + \bar{e}+ \frac{1+\bar{e}}{1+\sqrt{\bar{e}}}\right)\right)\\
	&\qquad +({M_v^{2\delta}} +   {M_v})\left(1+ {M_v^{2\delta}}+ M_v  +\sqrt{\bar{e}}\left(1+ \bar{e} + \frac{1+\bar{e}}{1+\sqrt{\bar{e}}} \right)\right\}
\end{align*}
We slightly enlarge the latter factor to maximize other expressions later in the proof by defining
\begin{align*}
     \tilde{A}_{e} &:= C \left\{1 + (\bar{e} + \sqrt{\bar{e}})\left(1+\sqrt{\underline{e}}+\frac{|e|_{C^1}}{\underline{e}}\right)\left(1+ \bar{e} + \frac{1+\bar{e}}{1+\sqrt{\bar{e}}}\right)^2 + M_{v}^{1-\delta-2\alpha} + M_{v}^{1-\beta}\right.\\
	&\qquad + \sqrt{\bar{e}}\left(1+ \bar{e}+ \frac{1+\bar{e}}{1+\sqrt{\bar{e}}}\frac{\bar{e}}{\underline{e}}\right)\left(1+ \sqrt{\bar{e}}^{1-\delta} + \sqrt{\bar{e}}\left(1+ M_{v} + \bar{e}+ \frac{1+\bar{e}}{1+\sqrt{\bar{e}}}\right)\right)\\
	&\qquad +(1 + {M_v^{2\delta}} + M_v^{1-\delta} + {M_v} + \sqrt{\bar{e}})\left(1+ {M_v^{2\delta}}+ M_v  +\sqrt{\bar{e}}\left(1+ \bar{e} + \frac{1+\bar{e}}{1+\sqrt{\bar{e}}} \right)\right\}.
\end{align*}
Then we set $A := \tilde{A}_e \left(\frac{1+\bar{e}}{\underline{e}}\right)^2 \left( \frac{\bar{e}}{\underline{e}} \right)^{3}$.
\end{proof}

\appendix
\section{An estimate for the fractional Laplacian} 

Recall that by $\mathcal{C}^{\beta}$ we denote the classical H\"older spaces, whereas by $C^{\beta}$ we denote the H\"older--Zygmund spaces. The following theorem provides estimates for the fractional Laplacian in $\mathcal{C}^{\beta}$.
\begin{theorem}[Interaction with H\"older spaces, Theorem B.1 of \cite{DeRosa_one-third}, Thm. 1.4 of \cite{RS16}] \label{thm:fract_Lap_Holder}
    Let $\alpha, \delta > 0$ and $\beta \geq 0$ such that $2\alpha + \beta + \delta < 1$, and let $f \colon \T^{3} \to \R^{3}$. If $f \in \mathcal{C}_{x}^{2\alpha+\beta+\delta}$, then $(-\D)^{\alpha}f \in \mathcal{C}_{x}^{\beta}$. Moreover, there exists a constant $C = C(\delta) > 0$ such that
    \begin{equation}
        \| (-\D)^{\alpha} f \|_{\mathcal{C}_{x}^{\beta}} \leq C [f]_{2\alpha + \beta + \delta}.
    \end{equation}
\end{theorem}

\section{Besov Spaces} 

Let us collect some crucial lemmas regarding Besov spaces.

\begin{lem}[Negative Regularity Paraproduct estimates, \cite{MW17}, Prop. A.7 (Fourth case)]\label{lem:paraproduct}
    Let $\alpha < 0 < \beta$ \textbf{and} $\alpha + \beta > 0$. Let $p_{1},p_{2},p, q \in [1,\infty]$ such that
    \begin{align*}
        \frac{1}{p} = \frac{1}{p_{1}} + \frac{1}{p_{2}}
    \end{align*}
    Then the mapping 
        \begin{align*}
            (f,g) \mapsto f \cdot g
        \end{align*}
    between continuous functions extends to a bilinear map from $B_{p_{1},q}^{\alpha} \times B_{p_{2},q}^{\beta}$ to $B_{p,q}^{\alpha},$
  i.e.
    \begin{align*}
        \| f \cdot g \|_{ B_{p,q}^{\alpha}} \leq C \| f \|_{B_{p_{1},q}^{\alpha}} \| g \|_{B_{p_{2},q}^{\beta}}.
    \end{align*}
\end{lem}
    
\begin{lem}[\cite{BCD11}, Prop. 2.78]\label{lem:Fourier_multipliers}
        Let $m \in \R$ and $h$ be a Fourier multiplier of class $S^{m}$. Define $h(D)u := \mathcal{F}^{-1}[ h \mathcal{F}u ]$, where $\mathcal{F}$ denotes the Fourier transform and $\mathcal{F}^{-1}$ the inverse Fourier transform. Then, for every $s \in \R$, $1 \leq p,r \leq \infty$, $h(D)$ is continuous from $B_{p,r}^{s} \to B_{p,r}^{s-m}$. 
\end{lem}

The next lemma is similar. We state and prove it because technically speaking, Lemma \ref{lem:Fourier_multipliers} requires that the symbol of a Fourier multiplier be smooth everywhere, which is not true for $f(\xi) = |\xi|^{\alpha}$ in $\xi = 0$.
\begin{lem}[Continuity of the fractional Laplacian]
    Let $s \in \R$ and $p,r \in [1,\infty]$. Then the fractional Laplacian is a continuous operator
    \begin{align*}
        (-\D)^{\alpha} \colon B_{p,r}^{s} \to B_{p,r}^{s-2\alpha}.
    \end{align*}
\end{lem}
\begin{proof}
    We follow the proof of \cite[Proposition 2.78]{BCD11}. Recall that $\D_{j}$ denotes the $j$-th Littlewood--Paley block.
 We need to show that
    \begin{align*}
        \forall j \geq -1: \quad 2^{j(s-2\alpha) } \| \D_{j} (-\D)^{\alpha} u \|_{L^{p}} \leq C 2^{js} \| \D_{j} u \|_{L^{p}}.
    \end{align*}    
    First, consider $j = -1$. Let $\theta \in C_{c}^{\infty}(\R^{d})$ such that $0 \leq \theta \leq 1$ and $\theta = 1$ on $\supp \chi$. Note that
    \begin{align*}
        \D_{-1}(-\D)^{\alpha} u = (\chi |\cdot|^{2\alpha} \hat{u})^{\vee} = \left( (|\cdot|^{2\alpha} \theta) \chi \hat{u} \right)^{\vee} = (|\cdot|^{2\alpha}\theta)(D) (\D_{-1} u) = \mathcal{F}^{-1}(|\cdot|^{2\alpha}\theta) * (\D_{-1} u).
    \end{align*}
    Note that since $\supp \theta \subset B_{R}(0)$, for some $R \geq 1$,
    \begin{align*}
        \mathcal{F}^{-1}(|\cdot|^{2\alpha}\theta)(x) = \sum_{k \in \Z^{d}} |k|^{2\alpha} \theta(k) e^{ik\cdot x} = \sum_{|k| \leq R} |k|^{2\alpha} \theta(k) e^{ik\cdot x},
    \end{align*}
    and hence
    \begin{align*}
        \| \mathcal{F}^{-1}(|\cdot|^{2\alpha}\theta) \|_{L^{1}} \leq |\T^{d}| \sum_{|k| \leq R} |k|^{2\alpha} \theta(k) \leq |\T^{d}| R^{d+2\alpha} \leq |\T^{d}| R^{d+2}  < \infty.
    \end{align*}
    Now, by Young's convolution inequality, we find
    \begin{align*}
        \| \D_{-1}(-\D)^{\alpha} u \|_{L^{p}} \leq \| \mathcal{F}^{-1}(|\cdot|^{2\alpha}\theta) \|_{L^{1}} \| \D_{-1} u \|_{L^{p}} \leq C(d) \| \D_{-1} u \|_{L^{p}},
    \end{align*}
    and we find that 
    \begin{align*}
        2^{(-1)(s - 2\alpha)} \| \D_{-1}(-\D)^{\alpha} u \|_{L^{p}} \leq 2^{(-1)s} 2^{ 2\alpha} C(d) \| \D_{-1} u \|_{L^{p}} \leq 2C(d) 2^{(-1)s} \| \D_{-1} u \|_{L^{p}}.
    \end{align*}
    Now consider $j \geq 0$. Similarly as before, we have, with $f(\xi) = |\xi|^{2\alpha}$,
    \begin{align*}
        \D_{j} (-\D)^{\alpha}u = (\varphi_{j} |\cdot|^{2\alpha} \hat{u})^{\vee} = f(D) \D_{j}u.
    \end{align*}
    Since the function $f(D)$ (recalling the notation for Fourier multipliers from Lemma \ref{lem:Fourier_multipliers}) is applied to the function $\D_{j}u$ which has Fourier support in an annulus $2^{j} \mathcal{C}$, we are in a position to apply \cite[Lemma 2.2]{BCD11}. Let us quickly check that $f$ satisfies the assumptions of the lemma with $m = 2\alpha$. We have
    \begin{align*}
        \partial_{i} f(\xi) &= 2\alpha |\xi|^{2\alpha-2} \xi_{i} \\
        \partial_{k} \partial_{i} f(\xi) &= - 4\alpha(1-\alpha) |\xi|^{2\alpha-4} \xi_{i}\xi_{k} + 2\alpha |\xi|^{2\alpha-2} \delta_{k,i} ,
    \end{align*}
    etc. This implies that
    \begin{align*}
        |\partial_{i}f(\xi)| &\leq 2\alpha |\xi|^{2\alpha-2} \left( \xi_{i}^{2} \right)^{1/2} \leq 2\alpha |\xi|^{2\alpha-2} |\xi| = 2\alpha |\xi|^{2\alpha-1}, \\
        |\partial_{k} \partial_{i} f(\xi)| &\leq 4\alpha(1-\alpha) |\xi|^{2\alpha-4}|\xi|^{2} + 2\alpha |\xi|^{2\alpha-2} = (4\alpha(1-\alpha) + 2\alpha) |\xi|^{2\alpha -2}.
    \end{align*}
    Continuing inductively, we see that for any $\beta \in \N_{0}^{d}$,
    \begin{align*}
        |\partial^{\beta} f(\xi)| \leq C_{\beta} |\xi|^{2\alpha - |\beta|}.
    \end{align*}
    Furthermore, $f \in C^{\infty}(\R^{d} \backslash \{ 0 \})$. Then by \cite[Lemma 2.2]{BCD11}, we get that
    \begin{align*}
        2^{j(s-2\alpha)}\| \D_{j} (-\D)^{\alpha}u \|_{L^{p}} = 2^{j(s-2\alpha)} \| f(D) \D_{j} u \|_{L^{p}} \leq C2^{j(s-2\alpha)} 2^{j 2\alpha} \| \D_{j} u \|_{L^{p}} = 2^{js} \| \D_{j} u \|_{L^{p}},
    \end{align*}
    which completes the proof.
\end{proof}

%


\begin{lem}[Increasing the $s$-index]\label{lem:sincr}
    Let $s' > s$, $p,r \in [1,\infty]$. Then there is a constant $C = C(|s'-s|)$ such that
    \begin{equation}\label{eq:s_embedding}
        \| u \|_{B^{s}_{p,r}} \leq C \| u \|_{B^{s'}_{p,r}}.
    \end{equation}
\end{lem}
\begin{proof}
    Without loss of generality we only consider the case $r < \infty$. The case $r = \infty$ works the same way.
    \begin{align*}
        \| u \|_{B^{s}_{p,r}}^{r} &\overset{\text{def.}}{=} \sum_{j=-1}^{\infty} 2^{jrs} \| \D_{j} u \|_{L^{p}}^{r} =  \sum_{j=-1}^{\infty} 2^{jrs'} 2^{jr(s-s')} \| \D_{j} u \|_{L^{p}}^{r} \\
        &=  2^{-rs'} 2^{-r(s-s')} \| \D_{-1} u \|_{L^{p}}^{r} +  \sum_{j=0}^{\infty}   2^{jrs'} \underbrace{2^{jr(s-s')}}_{\leq 1} \| \D_{j} u \|_{L^{p}}^{r} \\
        &\leq 2^{r|s-s'|} 2^{-rs'} \| \D_{-1} u \|_{L^{p}}^{r} +  \sum_{j=0}^{\infty}   2^{jrs'}  \| \D_{j} u \|_{L^{p}}^{r} \\
        &\leq (\max(2^{|s-s'|},1))^{r} \left( \sum_{j=-1}^{\infty} 2^{jrs} \| \D_{j} u \|_{L^{p}}^{r} \right) \\
        &= 2^{|s-s'|r} \left( \sum_{j=-1}^{\infty} 2^{jrs} \| \D_{j} u \|_{L^{p}}^{r} \right) \\
        \Rightarrow\quad  \| u \|_{B^{s}_{p,r}} &\leq 2^{|s-s'|} \| u \|_{B^{s'}_{p,r}},
    \end{align*}
    which proves the claim with $C = 2^{|s-s'|}$.
\end{proof}

\begin{lem}[Increasing the $q$-index]\label{lem_qincr}
    Let $p, q_{0}, q_{1} \in [1,\infty]$, $s \in \R$, $\ve > 0$ and $q_{0} < q_{1}$. Then we have the continuous embedding
    $$
        B_{p,q_{1}}^{s+\ve} \hookrightarrow B_{p,q_{0}}^{s},
    $$
    i.e. there is a constant $C = C_{\ve,q_{0},q_{1}}$ such that
    $
        \| u \|_{B_{p,q_{0}}^{s}} \leq C \| u \|_{B_{p,q_{1}}^{s+\ve}}.
    $
\end{lem}
\begin{proof}
    First, let $q_{1} < \infty$. Then
    \begin{align*}
        \| u \|_{B_{p,q_{0}}^{s}} &\overset{\text{def}}{=} \left( \sum_{j \geq -1}^{\infty} 2^{jsq_{0}} \| \D_{j} u \|_{L^{p}}^{q_{0}} \right)^{1/q_{0}} = \left( \sum_{j \geq -1}^{\infty} 2^{-j\ve q_{0}} \  2^{j(s+\ve)q_{0}} \| \D_{j} u \|_{L^{p}}^{q_{0}} \right)^{1/q_{0}} \\
        &\overset{\text{H\"older}}{\leq} \left( \sum_{j \geq -1}^{\infty} 2^{j(s+\ve)q_{0} \frac{q_{1}}{q_{0}}} \| \D_{j} u \|_{L^{p}}^{q_{0} \frac{q_{1}}{q_{0}}} \right)^{\frac{q_{0}}{q_{1}} \cdot \frac{1}{q_{0}}} \left( \sum_{j \geq -1}^{\infty} 2^{-j\ve \frac{q_{0}q_{1}}{q_{1}- q_{0}} } \right)^{\frac{1}{q_{0}}(1 - \frac{q_{0}}{q_{1}} )} \\
        &= \left( \sum_{j \geq -1}^{\infty} 2^{j(s+\ve)q_{1}} \| \D_{j} u \|_{L^{p}}^{q_{1}} \right)^{\frac{1}{q_{1}}} \left( \sum_{j \geq -1}^{\infty} 2^{-j\ve \frac{q_{0}q_{1}}{q_{1}- q_{0}} } \right)^{ \frac{q_{1} - q_{0}}{q_{0} q_{1}}} \\
        &= C_{\ve,q_{0},q_{1}} \| u \|_{B_{p,q_{1}}^{s+\ve}}.
    \end{align*}
    Now, let $q_{1} = \infty$. Then
    \begin{align*}
        \| u \|_{B_{p,q_{0}}^{s}} &\overset{\text{def}}{=} \left( \sum_{j \geq -1}^{\infty} 2^{jsq_{0}} \| \D_{j} u \|_{L^{p}}^{q_{0}} \right)^{1/q_{0}} = \left( \sum_{j \geq -1}^{\infty} 2^{-j\ve q_{0}} \  2^{j(s+\ve)q_{0}} \| \D_{j} u \|_{L^{p}}^{q_{0}} \right)^{1/q_{0}} \\
        &\leq \left( \sum_{j \geq -1}^{\infty} 2^{-j\ve q_{0}} \left( \sup_{j}  2^{j(s+\ve)q_{0}} \| \D_{j} u \|_{L^{p}} \right)^{q_{0}} \right)^{1/q_{0}} \\
        &=  \left( \sum_{j \geq -1}^{\infty} 2^{-j\ve q_{0}} \right)^{1/q_{0}} \| u \|_{B^{s+\ve}_{p,\infty}} = C_{\ve,q_{0}} \| u \|_{B^{s+\ve}_{p,\infty}}. \qedhere
    \end{align*}
\end{proof}

\begin{lem}[Increasing the $p$-index]\label{lem_pincr}
    Let $p_{0},p_{1},q \in [1,\infty]$, $s \in \R$ such that $p_{0} < p_{1}$. Then there is a constant $C = C_{p_{0},p_{1}}$ such that
    \begin{align*}
        \| u \|_{B_{p_{0},q}^{s}} \leq C \| u \|_{B_{p_{1},q}^{s}},
    \end{align*}
    i.e. we have the continuous embedding
        $B_{p_{1},q}^{s} \hookrightarrow B_{p_{0},q}^{s}.$
\end{lem}
\begin{proof}
    First let $p_{1} < \infty$. Then we have
    \begin{align*}
         \| u \|_{B_{p_{0},q}^{s}} &\overset{\text{def}}{=} \left( \sum_{j\geq -1} 2^{jsq} \|1 \cdot \D_{j} u \|_{L^{p_{0}}}^{q} \right)^{1/q} \\
         &\overset{\text{H\"older}}{\leq} \left( \sum_{j\geq -1} 2^{jsq} \left(\int 1 dx \right)^{\frac{q}{p_{0}} (1 - \frac{p_{0}}{p_{1}})} \left( \int | \D_{j} u |^{p_{0} \cdot \frac{p_{1}}{p_{0}}} dx \right)^{\frac{q}{p_{0}}\cdot \frac{p_{0}}{p_{1}}} \right)^{1/q} \\
         &= \left( \sum_{j\geq -1} 2^{jsq} \left(\int_{\T^{3}} 1 dx \right)^{\frac{q}{p_{0}} (1 - \frac{p_{0}}{p_{1}})} \left( \int_{\T^{3}} | \D_{j} u |^{p_{0} \cdot \frac{p_{1}}{p_{0}}} dx \right)^{\frac{q}{p_{0}}\cdot \frac{p_{0}}{p_{1}}} \right)^{1/q} \\
         &= |\T^{3}|^{\frac{p_{1}-p_{0}}{p_{0}p_{1}}} \left( \sum_{j\geq -1} 2^{jsq} \| \D_{j} u \|_{L^{p_{1}}}^{q} \right)^{1/q} = C_{p_{0},p_{1}} \| u \|_{B_{p_{1},q}^{s}}.
    \end{align*}
    Now let $p_{1} = \infty$. Then we find that
    \begin{align*}
        \| u \|_{B_{p_{0},q}^{s}} &\overset{\text{def}}{=} \left( \sum_{j\geq -1} 2^{jsq} \|1 \cdot \D_{j} u \|_{L^{p_{0}}}^{q} \right)^{1/q} \\
        &\overset{\text{H\"older}}{\leq}  \left( \sum_{j\geq -1} 2^{jsq} \left(\int 1 dx \right)^{\frac{q}{p_{0}}} \| \D_{j} u  \|_{L^{\infty}}^{q} \right)^{1/q} \\
        &= |\T^{3}|^{\frac{1}{p_{0}}} \left( \sum_{j\geq -1} 2^{jsq} \| \D_{j} u \|_{L^{\infty}}^{q} \right)^{1/q} = C_{p_{0},\infty} \| u \|_{B_{\infty,q}^{s}}. \qedhere
    \end{align*}
\end{proof}

Finally, we shall need the following Sobolev-type embedding theorem for Besov spaces. 
\begin{lem}[Sobolev-type embedding, \cite{BCD11}, Proposition 2.71]\label{lem:Sobolev_embedding}
    Let $1 \leq p_{1} \leq p_{2} \leq \infty$ and $1 \leq q_{1} \leq q_{2} \leq \infty$. Then, for any $s \in \R$, we have the continuous embedding
    \begin{align*}
        B_{p_{1},q_{1}}^{s} \hookrightarrow B_{p_{2},q_{2}}^{s - d\left( \frac{1}{p_{1}} - \frac{1}{p_{2}} \right)}, \quad \text{i.e.} \quad \| u \|_{B_{p_{2},q_{2}}^{s - d\left( \frac{1}{p_{1}} - \frac{1}{p_{2}} \right)}} \leq C \| u \|_{ B_{p_{1},q_{1}}^{s}},
    \end{align*}
    where $C$ is independent of $s, p_{i}$ and $q_{i}$, $i = 1,2$.
\end{lem}

\begin{lem}[Lemma C.1 of \cite{HLP23}]\label{lem:Cr_chain}
    Let $\phi = \phi_{n}$, $n \in \N$. For any $\delta \in [0,1)$ and $r \in \N$, $r \leq \kappa$ there exists a constant $C = C(\delta, r)$ such that the following holds almost surely for every $L \in \N$, $L \geq 1$. For every $f$ on $(-\infty, \mft_{L}] \times \T^{3}$ of class $C_{\leq \mft_{L}} C_{x}^{r+\delta}$ and every fixed $t \leq \mft_{L}$, we have
    \begin{align*}
        \| f \circ \phi \|_{C_{x}^{r+\delta}} &\leq C L^{r+\delta} \| f \|_{C_{x}^{r+\delta}}, \\
        \| f \circ \phi^{-1} \|_{C_{x}^{r+\delta}} &\leq C L^{r+\delta} \| f \|_{C_{x}^{r+\delta}}, 
    \end{align*}
    where $f \circ \phi$ denotes the map $(-\infty, \mft_{L}] \times \T^{3} \ni (t,x) \mapsto f(t, \phi(t,x))$, and similarly for $f \circ \phi^{-1}$.
\end{lem}

\begin{lem}[Lemma C.2 of \cite{HLP23}]\label{lem:B-alpha_chain}
    Let $\phi = \phi_{n}$, $n \in \N$. For every $s \in (0,2)$, $s \neq 1$ there exists $C = C(s)$ such that for every $L \in \N$, $L \geq 1$ and every continuous function $f$ on $\T^{3} \times (-\infty, \mft_{L}]$ it holds that
    \begin{align*}
        \| f \circ \phi^{\pm 1} \|_{B^{-s}_{\infty,\infty}} \leq CL^{4-s} \| f \|_{B^{-s}_{\infty,\infty}}.
    \end{align*}
\end{lem}
Recall the following estimates from \cite[Equ. $(C.1)$]{HLP23}:
\begin{lem}\label{lem:B11_chain}
    Let $\phi = \phi_{n}$, $n \in \N$. For every $s \in (0,1)$, there exists a $C = C(s)$ such that for every $L \in \N$ and every continuous function $f$ on $\T^{3} \times (-\infty, \mathfrak{t}_{L}]$ the following hold:
    \begin{align*}
        \| f \circ \phi \|_{B_{1,1}^{s}} &\leq CL^{4 - s} \| f \|_{B_{1,1}^{s}}, \\
        \| f \circ \phi^{-1} \|_{B_{1,1}^{s}} &\leq CL^{4 - s} \| f \|_{B_{1,1}^{s}}.
    \end{align*}
\end{lem}

\begin{lem}[Lemma C.4 of \cite{HLP23}]
        Let $\phi = \phi_{n}$, $n \in \N$. For any $\delta \in (0,1)$ and any $r \in \N$, $r+2 \leq \kappa$, there exists a constant $C = C(\delta,r)$ such that for every $v \colon \mathbb{T}^{3} \times (-\infty, \mft_{L}] \to \R^{3}$, $A \colon \mathbb{T}^{3} \times (-\infty, \mft_{L}] \to \R^{3 \times 3}$ almost surely for every $L \in \N$, $L \geq 1$, $t \leq \mft_{L}$
        \begin{align}
            \label{eq:QSchauder}
            \| \mathcal{Q}_{\phi} v \|_{C_{x}^{r+\delta}} &\leq CL^{2r + 2\delta} \| v \|_{C_{x}^{r+\delta}},
            \\
            \label{eq:antidiv-orderminusone}
            \| \mathcal{R}_{\phi} v \|_{C_{x}^{\delta}} &\leq C L^{5+4\delta} \| v \|_{B_{\infty,\infty}^{\delta-1}}, \\
            \label{eq:RdivCr} \| \mathcal{R}_{\phi} ( \mathrm{div}_{\phi} A) \|_{C_{x}^{r+\delta}} &\leq CL^{2r+2\delta} \| A \|_{C_{x}^{r+\delta}}, \\
            \label{eq:antidiv-orderminusone-r} \| \mathcal{R}_{\phi} v \|_{C_{x}^{r+1+\delta}} &\leq CL^{2r+1+2\delta} \| v \|_{C_{x}^{r+\delta}}.
        \end{align}
\end{lem}

\begin{lem}[Proposition C.5 of \cite{HLP23}] \label{lem:stat-phase}
 Let $\phi = \phi_{n}$, $n \in \N$. Let $a \in C^{\infty}(\T^{3})$ be a smooth function and let $k \in \Z^{3} \backslash \{ 0 \}$ and $\lambda \geq 1$ be fixed. Define $f(x) = a(x) e^{i \lambda k \cdot x}$.
\begin{enumerate}[(i)]
  \item For any $r \in \N$, we have almost surely for every $t \in \R$
  \begin{equation}\label{eq:SPL_integral}
        \left| \int_{\T^{3}} f(\phi(t,x)) dx \right| \leq \frac{[a]_{C_{x}^{r}}}{\lambda^{r}}.
  \end{equation}
  \item For any $\delta \in (0,1)$, $r \in \N$ such that $r+1 \leq \kappa$, we have almost surely for every $L \in \N$, $L \geq 1$ and $t \in (-\infty, \tau_{L}]$
\begin{align}
    \label{eq_SPL} \| \mcR^{\phi} (f \circ \phi) \|_{C^{\delta}_{x}} &\leq C L^{\delta} \left( \lambda^{\delta-1} \| a \|_{C^{0}_{x}} + \lambda^{\delta - r} [ a ]_{C^{r}_{x}} + \lambda^{-r} [a]_{C^{r+\delta}_{x}} \right), \\
    \label{eq_SPL_Q} \| \mcR^{\phi} Q^{\phi} (f \circ \phi) \|_{C^{\delta}_{x}} &\leq C L^{\delta} \left( \lambda^{\delta-1} \| a \|_{C^{0}_{x}} + \lambda^{\delta - r} [ a ]_{C^{r}_{x}} + \lambda^{-r} [a]_{C^{r+\delta}_{x}} \right),
\end{align}
where $C = C(\delta, r)$.
\end{enumerate}
\end{lem}

\begin{lem}[Lemma C.6 of \cite{HLP23}] \label{lem:Besov_convolution}
    Let $s \in \R$, $p_{1}, p_{2} \in (1,\infty)$ such that $\frac{1}{p_{1}} + \frac{1}{p_{2}} = 1$. Then, for every $\delta \in (0,1)$, there exists a $C$ such that for every $f \in B_{p_{1},\infty}^{s+\delta}(\T^{3})$ with mean zero and $g \in L^{p_{2}}(\T^{3})$, the spatial convolution on the torus $f *_{\T^{3}} g$ belongs to $B_{\infty,\infty}^{s}(\T^{3})$ and 
    \begin{align*}
        \| f *_{\T^{3}} g \|_{B_{\infty,\infty}^{s}} \leq C \| f \|_{B_{p_{1},\infty}^{s+\delta}} \| g \|_{L^{p_{2}}}. 
    \end{align*}
\end{lem}

\begin{lem}[Lemma C.7 of \cite{HLP23}] \label{lem:scaling}
    Let $\ell = 2^{-N}$ for some positive $N \in \N$. Let $f$ be a smooth function on $\R^{d}$ with $\supp f \subset (0,2\pi)^{d}$, and denote
    \begin{align*}
        f^{0} &:= f - \frac{1}{(2\pi)^{d}} \int_{[0,2\pi]^{d}} f, \\
        f_{\ell} &:= \ell^{-d} f(\cdot / \ell), \\
        f_{\ell}^{0} &:=  f_{\ell} - \frac{1}{(2\pi)^{d}} \int_{[0,2\pi]^{d}} f_{\ell} =  f_{\ell} - \frac{1}{(2\pi)^{d}} \int_{[0,2\pi]^{d}} f.
    \end{align*}
    Extend $f, f^{0}, f_{\ell}$ and $f_{\ell}^{0}$ periodically on $\T^{d}$. Then for every $s \in \R$, $p \in [1,\infty]$ it holds that
    \begin{align*}
        \| f_{\ell}^{0} \|_{B_{p,\infty}^{s}} = \ell^{d/p - d - s} \| f^{0} \|_{B_{p,\infty}^{s}} \leq \ell^{d/p - d - s} \| f \|_{B_{p,\infty}^{s}}. 
    \end{align*}

\end{lem}

\begin{lem}\label{lem:mollification}
    Let $\theta \in \R$. Let $\chi_{\ell}(t,x) = \psi_{\ell}(t) \varphi_{\ell}(x)$ be a space-time mollifier. Let $f \colon \R \times \T^{d} \to \R^{d}$ be a smooth function. Then for any $\kappa,\beta \in (0,1]$ and $t \geq 0$
    \begin{align*}
        \| f *_{t,x}  \chi_{\ell} - f \|_{B_{\infty,\infty}^{\theta}}(t) \leq C \ell^{\kappa} \| f \|_{C_{t}^{0}B_{\infty,\infty}^{\theta+\kappa}} + \ell^{\beta} \| f \|_{C^{\beta}_{t} B_{\infty,\infty}^{\theta}}.
    \end{align*}
\end{lem}
\begin{proof}
We find that
    \begin{align*}
    &\| (f * \chi_{\ell})(t) - f(t) \|_{B^{\theta}_{\infty,\infty}} = \left\| \int_{\R} \int_{\T^{3}} \left[f(t-s,\cdot-y) - f(t,\cdot) \right] \varphi_{\ell}(y) \psi_{\ell}(s) ~dyds \right\|_{B^{\theta}_{\infty,\infty}} \\
    &\overset{\text{def.}}{=} \sup_{j \geq -1} 2^{j\theta} \left\| \D_{j} \int_{\R} \int_{\T^{3}} \left[f(t-s,\cdot-y) - f(t,\cdot) \right] \varphi_{\ell}(y) \psi_{\ell}(s) ~dyds \right\|_{L^{\infty}} \\
    &= \sup_{j \geq -1} 2^{j\theta} \left\| \int_{\R} \int_{\T^{3}} \left[(\D_{j}  f)(t-s,\cdot-y) - (\D_{j} f)(t,\cdot) \right] \varphi_{\ell}(y) \psi_{\ell}(s) ~dyds \right\|_{L^{\infty}} \\
    &\leq \sup_{j \geq -1} 2^{j\theta} \sup_{x \in \T^{3}} \left| \int_{\R} \int_{\T^{3}} \left[(\D_{j}  f)(t-s,x-y) - (\D_{j}  f)(t-s,x) \right] \varphi_{\ell}(y) \psi_{\ell}(s) ~dyds \right| \\
    &\qquad + \sup_{j \geq -1} 2^{j\theta} \sup_{x \in \T^{3}} \left| \int_{\R} \int_{\T^{3}} \left[(\D_{j}  f)(t-s,x) - (\D_{j} f)(t,x) \right] \varphi_{\ell}(y) \psi_{\ell}(s) ~dyds \right| \\
    &=: A + B.
\end{align*}
We first estimate $A$:
\begin{align*}
    A &= \sup_{j \geq -1} 2^{j\theta} \sup_{x \in \T^{3}} \left| \int_{\R} \int_{\T^{3}} \left[(\D_{j}  f)(t-s,x-y) - (\D_{j}  f)(t-s,x) \right] \varphi_{\ell}(y) \psi_{\ell}(s) ~dyds \right| \\
    &\leq \sup_{j \geq -1} 2^{j\theta} \int_{\R} \int_{\T^{3}}  [ \D_{j}f (t-s,\cdot) ]_{C^{\kappa}_{x}} |y|^{\kappa} \varphi_{\ell}(y) \psi_{\ell}(s) ~dyds \\
    &\leq \ell^{\kappa} \sup_{j \geq -1} \int_{\R} \int_{\T^{3}} \sup_{r \in (-\infty, t]} 2^{j\theta}   [ \D_{j}f(r,\cdot) ]_{C^{\kappa}_{x}} \varphi_{\ell}(y) \psi_{\ell}(s) ~dyds \\
    &\leq \ell^{\kappa} \sup_{j \geq -1} \sup_{r \in (-\infty, t]} \underbrace{2^{j\theta}   [ \D_{j}f(r,\cdot) ]_{C^{\kappa}_{x}}}_{\leq \sup_{j \geq -1} \ldots} \\
    &\leq \ell^{\kappa}  \sup_{r \in (-\infty, t]} \sup_{j \geq -1} 2^{j\theta}   [ \D_{j}f(r,\cdot) ]_{C^{\kappa}_{x}} \\
    &\overset{\text{interpolation}}{\leq} 2 \ell^{\kappa}  \sup_{r \in (-\infty, t]} \sup_{j \geq -1} 2^{j\theta}  \left(  \| D_{x} \D_{j}f(r,\cdot) \|_{L^{\infty}_{x}} \right)^{\kappa} \| \D_{j}f(r,\cdot) \|_{L^{\infty}_{x}}^{1-\kappa}  \\
    &\overset{\text{\cite[Lemma 2.1]{BCD11}}}{\leq} 2\ell^{\kappa}  \sup_{r \in (-\infty, t]} \sup_{j \geq -1} 2^{j\theta} \left( 2^{j} C  \| \D_{j}f(r,\cdot) \|_{L^{\infty}_{x}} \right)^{\kappa} \| \D_{j}f(r,\cdot) \|_{L^{\infty}_{x}}^{1-\kappa} \\
    &= C \ell^{\kappa}  \sup_{r \in (-\infty, t]} \sup_{j \geq -1} 2^{j(\theta+\kappa)} \| \D_{j}f(r,\cdot) \|_{L^{\infty}_{x}}  \\
    &\overset{\text{def.}}{=} C\ell^{\kappa} \| f \|_{C_{t}^{0} B_{\infty,\infty}^{\theta+\kappa}}.
\end{align*}
Now let us estimate $B$: 
\begin{align*}
    B &= \sup_{j \geq -1} 2^{j\theta} \sup_{x \in \T^{3}} \left| \int_{\R}  \frac{(\D_{j}  v_{n})(t-s,x) - (\D_{j} v_{n})(t,x) }{|s|^{\beta}} |s|^{\beta}  \psi_{\ell}(s) ~ds \right| \\
    &\leq \int_{\R} \sup_{j \geq -1} 2^{j\theta} \sup_{x \in \T^{3}} \left| \frac{(\D_{j}  v_{n})(t-s,x) - (\D_{j} v_{n})(t,x) }{|s|^{\beta}} \right| |s|^{\beta}  \psi_{\ell}(s) ~ds \\
    &= \int_{\R} \frac{1}{|s|^{\beta}} \sup_{j \geq -1}  2^{j\theta}  \left\| (\D_{j}  v_{n})(t-s,x) - (\D_{j} v_{n})(t,x) \right\|_{L^{\infty}} |s|^{\beta}  \psi_{\ell}(s) ~ds \\
    &= \int_{\R} \frac{1}{|s|^{\beta}}  \left\|   v_{n}(t-s,\cdot) -  v_{n}(t,\cdot) \right\|_{B_{\infty,\infty}^{\theta}} |s|^{\beta}  \psi_{\ell}(s) ~ds \\
    &\leq  \int_{\R} \left( \sup_{\substack{u,r \in (-\infty,t] \\ u \neq r}} \frac{1}{|u-r|^{\beta}}  \left\|   v_{n}(u,\cdot) -  v_{n}(r,\cdot) \right\|_{B_{\infty,\infty}^{\theta}} \right) |s|^{\beta}  \psi_{\ell}(s) ~ds \\
    &\overset{\text{def.}}{=} \int_{\R} \| v_{n} \|_{C_{t}^{\beta}B_{\infty,\infty}^{\theta}} |s|^{\beta} \psi_{\ell}(s) ~ds \\
    &\leq \ell^{\beta} \| v_{n} \|_{C_{t}^{\beta}B_{\infty,\infty}^{\theta}}. \qedhere
\end{align*}
\end{proof}

\paragraph*{Acknowledgements.}
T.L. has received funding from the European Research Council (ERC) under the EU-HORIZON EUROPE ERC-2021-ADG research and innovation programme (project „Noise in Fluids“, grant agreement no. 101053472). A.S. gratefully acknowledges the support by the German Research Foundation (DFG) through the Walter Benjamin Programme, Project number 507913792. 

\bibliographystyle{plain}
\bibliography{bib-collection}
 \end{document}